\documentclass[twoside,11pt]{article}

\usepackage{blindtext}

\usepackage[abbrvbib, preprint]{jmlr2e}

\usepackage{lastpage}
\jmlrheading{xx}{xxxx}{1-\pageref{LastPage}}{3/25; Revised xx/xx}{xx/xx}{21-0000}{Author One and Author Two}

\usepackage{amsmath}
\usepackage{amssymb}
\usepackage{graphicx}
\usepackage{color}
\usepackage{ifpdf}
\usepackage{url}
\usepackage{algorithm}
\usepackage[usenames,dvipsnames]{xcolor}
\usepackage{paralist}
\usepackage{mathtools}

\usepackage{algorithm}
\usepackage{algpseudocode}
\usepackage{algorithmicx}
\usepackage[ruled,vlined, linesnumbered, algo2e]{algorithm2e}
\usepackage[ruled,vlined]{algorithm2e}
\usepackage{makecell}

\newtheorem{assumption}{Assumption}[section]
\newcommand{\Espace}{\mathbb{R}^p}
\newcommand{\Id}{\mathbb{I}}

\newcommand{\R}{\mathbb{R}}

\newcommand{\Rext}{\R\cup\{+\infty\}}

\newcommand{\set}[1]{\left\{#1\right\}}
\newcommand{\sets}[1]{\{#1\}}

\newcommand{\norms}[1]{\Vert#1\Vert}

\newcommand{\prox}{\mathrm{prox}}

\newcommand{\dom}[1]{\mathrm{dom}(#1)}

\newcommand{\zero}[1]{{\boldsymbol{0}}}

\newcommand{\Expsn}[2]{\mathbb{E}_{#1}\big[#2\big]}
\newcommand{\Exp}[1]{\mathbb{E}\left[#1\right]}
\newcommand{\Expn}[1]{\mathbb{E}\big[#1\big]}

\newcommand{\Expsb}[2]{\mathbb{E}_{#1}\big[#2\big]}

\newcommand{\zer}[1]{\mathrm{zer}(#1)}

\newcommand{\gra}[1]{\mathrm{gra}(#1)}

\newcommand{\mcal}[1]{\mathcal{#1}}

\newcommand{\Hc}{\mathcal{H}}

\newcommand{\Xc}{\mathcal{X}}

\newcommand{\Sc}{\mathcal{S}}

\newcommand{\Lc}{\mathcal{L}}
\newcommand{\Qc}{\mathcal{Q}}

\newcommand{\Tc}{\mathcal{T}}

\newcommand{\Fb}{\mathbf{F}}

\newcommand{\Fc}{\mathcal{F}}
\newcommand{\Uc}{\mathcal{U}}
\newcommand{\Vc}{\mathcal{V}}

\newcommand{\Ec}{\mathcal{E}}
\newcommand{\Nc}{\mathcal{N}}
\newcommand{\Pc}{\mathcal{P}}

\newcommand{\Wc}{\mathcal{W}}

\newcommand{\iprods}[1]{\langle #1\rangle}

\newcommand{\BigO}[1]{\mathcal{O}\left(#1\right)}
\newcommand{\BigOs}[1]{\mathcal{O}\big(#1\big)}
\newcommand{\SmallOs}[1]{o\big(#1\big)}
\newcommand{\SmallO}[1]{o\left(#1\right)}

\newcommand{\mblue}[1]{\textcolor{blue}{#1}}
\newcommand{\myblue}[1]{{\color{blue}{#1}}}

\newcommand{\mbf}[1]{\mathbf{#1}}

\newcommand{\beforesec}{\vspace{-1.25ex}}
\newcommand{\aftersec}{\vspace{-1ex}}
\newcommand{\beforesubsec}{\vspace{-1.25ex}}
\newcommand{\aftersubsec}{\vspace{-1.0ex}}
\newcommand{\beforesubsubsec}{\vspace{-1ex}}
\newcommand{\aftersubsubsec}{\vspace{-0.5ex}}

\newcommand{\myeq}[2]{%\vspace{-0.5ex}
\begin{equation}\label{#1}
{#2}
\end{equation}
}
\newcommand{\myeqn}[1]{%\vspace{-0.5ex}
\begin{equation*}
{#1}
\end{equation*}
}

\newcommand{\rv}[1]{{\color{black}#1}}
\newcommand{\rvs}[1]{{\color{black}#1}}
\newcommand{\rvt}[1]{{\color{black}#1}}

\title{Accelerated Variance-Reduced Fixed-Point-Based Methods for Generalized Equations: Better Convergence Guarantees}

\author{Quoc Tran-Dinh \vspace{0.25ex}\\
\newline {Department of Statistics and Operations Research}\\
\newline The University of North Carolina at Chapel Hill\\
318 Hanes Hall, UNC-Chapel Hill, NC 27599-3260.\\
\newline \textit{Email:} \url{quoctd@email.unc.edu}.}

\ShortHeadings{Variance-Reduced Fast Operator Splitting Methods}{Q. Tran-Dinh} % and N. Nguyen-Trung}
\firstpageno{1}

\begin{document}

\title{Variance-Reduced Fast Operator Splitting Methods for Generalized Equations}

\author{\name Quoc Tran-Dinh \email quoctd@email.unc.edu \\
%	\AND Nghia Nguyen-Trung$^{2}$ \email nghiant@email.unc.edu \\	
       \addr Department of Statistics and Operations Research \\
       The University of North Carolina at Chapel Hill\\
       318 Hanes Hall, CB \#3260, NC 27599-3260
%       \AND
%       \name Author Two \email two@cs.berkeley.edu \\
%       \addr Division of Computer Science\\
%       University of California\\
%       Berkeley, CA 94720-1776, USA
       }

\editor{My editor}

\maketitle

\begin{abstract}%   <- trailing '%' for backward compatibility of .sty file
We develop two variance-reduced fast operator splitting methods to approximate solutions of a class of generalized equations, covering fundamental problems such as \rvs{minimization}, minimax problems, and variational inequalities as special cases.
Our approach integrates recent advances in accelerated operator splitting and fixed-point methods, co-hypomonotonicity, and variance reduction.
First, we introduce a class of variance-reduced estimators and establish their variance-reduction bounds.
This class includes both unbiased and biased instances and comprises common estimators as special cases, including SVRG, SAGA, SARAH, and Hybrid-SGD.
Second, we design a novel accelerated variance-reduced forward-backward splitting (FBS) method using these estimators to solve generalized equations in both finite-sum and expectation settings.
Our algorithm achieves both $\BigOs{1/k^2}$ and $\SmallOs{1/k^2}$ convergence rates on the expected squared norm $\Expn{\norms{G_{\lambda}x^k}^2}$ of the FBS residual $G_{\lambda}$, where $k$ is the iteration counter.
Additionally, we establish almost sure convergence rates and the almost sure convergence of iterates to a solution of the underlying generalized equation.
Unlike existing stochastic operator splitting algorithms, our methods accommodate co-hypomonotone operators, which can include nonmonotone problems arising in recent applications.
\rvs{Third, we specify our method for each concrete estimator mentioned above and derive the corresponding oracle complexity, demonstrating that these variants achieve the best-known oracle complexity bounds without requiring additional enhancement techniques.}
Fourth, we develop a variance-reduced fast backward-forward splitting (BFS) method, which attains similar convergence results and oracle complexity bounds as our FBS-based algorithm.
Finally, we validate our results through numerical experiments and compare their performance with existing methods.
\end{abstract}


%%%%%%%%%%%%%%%%%%%%%%%%%%%%%%%%
%%%% 1. Introduction.
%%%%%%%%%%%%%%%%%%%%%%%%%%%%%%%%
%\beforesec
\section{Introduction}\label{sec:intro}
\aftersec
The generalized equation (GE), also known as the nonlinear inclusion, serves as a versatile framework with broad applications across various domains, including operations research, engineering, mechanics, economics, statistics, and machine learning, see, e.g., \citep{Bauschke2011,Facchinei2003,phelps2009convex,reginaset2008,ryu2022large,ryu2016primer}.
\rvs{The recent surge in modern machine learning and distributionally robust optimization has reinvigorated interest in minimax problems, which are special cases of GE.
These minimax models, particularly in the context of generative adversary, imitation learning, reinforcement learning, and distributionally robust optimization, can be effectively modeled and solved using the GE framework, see, e.g., \citep{arjovsky2017wasserstein,faghri2021bridging,goodfellow2014generative,kuhn2025distributionally,madry2018towards,namkoong2016stochastic,shi2022minimax,swamy2021moments,yu2022fast}.}
This paper develops two novel classes of stochastic accelerated operator splitting algorithms with variance reduction specifically designed for solving GEs.

\vspace{0.5ex}
\noindent\textbf{$\mathrm{(a)}$~Problem statement.} 
In this work, we focus on the following generalized equation (also known as an inclusion):
\myeq{eq:NI}{
\textrm{Find $x^{\star}\in \R^p$ such that:}~ \rv{0 \in \Phi{x}^{\star} := Fx^{\star} + Tx^{\star}},
\tag{GE}
}
where $F : \R^p \to \R^p$ is a single-valued mapping,  $T : \R^p \rightrightarrows 2^{\R^p}$ is a possibly multi-valued mapping, and \rv{$\Phi := F + T$}.
We consider two different settings of \eqref{eq:NI} as follows.
\begin{compactitem}
\item[($\mathsf{F}$)] [\textbf{Finite-sum setting}] $F$ is a large finite-sum of the form:
\vspace{-0.5ex}
\myeq{eq:finite_sum_form}{
Fx  =  \frac{1}{n}\sum_{i=1}^nF_ix,
\tag{$\mathsf{F}$}
\vspace{-0.5ex}
}
where $F_i : \R^p \to \R^p$ for $i \in [n] := \sets{1, \cdots, n}$ and $n$ is often sufficiently large.
\item[($\mathsf{E}$)] [\textbf{Expectation setting}] 
$F$ is equipped with an unbiased stochastic oracle $\mbf{F}(\cdot, \xi)$, where $\xi$ is a random variable defined on a given probability space $(\Omega, \mathbb{P}, \Sigma)$, i.e., for any $x\in\dom{F}$, we have
\myeq{eq:expectation_form}{
Fx = \Expsn{\xi}{ \mbf{F}(x, \xi)}.
\tag{$\mathsf{E}$}
}
\end{compactitem}
Note that the finite-sum setting \eqref{eq:finite_sum_form} can be viewed as a special case of the expectation setting \eqref{eq:expectation_form} by choosing $\mbf{F}(x, \xi) = \frac{1}{n\mbf{p}_i}F_ix$ for $\mbf{p}_i := \mathbb{P}(\xi = i) \in (0, 1)$.
However, since our algorithms have different oracle complexity bounds on each setting, we treat them separately.

\rvs{The mapping $T$ in \eqref{eq:NI} is possibly multi-valued and maximally $\rho$-co-hypomonotone (see Subsection~\ref{subsec:preliminary_defs} for the definition) as stated in Assumption~\ref{as:A1}(iii) below.}
%\rvs{\todo{Discussion the form of $T$.}}

%%%%%%%%%%%%%%%%%%%%%%%%%%%%%%%%%%%%%%%%%%%%%%%%
%%%  2.2. Fundamental Assumptions.
%%%%%%%%%%%%%%%%%%%%%%%%%%%%%%%%%%%%%%%%%%%%%%%% 
%\beforesubsec
%%\subsection{Fundamental Assumptions}\label{subsec:assumptions}
%\aftersubsec
\vspace{0.5ex}
\noindent\textbf{$\mathrm{(b)}$~Fundamental assumptions.} 
To develop our algorithms for solving \eqref{eq:NI}, we require the following assumptions on \eqref{eq:NI}.

%%% Assumption A.1.
\begin{assumption}\label{as:A1}
The generalized equation \eqref{eq:NI} satisfies the following conditions:
\begin{compactitem}
\item[$\mathrm{(i)}$] $\zer{\Phi} := \set{x^{\star} \in \R^p : 0 \in \Phi{x}^{\star} } \neq\emptyset$ $($i.e., there exists a solution $x^{\star}$ of \eqref{eq:NI}$)$.
\item[$\mathrm{(ii)}$] $($\textbf{Bounded variance}$)$ 
For the expectation setting \eqref{eq:expectation_form}, there exists $\sigma \geq 0$ such that $\Expsn{\xi}{\norms{\Fb(x, \xi) - Fx}^2 } \leq \sigma^2$ for all $x \in \dom{F}$.
\item[$\mathrm{(iii)}$] $($\textbf{Maximal $\rho$-co-hypomonotonicity}$)$ $T$ is maximally $\rho$-co-hypomonotone.
\end{compactitem}
\end{assumption}
%%%
Assumption~\ref{as:A1}(i) and Assumption~\ref{as:A1}(ii) are standard.
While Assumption~\ref{as:A1}(i) makes sure that \eqref{eq:NI} is solvable, \rv{Assumption~\ref{as:A1}(ii) has been widely used in various stochastic methods.
It was also modified and generalized in different ways, see, e.g., \citep{beznosikov2023stochastic,gorbunov2020unified}.
We use this assumption to derive our oracle complexities in the sequel that depend on $\sigma^2$ and a mega batch-size $n_k$ in the expectation setting  \eqref{eq:expectation_form}.
}

Assumption~\ref{as:A1}(iii) covers the case $T$ is maximally monotone, but also covers a class of nonmonotone operators as shown in Subsection \ref{subsec:assumptions_discussion} below with concrete examples.
If $T$ is maximally monotone, then it already encompasses the normal cone of a nonempty, closed, and convex set and the subdifferential of a proper, closed, and convex function as special cases.
Consequently, under Assumption~\ref{as:A1}, \eqref{eq:NI} includes constrained convex minimization and convex-concave minimax problems, and monotone [mixed] variational inequality problems (VIPs) as special cases.

%%% Assumption A1.
\begin{assumption}\label{as:A2}
The mapping $F$ in \eqref{eq:NI} satisfies one of the following assumptions.
\begin{compactitem}
\item[$(\mathsf{F})$] $[$\textbf{The finite-sum setting}$]$ 
For the finite-sum setting \eqref{eq:finite_sum_form}, $F$ is $\frac{1}{L}$-average co-coercive, i.e., for all $x, y \in\dom{F}$, there exists $L > 0$ such that 
\myeq{eq:F_averaged_cocoerciveness}{
\begin{array}{lcl}
\iprods{Fx - Fy, x - y} \geq \tfrac{1}{n L} \sum_{i=1}^n\norms{F_ix - F_iy}^2.
\end{array}
}
\item[$(\mathsf{E})$] $[$\textbf{The expectation setting}$]$ 
For the expectation setting \eqref{eq:expectation_form}, $F$ is $\frac{1}{L}$-co-coercive in expectation, i.e., for all $x, y \in \dom{F}$, there exists $L > 0$ such that:
\myeq{eq:F_expected_cocoerciveness}{
\hspace{-0.5ex}
\begin{array}{lcl}
\iprods{Fx - Fy, x - y} = \Expsn{\xi}{\iprods{\Fb(x,\xi) - \Fb(y, \xi), x - y}} \geq \tfrac{1}{L}\Expsn{\xi}{ \norms{ \Fb(x,\xi) - \Fb(y, \xi) }^2 }.
\end{array}
\hspace{-1ex}
}
\end{compactitem}
\end{assumption}
The average \rv{co-coercivity} \eqref{eq:F_averaged_cocoerciveness} is generally stronger than the co-coercivity of $F$ since we have $\frac{1}{n} \sum_{i=1}^n\norms{F_ix - F_iy}^2 \geq \norms{Fx - Fy}^2$ by \rvs{Jensen's} inequality.
Similarly, the \rv{co-coercivity} in expectation \eqref{eq:F_expected_cocoerciveness} is generally stronger than the co-coercivity of $F$ since $\Expsn{\xi}{ \norms{ \Fb(x,\xi) - \Fb(y, \xi) }^2 } \geq \norms{ \Expsn{\xi}{ \Fb(x,\xi) - \Fb(y,\xi)}}^2 = \norms{Fx - Fy}^2$ by Jensen's inequality. 
Both cases lead to $\iprods{Fx - Fy, x - y} \geq \frac{1}{L}\norms{Fx - Fy}^2$, i.e., $F$ is $\frac{1}{L}$-co-coercive.
Consequently, we get $\norms{Fx - Fy} \leq L\norms{x - y}$, showing that $F$ is $L$-Lipschitz continuous.
See Subsection \ref{subsec:assumptions_discussion} for further discussion and its connection to the gradient of a smooth and convex function.

\vspace{0.5ex}
\noindent\textbf{$\mathrm{(c)}$~Motivating examples.}
\ref{eq:NI} looks simple, but it is sufficiently general to cover various models across disciplines.
We recall some important special cases of \eqref{eq:NI} here.
We also refer to \citet{davis2022variance,peng2016arock,ryu2016primer} for additional examples.

%\rvs{This is overloaded -- Need to re-organize}
%%%% (i) Composite convex optimization.
\vspace{0.5ex}
\noindent\textit{$\mathrm{(i)}$~\textbf{Composite minimization.}}
Consider the following composite minimization problem:
\vspace{-0.5ex}
\begin{equation}\label{eq:opt_prob}
\min_{x \in \R^p} \big\{ \phi(x) := f(x) + g(x) \big\},
\tag{OP}
\vspace{-0.5ex}
\end{equation}
where $f : \R^p \to \R$ is convex and $L$-smooth (i.e., $\nabla{f}$ is $L$-Lipschitz continuous) and $g : \R^p\to\Rext$ is proper, closed, and not necessarily convex.

\rv{Let $\nabla{f}$ be the gradient of $f$ and  $\partial{g}$ be the [abstract] subdifferential of $g$, see \citep{bauschke2020generalized}.
Then, under appropriate regularity assumptions \citep{Rockafellar2004}, the optimality condition of \eqref{eq:opt_prob} is
\begin{equation}\label{eq:opt_cond_of_opt_prob}
0 \in \nabla{f}(x^{\star}) + \partial{g}(x^{\star}).
\end{equation}
If $g$ is convex, then $x^{\star}$ solves \eqref{eq:opt_cond_of_opt_prob} \rv{if and only if} it solves \eqref{eq:opt_prob}.
Clearly, \eqref{eq:opt_prob} is a special case of \eqref{eq:NI} with $F := \nabla{f}$ and $T := \partial{g}$.
We can also verify Assumptions~\ref{as:A1} and \ref{as:A2} for this special case.
For instance, in the finite-sum setting \eqref{eq:finite_sum_form}, if $f$ is $L$-average smooth, then $\nabla{f}$ satisfies Assumption~\ref{as:A2}.
If $g$ is proper, closed, and convex, then Assumption~\ref{as:A1}(iii) is automatically satisfied with $\rho = 0$.

Problem~\eqref{eq:opt_prob} covers many representative applications in machine learning and data science~\citep{Bottou2018,sra2012optimization,wright2017optimization}.
As a concrete example, consider the case where $f(x) := \frac{1}{n} \sum_{i=1}^n\ell(\iprods{Z_i, x}; y_i)$ represents an empirical loss associated with a dataset $\sets{(Z_i, y_i)}_{i=1}^n$, and $g(x) := \tau R(x)$ is a regularizer used to promote desirable structures in the solution $x$ (e.g., sparsity via an $\ell_1$-norm or a nonconvex SCAD regularizer).
In this form, \eqref{eq:opt_prob} captures many statistical learning problems such as linear regression, logistic regression, and support vector machines; see, e.g., \citep{friedman2001elements}.
}
%\rvs{\todo{Add references for each case, revise this to avoid repetition.}}

%%%% (ii) Convex-concave minimax problems.
\vspace{0.5ex}
\noindent\textit{$\mathrm{(ii)}$~\textbf{Minimax problem.}}
Consider the following minimax problem:
\begin{equation}\label{eq:minimax_prob}
\min_{u\in\R^{p_1}}\max_{v \in \R^{p_2}}\Big\{ \Lc(u, v) := f(u) + \Hc(u, v) - g^{*}(v) \Big\},
\tag{MP}
\end{equation}
where $f : \R^{p_1} \to\Rext$ and $g^{*} : \R^{p_2}\to\Rext$ are proper, closed, and not necessarily convex, and $\Hc : \R^{p_1}\times\R^{p_2}\to\R$ is a given jointly differentiable and convex-concave function.

\rv{Under appropriate regularity conditions  \citep{Bauschke2011}, the optimality condition of \eqref{eq:minimax_prob} becomes
\begin{equation}\label{eq:minimax_opt_cond}
0 \in \begin{bmatrix} \nabla_u{\Hc}(u^{\star}, v^{\star})  \\  -\nabla_v{\Hc}(u^{\star}, v^{\star}) \end{bmatrix} + \begin{bmatrix} \partial{f}(u^{\star}) \\   \partial{g}^{*}(v^{\star}) \end{bmatrix}.
\end{equation}
In particular, if $f$ and $g^{*}$ are convex, then \eqref{eq:minimax_opt_cond} is necessary and sufficient for $(u^{\star}, v^{\star})$ to be an optimal solution of \eqref{eq:minimax_prob}.
Otherwise, it is only a necessary condition.
If we define $x := [u, v]$, $F := [\nabla_u{\Hc}, -\nabla_v{\Hc}]$, and $T := [\partial{f}, \partial{g}^{*}]$, then \eqref{eq:minimax_opt_cond} is  a special case of \eqref{eq:NI}.
If $f$ and $g^{*}$ are convex, then $T$ is monotone and automatically satisfies Assumption~\ref{as:A1}(iii).
Moreover, under appropriate smoothness conditions on $\Hc$, Assumption~\ref{as:A2} also holds.

The minimax problem~\eqref{eq:minimax_prob} serves as a fundamental model in robust and distributionally robust optimization~\citep{Ben-Tal2001,rahimian2019distributionally,namkoong2016stochastic}, two-player games~\citep{ho2022game,kuhn1996work}, fair machine learning~\citep{du2021fairness,martinez2020minimax}, and generative adversarial networks (GANs)~\citep{arjovsky2017wasserstein,daskalakis2018training,goodfellow2014generative}, among many others.

As a specific example, if $f(u) := \delta_{\Delta_{p_1}}(u)$ and $g^{*}(v) := \delta_{\Delta_{p_2}}(v)$ are the indicators of the standard simplexes $\Delta_{p_1}$ and $\Delta_{p_2}$, respectively, and $\Hc(u, v) := \iprods{\mbf{L}u, v}$ is a bilinear form with a given payoff matrix $\mbf{L}$, then \eqref{eq:minimax_prob} reduces to the classical bilinear game problem.
Another representative example is a non-probabilistic robust optimization model derived from  Wald's minimax framework: $\min_{u \in \R^{p_1}}\big\{ \phi(u) := h(u) + f(u) \equiv \max_{ v \in \mathcal{V} } \Hc(u, v) + f(u) \big\}$, where $u$ is a decision variable, $v$ denotes an uncertainty vector over the uncertainty set $\mathcal{V}\subset\R^{p_2}$, and the function $h(u) := \max_{ v \in \mathcal{V} } \Hc(u, v)$ captures the worst-case risk across all possible realizations of $v$. 
See \citep{Ben-Tal2001} for concrete instances.
}

%%%% (iii) Monotone variational inequality problems.
\vspace{0.5ex}
\noindent\textit{$\mathrm{(iii)}$~\textbf{Variational inequality problems $\mathrm{(VIPs)}$.}}
If $T = \Nc_{\Xc}$, the normal cone of a nonempty, closed, and convex set $\Xc$ in $\R^p$, then \eqref{eq:NI} reduces to 
\begin{equation}\label{eq:VIP}
\textrm{Find $x^{\star} \in \Xc$ such that:} \   \iprods{Fx^{\star}, x - x^{\star}} \geq 0, \  \textrm{for all} \ x \in \Xc.
\tag{VIP}
\end{equation}
More generally, if $T = \partial{g}$, the subdifferential of a convex function $g$, then \eqref{eq:NI} reduces to a mixed VIP.
\rvs{Since $\Xc$ is convex, $T = \Nc_{\Xc}$ automatically satisfies Assumption~\ref{as:A1}(iii).}

The \ref{eq:VIP} covers many well-known problems in practice, including unconstrained and constrained minimization, minimax problems, complementarity problems, and Nash's equilibria, see also \citep{Facchinei2003,Konnov2001} for more details and direct applications in traffic networks and economics.

%%%% (iv) Fixed-point problems of non-expansive mapping.
\vspace{0.5ex}
\noindent\textit{$\mathrm{(iv)}$~\textbf{Fixed-point problem of nonexpansive mapping.}}
The fixed-point problem is a fundamental topic in computational mathematics, with numerous applications in numerical analysis, ordinary and partial differential equations, engineering, and physics \citep{agarwal2001fixed,Bauschke2011,Combettes2011a}.
Let $P : \R^p \to \R^p$ be a given nonexpansive mapping, i.e., $\norms{Px - Py} \leq \norms{x - y}$ for all $x, y \in \R^p$.
Then, the classical fixed-point problem is stated as follows:
\myeq{eq:fixed_point}{
\textrm{Find $x^{\star} \in \R^p$ such that:} \quad  x^{\star} = Px^{\star}.
\tag{FP}
}
This problem is equivalent to \eqref{eq:NI} with $F := \Id - P$ and $T = 0$, where $\Id$ is the identity mapping.
\rvs{It is well-known that  $P$ is nonexpansive if and only if $F$ is $(1/2)$-co-coercive.}
The algorithms developed in this paper for \eqref{eq:NI} can be applied to solve  \eqref{eq:fixed_point}.
We can also generalize \eqref{eq:fixed_point} to a Kakutani's fixed-point problem $x^{\star} \in Px^{\star} + Tx^{\star}$ of a single-valued mapping $F$ and a multi-valued mapping $T$, which is also equivalent to \eqref{eq:NI} with $F = P - \Id$.

%%% Add: sadiev2024stochastic -- Reference - R1.
\vspace{0.5ex}
\noindent\textbf{$\mathrm{(d)}$~Motivation and challenges.}
Advanced numerical methods such as acceleration, stochastic approximation, and variance reduction, have received significant attention over the past few decades for solving special cases of \eqref{eq:NI}, including \eqref{eq:opt_prob}, \eqref{eq:minimax_prob}, and \eqref{eq:VIP}, due to their broad applications in modern machine learning and data science.
Relevant works include, but are not limited to, \citep{alacaoglu2021stochastic,alacaoglu2021forward,davis2016smart,davis2022variance,Defazio2014,emmanouilidis2024stochastic,SVRG,nguyen2017sarah,sadiev2024stochastic,Tran-Dinh2019a}.
Moreover, biased variance-reduced estimators such as SARAH \citep{nguyen2017sarah} and Hybrid-SGD \citep{Tran-Dinh2019a} have demonstrated better oracle complexity than their unbiased counterparts, such as SVRG \citep{SVRG} and SAGA \citep{Defazio2014}; see also \citep{driggs2019accelerating,Pham2019}.
However, designing new algorithms for \eqref{eq:NI} that combine both acceleration and biased variance-reduction remains a largely unexplored topic.
This is due to several challenges, including the following.
\begin{compactenum}
\item[$\mathrm{(i)}$] 
First, most convergence analyses of  \rvs{Nesterov's accelerated} randomized and stochastic methods for merely convex optimization and convex-concave minimax problems rely on the objective function as a key metric for constructing a suitable Lyapunov function.
However, such a function does not exist in \eqref{eq:NI}, and it remains unclear how to define an alternative metric that can play a similar role.

\item[$\mathrm{(ii)}$] 
\rvs{
Second, unlike stochastic methods in optimization, there is a lack of convergence analysis techniques for handling biased estimators in algorithms for solving \eqref{eq:NI}.
}

\item[$\mathrm{(iii)}$] 
Third, in convex optimization, accelerated methods achieve faster convergence rates by leveraging convexity (or, equivalently, the monotonicity of [sub]gradients).
Extending such acceleration to settings beyond convexity or monotonicity, particularly in stochastic methods for solving \eqref{eq:NI}, remains a significant challenge.
\end{compactenum}
%%%
%%% Add: condat2022murana --- Reference - R1
Hitherto, most existing methods for \eqref{eq:NI} still face these challenges \citep{driggs2019accelerating,driggs2019bias}, and only a few works have managed to overcome some of them without compromising convergence guarantees \citep{cai2022stochastic,cai2023variance,condat2022murana}.
In this paper, we develop new classes of accelerated schemes for solving \eqref{eq:NI} that can incorporate both unbiased and biased estimators.
This capability enables our methods to achieve better oracle complexity than several existing approaches and cover a broad family of algorithms.

\vspace{0.5ex}
\noindent\textbf{$\mathrm{(e)}$~Our contributions.}
Our contributions in this paper consist of the following.
\begin{compactitem}
\item[(i)] 
We introduce a class of variance-reduced estimators that includes a broad spectrum of both unbiased and biased variance-reduced instances. 
We demonstrate that common estimators, including SVRG, SAGA, SARAH, and Hybrid-SGD, fall within this class. 
Furthermore, we establish the necessary bounds required for our convergence analysis.

\item[(ii)] 
We develop a new class of accelerated forward-backward splitting (FBS) methods with variance reduction to approximate a solution of \eqref{eq:NI} in both the finite-sum and expectation settings under appropriate \rv{co-coercivity} of $F$ and co-hypomonotonicity of $T$.
Our algorithm is single-loop and simple to implement.
It also covers a broad range of variance-reduced estimators introduced in (i).
Our method achieves both $\BigOs{1/k^2}$ and $\SmallOs{1/k^2}$ convergence rates in expectation on the squared norm of the FBS residual, along with several summability bounds.
We further prove  $\SmallOs{1/k^2}$ almost sure convergence rates.
We also show that the sequences of iterates generated by our method  almost surely converge to a solution of \eqref{eq:NI}.

\item[(iii)] 
We specify our method to cover four specific estimators: SVRG, SAGA, SARAH, and Hybrid-SGD, each achieving the ``best-known''\footnote{They may be different by  a poly-logarithmic factor $\log^{\nu}(n)$ or $\log^{\nu}(1/\epsilon)$.} oracle complexity.
For the SVRG and SAGA estimators, we establish a complexity of $\widetilde{\mathcal{O}}(n + n^{2/3}\epsilon^{-1})$ in the finite-sum setting \eqref{eq:finite_sum_form}, and $\widetilde{\mathcal{O}}(\epsilon^{-3})$ in the expectation setting \eqref{eq:expectation_form}.
For SARAH and Hybrid-SGD, this complexity improves to $\widetilde{\mathcal{O}}(n + n^{1/2}\epsilon^{-1})$  and $\BigOs{\epsilon^{-3}}$, respectively.

\item[(iv)] 
Alternatively, we also propose a class of accelerated backward-forward splitting (BFS) algorithms with variance reduction for solving \eqref{eq:NI}, which attain the same convergence properties and oracle complexities as our accelerated FBS method.
\end{compactitem}

\begin{table}[ht!]
\vspace{-2ex}
\newcommand{\cell}[1]{{\!\!}{#1}{\!\!}}
\begin{center}
\caption{Comparison of existing variance-reduced single-loop methods and our algorithms}\label{tbl:existing_vs_ours}
\vspace{1ex}
\begin{small}
\resizebox{\textwidth}{!}{  
\begin{tabular}{|c|c|c|c|c|c|c|} \hline
Refs & \cell{Ass. on $F$} & Add. Ass. & Estimators & {\!\!\!} Setting {\!\!\!} &  \cell{Residual Rates} & Complexity \\ \hline
\textbf{Work 1}             & co-coercive     & $T=0$, $\mu$-SQM. & {\!\!\!} SVRG \& SAGA {\!\!\!} &  \eqref{eq:finite_sum_form} & $\BigOs{1/k}$  & $\BigOs{(L/\mu)\log(\epsilon^{-1})}$  \\ \hline
\textbf{Work 2}             & monotone       & monotone $T$ & SVRG & \eqref{eq:finite_sum_form} & $\BigOs{1/k}$  & --  \\ \hline
\textbf{Work 3}             & co-coercive     & monotone $T$ & a class & \eqref{eq:finite_sum_form} & $\BigOs{1/k^2}$  & $\mcal{O}( n + n^{2/3}\epsilon^{-1} )$  \\ \hline
\textbf{Work 4}             & co-coercive     & monotone $T$ & SARAH  &  \eqref{eq:finite_sum_form}   & $\BigOs{1/k^2}$ & $\widetilde{\mcal{O}}\big( n + n^{1/2}\epsilon^{-1} \big)$ \\ \hline
\textbf{Work 5}             & co-coercive     & monotone $T$ & SARAH  &  \eqref{eq:expectation_form} &  -- &  $\BigOs{\epsilon^{-3}}$ \\ \hline
\myblue{\textbf{Ours}} & co-coercive     & {\!\!\!} \myblue{co-hypomonotone $T$} {\!\!\!} & \myblue{a class} & \eqref{eq:finite_sum_form}   &{\!\!\!\!\!}  \makecell{\myblue{ $\BigOs{1/k^2}$, $\SmallOs{1/k^2}$} \\ \myblue{$x^k \to x^{\star}$ \textit{a.s.} } } {\!\!\!\!\!} &{\!\!\!}  \makecell{$\widetilde{\mcal{O}}\big(n + n^{2/3}\epsilon^{-1}\big)$ \\  $\to \widetilde{\mcal{O}}\big(n + n^{1/2}\epsilon^{-1}\big)$ } {\!\!\!}  \\ \hline
\myblue{\textbf{Ours}} & co-coercive     & {\!\!\!} \myblue{co-hypomonotone $T$} {\!\!\!} & \myblue{a class} & \eqref{eq:expectation_form} &{\!\!\!\!\!}  \makecell{\myblue{ $\BigOs{1/k^2}$, $\SmallOs{1/k^2}$} \\ \myblue{$x^k \to x^{\star}$ \textit{a.s.} } } {\!\!\!\!\!}  &  $\widetilde{\mcal{O}}(\epsilon^{-3}) \to  \mcal{O}(\epsilon^{-3})$ \\ \hline
\end{tabular}}
\end{small}
\end{center}
%\\\vspace{1ex}
{\footnotesize
\vspace{-1ex}
\textbf{Abbreviations:} 
\textbf{Refs} $=$ References; 
\textbf{Ass.} $=$ Assumptions; 
\textbf{Add. Ass.} $=$ Additional Assumptions; 
\textbf{SQM} $=$ strong quasi-monotonicity;  
\eqref{eq:finite_sum_form} $=$ the finite-sum setting,  and \eqref{eq:expectation_form} $=$ the expectation setting; 
\textbf{Residual rate} $=$ the convergence rate on $\Expn{\norms{G_{\lambda}x^k}^2}$, where $G_{\lambda}$ is either the equation operator or the FBS residual in \eqref{eq:FBS_residual}; 
\textbf{\myblue{a class}} $=$ a class of variance-reduced estimators satisfying Definition~\ref{de:VR_Estimators};
and \textit{a.s.} $=$ almost surely.

\textbf{References:} \textbf{Work 1} is \citet{davis2022variance}; 
\textbf{Work 2} is \citet{alacaoglu2021forward,alacaoglu2021stochastic}; 
\textbf{Work 3} is \citet{tran2024accelerated}; 
\textbf{Work 4} is \citet{cai2023variance}; 
and 
\textbf{Work 5} is \citet{cai2022stochastic}.
}
\vspace{-1ex}
\end{table}

Table~\ref{tbl:existing_vs_ours} summarizes the most related results to our work.
Let us further discuss in detail our contributions and compare our results with the most related works.
%%%
First, our approach is indirect compared to \citep{cai2022stochastic,cai2023variance}, i.e., we reformulate \eqref{eq:NI} into an equation (or equivalently, a fixed-point problem) before developing our algorithms.
This approach offers certain advantages:
(i) it enables us to handle a co-hypomonotone operator $T$ with a co-hypomonotonicity modulus $\rho$ that is independent of the algorithmic parameters;
(ii) it enhances the flexibility of our method, making it readily applicable to other reformulations such as FBS and BFS.
%%%
Second, unlike \citep{alacaoglu2021forward,alacaoglu2021stochastic,bot2019forward,davis2022variance}, our new class of variance-reduced estimators is sufficiently broad to cover many existing ones as special cases and can  potentially accommodate new estimators.
%%%
Third, our methods differ from Halpern’s fixed-point iterations in \citep{cai2022stochastic,cai2023variance}, \rvt{which enables us to employ different parameter update rules than those in Halpern’s schemes.}
This distinction is crucial for achieving faster convergence rates of $\SmallOs{1/k^2}$ and allows us to establish both the almost sure convergence of the iterates and the  $\SmallOs{1/k^2}$ almost sure convergence rates.
%%%
Fourth, our algorithms are accelerated and single-loop, making them easier to implement compared to double-loop or catalyst methods \citep{khalafi2023accelerated,yang2020catalyst}.
Fifth, our rates and oracle complexity rely on the metric $\mathbb{E}\big[ \norms{G_{\lambda}x^k}^2 \big]$.
This differs from existing results  using a gap or a restricted gap function, which only works for monotone problems. 
Sixth, our rate offers a $1/k$ factor improvement over non-accelerated methods \citep{alacaoglu2021forward,alacaoglu2021stochastic,davis2022variance}.
Finally, our oracle complexity matches the best-known results for methods using SARAH estimators, without requiring any enhancement strategies such as scheduled restarts or multiple loops, as employed, e.g., in \citet{cai2022stochastic,cai2023variance}.

\vspace{0.5ex}
\noindent\textbf{$\mathrm{(f)}$~Related work.}
Problem \eqref{eq:NI} and its special cases are well-studied in the literature, see, e.g., \citep{Bauschke2011,reginaset2008,Facchinei2003,phelps2009convex,ryu2022large,ryu2016primer}. 
We focus on the most recent works relevant to our methods in both the finite-sum and expectation settings.
%\rvs{\todo{Discussing the paper: gorbunov2022convergence}}

%%% Add reference: boct2024generalized -- Reference - R1.
%\rv{\todo{Discuss the paper -- boct2024generalized.}}
\textit{Accelerated methods.}
\rv{Deterministic accelerated methods have been broadly developed to solve \eqref{eq:NI} and its special cases in early works \citep{he2016accelerated,kolossoski2017accelerated,attouch2019convergence}, and further studied in subsequent papers  \citep{adly2021first,attouch2020convergence,attouch2022ravine,boct2024generalized,chen2017accelerated,gorbunov2022convergence,kim2021accelerated,mainge2021accelerated,park2022exact,tran2022connection}.}
These methods are based on Nesterov’s acceleration technique \citep{Nesterov1983}.
However, unlike in convex optimization, extending Nesterov’s acceleration to monotone inclusions presents a fundamental challenge due to the absence of an objective function, which complicates the construction of a suitable Lyapunov function as mentioned earlier.
This limitation necessitates a different approach for solving \eqref{eq:NI} \citep{attouch2019convergence,mainge2021accelerated}.
Our approach builds on insights from \citep{alcala2023moving,attouch2019convergence,mainge2021accelerated,tran2022connection,yuan2024symplectic}, combined with variance reduction strategies to develop new methods.

Alternatively, Halpern's fixed-point iteration  \citep{halpern1967fixed}  has recently been proven to achieve a better convergence rates, see \citep{diakonikolas2020halpern,lieder2021convergence,sabach2017first}, matching Nesterov's acceleration schemes. 
\citet{yoon2021accelerated} extended Halpern's method to extragradient-type schemes, relaxing the co-coercivity assumption. 
Many subsequent works have exploited this idea to other methods, e.g., \citep{alcala2023moving,cai2022accelerated,cai2022baccelerated,lee2021semi,lee2021fast,park2022exact,tran2021halpern,tran2023extragradient,tran2022connection}. 
Recently, \citet{tran2022connection} established a connection between Nesterov’s and Halpern’s accelerations for various iterative schemes.

%%% Add reference: demidovich2023guide -- Reference -- R1.
\textit{Stochastic methods.}
Stochastic methods for \eqref{eq:NI} and its special cases have been extensively developed; see, e.g., \citep{juditsky2011solving,kotsalis2022simple,pethick2023solving}.
Some approaches exploit mirror-prox and averaging techniques, such as those in \citep{juditsky2011solving,kotsalis2022simple}, while others rely on projection or extragradient-type schemes, e.g., \citep{bohm2022two,cui2021analysis,iusem2017extragradient,kannan2019optimal,mishchenko2020revisiting,pethick2023solving,yousefian2018stochastic}.
Many algorithms employ standard Robbins–Monro's stochastic approximation with fixed or increasing batch sizes.
\rv{Other works extend the analysis to a broader class of algorithms, including both unbiased and biased estimators, e.g., \citep{beznosikov2023stochastic,demidovich2023guide,gorbunov2022stochastic,loizou2021stochastic}, thereby covering standard stochastic and unbiased variance-reduction methods.}
The complexity typically depends on an upper bound of the variance, which often leads to inefficient oracle complexity bounds.

%%% Variance reductions.
\textit{Variance-reduction methods.}	
Variance-reduction schemes using control variate techniques are widely developed in optimization, where many estimators have been proposed, including SAGA \citep{Defazio2014}, SVRG \citep{SVRG}, SARAH \citep{nguyen2017sarah}, and Hybrid-SGD \citep{Tran-Dinh2019}. 
Researchers have adopted these estimators to develop methods for solving \eqref{eq:NI}. 
For instance,  \citet{davis2016smart,davis2022variance} proposed SAGA-type methods for \eqref{eq:NI}, under a ``star'' \rv{co-coercivity} and strong quasi-monotonicity, most relevant to our work. 
However, we focus on accelerated methods that achieve better convergence rates and complexity.
The authors in  \citet{alacaoglu2021forward,alacaoglu2021stochastic} employed SVRG estimators to develop variance-reduced extragradient-type methods to solve \eqref{eq:VIP}, but these are non-accelerated. 
Other works can be found in \citet{bot2019forward,carmon2019variance,chavdarova2019reducing,huang2022accelerated,palaniappan2016stochastic,yu2022fast}, some of which focus on minimax problems or bilinear matrix games. 
More recently, \citet{cai2022stochastic,cai2023variance} exploited Halpern's fixed-point iteration to develop variance-reduced methods, often achieving better oracle complexity by employing the SARAH estimator.
All these results differ from ours due to the generalization of Definition~\ref{de:VR_Estimators} and the new accelerated methods that we develop in this paper.
  
\vspace{0.25ex}
\noindent\textbf{$\mathrm{(g)}$~Paper organization.}
The rest of this paper is organized as follows.
In Section~\ref{sec:preleminary}, we recall some related notations, concepts, and technical results used in this paper.
We also further discuss our assumptions imposed on \eqref{eq:NI}.
Section~\ref{sec:VR_estimators} introduces a class of variance-reduced estimators and establishes their bounds.
Section~\ref{sec:VR_AFBS_method} develops our accelerated forward-backward splitting method with variance-reduction to solve \eqref{eq:NI} and establishes its convergence properties. 
We also specify this algorithm for each concrete estimator to obtain the corresponding variant, and estimate its oracle complexity bound.
Section~\ref{sec:VrABFS_method} presents an alternative: an accelerated backward-forward splitting method with variance reduction for solving \eqref{eq:NI} and establishes similar convergence results as in Section~\ref{sec:VR_AFBS_method}.
Section~\ref{sec:num_examples} provides two numerical examples to validate our results, and compare different methods.
For clarity of presentation, all the technical proofs are deferred to the appendix.

%%%%%%%%%%%%%%%%%%%%%%%%%%%%%%%%%%%%%%%%%%
%%%% 2. Background, Fundamental Assumptions, and Mathematical Tools
%%%%%%%%%%%%%%%%%%%%%%%%%%%%%%%%%%%%%%%%%%
\vspace{-0.5ex}
\beforesec
\section{Background and Mathematical Tools}\label{sec:preleminary}
\aftersec
\vspace{-0.25ex}
First, we recall the necessary notations and concepts.
Next, we further discuss our Assumptions \ref{as:A1} and \ref{as:A2}.
Finally, we prove a key result essential for developing our algorithms.

%%\rvs{\todo{Summarize all notations in a table since it is hard to follow.}}
%%%%%%%%%%%%%%%%%%%%%%%%%%%%%%%%%%%%%%%%%%%%%%%%
%%% 2.1. Notations and concepts.
%%%%%%%%%%%%%%%%%%%%%%%%%%%%%%%%%%%%%%%%%%%%%%%%
\vspace{-1ex}
\beforesubsec
\subsection{Notations and basic concepts}\label{subsec:preliminary_defs}
\aftersubsec
We work with the finite dimensional space $\Espace$ equipped with the standard inner product $\iprods{\cdot, \cdot}$ and the Euclidean norm $\norms{\cdot}$.
For a single-valued or multi-valued mapping $T : \Espace \rightrightarrows 2^{\Espace}$, $\dom{T} = \set{x \in \Espace : Tx \not= \emptyset}$ denotes its domain, $\gra{T} = \set{(x, u) \in \Espace \times \Espace : u \in Tx}$ denotes its graph, where $2^{\Espace}$ is the set of all subsets of $\Espace$.

For a convex function $f$, $\nabla{f}$ denotes its [sub]gradient, and $\partial{f}$ denotes its [abstract] subdifferential.
For a given symmetric matrix $\mbf{X}$, $\lambda_{\max}(\mbf{X})$ and $\lambda_{\min}(\mbf{X})$ denote its largest and smallest eigenvalues, respectively.
We also use standard $\mcal{O}(\cdot)$ and $o(\cdot)$ for convergence rates and complexity bounds, and $\widetilde{\mcal{O}}(s)$ is $\mcal{O}(s\log^{\nu}(s))$ (hiding  a poly-log factor).

Let $\Fc_k$ be the $\sigma$-algebra generated by all the randomness arising from the algorithm, including $x^0, x^1, \cdots, x^k$, up to the current iteration $k$.
Let $\mathbb{E}_k[{\cdot}] := \mathbb{E}[ \cdot \mid \Fc_k ]$ be the conditional expectation and $\mathbb{E}[\cdot]$ be the total expectation. 

\rv{
Next, let us recall the concepts of co-hypomonotonicity, monotonicity, and co-coercivity for operators, see, e.g., \citep{Bauschke2011,bauschke2020generalized}.

%%% a. Monotonicity.
\begin{definition}\label{de:monotonicity}
For a multi-valued mapping $T : \Espace \rightrightarrows 2^{\Espace}$, we say that:
\begin{compactitem}
\item $T$ is $\rho$-co-hypomonotone  if there exists $\rho \geq 0$ such that
\myeq{eq:co_hypomonotone_def}{
\iprods{u - v, x - y} \geq -\rho\norms{u - v}^2, \quad \textrm{for all}~(x, u),  (y, v) \in \gra{T}.
}
Here, $\rho$ is referred to as the co-hypomonotonicity modulus of $T$.
\item $T$ is monotone if \eqref{eq:co_hypomonotone_def} holds with $\rho = 0$, i.e., $\iprods{u - v, x - y} \geq 0$ for $(x, u),  (y, v) \in \gra{T}$.
\item $T$ is maximally [co-hypo]monotone if $\gra{T}$ is not properly contained in the graph of any other [co-hypo]monotone operator.
\end{compactitem}
\end{definition}
If $T$ is single-valued, then \eqref{eq:co_hypomonotone_def} reduces to $\iprods{Tx - Ty, x - y} \geq -\rho\norms{Tx - Ty}^2$ for all $x, y \in \dom{T}$.
A co-hypomonotone operator is not necessarily monotone, see Subsection~\ref{subsec:assumptions_discussion} for concrete examples.
The co-hypomonotonicity concept in Definition~\ref{de:monotonicity} is global.
We say that $T$ is locally $\rho$-co-hypomonotone around $(\bar{x},\bar{u}) \in \gra{T}$ if there exists a neighborhood $\Wc$ of $(\bar{x},\bar{u})$ such that for all $(x, u), (y,v) \in \gra{T}\cap\Wc$, we have $\iprods{u - v, x - y} \geq -\rho\norms{u - v}^2$.

%%% b. Lipschitz continuity and \rv{co-coercivity}.
\begin{definition}\label{de:co_coercivity}
Given a single-valued mapping $F : \R^p \to \R^p$, we say that:
\begin{compactitem}
\item $F$ is $\frac{1}{L}$-co-coercive if there exists $L > 0$ such that
\myeqn{
\iprods{Fx - Fy, x - y} \geq \tfrac{1}{L}\norms{Fx - Fy}^2, \quad \textrm{for all}~x, y \in\dom{F}.
}
\item $F$ is $L$-Lipschitz continuous if $\norms{Fx - Fy} \leq L\norms{x - y}$ for all $x, y\in\dom{F}$, where $L \geq 0$ is the Lipschitz constant. 
In particular, if $L = 1$, then $F$ is nonexpansive.  
\end{compactitem}
\end{definition}
If $F$ is $\frac{1}{L}$-co-coercive, then it is also monotone and $L$-Lipschitz continuous.
}

%%%% c. Resolvent and reflection operators.
The operator $J_Tx := \set{w \in \Espace : x \in w + Tw}$ is called the resolvent of $T$, often denoted by $J_Tx = (\Id + T)^{-1}x$, where $\Id$ is the identity mapping.
If $T$ is monotone, then $J_T$ is singled-valued, and if $T$ is maximally monotone, then $J_T$ is singled-valued and $\dom{J_T} = \Espace$.

%%%%%%%%%%%%%%%%%%%%%%%%%%%%%%%%%%%%%%%%%%%%%%%%%%%
%%%%%%  2.2. Fundamental Assumptions.
%%%%%%%%%%%%%%%%%%%%%%%%%%%%%%%%%%%%%%%%%%%%%%%%%%% 
\beforesubsec
\subsection{Further discussion of Assumptions~\ref{as:A1} and \ref{as:A2}}\label{subsec:assumptions_discussion}
\aftersubsec
%\vspace{0.5ex}
\noindent\textbf{$\mathrm{(a)}$~Pros and cons of Assumption~\ref{as:A2}.}
Similar to variance-reduction methods using control variate techniques in optimization, we require Assumption \ref{as:A2} in our methods.
This assumption has some pros and cons as follows.

\textit{$\mathrm{(i)}$~\textbf{Pros}}.
First, if $Fx = \nabla{f}(x)$, the gradient of a differentiable convex function, then the $\frac{1}{L}$-\rv{co-coercivity} of $F$ is equivalent to the convexity and $L$-smoothness of $f$ (i.e., $\nabla{f}$ is $L$-Lipschitz continuous).
Therefore, Assumption~\ref{as:A2} covers convex and $L$-smooth functions as special cases, including the finite-sum and expectation settings.

\rv{Second, the \rv{$\frac{1}{L}$-co-coercivity} of $F$ is equivalent to the nonexpansiveness of $G = \Id - \frac{2}{L} F$, see \citep[Proposition 4.11]{Bauschke2011}.
Therefore, our methods can also be applied to find approximate fixed-points of a nonexpansive operator.
Note that several problems without satisfying Assumption~\ref{as:A2} can be reformulated equivalently to a fixed-point problem of a nonexpansive operator, and thus can be indirectly solved by our methods.}
More specifically,  as shown in \citet{tran2022connection,tran2022accelerated}, there are several ways to reformulate \eqref{eq:NI} and its special cases into a co-coercive equation (e.g., using the Douglas-Rachford splitting or three-operator splitting techniques).
This approach possibly expands the applicability of  our methods to other problem classes.

\textit{$\mathrm{(ii)}$~\textbf{Cons}}.
Though Assumption~\ref{as:A2} is reasonable and relatively broad, it may have some extreme cases.
For instance, it does not directly cover general linear mappings $Fx := \mathbb{F}x + q$ for a given square matrix $\mathbb{F}$ and a vector $q$, unless $\mathbb{F}$ is positive definite.
One way to handle this extreme case is to consider its Moreau-Yosida's approximation instead of $F$ itself.

\vspace{0.5ex}
\noindent\textbf{$\mathrm{(b)}$~Examples of co-hypomonotone operators.}
We provide here two examples of co-hypomonotone operators.
However, other examples exist, see, e.g., \citep{evens2023convergence}.

\rv{
\textit{Example 1.}
Consider $Tx := \mathbb{T}x + s$, where $\mathbb{T}$ is symmetric and invertible, but not positive semidefinite, and $s \in\R^p$ is given.
Assume that $\rho := -\lambda_{\min}(\mathbb{T}^{-1}) > 0$.
Then, $T$ is $\rho$-co-hypomonotone and thus satisfies Assumption~\ref{as:A1}(iii).
Generally, it is easy to check that if $\mathbb{T} + \mathbb{T}^{\top} + 2\rho\mathbb{T}^{\top}\mathbb{T} \succeq 0$ for some $\rho \geq 0$, then $T$ is $\rho$-co-hypomonotone.

Next, we consider $Fx := \mathbb{F}x + q$, where $\mathbb{F}$ is a symmetric positive semidefinite matrix and $q \in \R^p$.
Clearly, $F$ satisfies Assumption~\ref{as:A2} with $L := \lambda_{\max}(\mathbb{F})$.
However, $\Phi := F + T$ is nonmonotone if we impose \rvs{$\lambda_{\min}(\mathbb{F} + \mathbb{T}) < 0$}.
In addition, it is possible to choose \rvs{$\mathbb{F}$} and $\mathbb{T}$ such that $L \rho < 1$, which satisfies the range condition of $L\rho$ in Lemma~\ref{le:FBS_cocoerciveness} below.

\textit{Example 2.}
Let $\Psi : \R^{p_1}\times\R^{p_2} \to \R$ be a twice continuously differentiable and $\mu$-convex-concave function for some $\mu\in\R$, i.e., $\nabla^2_{uu}\Psi(x) \succeq \mu\Id$ and $-\nabla^2_{vv}\Psi(x) \succeq \mu\Id$ for all $x = [u, v] \in \R^{p_1+p_2}$.
We say that $\Psi$ is $\alpha$-interaction dominate, see \citep{grimmer2023landscape}, if there exist some $\alpha > -\frac{1}{\rho}$ and $\rho \in \big(0, \frac{1}{\max\sets{-\mu, 0}}\big)$ such that for all $x \in \R^{p_1+p_2}$, we have
\begin{equation*}
\arraycolsep=0.2em
\begin{array}{lcl}
\nabla^2_{uu}\Psi(x) + \nabla^2_{uv}\Psi(x)(\frac{1}{\rho}\Id - \nabla^2_{vv}\Psi(x))^{-1}\nabla^2_{vu}\Psi(x) & \succeq & \alpha\Id, \vspace{1ex}\\
-\nabla^2_{vv}\Psi(x) + \nabla^2_{vu}\Psi(x)(\frac{1}{\rho}\Id + \nabla^2_{uu}\Psi(x))^{-1}\nabla^2_{uv}\Psi(x) & \succeq & \alpha\Id.
\end{array}
\end{equation*}
As proven in \citet[Proposition 4.17]{evens2023convergence}, if $\alpha > 0$, then the saddle mapping $Tx := [\nabla_{u}\Psi(x), -\nabla_{v}\Psi(x)]$ is $\rho$-co-hypomonotone. 
The $\alpha$-interaction dominance notion was studied in \citet{grimmer2023landscape} for nonconvex-nonconcave minimax problems.
}

%%%%%%%%%%%%%%%%%%%%%%%%%%%%%%%%%%%%%%%%%%%%%%%%
%%% 2.3. Equivalent Reformulations.
%%%%%%%%%%%%%%%%%%%%%%%%%%%%%%%%%%%%%%%%%%%%%%%%
\beforesubsec
\subsection{Equivalent reformulations}\label{subsec:NI_as_NE}
\aftersubsec
The first step of our approach is to reformulate \eqref{eq:NI} into an equation using either the forward-backward splitting (FBS) residual or the backward-forward splitting (BFS) residual mapping.
These reformulations are also equivalent to the fixed-point problem \eqref{eq:fixed_point}.

\vspace{0.5ex}
\noindent\textbf{$\mathrm{(a)}$~Forward-backward splitting reformulation.}
We consider the following \textbf{forward-backward splitting} residual mapping of \eqref{eq:NI}:
\begin{equation}\label{eq:FBS_residual}
G_{\lambda}x := \tfrac{1}{\lambda}\big(x - J_{\lambda T}(x - \lambda Fx) \big),
\end{equation}
for a given $\lambda > 0$, and $J_{\lambda T}$ is the resolvent of $\lambda T$.
Then, \textbf{$x^{\star}$ solves \eqref{eq:NI} iff $G_{\lambda}x^{\star} = 0$}, i.e.:
\begin{equation}\label{eq:FBS_reform}
0 \in \Phi{x}^{\star} := Fx^{\star} + Tx^{\star} \quad \Leftrightarrow \quad G_{\lambda}x^{\star} = 0.
\end{equation}
\rv{
The equation $G_{\lambda}x^{\star} = 0$ can also be rewritten  equivalently to the fixed-point formulation: $x^{\star} = J_{\lambda T}(x^{\star} - \lambda Fx^{\star})$ of the FBS operator $J_{\lambda T}((\cdot) - \lambda F(\cdot))$.
}

\vspace{0.5ex}
\noindent\textbf{$\mathrm{(b)}$~Backward-forward splitting reformulation.}
Alternatively, we can also consider the following \textbf{backward-forward splitting} residual mapping of \eqref{eq:NI}:
\begin{equation}\label{eq:BFS_residual}
S_{\lambda}u := F(J_{\lambda T}u) + \tfrac{1}{\lambda}(u - J_{\lambda T}u),
\end{equation}
for a given $\lambda > 0$.
Then, \textbf{$x^{\star}$ solves \eqref{eq:NI} iff $u^{\star}$ solves $S_{\lambda}u^{\star} = 0$ and $x^{\star} = J_{\lambda T}u^{\star}$}, i.e.:
\begin{equation}\label{eq:BFS_reform}
0 \in \Phi{x}^{\star} := Fx^{\star} + Tx^{\star} \quad \Leftrightarrow \quad S_{\lambda}u^{\star} = 0 \quad\textrm{and} \quad x^{\star} = J_{\lambda T}u^{\star}.
\end{equation}
%%%
\rv{\noindent\textbf{$\mathrm{(c)}$~Comparison between \eqref{eq:FBS_residual} and \eqref{eq:BFS_residual}.}
Note that $G_{\lambda}$ in \eqref{eq:BFS_residual} does not preserve the finite-sum or expectation structure as of $F$ due to the composition $J_{\lambda T} \circ (\Id - \lambda F)$.
In contrast, $S_{\lambda}$ in \eqref{eq:BFS_residual} maintains this structure as of $F$.
Moreover, the primal variable of $G_{\lambda}$ is $x$, which directly corresponds to the variable $x$ in \eqref{eq:NI}.
However, the primal variable of $S_{\lambda}$ is $u$, which is indirectly related to $x$ via the resolvent $x = J_{\lambda T} u$, making $x$ a shadow variable.
}
%\rv{\todo{Discuss the relation between FBS and BFS --- Are they equivalent?}}

To develop our algorithms, we require further properties of $G_{\lambda}$ and $S_{\lambda}$ as stated in the following lemma, whose proof can be found in Appendix~\ref{apdx:le:FBS_cocoerciveness}.

%%% Lemma 3.1.
\begin{lemma}\label{le:FBS_cocoerciveness}
For \eqref{eq:NI}, suppose that $F$ is $\frac{1}{L}$-co-coercive and $T$ is maximally $\rho$-co-hypomonotone such that $L\rho < 1$.
For given $\hat{L}$ and $\lambda$ such that $\hat{L} \geq L$, $\hat{L}\rho < 1$, and $\rho < \lambda \leq \frac{2(1 + \sqrt{1 - \hat{L}\rho})}{\hat{L}}$, we define $\bar{\beta} := \frac{\lambda(4 - \hat{L}\lambda) - 4\rho}{4(1 - \rho\hat{L})} \geq 0$ and $\Lambda := \frac{\hat{L}-L}{L\hat{L}} \geq 0$.
Then
\begin{compactitem}
\item[$\mathrm{(i)}$] 
if $G_{\lambda}$ is defined by \eqref{eq:FBS_residual}, then for all $x, y \in \dom{\Phi}$, we have 
\begin{equation}\label{eq:G_cocoerciveness} 
\arraycolsep=0.2em
\begin{array}{lcl}
\iprods{G_{\lambda}x - G_{\lambda}y, x - y} \geq \bar{\beta} \norms{G_{\lambda}x - G_{\lambda}y}^2 + \Lambda L \iprods{Fx - Fy, x - y}.
\end{array} 
\end{equation}

\item[$\mathrm{(ii)}$] 
if $S_{\lambda}$ is defined by \eqref{eq:BFS_residual}, then for all $u, v \in \R^p$, we have 
\begin{equation}\label{eq:S_cocoerciveness} 
\hspace{-2ex}
\arraycolsep=0.1em
\begin{array}{lcl}
\iprods{S_{\lambda}u - S_{\lambda}v, u - v}  & \geq & \bar{\beta} \norms{S_{\lambda}u - S_{\lambda}v}^2 + \Lambda L  \iprods{F(J_{\lambda T}u) - F(J_{\lambda T}v), J_{\lambda T}u  - J_{\lambda T}v}.
\end{array} 
\hspace{-1ex}
\end{equation}
\item[$\mathrm{(iii)}$]  
if additionally $\lambda \geq 2\rho$, then $\norms{J_{\lambda T}x - J_{\lambda T}y} \leq \norms{x - y}$ for all $x, y \in \dom{T}$, i.e., the resolvent $J_{\lambda T}$ is nonexpansive.

\end{compactitem}
\end{lemma}
Unlike deterministic methods that only require the \rv{co-coercivity} of $G_{\lambda}$ and $S_{\lambda}$, which were already proven in the literature, see, e.g., \citep{Bauschke2011}, we need the new bounds \eqref{eq:G_cocoerciveness}  and \eqref{eq:S_cocoerciveness} for our analysis, which allow us to cover a $\rho$-co-hypomonotone operator $T$ in \eqref{eq:NI}.
This mapping is not necessarily monotone as demonstrated earlier.

%%%%%%%%%%%%%%%%%%%%%%%%%%%%%%%%%%%%%%%%%%%%%%%%
%%%% 3. A Class of Variance-Reduced Estimators.
%%%%%%%%%%%%%%%%%%%%%%%%%%%%%%%%%%%%%%%%%%%%%%%%
\beforesec
\section{A Class of Variance-Reduced Estimators}\label{sec:VR_estimators}
\aftersec
We are given an unbiased stochastic oracle $\mbf{F}(\cdot,\xi)$ of $F$ such that for any $x \in \dom{F}$ we have $Fx = \Expsn{\xi}{\mbf{F}(x, \xi)}$ in both settings \eqref{eq:finite_sum_form} and \eqref{eq:expectation_form}.
We denote by $\widetilde{F}(x, \Sc)$ an unbiased  stochastic estimator of $Fx$  constructed from an i.i.d. sample $\Sc$ of $\xi$ by querying  $\mbf{F}(\cdot,\xi)$.

%%%%%%%%%%%%%%%%%%%%%%%%%%%%%%%%%%%%%%%%%%%%%%%%
%%%% 3.1. The Class of Variance-Reduced Stochastic Estimators.
%%%%%%%%%%%%%%%%%%%%%%%%%%%%%%%%%%%%%%%%%%%%%%%%
\beforesubsec
\subsection{The class of variance-reduced stochastic estimators}\label{subsec:Vr_Estimator}
\aftersubsec
Given a sequence of iterates $\sets{x^k}$ generated by our algorithm, we consider the following class of variance-reduced estimators $\widetilde{F}^k$ of $Fx^k$ that covers various instances specified later.

%%% Definition 2.1.
\begin{definition}\label{de:VR_Estimators}
Let $\sets{x^k}_{k \geq 0}$ be a given sequence of iterates and $\widetilde{F}^k := \widetilde{F}(x^k, \Sc_k)$ be a stochastic estimator of $Fx^k$ constructed from an i.i.d. sample $\Sc_k$ adapted to the filtration $\Fc_k$.
We say that $\widetilde{F}^k$ satisfies a $\textbf{VR}(\Delta_k; \kappa_k, \Theta_k, \sigma_k)$  $($variance-reduction$)$  property if there exist three parameters $\kappa_k \in (0, 1]$, $\Theta_k \geq  0$, and $\sigma_k \geq 0$, and a random quantity $\Delta_k \geq 0$  such that
\myeq{eq:VR_property}{
\arraycolsep=0.2em
\begin{array}{lcl}
\Expsn{k}{ \norms{\widetilde{F}^k - Fx^k }^2 } & \leq & \Expsn{k}{ \Delta_k }, \vspace{1ex}\\
\Expsn{k}{ \Delta_k }  & \leq & (1- \kappa_k)\Delta_{k-1}  + \Theta_k  \bar{\Ec}_k + \sigma_k^2,
\end{array}
}
almost surely for $k\geq 0$, where $\bar{\Ec}_k := \Expsn{\xi}{\norms{\mathbf{F}(x^{k}, \xi) - \mbf{F}(x^{k-1}, \xi)}^2 }$, $x^{-1} := x^0$, and $\Delta_{-1} := 0$.
\end{definition}

Note that $\widetilde{F}^k$ covers both unbiased and biased estimators of $Fx^k$ since we do not impose the unbiased condition $\Expsn{\Sc_k}{\widetilde{F}^k} = Fx^k$.
In the finite-sum setting \eqref{eq:finite_sum_form}, $\bar{\Ec}_k$ in Definition~\ref{de:VR_Estimators} reduces to $\bar{\Ec}_k := \frac{1}{n}\sum_{i=1}^n  \norms{ F_ix^{k} - F_ix^{k-1} }^2$.
Moreover, from \eqref{eq:VR_property}, we also have $\Expn{\Delta_0} \leq \sigma_0^2$.

We highlight that Definition~\ref{de:VR_Estimators} is different from existing works, including \citep{driggs2019accelerating}, as we only require the condition~\eqref{eq:VR_property}, making it broader to cover several existing estimators that may violate the definition in \citet{driggs2019accelerating}.
It is also broader than and different from the class of unbiased estimators in  \citet{tran2024accelerated} and \citet{TranDinh2024}.
Our class is also different from other general classes of estimators such as \citep{beznosikov2023stochastic,gorbunov2022stochastic,loizou2021stochastic} for stochastic optimization algorithms since they require the unbiasedness and  additional or different conditions.

\beforesubsec
\subsection{Concrete variance-reduced estimators}\label{subsec:VR_estimator_examples}
\aftersubsec
In this section we present four different variance-reduced estimators satisfying Definition~\ref{de:VR_Estimators}: Loopless-SVRG, SAGA, Loopless-SARAH, and Hybrid-SGD, which will be used to develop methods for solving \eqref{eq:NI} in this paper.
Note that the Loopless-SVRG and SAGA are unbiased, while the Loopless-SARAH and Hybrid-SGD are biased.
Though we only focus on these four estimators, we believe that other variance-reduced estimators such as SAG \citep{LeRoux2012,schmidt2017minimizing}, SARGE \citep{driggs2019bias}, SEGA \citep{hanzely2018sega}, and JacSketch \citep{gower2021stochastic} can possibly be used in our methods.

%%% 1. Loopless-SVRG estimator.
\beforesubsubsec
\subsubsection{Loopless-SVRG estimator}\label{subsubsec:svrg_estimator}
\aftersubsubsec
The SVRG estimator was introduced in \citet{SVRG}, and its loopless version was proposed in \citet{kovalev2019don}.
We show that this estimator satisfies Definition~\ref{de:VR_Estimators}.

For given $Fx$ in \eqref{eq:NI}, its unbiased stochastic oracle $\mbf{F}(\cdot,\xi)$, two iterates $x^k$ and $\tilde{x}^k$, and an i.i.d. sample $\Sc_k$ of  size $b_k$, we construct
\begin{equation}\label{eq:loopless_svrg}
\widetilde{F}^k := \bar{F}\tilde{x}^k  + \Fb(x^k, \Sc_k) -  \Fb(\tilde{x}^k,  \Sc_k),
\tag{L-SVRG}
\end{equation}
where, for $k\geq 1$,  $\tilde{x}^k$ is updated by
\begin{equation}\label{eq:xy_hat}
\tilde{x}^{k}  := \begin{cases}
x^{k-1} &\text{with probability}~\mbf{p}_{k} \vspace{1ex}\\
\tilde{x}^{k-1} &\text{with probability}~1-\mbf{p}_{k}.
\end{cases}
\end{equation}
Here, $\mbf{p}_k \in [\underline{\mbf{p}}, 1)$ is a given probability, $\tilde{x}^0 := x^0$, and $\Fb(\cdot, \Sc_k) := \frac{1}{b_k}\sum_{\xi_i \in \Sc_k}\Fb(\cdot, \xi_i)$ is a mini-batch estimator of $F(\cdot)$.
The quantity $\bar{F}\tilde{x}^k$ is constructed by one of the two options:
\begin{compactitem}
\item \textbf{Full-batch evaluation.} We choose $\bar{F}\tilde{x}^k  := F\tilde{x}^k$.
\item \textbf{Mega-batch evaluation.} We compute $\bar{F}\tilde{x}^k := \frac{1}{n_k}\sum_{\xi \in \bar{\Sc}_k}\mbf{F}(\tilde{x}^k, \xi)$, an unbiased estimator of $F\tilde{x}^k$ using a mega-batch $\bar{\Sc}_k$ of size $n_k$.
Then, we have $\Expsb{\bar{\Sc}_k}{\bar{F}\tilde{x}^k} = F\tilde{x}^k$ and $\Expsb{\bar{\Sc}_k}{ \norms{\bar{F}\tilde{x}^k - F\tilde{x}^k}^2 } \leq \frac{\sigma^2}{n_k}$, where $\sigma^2$ is given in Assumption~\ref{as:A1}(ii) and $n_k \geq n_{k-1}$.
\end{compactitem}
We assume that the full batch evaluation $\bar{F}\tilde{x}^k  := F\tilde{x}^k$ is computed for the finite-sum setting \eqref{eq:finite_sum_form}, while we often use a mega-batch evaluation for the expectation setting \eqref{eq:expectation_form}.

%%%%%%%%
The following lemma shows that $\widetilde{F}^k$ satisfies Definition~\ref{de:VR_Estimators}, whose proof is in Appendix~\ref{apdx:proof_SVRG_estimator_bound}.

%%% Lemma 7.
\vspace{-1ex}
\begin{lemma}\label{le:loopless_svrg_bound}
Let $\widetilde{F}^k$ be constructed by \eqref{eq:loopless_svrg} and $\hat{\Delta}_k := \frac{1}{b_k}\Expsn{\xi}{ \norms{\Fb(x^k, \xi) - \Fb(\tilde{x}^k, \xi)}^2 }$ for $b_k \geq b_{k-1}$.
Then, for any $\tau > 0$, we have 
\begin{equation}\label{eq:loopless_svrg_bound2}
\Expsn{k}{\norms{\widetilde{F}^k - Fx^k}^2}  \leq  (1+\tau) \Expsn{k}{\hat{\Delta}_k} +  \tfrac{1+\tau}{\tau}\Expsn{k}{ \norms{ \bar{F}\tilde{x}^k - F\tilde{x}^k }^2 }.  
\end{equation}
For any $\alpha \in (0, 1)$ and $\bar{\Ec}_k$ defined in Definition~\ref{de:VR_Estimators}, we have
\begin{equation}\label{eq:loopless_svrg_bound}
\arraycolsep=0.2em
\begin{array}{lcl}
\Expsn{k}{ \hat{\Delta}_k } \leq  \left(1- \alpha \mbf{p}_k \right) \hat{\Delta}_{k-1}  + \frac{1}{(1-\alpha) b_k \mbf{p}_k} \cdot \bar{\Ec}_k.
\end{array}
\end{equation}
Consequently, $\widetilde{F}^k$ satisfies the $\textbf{VR}(\Delta_k; \kappa_k, \Theta_k,  \sigma_k)$ property in Definition~\ref{de:VR_Estimators} with $\Delta_k := 2\hat{\Delta}_k + \frac{2\sigma^2}{\tau n_k}$, $\kappa_k := \alpha \mbf{p}_k$, $\Theta_k := \frac{2}{ (1-\alpha) b_k \mbf{p}_k}$,  and $\sigma_k^2 := \frac{2\alpha\mbf{p}_k \sigma^2}{n_k}$.

In particular, if we choose $\bar{F}\tilde{x}^k := F\tilde{x}^k$, then $\widetilde{F}^k$ satisfies the $\textbf{VR}(\Delta_k; \kappa_k, \Theta_k, \sigma_k)$ property in Definition~\ref{de:VR_Estimators} with $\Delta_k := \hat{\Delta}_k$, $\kappa_k := \alpha \mbf{p}_k$, $\Theta_k := \frac{1}{ (1-\alpha) b_k \mbf{p}_k}$, and $\sigma^2_k := 0$.
\end{lemma}

%%% 2. SAGA Estimator.
\beforesubsubsec
\subsubsection{SAGA estimator for finite-sum setting \eqref{eq:finite_sum_form}}\label{subsubsec:saga_estimator}
\aftersubsubsec
The SAGA estimator was introduced in \citet{Defazio2014} for solving finite-sum optimization problems.
We apply it to the finite-sum setting \eqref{eq:finite_sum_form}.
It  is constructed as follows.

For a given $F$ defined in the finite-sum setting \eqref{eq:finite_sum_form} of \eqref{eq:NI}, a given sequence $\sets{x^k}_{k\geq 0}$, and a given i.i.d. sample $\Sc_k$ of size $b_k$, for $k\geq 1$, we update $\hat{F}^k_i$ for all $i\in [n]$ as
\begin{equation}\label{eq:SAGA_ref_points}
\hat{F}_i^k := \left\{\begin{array}{lcl}
F_ix^{k-1}  &\text{if}~ i \in \Sc_k, \vspace{1ex} \\
\hat{F}_i^{k-1} & \text{if}~i \notin \Sc_k.
\end{array}\right.
\end{equation}
Then, we  construct a SAGA estimator for $Fx^k$ as follows:
\begin{equation}\label{eq:SAGA_estimator}
\arraycolsep=0.2em
\begin{array}{lcl}
\widetilde{F}^k & := & \frac{1}{n} \sum_{i=1}^n\hat{F}_i^{k}  + F_{\Sc_k}x^k  - \hat{F}_{\Sc_k}^{k}, 
\end{array}
\tag{SAGA}
\end{equation}
where $F_{\Sc_k}x^k := \frac{1}{b_k}\sum_{i\in\Sc_k}F_ix^k$ and $\hat{F}_{\Sc_k}^{k} := \frac{1}{b_k}\sum_{i\in\Sc_k}\hat{F}_i^{k}$.
Moreover, we need to store $n$ component $\hat{F}^k_i$ computed so far for all $i \in [n]$ in a table $\mbf{T}_k := [\hat{F}_1^k, \hat{F}_2^k, \cdots, \hat{F}_n^k]$ initialized at $\hat{F}_i^0 := F_ix^0$ for all $i \in [n]$.
SAGA requires significant memory to store $\mbf{T}_k$ if $n$ and $p$ are both large.
We have the following result, whose proof is given in Appendix~\ref{apdx:proof_SAGA_estimator_bound}.

%%% Lemma 2.2.
\vspace{-1ex}
\begin{lemma}\label{le:SAGA_estimator_full}
Let $\widetilde{F}^k$ be constructed by \eqref{eq:SAGA_estimator} such that $b_{k-1} - \frac{(1-\alpha)b_kb_{k-1}}{2n} \leq b_k \leq b_{k-1}$ for all $k\geq 1$ and a given $\alpha \in (0, 1)$, and $\Delta_k :=   \frac{1}{nb_k} \sum_{i=1}^n \norms{ F_ix^k -  \hat{F}_{i}^k }^2$.
Then, we have 
\vspace{-0.5ex}
\begin{equation}\label{eq:SAGA_variance}
\left\{\begin{array}{lcl}
\Expsn{k}{ \norms{\widetilde{F}^k - Fx^k }^2 }  & \leq &  \Expsn{k}{ \Delta_k }, \vspace{1ex}\\
\Expsn{k}{ \Delta_k } & \leq &  \big(1 -  \frac{\alpha b_k}{n} \big)  \Delta_{k-1}   +   \frac{\Theta_k }{n} \sum_{i=1}^n \norms{F_ix^k - F_ix^{k-1}}^2,
\end{array}\right.
\vspace{-0.5ex}
\end{equation}
where $\Theta_k := \frac{(3 - \alpha)n}{(1-\alpha)b_k^2}$.

Consequently, $\widetilde{F}^k$ satisfies the $\textbf{VR}(\Delta_k; \kappa_k, \Theta_k, \sigma_k)$ property in Definition~\ref{de:VR_Estimators} with $\kappa_k := \frac{\alpha b_k }{n} \in (0, 1]$, $\Delta_k$ and $\Theta_k$ given above, and $\sigma_k^2 = 0$.
\end{lemma}

%%% 3. Loopless-SARAH.
\beforesubsubsec
\subsubsection{Loopless-SARAH estimator}\label{subsubsec:sarah_estimator}
\aftersubsubsec
The SARAH estimator was introduced in \citet{nguyen2017sarah} for finite-sum convex optimization.
Its loopless variant  was studied in \citet{driggs2019accelerating,Li2019,li2020page}, and recently in \citet{cai2022stochastic,cai2023variance}.
%%% Add reference: Li2019 -- Reference - R1.
This estimator is constructed as follows.

Given $\sets{x^k}$ and an i.i.d. sample $\Sc_k$ of size $b_k$, we construct an estimator $\widetilde{F}^k$ of $Fx^k$ as
\begin{equation}\label{eq:loopless_sarah}
\widetilde{F}^k := \begin{cases} 
\widetilde{F}^{k-1}  + \Fb(x^k, \Sc_k) - \Fb(x^{k-1}, \Sc_k) &\text{with probability}~1 - \mbf{p}_k, \vspace{1ex} \\
\bar{F}x^k &\text{with probability}~\mbf{p}_k,
\end{cases} 
\tag{L-SARAH}
\end{equation}
where $\Fb(\cdot, \Sc_k) := \frac{1}{b_k}\sum_{\xi_i\in\Sc_k}\Fb(\cdot, \xi_i)$, $\widetilde{F}^0 := \bar{F}x^0$, and $\mbf{p}_k \in [\underline{\mbf{p}}, 1)$ is a given probability of the switching rule in \eqref{eq:loopless_sarah}.
The quantity $\bar{F}x^k$ is constructed by one of the two options:
\begin{compactitem}
\item \textbf{Full-batch evaluation.} We compute $\bar{F}x^k  := Fx^k$.
\item \textbf{Mega-batch evaluation.} We compute $\bar{F}x^k := \frac{1}{n_k}\sum_{\xi\in\bar{\Sc}_k}\mbf{F}(x^k, \xi)$ using a mega-batch $\bar{\Sc}_k$ of size $n_k$.
In this case, we have $\Expsn{\bar{\Sc}_k}{\bar{F}x^k} = Fx^k$ and $\Expsn{\bar{\Sc}_k}{ \norms{\bar{F}x^k - Fx^k}^2 } \leq \frac{\sigma^2}{n_k}$, where $\sigma^2$ is given in Assumption~\ref{as:A1}(ii).
\end{compactitem}
The following lemma shows that $\widetilde{F}^k$ satisfies Definition~\ref{de:VR_Estimators}, whose proof is in Appendix~\ref{apdx:proof_SARAH_estimator_bound}.

%%% Lemma 8.
\vspace{-1ex}
\begin{lemma}\label{le:loopless_sarah_bound}
Let $\widetilde{F}_k$ be constructed by \eqref{eq:loopless_sarah} and $\bar{\Ec}_k$ be defined in Definition~\ref{de:VR_Estimators}.
Then
\begin{equation}\label{eq:loopless_sarah_var_bound}
\arraycolsep=0.2em
\begin{array}{lcl}
\Expsn{k}{\norms{\widetilde{F}^k - Fx^k }^2 }  & \leq & (1 - \mbf{p}_k)  \norms{ \widetilde{F}^{k-1} - Fx^{k-1} }^2   +  \mbf{p}_k  \norms{\bar{F}x^k - Fx^k }^2 + \frac{ 1-\mbf{p}_k }{b_k} \cdot \bar{\Ec}_k.
\end{array}
\end{equation}
Consequently, $\widetilde{F}^k$ satisfies the $\textbf{VR}(\Delta_k; \kappa_k, \Theta_k, \sigma_k)$ property in Definition~\ref{de:VR_Estimators} with 
$\Delta_k :=  \norms{\widetilde{F}^k - Fx^k }^2$,  $\kappa_k =  \mbf{p}_k$,  $\Theta_k  := \frac{1}{b_k}$, and $\sigma_k^2 := \frac{\mbf{p}_k \sigma^2}{n_k}$.
In particular, if we choose $\bar{F}x^k := Fx^k$, then $\widetilde{F}^k$ satisfies Definition~\ref{de:VR_Estimators} with the same $\Delta_k$, $\kappa_k$, and $\Theta_k$, but with $\sigma_k = 0$. 
\end{lemma}

%%% 1. Hybrid SGD estimator.
\vspace{-2ex}
\beforesubsubsec
\subsubsection{Hybrid SGD estimator}\label{subsec:HSGD_estimator}
\aftersubsubsec
\rvs{The hybrid stochastic gradient estimator (HSGD), was introduced in  \citet{Tran-Dinh2019,Tran-Dinh2019a}} to construct biased variance-reduced estimators for nonconvex optimization. 
We extend it here for operator $F$ to solve \eqref{eq:NI}.
It is constructed as follows.

Given  $\sets{x^k}$ generated by our algorithm, an initial estimate $\widetilde{F}^0$ such that $\Expsn{0}{\norms{\widetilde{F}^0 - Fx^0}^2 } \leq \frac{\sigma^2}{n_0}$, and an i.i.d. sample $\Sc_k$ of size $b_k$, we construct $\widetilde{F}^k$ for $Fx^k$ as follows:
\begin{equation}\label{eq:HSGD_estimator}
\arraycolsep=0.2em
\begin{array}{lcl}
\widetilde{F}^k  := (1 - \tau_k)\big[ \widetilde{F}^{k-1} + \Fb(x^k, \Sc_k) - \Fb(x^{k-1}, \Sc_k) \big] +  \tau_k \bar{F}x^k,
\end{array}
\tag{HSGD}
\end{equation}
where $\tau_k \in [0, 1]$  is a given weight, $\bar{F}x^k$ is an unbiased estimator of $Fx^k$ constructed from an i.i.d. sample $\hat{\Sc}_k$ of size $\hat{b}_k$, i.e., $\Expsn{\hat{\Sc}_k}{ \bar{F}x^k } = Fx^k$ and $\Expsn{\hat{\Sc}_k}{ \norms{ \bar{F}x^k - Fx^k}^2 } \leq \frac{\sigma^2}{\hat{b}_k}$ for $k\geq 1$.
Here, we allow $\Sc_k$ and $\hat{\Sc}_k$ to be dependent or even identical.
Our \ref{eq:HSGD_estimator} estimator covers the following special cases.
\begin{compactitem}
\item If $\tau_k = 0$, then it reduces to the SARAH estimator \citep{nguyen2017sarah}.
\item If $\tau_k = 1$, then $\widetilde{F}^k = \bar{F}x^k$ as a mini-batch unbiased estimator.
\item If $\bar{F}x^k = \Fb(x^k, \Sc_k)$ (i.e., $\Sc_k \equiv \hat{\Sc}_k$), then $\widetilde{F}_k$ reduces to the STORM estimator in \citet{Cutkosky2019}.
\end{compactitem}
The following lemma provides a key property of \eqref{eq:HSGD_estimator}, whose proof is in Appendix~\ref{apdx:proof_HSGD_estimator_bound}.

%%% Lemma 1.
\vspace{-1ex}
\begin{lemma}\label{le:HSGD_estimator_bound}
Let $\widetilde{F}^k$ be constructed by \eqref{eq:HSGD_estimator}, $\Delta_k :=  \norms{ \widetilde{F}^k  - Fx^k }^2$, $\delta_k^2 :=  \Expsn{\hat{\Sc}_k }{\norms{\bar{F}x^k  - Fx^k}^2 }$, and $\bar{\Ec}_k$ be defined in Definition~\ref{de:VR_Estimators}.
Then, the following statements hold.
\begin{compactitem}
\item[$\mathrm{(i)}$] If $\Sc_k$ is independent of $\hat{\Sc}_k$, then
\begin{equation}\label{eq:HSGD_key_est_a}
\hspace{-0.5ex}
\arraycolsep=0.1em
\begin{array}{lcl}
\Expsn{k}{ \Delta_k } &\leq & (1-\tau_k)^2  \Delta_{k-1}  +  \frac{ (1-\tau_k)^2}{b_k} \bar{\Ec}_k  +  \tau_k^2 \delta_k^2.
\end{array}
\hspace{-3ex}
\end{equation}
\item[$\mathrm{(ii)}$] If $\Sc_k$ and $\hat{\Sc}_k$ are dependent or identical, then
\begin{equation}\label{eq:HSGD_key_est_b}
\hspace{-0.5ex}
\arraycolsep=0.2em
\begin{array}{lcl}
\Expsn{k}{ \Delta_k } &\leq & (1-\tau_k)^2  \Delta_{k-1}  +  \frac{ 2(1-\tau_k)^2}{b_k} \bar{\Ec}_k + 2\tau_k^2 \delta_k^2.
\end{array}
\hspace{-3ex}
\end{equation}
\item[$\mathrm{(iii)}$] The estimator $\widetilde{F}^k$ satisfies the $\textbf{VR}(\Delta_k; \kappa_k, \Theta_k, \sigma_k)$ property in Definition~\ref{de:VR_Estimators}  with  $\Delta_k$ given above, $\kappa_k := 1 - (1-\tau_k)^2$, and
\begin{compactitem}
\item[$\bullet$] $\Theta_k = \frac{(1-\tau_k)^2}{b_k}$ and $\sigma_k^2 = \frac{\tau_k^2\sigma^2}{\hat{b}_k}$, if $\Sc_k$ is independent of $\hat{\Sc}_k$;
\item[$\bullet$] $\Theta_k = \frac{ 2(1-\tau_k)^2}{b_k}$ and $\sigma_k^2 = \frac{2\tau_k^2 \sigma^2}{ b_k}$, otherwise.
\end{compactitem}
\end{compactitem}
\end{lemma}

%%%%%%%%%%%%%%%%%%%%%%%%%%%%%%%%%%%%%%%%%%%
%%%% 3.4. Variance-Reduced Accelerated FBS Methods.
%%%%%%%%%%%%%%%%%%%%%%%%%%%%%%%%%%%%%%%%%%%
\vspace{-1ex}
\beforesec
\section{Variance-Reduced Accelerated Forward-Backward Splitting Method}\label{sec:VR_AFBS_method}
\aftersec
In this section, we propose a novel \textbf{V}ariance-reduced \textbf{F}ast \textbf{O}perator \textbf{S}plitting (forward-backward splitting)  \textbf{A}lgorithm to solve \eqref{eq:NI},  abbreviated by \eqref{eq:VrAFBS4NI}.
Instead of using a specific stochastic estimator as in many existing works, like \citep{cai2023variance,davis2022variance}, we develop an algorithmic framework that covers all estimators satisfying Definition~\ref{de:VR_Estimators}.

%%%%%%%%%%%%%%%%%%%%%%%%%%%%%%%%%%%%%%%%%%%%%%%%
%%% 3.2. A class of variance-reduced accelerated fixed-point algorithms
%%%%%%%%%%%%%%%%%%%%%%%%%%%%%%%%%%%%%%%%%%%%%%%%
\vspace{-0.5ex}
\beforesubsec
\subsection{Variance-reduced fast FBS algorithmic framework}\label{subsec:VR_AFP4NI}
\aftersubsec
\noindent\textbf{$\mathrm{(a)}$~The proposed algorithm.}
Motivated by Nesterov's acceleration techniques \citep{Nesterov1983,Nesterov2004}, our \ref{eq:VrAFBS4NI} framework for solving \eqref{eq:NI} is presented as follows:
Starting from $x^0 \in \dom{\Phi}$, set $z^0 := x^0$, and at each iteration $k \geq 0$, we update
\begin{equation}\label{eq:VrAFBS4NI}
\arraycolsep=0.2em
\left\{\begin{array}{lcl}
y^k & := &  \frac{t_k - 1}{t_k} x^k + \frac{1}{t_k} z^k, \vspace{1ex}\\
x^{k+1} &:= & y^k - \eta_k \widetilde{G}_{\lambda}^k, \vspace{1ex}\\
z^{k+1} &:= &  z^k + \nu(x^{k+1} - y^k),
\end{array}\right.
\tag{VFOSA$_{+}$}
\end{equation}
where $t_k > 0$, $\eta_k > 0$, and $\nu \in (0, 1]$ are given parameters, determined later.
Here, $ \widetilde{G}_{\lambda}^k$ is a stochastic estimator of $G_{\lambda}x^k$ defined by \eqref{eq:FBS_residual}, which is constructed as follows:
\vspace{-0.5ex}
\begin{equation}\label{eq:VrAFBS4NI_estimator}
\arraycolsep=0.2em
\begin{array}{lcl}
 \widetilde{G}_{\lambda}^k := \frac{1}{\lambda}\big(x^k - J_{\lambda T}(x^k - \lambda \widetilde{F}^k) \big),
\end{array}
\vspace{-0.5ex}
\end{equation}
where $\widetilde{F}^k$ is a stochastic estimator of $Fx^k$ satisfying the $\mathbf{VR}(\Delta_k; \kappa_k, \Theta_k, \sigma_k)$ property in Definition~\ref{de:VR_Estimators}.
By the non-expansiveness of $J_{\lambda T}$ from Lemma~\ref{le:FBS_cocoerciveness}, one can easily show that
\begin{equation}\label{eq:G_estimator_bound1}
\norms{ \widetilde{G}_{\lambda}^k - G_{\lambda}x^k} \leq \norms{\widetilde{F}^k - Fx^k}.
\end{equation}
This relation shows that if $\widetilde{F}^k$ well approximates $Fx^k$, then $\widetilde{G}_{\lambda}^k$ well approximates $G_{\lambda}x^k$.

\vspace{0.5ex}
\noindent\textbf{$\mathrm{(b)}$~The implementation version.}
%%\rvs{\todo{Using pseudo code to present the algorithm. For the specific settings, they can be presented in the appendix.}}
By combining \eqref{eq:VrAFBS4NI}, \eqref{eq:VrAFBS4NI_estimator}, and the update rules in Theorem~\ref{th:VrAFBS4NI_convergence}, we obtain Algorithm~\ref{alg:A1}, which is presented for implementation purposes.

%%%% Algorithm 1.
\rv{
\begin{algorithm}[hpt!]\caption{(\textbf{V}ariance-reduced \textbf{F}ast [FB] \textbf{O}perator \textbf{S}plitting \textbf{A}lgorithm (\textbf{VFOSA$_{+}$}))}\label{alg:A1}
\normalsize
\begin{algorithmic}[1]
\State\label{step:A1_i0}{\bfseries Initialization:} Take an initial point $x^0 \in \dom{\Phi}$.

\State \hspace{2.5ex}Choose $\mu$, $\lambda$, and $\beta$ from Theorem~\ref{th:VrAFBS4NI_convergence}.
Set $\nu := \frac{\mu}{2}$, $r := 2 + \frac{1}{\mu}$, and $z^0 := x^0$.
\State\hspace{0ex}\label{step:A1_o1}{\bfseries For $k := 0,\cdots, k_{\max}$ do}
\vspace{0.25ex}   
\State\hspace{2.5ex}\label{step:A1_o2}Construct an estimator $\widetilde{F}^k$ of $Fx^k$ satisfying Definition~\ref{de:VR_Estimators}.
\State\hspace{2.5ex}\label{step:A1_o3}Update $t_k := \mu(k + r)$ and $\eta_k := \frac{2\beta(t_k - 1)}{t_k - \nu}$.
\State\hspace{2.5ex}\label{step:A1_o4}Update the following iterates:
\vspace{-0.5ex}
\myeq{eq:VrAFBS4NI_impl}{
\arraycolsep=0.2em
\left\{\begin{array}{lcl}
y^k & := &  \frac{t_k - 1}{t_k} x^k + \frac{1}{t_k} z^k, \vspace{1ex}\\
w^k &:= &  J_{\lambda T}(x^k - \lambda \widetilde{F}^k), \vspace{1ex}\\
x^{k+1} &:= & y^k - \frac{\eta_k}{\lambda}(x^k - w^k), \vspace{1ex}\\
z^{k+1} &:= &  z^k + \nu(x^{k+1} - y^k).
\end{array}\right.
\vspace{-0.5ex}
}
\State\hspace{0ex}{\bfseries End For}
\end{algorithmic}
\end{algorithm}
}

%%%%
Algorithm~\ref{alg:A1} is single-loop, and at each iteration $k$, it requires one evaluation $\widetilde{F}^k$ of $Fx^k$ and one evaluation $J_{\lambda T}$ of $T$, while other steps are only scalar-vector multiplications or vector additions.
\rv{For simplicity of implementation, we can use the following parameters:
\begin{compactitem}
\item Choose $\mu := 0.95\cdot\frac{2}{3}$ and $r := 5$ (but other values of $r$ such as $r  = 10$ still work).
\item Given an estimate of $L$, choose $\hat{L} := L + \zeta$ and $\lambda := \frac{1}{L+\zeta}$ for some small $\zeta > 0$.
\item Given $\rho \geq 0$, compute $\bar{\beta} :=  \frac{\lambda(4 - \hat{L}\lambda) - 4\rho}{4(1 - \rho\hat{L})}$ and set $\beta := \frac{(2-\mu)\bar{\beta}}{2+\mu}$.
In particular, if $\rho = 0$ (i.e., $T$ is monotone), then we can choose $\lambda := \frac{1}{\hat{L}}$ and $\bar{\beta} := \frac{\lambda(4 - \hat{L}\lambda)}{4}$.
\end{compactitem}
This parameter configuration reflects what we mainly used for our experiments in Section~\ref{sec:num_examples}.
However, in practice, depending on applications, we can find appropriate parameters which possibly improve the performance of Algorithm~\ref{alg:A1}, while still still satisfying Theorem~\ref{th:VrAFBS4NI_convergence}.
}

\vspace{0.5ex}
\noindent\textbf{$\mathrm{(c)}$~Comparison to Nesterov's acceleration in convex optimization.}
Suppose that we apply \eqref{eq:VrAFBS4NI} to solve the composite convex minimization problem \eqref{eq:opt_prob}, where $f$ is convex and $L$-smooth and $g$ is proper, closed, and convex. 
In this case, Assumptions~\ref{as:A1}(iii) and \ref{as:A2} automatically hold.
The key step is $x^{k+1} := y^k - \eta_k\widetilde{G}_{\lambda}^k$, which becomes $x^{k+1} := y^k - \frac{\eta_k}{\lambda}(x^k - \prox_{\lambda g}(x^k - \lambda\widetilde{\nabla}{f}(x^k))$. 
This is a proximal-gradient step using the gradient mapping $\mcal{G}_{\lambda}(x) := \lambda^{-1}(x - \prox_{\lambda g}(x - \lambda \nabla{f}(x)))$.
Thus, \eqref{eq:VrAFBS4NI} reduces to
\vspace{-0.5ex}
\begin{equation}\label{eq:VrAFP4Cvx}
\arraycolsep=0.2em
\left\{\begin{array}{lcl}
y^k & := &  \big(1 - \frac{1}{t_k} \big) x^k + \frac{1}{t_k} z^k, \vspace{1ex}\\
x^{k+1} & := & y^k - \frac{\eta_k}{\lambda}(x^k - \prox_{\lambda g}(x^k - \lambda\widetilde{\nabla}{f}(x^k)), \vspace{1ex}\\
z^{k+1} &:= &  z^k + \nu(x^{k+1} - y^k),
\end{array}\right.
\vspace{-0.5ex}
\end{equation}
where $\widetilde{\nabla}{f}(x^k)$ is a stochastic estimator of $\nabla{f}(x^k)$.
This scheme is a new algorithm for solving the convex optimization problem \eqref{eq:opt_prob}.

\rvs{Note that the proximal-gradient variant of Nesterov's accelerated methods \citep{Nesterov1983} is equivalent to FISTA in \citep{Beck2009}}  for solving \eqref{eq:opt_prob}, which can be expressed as follows using the gradient mapping value $\mcal{G}_{\lambda}(y^k)$ and $\lambda = \eta_k$:
\vspace{-0.5ex}
\begin{equation}\label{eq:VrNes4Cvx}
\arraycolsep=0.2em
\left\{\begin{array}{lcl}
y^k & := &  \big(1 - \frac{1}{t_k} \big) x^k + \frac{1}{t_k} z^k, \vspace{1ex}\\
x^{k+1} & := & y^k - \frac{\eta_k}{\lambda}(y^k - \prox_{\lambda g}(y^k - \lambda\nabla{f}(y^k)) \equiv \prox_{\lambda g}(y^k - \lambda\nabla{f}(y^k)), \vspace{1ex}\\
z^{k+1} &:= &  z^k + t_k(x^{k+1} - y^k).
\end{array}\right.
\vspace{-0.5ex}
\end{equation}
Clearly,  our scheme \eqref{eq:VrAFP4Cvx} has some similarity to \eqref{eq:VrNes4Cvx} in terms of structure and iterate updates $y^k$ and $x^{k+1}$.
However, there are two key differences between \eqref{eq:VrAFP4Cvx} and \eqref{eq:VrNes4Cvx}.
\begin{compactitem}
\item First, \eqref{eq:VrAFP4Cvx} evaluates the gradient mapping  $\mcal{G}_{\lambda}$ at $x^k$ instead of at $y^k$ as in \eqref{eq:VrNes4Cvx}.
\item Second, \eqref{eq:VrAFP4Cvx} uses a fixed parameter $\nu \in (0, 1)$ in the last step instead of $t_k$ as in \eqref{eq:VrNes4Cvx}.
\end{compactitem}
These two differences fundamentally change the convergence analysis of our method.

Due to the second and third lines, \ref{eq:VrAFBS4NI} is also different from the accelerated schemes in  \citep{attouch2020convergence,kim2021accelerated,mainge2021fast,mainge2021accelerated,tran2022connection} since these methods were though derived from Nesterov's accelerated techniques  such as in \citep{Nesterov1983}, they were aided by a [gradient] correction term or a Hessian-driven damping term.

\beforesubsec
\subsection{Key estimates for convergence analysis}\label{subsec:VrAFBS_key_estimates}
\aftersubsec
\textbf{$\mathrm{(a)}$~Lyapunov function.}
To analyze the convergence of \eqref{eq:VrAFBS4NI}, 
for given $\Lambda \geq 0$ and $\bar{\beta} \geq 0$ in Lemma~\ref{le:FBS_cocoerciveness}, we construct the following functions w.r.t. the iteration counter $k$:
\begin{equation}\label{eq:VrAFBS_Lyapunov_func}
\arraycolsep=0.2em
\begin{array}{lcl}
\Lc_k & := &  \beta a_k \norms{G_{\lambda}x^k}^2 + t_{k-1} \iprods{G_{\lambda}x^k, x^k - z^k} + \frac{c_k}{2\nu\beta} \norms{z^k - x^{\star}}^2, \vspace{1ex}\\
\Qc_k &:= & \Lc_k + \big[ \bar{\beta} - (1+s) \beta  \big] t_{k-1} (t_{k-1} -1)  \norms{G_{\lambda}x^k - G_{\lambda}x^{k-1} }^2 \vspace{1ex}\\
&& + {~} \Lambda L t_{k-1}(t_{k-1}-1)  \iprods{Fx^k - Fx^{k-1}, x^k - x^{k-1}}, \vspace{1ex}\\
\Pc_k &:= & \Qc_k + \frac{[\mu(1-\kappa_k )\Gamma_k + \beta] t_{k-1}(t_{k-1} -1)}{2\mu}  \Delta_{k-1},

\end{array}
\end{equation}
where  $\beta > 0$, $\mu \in (0, 1]$, and $\nu \in [0, 1]$ are given in \eqref{eq:VrAFBS4NI}, and $s > 0$ and $\Gamma_k \geq 0$ are given parameters, determined later.
The quantities $\Delta_k$ and $\kappa_k$ are given in Definition~\ref{de:VR_Estimators}, and the coefficients $a_k$ and $c_k$ are respectively given by
\begin{equation}\label{eq:VrAFBS_coefficients}
\arraycolsep=0.2em
\begin{array}{lcl}
a_k & := & t_{k-1} [ t_{k-1} - 1 - s(1 - \nu) ] \quad \textrm{and} \quad  c_k  :=  \frac{(1-\mu) [(t_k-\nu)(t_{k-1}-1) + \mu(1 - \nu)]}{2(t_{k-1}-1)(t_k-1)}.
%q_k & := &  \big[ \frac{2\Lambda \mu(1-\kappa_k)}{\Theta_k} + \beta \big] t_{k-1}(t_{k-1}-1).
\end{array}
\end{equation}
%%%

\vspace{0.5ex}
\noindent\textbf{$\mathrm{(b)}$~Key lemmas.}
The following three lemmas provide key bounds for our analysis.
We first state the first important lemma, whose proof can be found in Appendix~\ref{apdx:le:VrAFBS4NI_descent_property}.

%%% Lemma 2.
\vspace{-1ex}
\begin{lemma}\label{le:VrAFBS4NI_descent_property}
Suppose that Assumptions~\ref{as:A1} and \ref{as:A2} hold for \eqref{eq:NI}.
Let  $\Lc_k$ be defined by \eqref{eq:VrAFBS_Lyapunov_func} and $\sets{(x^k, z^k)}$ be generated by \eqref{eq:VrAFBS4NI} using $t_k$ and $\eta_k$ respectively as
\begin{equation}\label{eq:VrAFBS4NI_tk_etak_update}
\begin{array}{l}
t_k := \mu(k+r) \quad \textrm{and} \quad \eta_k = \frac{2\beta(t_k - 1)}{t_k - \nu},
\end{array}
\end{equation}
for given $\mu \in (0, 1]$,  $r \geq 0$, and $\beta > 0$.
Then, for any $s > 0$, we have
\begin{equation}\label{eq:VrAFBS4NI_desecent_property}
\arraycolsep=0.2em
\begin{array}{lcl}
\Lc_k - \Lc_{k+1} & \geq & \beta \varphi_k  \norms{G_{\lambda}x^k}^2 + (1-\mu) \iprods{G_{\lambda} x^k, x^k - x^{\star}}   +  \Lambda t_k(t_k-1) \Ec_{k+1}   \vspace{1ex}\\
&& + {~} \big[ \bar{\beta} - (1+s)\beta  \big] t_k(t_k-1)\norms{G_{\lambda}x^{k+1} - G_{\lambda}x^k}^2 - \psi_k  \norms{e^k}^2,
\end{array}
\end{equation}
where $\bar{\beta}$ and $\Lambda$ are given constants in Lemma~\ref{le:FBS_cocoerciveness},  $e^k := \widetilde{F}^k - Fx^k$,  $\Ec_{k+1} :=  L \iprods{Fx^{k+1} - Fx^k, x^{k+1} - x^k}$,  and the coefficients $\varphi_k$ and $\psi_k$ are respectively given by
\begin{equation}\label{eq:VrAFBS4NI_mk_param}
\arraycolsep=0.2em
\begin{array}{lcl}
\varphi_k & := & t_{k-1} \big[ t_{k-1} - 1 - s(1 - \nu) \big] - \frac{ (t_k-1) }{t_k - \nu} \big[ t_k(t_k - 2 + \nu - s(1-\nu)) + 2(1-\mu)\nu \big], \vspace{1ex}\\
\psi_k &:= & \beta(t_k-1)\big[ \frac{t_k(t_k-1)}{s(t_k - \nu)} +  \frac{2\nu(1-\mu)}{t_k-\nu} + \frac{(1-\mu)\nu(t_{k-1} - 1) }{\mu(1-\nu)}  \big].
\end{array}
\end{equation}
\end{lemma}

Lemma~\ref{le:VrAFBS4NI_descent_property} is of independent interest as it can serve as a core step to analyze inexact variants of \eqref{eq:VrAFBS4NI} beyond this work.
For example, we can use it to analyze inexact methods (either deterministic or stochastic), where the approximation error $e^k := \widetilde{F}^k - Fx^k$ between $\widetilde{F}^k$ and $Fx^k$ is adaptively controlled along the iterations. 
In particular, we can assume that $\Expsn{k}{ \norms{e^k}^2 } \leq \delta_0 \norms{Fx^k -Fx^{k-1}}^2$ for a given $\delta_0 \geq 0$.

%%% The lower bound of $\Lc_k$.
Next, we show that $\Lc_k$ is lower bounded in Lemma~\ref{le:VrAFBS4NI_Lk_lowerbound}, whose proof is in Appendix~\ref{apdx:le:VrAFBS4NI_Lk_lowerbound}.

%%% Lemma 3.3.
\vspace{-1ex}
\begin{lemma}\label{le:VrAFBS4NI_Lk_lowerbound}
Under the same setting as in Lemma~\ref{le:VrAFBS4NI_descent_property}, $\Lc_k$ defined by \eqref{eq:VrAFBS_Lyapunov_func} satisfies
\begin{equation}\label{eq:VrAFBS4NI_Lyapunov_func_lowerbound}
\arraycolsep=0.2em
\begin{array}{lcl}
\Lc_k \geq \frac{A_k}{2}  \norms{G_{\lambda}x^k }^2 +  \frac{(t_{k-1}-1)[(1-\mu-2\nu)t_k + \nu(1+\mu)] + \mu(1-\mu)(1-\nu)}{4\nu\beta(t_{k-1}-1)(t_k-1)}    \norms{z^k - x^{\star}}^2,
\end{array}
\end{equation}
where $A_k := \beta t_{k-1}[ t_{k-1} - 2 - 2s(1 - \nu)] +   2\bar{\beta}t_{k-1}$.
Moreover, $\Qc_k$ and $\Pc_k$ defined by \eqref{eq:VrAFBS_Lyapunov_func} satisfy $\Pc_k \geq \Qc_k \geq \Lc_k \geq 0$.
\end{lemma}

%% Descent property of Lyapunov function.
Finally, Lemma~\ref{le:VrAFBS4NI_descent_property2} states a descent property of $\Pc_k$, whose proof is in Appendix~\ref{apdx:le:VrAFBS4NI_descent_property2}.

%%% Lemma 3.4.
\vspace{-1ex}
\begin{lemma}\label{le:VrAFBS4NI_descent_property2}
Under the same setting as in Lemma~\ref{le:VrAFBS4NI_descent_property} and $\lambda \geq 2\rho$, suppose further that $\kappa_k$ in Definition~\ref{de:VR_Estimators},  $\Gamma_k$ in \eqref{eq:VrAFBS_Lyapunov_func}, and $\psi_k$ in \eqref{eq:VrAFBS4NI_mk_param} satisfy the following condition:
\begin{equation}\label{eq:VrAFBS4NI_param2}
\arraycolsep=0.2em
\begin{array}{lcl}
2 \psi_k  + \big[ \Gamma_{k+1}(1-\kappa_{k+1})  + \frac{\beta}{\mu} \big] t_k(t_k-1) \leq \Gamma_k  t_{k-1}(t_{k-1}-1).
\end{array}
\end{equation}
Then, for $\bar{\Ec}_k$ defined in Definition \ref{de:VR_Estimators} and $\Pc_k$ defined by \eqref{eq:VrAFBS_Lyapunov_func}, we have
\begin{equation}\label{eq:VrAFBS4NI_desecent_property2}
\arraycolsep=0.2em
\begin{array}{lcl}
\Pc_k  & \geq &  \Expsn{k}{ \Pc_{k+1}} + \beta \varphi_k \norms{G_{\lambda}x^k}^2  +  (1-\mu)  \iprods{G_{\lambda} x^k, x^k - x^{\star}}  \vspace{1ex}\\
&& + {~} \big[ \bar{\beta} - (1+s)\beta  \big] t_{k-1}(t_{k-1} - 1) \norms{G_{\lambda}x^k - G_{\lambda}x^{k-1} }^2 \vspace{1ex}\\
&& - {~} \frac{ \Gamma_k t_{k-1}(t_{k-1} - 1)  }{2}  \sigma_k^2  + \frac{\beta t_{k-1}(t_{k-1} - 1) }{2\mu}\Delta_{k-1}  +  \frac{( 2\Lambda - \Gamma_k\Theta_k ) t_{k-1}(t_{k-1}-1)}{2} \bar{\Ec}_k.
\end{array}
\end{equation}
\end{lemma}
In the proofs of Lemmas~\ref{le:VrAFBS4NI_descent_property}, \ref{le:VrAFBS4NI_Lk_lowerbound}, and \ref{le:VrAFBS4NI_descent_property2}, we use Young's inequality several times.
We do not attempt to optimize its coefficients, resulting in somewhat loose bounds on our results.

%%%% 4.3. Convergence Analysis of SAFBS.
\beforesubsec
\subsection{Convergence analysis of \ref{eq:VrAFBS4NI}}\label{subsec:convergence_of_VrAFBS4NI}
\aftersubsec
Now, utilizing the above three technical lemmas, we are ready to state and prove our convergence results in the following theorem, whose proof is given in Appendix~\ref{apdx:subsec:th:VrAFBS4NI_convergence}

%%%% Theorem 3.1.
\begin{theorem}\label{th:VrAFBS4NI_convergence}
Suppose that Assumptions~\ref{as:A1} and \ref{as:A2} hold for \eqref{eq:NI} such that $L\rho < 1$.
Let $\sets{(x^k, y^k, z^k)}$ be generated by \eqref{eq:VrAFBS4NI} using an estimator $\widetilde{F}^k$ for $Fx^k$ satisfying Definition~\ref{de:VR_Estimators}.
Let $\hat{L} > L$ be such that $\hat{L}\rho < 1$, and  $\bar{\beta}$ and $\Lambda$ be defined in Lemma~\ref{le:FBS_cocoerciveness}.
Assume  that we choose $\lambda$, $\mu$, $r$, $\nu$, and $\beta$, and update $t_k$ and $\eta_k$ as follows:
\begin{equation}\label{eq:VrAFBS4NI_param_update}
\arraycolsep=0.1em
\begin{array}{ll}
& 2\rho \leq \lambda < \frac{2(1 + \sqrt{1-\hat{L}\rho})}{\hat{L}}, \quad 0 < \mu < \frac{2}{3}, \quad r \geq 2 + \frac{1}{\mu},  \quad \nu := \frac{\mu}{2}, \vspace{1ex}\\
&  0 < \beta \leq \frac{(2-\mu)\bar{\beta}}{2+\mu}, \quad t_k := \mu(k + r),  \quad  \textrm{and} \quad \eta_k := \frac{2\beta(t_k - 1)}{t_k - \nu}.
\end{array}
\end{equation}
Suppose further that, for all $k\geq 0$, $\kappa_k$ and $\Theta_k$ in Definition~\ref{de:VR_Estimators} and $\Gamma_k$ in \eqref{eq:VrAFBS_Lyapunov_func} satisfy
\begin{equation}\label{eq:VrAFBS4NI_param_cond}
\arraycolsep=0.2em
\begin{array}{ll}
\kappa_k \geq 1 - \frac{\Gamma_{k-1} t_{k-2}(t_{k-2}-1)}{\Gamma_k t_{k-1}(t_{k-1}-1)} + \frac{5\beta}{\mu\Gamma_k }  \quad  \textrm{and} \quad \Gamma_k\Theta_k \leq 2\Lambda.
\end{array}
\end{equation}
Then, for all $K \geq 0$,  $G_{\lambda}$ defined by \eqref{eq:FBS_residual} satisfies
\begin{equation}\label{eq:VrAFBS4NI_th1_convergence1} 
\Expn{ \norms{G_{\lambda}x^K }^2}   \leq   \frac{2( \Psi_0^2 + E_0^2 + B_{K-1}) }{\mu^2\mblue{(K + r - 1)^2}},
\end{equation}
where $\Psi_0^2$, $E_0^2$, and $B_K$ are respectively given by
\begin{equation}\label{eq:VrAFBS4NI_th1_convergence1_init} 
\arraycolsep=0.2em
\begin{array}{lcl}
\Psi_0^2 & := & \mu^2r^2 \Expn{ \norms{G_{\lambda}x^0}^2 }  + \frac{2r-1}{4\beta^2(\mu r - 1)}\norms{x^0 - x^{\star}}^2, \vspace{1ex}\\
E_0^2 & := &  \frac{\mu r^2(\Gamma_0\mu + \beta)}{2\beta} \sigma_0^2, \vspace{1ex}\\
B_K & := & \frac{\Lambda}{\beta}\sum_{k=0}^K   t_{k-1}(t_{k-1} - 1) \frac{\sigma_k^2}{\Theta_k}.
\end{array}
\end{equation}
In addition, for any $K \geq 0$, we also have
\begin{equation}\label{eq:VrAFBS4NI_th1_summable_bound1} 
\arraycolsep=0.2em
\begin{array}{lcl}
\sum_{k=0}^K\big[   (2 - 3\mu)\mu(k+r)  + 6\mu^2  \big] \Expn{\norms{ G_{\lambda}x^k }^2 } & \leq &   2\Psi_0^2 + 2E_0^2 + 2B_K, \vspace{1ex}\\
\sum_{k=0}^{K} (k+r)(k+r - \mu^{-1})  \Expn{ \norms{\widetilde{F}^k - Fx^k}^2 } & \leq & \frac{ 2 (\Psi_0^2 + E_0^2 + B_K ) }{\beta \mu}, \vspace{1ex}\\
\sum_{k=0}^K (k+r)(k+r-\mu^{-1}) \Expn{ \norms{ G_{\lambda}x^{k+1} - G_{\lambda}x^k }^2 } & \leq &  \frac{ (2-\mu)\beta ( \Psi_0^2 + E_0^2 + B_K ) }{\mu^2[(2-\mu)\bar{\beta} - (2+\mu)\beta]}.
\end{array}
\end{equation}
Here, the last summability bound in \eqref{eq:VrAFBS4NI_th1_summable_bound1} requires $0 < \beta < \frac{(2-\mu)\bar{\beta}}{2+\mu}$.
\end{theorem}

Note that the conclusions of Theorem~\ref{th:VrAFBS4NI_convergence} hold under the  condition \eqref{eq:VrAFBS4NI_param_cond}, and the right-hand side bounds depend on the quantity $B_K$.
Next, we state both $\BigOs{1/k^2}$ and $\SmallOs{1/k^2}$ convergence rates of \eqref{eq:VrAFBS4NI} under the following condition:
\begin{equation}\label{eq:summable_variance}
B_{\infty} := \frac{\Lambda}{\beta} \sum_{k=0}^{\infty} t_{k-1}(t_{k-1}-1)  \frac{\sigma_k^2}{\Theta_k} < +\infty.
\end{equation}
The proof of the following theorem is deferred to Appendix~\ref{apdx:subsec:th:VrAFBS4NI_o_rates_convergence2}.

%%%% Theorem 4.2.
\begin{theorem}\label{th:VrAFBS4NI_o_rates_convergence2}
Under the same conditions and settings as in Theorem~\ref{th:VrAFBS4NI_convergence}, and assuming that the condition \eqref{eq:summable_variance} also holds, for all $k \geq 0$, we have
\begin{equation}\label{eq:VrAFBS4NI_BigO_rate} 
\Expn{ \norms{G_{\lambda}x^k }^2}   \leq  \frac{2(\Psi_0^2 + E_0^2 + B_{\infty} ) }{\mu^2( k + r - 1)^2 } .
\end{equation}
In addition, we have the following summability bounds:
\begin{equation}\label{eq:VrAFBS4NI_summable_bound2}
\begin{array}{llcl}
& \sum_{k=0}^{\infty} (k+1)  \Expn{\norms{ G_{\lambda}x^k }^2 } & < & +\infty, \vspace{1.5ex}\\
& \sum_{k=0}^{\infty} (k+1)^2  \Expn{ \norms{\widetilde{F}^k - Fx^k}^2 } & < & +\infty, \vspace{1.5ex}\\
& \sum_{k=0}^{\infty} (k+1)  \Expn{ \norms{x^{k+1} - x^k}^2 } & < & +\infty.
\end{array}
\end{equation}
We also have the following $\SmallO{1/k^2}$ convergence rates in expectation:
\begin{equation}\label{eq:VrAFBS4NI_small_o_rates}
\arraycolsep=0.2em
\begin{array}{llcl}
& \lim_{k\to\infty}k^2  \Expn{ \norms{x^{k+1} - x^k}^2 } &= & 0, \vspace{1.5ex}\\
& \lim_{k\to\infty}k^2   \Expn{ \norms{G_{\lambda}x^k}^2 } &= & 0. 
\end{array}
\end{equation}
\end{theorem}

Finally, we state the following almost sure  convergence properties of \eqref{eq:VrAFBS4NI}.
The proof of this theorem can be found in Appendix \ref{apdx:th:VrAFBS4NI_o_rates_convergence3}.

%%% Theorem 4.4.
\begin{theorem}\label{th:VrAFBS4NI_o_rates_convergence3}
Under the same conditions and settings as in Theorem~\ref{th:VrAFBS4NI_convergence}, and assuming that \eqref{eq:summable_variance} holds, $0 < \beta < \frac{(2-\mu) \bar{\beta} }{2 + \mu}$, and there exist $\underline{\Theta} > 0$ and $\underline{\Gamma} > 0$ such that
\begin{equation}\label{eq:VrAFBS4NI_param_cond_new}
\arraycolsep=0.2em
\begin{array}{ll}
\Gamma_k \geq\underline{\Gamma}, \quad \Theta_k \geq \underline{\Theta}, \quad \textrm{and} \quad \Gamma_k\Theta_k \leq \Lambda, 
\end{array}
\end{equation}
we have 
\begin{equation}\label{eq:VrAFBS4NI_summable_bound2_as}
\arraycolsep=0.2em
\begin{array}{llcll}
& \sum_{k=0}^{\infty} (k+1)  \norms{ G_{\lambda}x^k }^2 & < & +\infty &\textrm{almost surely}, \vspace{1ex}\\
& \sum_{k=0}^{\infty} (k+1)^2  \norms{\widetilde{F}^k - Fx^k}^2  & < & +\infty &\textrm{almost surely}, \vspace{1ex}\\
& \sum_{k=0}^{\infty} (k+1)  \norms{x^{k+1} - x^k}^2  & < & +\infty &\textrm{almost surely}.
\end{array}
\end{equation}
The following almost sure limits also hold $($showing $\SmallOs{1/k^2}$ almost sure convergence rates$)$:
\begin{equation}\label{eq:VrAFBS4NI_small_o_rates_as}
\arraycolsep=0.2em
\begin{array}{llcl}
& \lim_{k\to\infty}k^2    \norms{G_{\lambda}x^k}^2  &= & 0, \vspace{1ex}\\
& \lim_{k\to\infty}k^2  \norms{x^{k+1} - x^k}^2   &= & 0.
\end{array}
\end{equation}
Moreover, both $\sets{x^k}$ and $\sets{z^k}$ almost surely converge  to a $\zer{\Phi}$-valued random variable $x^{\star}$ as a solution of \eqref{eq:NI}.
\end{theorem}
%%%

The second condition of  \eqref{eq:VrAFBS4NI_param_cond_new} comes from  the second condition of \eqref{eq:VrAFBS4NI_param_cond}.
As we will see from Subsections~\ref{subsec:complexity_bounds1} and \ref{subsec:complexity_bounds2}, there exists $\underline{\Theta} > 0$ such that $\Theta_k \geq \underline{\Theta} > 0$ and we choose  $\Gamma_k = \Gamma > 0$ for all $k \geq 0$ such that $\Gamma\Theta_k \leq \Lambda$.
Thus, the  condition \eqref{eq:VrAFBS4NI_param_cond_new} automatically holds.

%%% Remark 15.
\rv{\begin{remark}[Sufficient condition for \eqref{eq:summable_variance}]\label{re:finiteness_of_Binf}
Since $t_k := \mu(k+r)$, from \eqref{eq:summable_variance}, we have
\begin{equation*}
B_{\infty} = \frac{\mu^2\Lambda}{\beta}{\displaystyle\sum_{k=0}^{\infty}} \frac{(k+r-1)(k+r-2)\sigma_k^2}{\Theta_k}.
\end{equation*}
If $\Theta_k := \Theta > 0$ is fixed for all $k \geq 0$, then a sufficient condition for $B_{\infty} < +\infty$ is either $\sigma_k = 0$ or $\sigma_k^2 = \BigOs{\frac{\sigma^2}{k^{3+\omega}}}$ for any $\omega > 0$ and $\sigma^2$ given in Assumption~\ref{as:A1}$\mathrm{(ii)}$.
If we use either SVRG or SARAH estimators to construct $\widetilde{F}^k$, then we need to choose an increasing mega-batch size $n_k$ for evaluating $\bar{F}\tilde{x}^k$ or $\bar{F}x^k$ such that $n_k := \BigOs{k^{3+\omega}}$ to guarantee $B_{\infty} < +\infty$.
\end{remark}}

\begin{remark}\label{re:strong_monotonicity}
\rv{
In this paper, we do not consider the strong monotonicity of $\Phi$ in \eqref{eq:NI}.
This case was studied in, e.g.,  \citet{tran2024accelerated} when $T=0$ using a different class of unbiased variance-reduced estimators.
We believe that our methods can be customized to handle the strong monotonicity of $\Phi$ and can achieve a linear convergence rate as in \citet{tran2024accelerated}.
{\!\!\!}}
\end{remark}

\begin{remark}\label{re:local_property}
\rv{
Our analysis above relies on the main inequality \eqref{eq:G_cocoerciveness} for $x = x^{k+1}$ and $y = x^k$, and for $x = x^k$ and $y = x^{\star}$.
We have proven that both $\sets{ \norms{x^{k+1}-x^k}^2}$ and $\sets{\norms{x^k - x^{\star}}^2}$ almost surely converge to zero $($the former converges with a $\SmallOs{1/k^2}$ rate$)$.
Consequently, when $k$ is sufficiently large, both $\norms{x^{k+1} - x^k}$ and $\norms{x^k - x^{\star}}$ are almost surely sufficiently small.
Thus, we just require  \eqref{eq:G_cocoerciveness} to holds locally, which can be guaranteed if $T$ is locally $\rho$-co-hypomonotone.
This weaker condition expands the potential applicability of our methods to locally $\rho$-co-hypomonotone operators $T$ in \eqref{eq:NI}.
}
\end{remark}

%\rvs{\todo{The algorithm has many hyperparameters -- Add remarks on how to choose these parameters.}}
%
%\rvt{\todo{Recommend specific parameter setting, e.g., use in the experiments.}}

%\begin{remark}\label{re:strongly_monotone_case}
%If $\Phi$ in \eqref{eq:NI} is $\mu$-strongly monotone, then under Assumptions~\ref{as:A1} and \ref{as:A2}, one can also show that $G_{\lambda}$ is also $\mu_G$-strongly monotone.
%As shown in \citet{tran2024accelerated}, ???
%\end{remark}

%%%%%%%%%%%%%%%%%%%%%%%%%%%%%%%%%%%%%%%%%%%%%%%%%%%%%%
%%%% 3.4. Oracle complexity estimate for specific estimators.
%%%%%%%%%%%%%%%%%%%%%%%%%%%%%%%%%%%%%%%%%%%%%%%%%%%%%%
\beforesubsec
\subsection{Complexity of \ref{eq:VrAFBS4NI} for specific estimators in the finite-sum setting}\label{subsec:complexity_bounds1}
\aftersubsec
In this section, we apply Theorem~\ref{th:VrAFBS4NI_convergence}  to concrete estimators described in Section~\ref{sec:VR_estimators} to obtain explicit complexity bound for three cases: L-SVRG, SAGA, and L-SARAH.
The parameters and constants $\beta$, $\mu$, $\nu$, $r$, $\Lambda$, and $\Psi_0$ used in what follows are given in Theorem~\ref{th:VrAFBS4NI_convergence}.
The proof of these results can be found in Appendices~\ref{apdx:co:SVRG_complexity}, \ref{apdx:co:SAGA_complexity}, and \ref{apdx:co:SARAH_complexity}, respectively.

%%% Corollary 4.1.
\begin{corollary}[L-SVRG]\label{co:SVRG_complexity}
Suppose that Assumptions~\ref{as:A1} and \ref{as:A2} hold for \eqref{eq:NI} in the finite-sum setting \eqref{eq:finite_sum_form}.
Let $\sets{(x^k, z^k)}$ be generated by \eqref{eq:VrAFBS4NI} using  $t_k$, $\lambda$, and $\beta$ as in Theorem~\ref{th:VrAFBS4NI_convergence}.
Suppose that $\widetilde{F}^k$ is constructed by \eqref{eq:loopless_svrg} with $\tilde{x}^0 = x^0$, $\bar{F}\tilde{x}^k := F\tilde{x}^k$, and
\begin{equation*}
\begin{array}{lcl}
b_k := \lfloor c_b  n^{2\omega} \rfloor  \quad \textrm{and} \quad 
\mbf{p}_k := \begin{cases}
\frac{2}{c_p n^{\omega}} + \frac{4\mu}{\mu(k+r-1) - 1} &\textrm{if}~0 \leq k \leq K_0 := \lfloor 4c_p n^{\omega} - r + 1 + \mu^{-1} \rfloor, \\
\frac{3}{c_p n^{\omega}}  & \textrm{otherwise},
\end{cases}
\end{array}
\end{equation*}
where $c_p > 0$ is a given constant, $r > 5 + \frac{1}{\mu}$, $n^{\omega} \geq \frac{1}{c_p}\max\big\{ \frac{2\mu(r-1)-2}{\mu(r-5)-1}, \frac{\mu(r-1) -1 }{4\mu} \big\}$ for a fixed $\omega \in [0, 1]$, and $c_b := \frac{5\beta c_p^2}{\mu\Lambda}$.

Then, for a given $\epsilon > 0$, the expected total number $\bar{\Tc}_K$ of oracle calls $F_i$ and evaluations  $J_{\lambda T}$ to obtain $x^K$ such that $\Expn{ \norms{G_{\lambda}x^K}^2} \leq \epsilon^2$ is at most
\begin{equation*}
\begin{array}{lcl}
\bar{\Tc}_K := \Big\lfloor n + 4 n\big[ 2 + \ln(4c_pn^{\omega}) \big] + \frac{\sqrt{2}\Psi_0}{\mu \epsilon}  \big( 2c_bn^{2\omega} + \frac{ 3n^{1-\omega} }{c_p} \big) \Big\rfloor.
\end{array}
\end{equation*}
In particular, if we choose $\omega := \frac{1}{3}$, then our oracle complexity is $\bar{\Tc}_K = \BigOs{ n\ln(n) +   \frac{n^{2/3} }{\epsilon}}$.
\end{corollary}

Note that when $\omega = \frac{1}{3}$, our mini-batch size $b_k$ is $b_k = \BigOs{n^{2/3}}$ and our probability $\mbf{p}_k = \BigO{n^{-1/3}}$.
The oracle complexity $\BigOs{ n\ln(n) +   \frac{n^{2/3} }{\epsilon}}$ appears to be slightly worse than the  $\BigOs{n +   \frac{n^{2/3} }{\epsilon}}$ bound obtained in \citet{tran2024accelerated} by a $\ln(n)$ factor.
\rv{In our experiments, we often set $\mbf{p}_k := \frac{1}{2n^{1/3}}$ and $b_k := \lfloor \frac{n^{2/3}}{2}\rfloor$, but we can appropriately adjust these parameters.}

%\rvt{\todo{Recommend specific parameters}}

%%% Corollary 4.2.
\begin{corollary}[SAGA]\label{co:SAGA_complexity}
Suppose that Assumptions~\ref{as:A1} and \ref{as:A2} hold for \eqref{eq:NI} in the finite-sum setting \eqref{eq:finite_sum_form}.
Let $\sets{(x^k, z^k)}$ be generated by \eqref{eq:VrAFBS4NI} using $t_k$, $\lambda$, and $\beta$ from Theorem~\ref{th:VrAFBS4NI_convergence}.
Suppose that $\widetilde{F}^k$ is constructed by \eqref{eq:SAGA_estimator} with 
\begin{equation*}
b_k := \begin{cases}
2c_bn^{2/3} + \frac{4\mu n}{\mu(k+r-1)-1} & \textrm{if}~~0 \leq k \leq K_0 := \lfloor 4n^{1/3} + 1 + \mu^{-1} - r \rfloor, \\
3c_bn^{2/3}, &\textrm{otherwise},
\end{cases}
\end{equation*}
where $c_b :=  \frac{5}{2}\sqrt{ \frac{\beta }{\mu\Lambda} }$, $r > 5 + \frac{1}{\mu}$, and $n^{1/3} \geq \max\big\{ \frac{2c_b[\mu(r-1)-1]}{\mu(r-5)-1}, \frac{\mu(r-1) -1 }{4\mu} \big\}$.

Then, for a given tolerance $\epsilon > 0$, the expected total number $\bar{\Tc}_K$ of oracle calls $F_i$ and $J_{\lambda T}$ evaluations to obtain $x^K$ such that $\Expn{ \norms{G_{\lambda}x^K}^2} \leq \epsilon^2$ is at most
\begin{equation*}
\begin{array}{lcl}
\bar{\Tc}_K :=  \Big\lfloor [ 8c_b + 4 \ln(4n^{1/3})] n + \frac{3\sqrt{2} c_b \Psi_0 n^{2/3}}{\mu\epsilon} \Big\rfloor.
\end{array}
\end{equation*}
\end{corollary}

Again, the complexity of the SAGA estimator stated in Corollary \ref{co:SAGA_complexity} is $\mcal{O}\big(n\ln(n) + n^{2/3}\epsilon^{-1} \big)$, which is slightly worse than $\BigOs{n + n^{2/3}\epsilon^{-1}}$ in \citet{tran2024accelerated}.
\rv{In our experiments, we often choose $b_k := \min\sets{n, \lfloor \frac{n^{2/3}}{2}\rfloor}$, but we can appropriately adjust $b_k$.}

%\rvt{\todo{Recommend specific parameters}}

%%% Corollary 4.3.
\begin{corollary}[L-SARAH]\label{co:SARAH_complexity}
Suppose that Assumptions~\ref{as:A1} and \ref{as:A2} hold for \eqref{eq:NI} in the finite-sum setting \eqref{eq:finite_sum_form}.
Let $\sets{(x^k, z^k)}$ be generated by \eqref{eq:VrAFBS4NI} using $t_k$, $\lambda$, and $\beta$ as in Theorem~\ref{th:VrAFBS4NI_convergence}.
Suppose that $\widetilde{F}^k$ is computed by \eqref{eq:loopless_sarah} with 
\begin{equation*}
b_k := \lfloor c_b n^{\omega} \rfloor  \quad \textrm{and} \quad \mbf{p}_k := \begin{cases}
\frac{1}{ c_p n^{\omega} } + \frac{2\mu}{\mu(k+r-1) - 1} &\textrm{if}~0 \leq k \leq K_0 := \lfloor 2c_p n^{\omega} - r + 1 + \mu^{-1}  \rfloor, \\
\frac{2}{c_p n^{\omega}}  & \textrm{otherwise},
\end{cases}
\end{equation*}
where $c_p > 0$ is a given constant,  $r > 3 + \frac{1}{\mu}$, and $n^{\omega} \geq \frac{1}{c_p}\max\big\{ \frac{\mu(r-1) - 1}{\mu(r-3)-1}, \frac{\mu(r-1)-1}{2\mu} \big\}$ for a fixed $\omega \in [0, 1]$, and $c_b := \frac{5\beta c_p}{\mu\Lambda}$.

Then, for a given $\epsilon > 0$, the expected total number $\bar{\Tc}_K$ of oracle calls $F_i$ and $J_{\lambda T}$ evaluations to obtain $x^K$ such that $\Expn{ \norms{G_{\lambda}x^K}^2} \leq \epsilon^2$ is at most
\begin{equation*}
\begin{array}{lcl}
\bar{\Tc}_K := \Big\lfloor n + 2 n\big[ 2 + \ln(2c_p n^{\omega}) \big] + \frac{\sqrt{2}\Psi_0}{\mu \epsilon}  \big( c_bn^{\omega} + \frac{2n^{1-\omega}}{c_p} \big) \Big\rfloor.
\end{array}
\end{equation*}
In particular, if we choose $\omega := \frac{1}{2}$, then our oracle complexity is $\bar{\Tc}_K := \BigOs{ n\ln(n^{1/2}) + \frac{\sqrt{n}}{\epsilon}}$.
\end{corollary}

\rv{In our experiments,  we often choose $\mbf{p}_k := \frac{1}{2n^{1/2}}$ and $b_k := \lfloor \frac{\sqrt{n}}{2}\rfloor$, but again, we can appropriately adjust these parameters.}
Note that we can also apply the \ref{eq:HSGD_estimator} estimator to the finite-sum setting \eqref{eq:finite_sum_form} of \eqref{eq:NI}, but we still need to assume that $\Expsn{i}{\norms{\bar{F}x^k - Fx^k}^2} \leq \sigma^2$ in order to  estimate its oracle complexity.
Nevertheless, we omit this result here.

%\rvt{\todo{Recommend specific parameters}}

Finally, we specify Theorem~\ref{th:VrAFBS4NI_o_rates_convergence2} and Theorem~\ref{th:VrAFBS4NI_o_rates_convergence3}  for concrete estimators in Section~\ref{sec:VR_estimators}.
The following result is a direct consequence of Theorems \ref{th:VrAFBS4NI_o_rates_convergence2} and \ref{th:VrAFBS4NI_o_rates_convergence3} since $B_{\infty} = 0$.

%%% Corollary 4.3.
\begin{corollary}\label{co:VrAFBS4NI_o_rates_convergence2}
For the finite-sum setting \eqref{eq:finite_sum_form} of \eqref{eq:NI}, suppose that the L-SVRG, SAGA, and L-SARAH estimators are constructed as in Corollaries~\ref{co:SVRG_complexity}, \ref{co:SAGA_complexity}, and \ref{co:SARAH_complexity}, respectively. 
Then, the conditions~\eqref{eq:VrAFBS4NI_param_cond} of Theorem~\ref{th:VrAFBS4NI_convergence} are fulfilled and $B_{\infty} = 0$ in \eqref{eq:summable_variance}.

Consequently, the conclusions of both Theorem~\ref{th:VrAFBS4NI_o_rates_convergence2} and Theorem~\ref{th:VrAFBS4NI_o_rates_convergence3} are valid for \eqref{eq:VrAFBS4NI} using these three estimators.
\end{corollary}

Let us discuss our results and compare them with existing works.
The complexity of L-SARAH is better than that of L-SVRG and SAGA by a factor of $n^{1/6}$.
The complexity of the L-SARAH variant is the same as the stochastic Halpern method in \citet{cai2023variance}.
Nevertheless, our method is different from \citet{cai2023variance} and we also achieve better $\SmallO{1/k^2}$ convergence rates, several summability results, almost sure convergence rates, and the almost sure convergence of iterates.

%%%%%%%%%%%%%%%%%%%%%%%%%%%%%%%%%%%%%%%%%%%%%%%%%%%%%%
%%%% 3.4. Oracle complexity estimate for specific estimators.
%%%%%%%%%%%%%%%%%%%%%%%%%%%%%%%%%%%%%%%%%%%%%%%%%%%%%%
\vspace{-1ex}
\beforesubsec
\subsection{Complexity of \ref{eq:VrAFBS4NI} for specific estimators in the expectation setting}\label{subsec:complexity_bounds2}
\aftersubsec
Now, we derive the oracle complexity of  \eqref{eq:VrAFBS4NI} to solve \eqref{eq:NI} in the expectation setting \eqref{eq:expectation_form}.
We have three variants corresponding to the L-SVRG, L-SARAH, and HSGD estimators.
The proof of these results are given in Appendicies~\ref{apdx:co:SVRG_complexity_Esetting}, \ref{apdx:co:SARAH_complexity_Esetting}, and \ref{apdx:co:HSGD_complexity_Esetting}.

%%% Corollary 4.1.
\begin{corollary}[L-SVRG]\label{co:SVRG_complexity_Esetting}
Suppose that Assumptions~\ref{as:A1} and \ref{as:A2} hold for \eqref{eq:NI} in the expectation setting \eqref{eq:expectation_form}.
Let $\sets{(x^k, z^k)}$ be generated by \eqref{eq:VrAFBS4NI} using  $t_k$, $\lambda$, and $\beta$ as in Theorem~\ref{th:VrAFBS4NI_convergence}.
Suppose that $\widetilde{F}^k$ is constructed by \eqref{eq:loopless_svrg} with
\begin{equation*}
\arraycolsep=0.2em
\begin{array}{lcl}
b_k = b := \big\lfloor \frac{c_b}{ \epsilon^2} \big\rfloor, \quad n_k = n := \big\lfloor \frac{c_n}{\epsilon^3} \big\rfloor,  \quad \textrm{and} \quad \mbf{p}_k := 2\epsilon  + \frac{4\mu}{\mu(k+r - 1) - 1},
\end{array}
\vspace{-0.5ex}
\end{equation*}
where  $r > 5 + \frac{1}{\mu}$, $\epsilon \in \big(0, \frac{\mu(r-5)-1}{2\mu(r-1)-2} \big]$, $c_b := \frac{10\beta}{\mu\Lambda}$, and $c_n := 12\sigma^2 \max\sets{1, \frac{2\sqrt{2}\Psi_0}{\mu^2}}$.

Then, the expected total number $\bar{\Tc}_K$ of oracle calls $\Fb(\cdot, \xi)$ and  evaluations $J_{\lambda T}$ to obtain $x^K$ such that $\Expn{ \norms{G_{\lambda}x^K}^2} \leq \epsilon^2$ is at most $\bar{\Tc}_K :=  \BigOs{ \epsilon^{-3} +  \epsilon^{-3}\ln(\epsilon^{-1})}$.
\end{corollary}

%%% Corollary 4.3.
\begin{corollary}[L-SARAH]\label{co:SARAH_complexity_Esetting}
Suppose that Assumptions~\ref{as:A1} and \ref{as:A2} hold for \eqref{eq:NI} in the expectation setting \eqref{eq:expectation_form}.
Let $\sets{(x^k, z^k)}$ be generated by \eqref{eq:VrAFBS4NI} using $t_k$, $\lambda$, and $\beta$ as in Theorem~\ref{th:VrAFBS4NI_convergence}.
Suppose that $\widetilde{F}^k$ is constructed by \eqref{eq:loopless_sarah} with 
\begin{equation*}
\arraycolsep=0.2em
\begin{array}{ll}
b_k = b := \big\lfloor \frac{c_b}{\epsilon} \big\rfloor, \quad n_k = n := \big\lfloor \frac{c_n}{\epsilon^3} \big\rfloor,  \quad \textrm{and} \quad \mbf{p}_k := \epsilon + \frac{2\mu}{\mu(k+r-1) - 1},
\end{array}
\end{equation*}
where  $r > 3 + \frac{1}{\mu}$, $\epsilon \in \big(0, \frac{\mu(r-3)-1}{\mu(r-1)-1} \big]$, $c_b := \frac{5\beta}{\mu\Lambda}$, and $c_n := 24\sigma^2 \max\set{  \frac{\sqrt{2}\Psi_0}{\mu^2}, \frac{1}{\mu\beta} }$.

Then, the expected total number $\bar{\Tc}_K$ of oracle calls $\Fb(\cdot, \xi)$ and  evaluations $J_{\lambda T}$ to obtain $x^K$ such that $\Expn{ \norms{G_{\lambda}x^K}^2} \leq \epsilon^2$ is at most $\bar{\Tc}_K := \BigOs{ \epsilon^{-2} + \epsilon^{-3} +  \epsilon^{-3}\ln(\epsilon^{-1})}$.
\end{corollary}

Unlike the variance-reduced Halpern fixed-point method in \citet{cai2022stochastic}, we choose a fixed mini-batch size $b_k$ instead of varying it.
It leads to a $\log$-factor in our complexity.

%%% Corollary 4.4.
\begin{corollary}[HSGD]\label{co:HSGD_complexity_Esetting}
Suppose that Assumptions~\ref{as:A1} and \ref{as:A2} hold for \eqref{eq:NI} in the expectation setting \eqref{eq:expectation_form}.
Let $\sets{(x^k, z^k)}$ be generated by \eqref{eq:VrAFBS4NI} using the parameters $t_k$, $\lambda$, and $\beta$ as in Theorem~\ref{th:VrAFBS4NI_convergence}.
Suppose that $\widetilde{F}^k$ is constructed by \eqref{eq:HSGD_estimator} with 
\begin{equation*}
\arraycolsep=0.2em
\begin{array}{ll}
& \tau_k := 1 - \sqrt{\frac{(1- \theta)t_{k-1}(t_{k-1}-1)}{t_k(t_k-1)}}  \quad\textrm{for}~~ \theta := \epsilon, \vspace{1ex}\\
&b_k = b := \big\lfloor \frac{ c_b}{\epsilon} \big\rfloor, \quad  \hat{b}_k = \hat{b} := \big\lfloor \frac{ \hat{c}_b}{\epsilon^2 } \big\rfloor \quad  \textrm{and} \quad n_0 := \big\lfloor \frac{ c_n}{\epsilon} \big\rfloor,
\end{array}
\end{equation*}
where $\epsilon \in (0, 1/2]$ is a given tolerance,  $r \geq 5 + \frac{1}{\mu}$, and $c_b$, $\hat{c}_b$ and $c_n$ are constants independent of $\epsilon$.
Then, the expected  total number $\bar{\Tc}_K$ of oracle calls $\Fb(\cdot, \xi)$ and  evaluations $J_{\lambda T}$ to obtain $x^K$ such that $\Expn{ \norms{G_{\lambda}x^K}^2} \leq \epsilon^2$ is at most $\bar{\Tc}_K :=  \BigOs{ \epsilon^{-3}}$.
\end{corollary}

%%% Remark 4.1.
Three variants: L-SVRG, L-SARAH, and HSGD, discussed in this subsection offer either  $\widetilde{\mcal{O}}(\epsilon^{-3} )$ or $\BigOs{\epsilon^{-3}}$ oracle complexity to achieve an $\epsilon$-solution.
However, each variant has its own advantages and disadvantages.
For instance, L-SVRG seems to have the worst complexity among three variants, but it is unbiased, which is often preferable in practice.
HSGD has the best complexity of $\BigOs{ \epsilon^{-3}}$.
The chosen mini-batch $b$ in L-SARAH is $\BigOs{\epsilon^{-1}}$ which is smaller than $\BigOs{\epsilon^{-2}}$ in both L-SVRG and HSGD.
Both L-SVRG and L-SARAH occasionally require mega-batches of the size $\BigOs{\epsilon^{-3}}$ to evaluate the snapshot point with a probability $\mbf{p}_k$, while HSGD does not need any mega-batch, except for the initial epoch. 
Note that our theoretical parameters given in this subsection aim at achieving the best theoretical complexity bounds.
However, one can adjust the values of parameters $r$, $\beta$, $b_k$, $\hat{b}_k$, $\mbf{p}_k$, and $\tau_k$ when implementing these methods, which may work better in practice. 

%%%%%%%%%%%%%%%%%%%%%%%%%%%%%%%%%%%%%%%%
%%% 4. Backward-Forward Splitting Method
%%%%%%%%%%%%%%%%%%%%%%%%%%%%%%%%%%%%%%%%
\beforesec
\section{Variance-Reduced Accelerated  Backward-Forward Splitting  Method}\label{sec:VrABFS_method}
\aftersec
In this section, we explore the BFS mapping $S_{\lambda}$ in \eqref{eq:BFS_residual} to develop an alternative algorithmic framework to solve \eqref{eq:NI}.
This algorithm is less popular in optimization and can be viewed as a gradient-proximal method in contrast to the proximal-gradient method.

\beforesubsec
\subsection{The algorithm and its implementation}\label{subsec:VrABFS4NI}
\aftersubsec
Our method relies on approximating the BFS mapping $S_{\lambda}(\cdot)$ at $u^k$ defined by \eqref{eq:BFS_residual} by a stochastic estimator $\widetilde{S}_{\lambda}^k$.
This estimator is constructed as follows:
\begin{equation}\label{eq:S_estimator}
\widetilde{S}_{\lambda}^k := \widetilde{F}^k + \tfrac{1}{\lambda}(u^k - x^k), 
\end{equation}
where $x^k := J_{\lambda T}u^k$ and $\widetilde{F}^k$ is a variance-reduced estimator of $Fx^k$ satisfying Definition~\ref{de:VR_Estimators}.
It is obvious to see that $\norms{\widetilde{S}_{\lambda}^k - S_{\lambda}u^k} = \norms{\widetilde{F}^k - Fx^k}$.

Now, we are ready to  present the following algorithm, Algorithm~\ref{alg:A2}, for solving \eqref{eq:NI}.

%%%% Algorithm 2.
\begin{algorithm}[hpt!]\caption{(\textbf{V}ariance-reduced \textbf{F}ast [BF] \textbf{O}perator \textbf{S}plitting \textbf{A}lgorithm (\textbf{VFOSA$_{-}$}))}\label{alg:A2}
\normalsize
\rv{
\begin{algorithmic}[1]
\State\label{step:A2_i0}{\bfseries Initialization:} 
Take an initial point $x^0 \in \dom{\Phi}$.

\State \hspace{2.5ex}Choose $\mu$, $\lambda$, and $\beta$ as in Theorem~\ref{th:VrAFBS4NI_convergence}.
Set $\nu := \frac{\mu}{2}$ and $r := 2 + \frac{1}{\mu}$.

\State \hspace{2.5ex}Compute $u^0 := x^0 + \lambda\xi^0$ for any $\xi^0 \in Tx^0$ and set $s^0 := u^0$.
\State\hspace{0ex}\label{step:A2_o1}{\bfseries For $k := 0,\cdots, k_{\max}$ do}
\vspace{0.25ex}   
\State\hspace{2.5ex}\label{step:A2_o2}Compute \textit{the shadow iterate} $x^k :=  J_{\lambda T}u^k$.
\State\hspace{2.5ex}\label{step:A2_o3}Construct an estimator $\widetilde{F}^k$ of $Fx^k$ satisfying Definition~\ref{de:VR_Estimators}.
\State\hspace{2.5ex}\label{step:A2_o4}Update $t_k := \mu(k + r)$ and $\eta_k := \frac{2\beta(t_k - 1)}{t_k - \nu}$.
\State\hspace{2.5ex}\label{step:A2_o5}Update the following iterates:
\vspace{-0.5ex}
\myeq{eq:VrABFS4NI}{
\arraycolsep=0.2em
\left\{\begin{array}{lcl}
v^k & := &  \frac{t_k - 1}{t_k} u^k + \frac{1}{t_k} s^k, \vspace{1ex}\\
u^{k+1} &:= & v^k - \eta_k \widetilde{S}_{\lambda}^k \equiv v^k - \frac{\eta_k}{\lambda}(u^k - x^k) - \eta_k\widetilde{F}^k, \vspace{1ex}\\
s^{k+1} &:= &  s^k + \nu(u^{k+1} - v^k).
\end{array}\right.
\tag{VFOSA$_{-}$}
\vspace{-0.5ex}
}
\State\hspace{0ex}{\bfseries End For}
\end{algorithmic}}
\end{algorithm}
%}

%Starting from $x^0 \in \dom{\Phi}$, compute $\xi^0 \in Tx^0$, and $u^0 := x^0 + \lambda\xi^0$, set $s^0 := u^0$, and at each iteration $k \geq 0$, we update
%\begin{equation}\label{eq:VrABFS4NI}
%\arraycolsep=0.2em
%\left\{\begin{array}{lcl}
%v^k & := &  \frac{t_k - 1}{t_k} u^k + \frac{1}{t_k} s^k, \vspace{1ex}\\
%x^k &:= & J_{\lambda T}u^k, \vspace{1ex}\\
%%\widetilde{S}_{\lambda}^k & := &  \widetilde{F}^k + \tfrac{1}{\lambda}(u^k - x^k), \vspace{1ex}\\
%u^{k+1} &:= & v^k - \eta_k \widetilde{S}_{\lambda}^k \equiv v^k - \frac{\eta_k}{\lambda}(u^k - x^k) - \eta_k\widetilde{F}^k, \vspace{1ex}\\
%s^{k+1} &:= &  s^k + \nu(u^{k+1} - v^k),
%\end{array}\right.
%\tag{VFOSA$_{-}$}
%\end{equation}
%where $t_k > 0$, $\eta_k > 0$, and $\nu \in (0, 1]$ are given parameters as in \eqref{eq:VrAFBS4NI}.
Algorithm~\ref{alg:A2} also requires only one evaluation $\widetilde{F}^k$ and one evaluation $J_{\lambda T}$ per iteration.
Hence, this method has the same per-iteration complexity as \eqref{eq:VrAFBS4NI}.

\beforesubsec
\subsection{Convergence analysis}\label{subsec:VrABFS4NI_convergence}
\aftersubsec
Now, we apply the analysis in Section~\ref{sec:VR_AFBS_method} to establish the convergence of Algorithm~\ref{alg:A2} in the following theorem, whose proof is given in Appendix~\ref{apdx:th:VrABFS4NI_O_rates}.

%%% Theorem 5.1.
\begin{theorem}\label{th:VrABFS4NI_O_rates}
Suppose that Assumptions~\ref{as:A1} and \ref{as:A2} hold for \eqref{eq:NI}.
Let $\sets{(x^k, u^k)}$ be generated by \eqref{eq:VrABFS4NI} using an estimator $\widetilde{F}^k$  of $Fx^k$ satisfying Definition~\ref{de:VR_Estimators}.
Under the same parameters $\lambda$, $\nu$, $\beta$, $t_k$ and $\eta_k$ and the same conditions as in Theorem~\ref{th:VrAFBS4NI_convergence}, let  $S_{\lambda}$ be defined by \eqref{eq:BFS_residual} and $\xi^k := \frac{1}{\lambda}(u^k - x^k) \in Tx^k$ for all $k \geq 0$.
Then, we have 
\begin{equation}\label{eq:VrABFS4NI_th1_convergence1} 
\Expn{ \norms{Fx^K + \xi^K }^2 }   \leq   \frac{2( \bar{\Psi}_0^2 + \bar{E}_0^2 + B_{K-1}) }{\mu^2(K + r - 1)^2},
\end{equation}
where $\bar{\Psi}_0^2 := \mu^2r^2 \Expn{ \norms{Fx^0 + \xi^0}^2}  + \frac{2r-1}{4\beta^2(\mu r -1)} \norms{u^0 - u^{\star}}^2$, $\bar{E}_0^2 :=  \frac{\mu r^2(\mu\Gamma_0 + \beta)\sigma_0^2}{2\beta}$, and $B_K$ is given in Theorem~\ref{th:VrAFBS4NI_convergence}.

Moreover, we also have
\begin{equation}\label{eq:VrABFS4NI_th1_summable_bound1} 
\arraycolsep=0.2em
\begin{array}{lcl}
\sum_{k=0}^K\big[   (2 - 3\mu)\mu(k+r)  + 6\mu^2  \big] \Expn{\norms{ Fx^k + \xi^k }^2 } & \leq &   2\big(\bar{\Psi}_0^2 + \bar{E}_0^2 + B_K\big), \vspace{1.5ex}\\
\sum_{k=0}^{K} (k+r)(k+r - \mu^{-1})  \Expn{ \norms{\widetilde{F}^k - Fx^k}^2 } & \leq & \frac{ 2 (\bar{\Psi}_0^2 + \bar{E}_0^2 + B_K ) }{\beta \mu}.
\end{array}
\end{equation}
%Here, the last bound requires $\beta < \frac{(2-\mu)\bar{\beta}}{2 + \mu}$.
\end{theorem}

Similarly, we can also state both $\BigO{1/k^2}$ and $\SmallOs{1/k^2}$-convergence rates of \eqref{eq:VrABFS4NI}, whose proof is very similar to the proof of Theorem~\ref{th:VrAFBS4NI_o_rates_convergence2}, see Appendix~\ref{apdx:th:VrABFS4NI_o_rates}.

%%%% Theorem 5.2.
\begin{theorem}\label{th:VrABFS4NI_o_rates}
Under the same conditions and settings as in Theorems~\ref{th:VrAFBS4NI_o_rates_convergence2} and \ref{th:VrABFS4NI_O_rates}, if, in addition, \eqref{eq:summable_variance} holds, then for $\xi^k := \frac{1}{\lambda}(u^k - x^k) \in Tx^k$ and for all $k \geq 0$, we have 
\begin{equation*} 
\Expn{ \norms{Fx^k + \xi^k }^2}   \leq  \frac{2(\bar{\Psi}_0^2 + \bar{E}_0^2 + B_{\infty} ) }{\mu^2 (k + r - 1)^2 }.
\end{equation*}
We have the following summability bounds:
\begin{equation*} 
\begin{array}{llcl}
& \sum_{k=0}^{\infty} (k+1)  \Expn{\norms{ Fx^k + \xi^k }^2 } & < & +\infty, \vspace{1ex}\\
& \sum_{k=0}^{\infty} (k+1)^2   \Expn{ \norms{\widetilde{F}^k - Fx^k}^2 } & < & +\infty, \vspace{1ex}\\
& \sum_{k=0}^{\infty} (k+1)   \Expn{ \norms{u^{k+1} - u^k}^2 } & < & +\infty, \vspace{1ex}\\
& \sum_{k=0}^{\infty} (k+1)  \Expn{ \norms{x^{k+1} - x^k}^2 } & < & +\infty,
\end{array}
\end{equation*}
where the last three bound require $0 < \beta < \frac{(2-\mu) \bar{\beta} }{2+\mu}$. 
We also obtain the following limits:
\begin{equation*} 
\arraycolsep=0.2em
\begin{array}{llcl}
& \lim_{k\to\infty}k^2  \Expn{ \norms{u^{k+1} - u^k}^2 } &= & 0, \vspace{1ex}\\
& \lim_{k\to\infty}k^2  \Expn{ \norms{x^{k+1} - x^k}^2 } &= & 0, \vspace{1ex}\\
& \lim_{k\to\infty}k^2  \Expn{ \norms{Fx^k + \xi^k}^2 } &= & 0.
\end{array}
\end{equation*}
\end{theorem}

\begin{theorem}\label{th:VrABFS4NI_almost_sure_convergence}
Under the same the conditions and settings as in Theorems~\ref{th:VrAFBS4NI_o_rates_convergence3} and \ref{th:VrABFS4NI_o_rates},  we have the following summability bounds:
\begin{equation*} 
\begin{array}{llcl}
& \sum_{k=0}^{\infty} (k+1)  \norms{ Fx^k + \xi^k }^2 & < & +\infty, \vspace{1.5ex}\\
& \sum_{k=0}^{\infty} (k+1)^2    \norms{\widetilde{F}^k - Fx^k}^2  & < & +\infty, \vspace{1.5ex}\\
& \sum_{k=0}^{\infty} (k+1)    \norms{u^{k+1} - u^k}^2  & < & +\infty, \vspace{1.5ex}\\
& \sum_{k=0}^{\infty} (k+1)   \norms{x^{k+1} - x^k}^2  & < & +\infty.
\end{array}
\end{equation*}
The following almost sure limits also hold  $($showing $\SmallOs{1/k^2}$ almost sure convergence rates$)$:
\vspace{-0.5ex}
\begin{equation*} 
\arraycolsep=0.2em
\begin{array}{llcl}
& \lim_{k\to\infty}k^2   \norms{x^{k+1} - x^k}^2   &= & 0, \vspace{1ex}\\
& \lim_{k\to\infty}k^2   \norms{Fx^k + \xi^k}^2  &= & 0.
\end{array}
\vspace{-0.5ex}
\end{equation*}
Moreover, $\sets{u^k}$   almost surely converges to a $ \zer{S_{\lambda}}$-valued random variable $u^{\star} \in \zer{S_{\lambda}}$.
Consequently, $\sets{x^k}$ also almost surely converges  to $x^{\star} := J_{\lambda T}u^{\star} \in \zer{\Phi}$.
\end{theorem}

%\rv{\todo{To do: Discuss the condition such that $B_{\infty} < +\infty$.}}

If we specify \eqref{eq:VrABFS4NI} (or equivalently, Algorithm~\ref{alg:A2}) to obtain  a specific variant corresponding to each estimator in Section~\ref{subsec:VR_estimator_examples}, then we still obtain the same oracle complexity as in \eqref{eq:VrAFBS4NI}.
We omit these variants as they are similar to Subections~\ref{subsec:complexity_bounds1} and \ref{subsec:complexity_bounds2}.

\vspace{1ex}
\rv{\noindent\textbf{Comparison between \ref{eq:VrAFBS4NI} and \ref{eq:VrABFS4NI}.}
Both \ref{eq:VrAFBS4NI} and \ref{eq:VrABFS4NI} require the same assumptions to guarantee convergence. 
They also share the same convergence properties. 
However, $\sets{x^k}$ is the primal sequence of \ref{eq:VrAFBS4NI}, and it directly converges to a solution $x^{\star}$ of \eqref{eq:NI} almost surely. 
The primal sequence of \ref{eq:VrABFS4NI} is $\sets{u^k}$, which does not converge to $x^{\star}$. 
Instead, the shadow sequence $\sets{x^k}$ almost surely converges to $x^{\star}$, where $x^k := J_{\lambda T} u^k$. 
As mentioned earlier, \ref{eq:VrAFBS4NI} is more popular than \ref{eq:VrABFS4NI} in the literature. 
It covers the well-known proximal-gradient-type methods in convex optimization.
}

%%%%%%%%%%%%%%%%%%%%%%%%%%%%%%%%%%%%%%%%%%%%%%%%%
%%% Numerical Experiments
%%%%%%%%%%%%%%%%%%%%%%%%%%%%%%%%%%%%%%%%%%%%%%%%%
\beforesec
\section{Numerical Experiments}\label{sec:num_examples}
\aftersec
In this section, we present two numerical examples to validate our methods and compare the performance of different algorithms, including ours and recent methods from the literature.
All the algorithms are implemented in Python and executed on a MacBook Pro with Apple M4 processor and 24Gb of memory.

%%%\rv{\todo{Add numerical examples to compare with L-Katyusha, and non-accelerated methods, and deterministic Nesterov's accelerated methods.}}
%%%%%%%%%%%%%%%%%%%%%%%%%%%%%%%%%%%%%
%%%% 6.1. Robust regularized logistic regression with ambiguous features.
%%%%%%%%%%%%%%%%%%%%%%%%%%%%%%%%%%%%%
\beforesubsec
\subsection{Robust regularized logistic regression with ambiguous features}\label{subsec:logistic_reg_with_ambiguous_feature}
\aftersubsec
We apply our methods to solve  the regularized logistic regression problem with ambiguous features studied in \citet{tran2022accelerated} and compare them with the accelerated Halpern's fixed-point method using SARAH in \citet{cai2023variance}. 

\beforesubsubsec
\subsubsection{Mathematical model}\label{subsubsec:math_model}
\aftersubsubsec
We are given a dataset $\sets{(\hat{X}_i, y_i)}_{i=1}^n$, where $\hat{X}_i$ is an i.i.d. sample of a feature vector in $\R^{p_1}$ and $y_i \in \{0, 1\}$ is the corresponding label of $\hat{X}_i$.
We assume that $\hat{X}_i$ is ambiguous, i.e., it belongs to one of $p_2$ possible examples  $\{X_{ij} \}_{j=1}^{p_2}$.
Since we do not know $\hat{X}_i$ to evaluate the loss, we consider the worst-case loss $f_i(u) := \max_{1\leq j \leq p_2}\ell( \iprods{X_{ij}, u}, y_i)$ computed from $p_2$ examples $\{X_{ij} \}_{j=1}^{p_2}$, where $\ell(\tau, s) := \log(1 + \exp(\tau)) - s\tau$ is the standard logistic loss.

Since $\max_{1\leq j\leq p_2}\ell_j(\cdot) = \max_{v\in\Delta_{p_2}}\sum_{j=1}^{p_2} v_j \ell_j(\cdot)$, where $\Delta_{p_2}$ is  the standard simplex in $\R^{p_2}$, we can model this robustification of the regularized logistic regression problem into the following [non]convex-concave minimax formulation:
\vspace{-1ex}
\begin{equation}\label{eq:logistic_reg_exam}
\min_{u \in\R^{p_1} }  \Big\{ \phi(u) := \max_{v\in \R^{p_2} }\Big\{ \mcal{H}(u, v) := \frac{1}{n} \sum_{i=1}^{n} \sum_{j=1}^{p_2} v_j  \ell ( \iprods{X_{ij}, u}, y_i ) - \delta_{\Delta_{p_2}}(v) \Big\} + \bar{\lambda} R(u)  \Big\},
\vspace{-0.5ex}
\end{equation}
where $R(\cdot)$ is a given regularizer chosen later, $ \bar{\lambda} > 0$ is a regularization parameter, and $\delta_{\Delta_{p_2}}$ is the indicator of $\Delta_{p_2}$ that handles the constraint $v \in \Delta_{p_2}$.

Let us denote by $x := [u; v]$ and
\vspace{-0.5ex}
\begin{equation*}
\arraycolsep=0.2em
\begin{array}{lcl}
F_ix & := & \big[ \sum_{j=1}^{p_2}v_j \ell'(\iprods{X_{ij}, u}, y_i) X_{ij};\ -\ell(\iprods{X_{i1}, u}, y_i); \cdots; -\ell(\iprods{X_{ip_2}, u}, y_i) \big], \vspace{1ex} \\
Tx & := & [  \bar{\lambda}  \partial{R}(u); \partial{\delta_{\Delta_{p_2}}}(v)],
\end{array}
\vspace{-0.5ex}
\end{equation*}
 where $\ell'(\tau, s) = \frac{\exp(\tau)}{1+\exp(\tau)} - s$.
 Then, we have $Fx = \frac{1}{n}\sum_{i=1}^nF_ix$ as given in \eqref{eq:finite_sum_form}.
 The optimality condition of \eqref{eq:logistic_reg_exam} can be written as $0 \in Fx + Tx$, which is a special case of \eqref{eq:NI}.

\beforesubsubsec
\subsubsection{Implementation details and input data}\label{subsec:implementation_of_experiments}
\aftersubsubsec
We implement both methods: \ref{eq:VrAFBS4NI} and \ref{eq:VrABFS4NI}  to solve \eqref{eq:logistic_reg_exam}.
Each method consists of $4$ different variants: L-SVRG, SAGA, L-SARAH, and Hybrid-SGD.
They are abbreviated by \texttt{\ref{eq:VrAFBS4NI}{\!}-Svrg}, \texttt{\ref{eq:VrAFBS4NI}{\!}-Saga}, \texttt{\ref{eq:VrAFBS4NI}{\!}-Sarah}, \texttt{\ref{eq:VrAFBS4NI}{\!}-Hsgd}, \texttt{\ref{eq:VrABFS4NI}{\!}-Svrg}, \texttt{\ref{eq:VrABFS4NI}{\!}-Saga}, \texttt{\ref{eq:VrABFS4NI}{\!}-Sarah}, and \texttt{\ref{eq:VrABFS4NI}{\!}-Hsgd}, respectively.
We consider two cases of \eqref{eq:logistic_reg_exam} as follows.
\begin{compactitem}
\item \textbf{The monotone case.} 
We choose $R(u) := \norms{u}_1$ as an $\ell_1$-norm regularizer, leading to a convex-concave instance of \eqref{eq:logistic_reg_exam}. This corresponds to a monotone $T$ in \eqref{eq:NI}.
In this case, Assumption~\ref{as:A1}(iii) holds with $\rho = 0$.

\item \textbf{The nonmonotone case}: We choose $R(u)$ to be the SCAD (smoothly clipped absolute deviation)  regularizer from \citep{fan2001variable}, which promotes sparsity of $u$ using a nonconvex regularizer.
Hence, \eqref{eq:logistic_reg_exam} is nonconvex-concave, leading to a nonmonotone $T$ in \eqref{eq:NI}.
\rv{In fact, $T$ is locally $\rho$-co-hypomonotone with $\rho := a - 1 = 2.7$.
By Remark~\ref{re:local_property}, we can still apply our methods to solve \eqref{eq:logistic_reg_exam}.}
\end{compactitem}
For comparison, we also implement the variance-reduced Halpern's fixed-point method from  \citep{cai2023variance}, which is abbreviated  by \texttt{VrHalpern}.
For further comparison with other methods such as extragradient,  variance-reduced extragradient, and  extra-anchor gradient methods \citep{yoon2021accelerated}, we refer to Subsection~\ref{subsec:matrix_game} and also \citep{cai2023variance}.

Since it is difficult to exactly evaluate the co-coercive constant $L$, we approximate it by $L := \frac{1}{4}\Vert X^TX \Vert$, which represents the smoothness constant of the logistic loss.
Then, we choose $\lambda := \frac{1}{2L} < \frac{2}{L}$, and  $\bar{\beta} := \frac{\lambda(4  - L\lambda)}{4}$ as suggested by Lemma~\ref{le:FBS_cocoerciveness} and Theorem~\ref{th:VrAFBS4NI_convergence}.
From our theory and \citep{cai2023variance}, we choose the parameters for each variant as follows.
\begin{compactitem}
\item We  choose $\mu := \frac{0.95 \times 2}{3} < \frac{2}{3}$ and $r := 2 + \frac{1}{\mu}$ for all variants of our methods.
\item For the L-SVRG variants, we choose $\mbf{p}_k = \frac{1}{2n^{1/3}}$ and $b := \lfloor \frac{n^{2/3}}{2} \rfloor$, see Corollary~\ref{co:SVRG_complexity}.
\item For the SAGA variants, we choose $b := \lfloor \frac{n^{2/3}}{2} \rfloor$, see  Corollary~\ref{co:SAGA_complexity}.
\item For the L-SARAH variants, we choose $\mbf{p}_k := \frac{1}{2\sqrt{n}}$ and $b := \lfloor \frac{\sqrt{n}}{2} \rfloor$, see Corollary~\ref{co:SARAH_complexity}.
\item For the Hybrid-SGD variants, we choose $\theta := \frac{1}{n}$ and $b := \lfloor \frac{\sqrt{n}}{2} \rfloor$, see Corollary~\ref{co:HSGD_complexity_Esetting}.
\item For \texttt{VrHalpern}, we choose $\mbf{p}_k := \frac{1}{2\sqrt{n}}$ and $b := \lfloor \frac{\sqrt{n}}{2} \rfloor$, see \citep{cai2023variance}.
\end{compactitem}
We report the relative norm $\norms{G_{\lambda}x^k}/\norms{G_{\lambda}x^0}$ against the number of epochs as our main criterion in each experiment.
We choose the initial point $x^0 := 0.25 \times \mathrm{randn}(p)$ in all methods, and run each algorithm for $N_e := 200$ epochs.
Note that in the following experiments, we do not implement any tuning strategy for our parameters.

We use 4 different real datasets from \texttt{LIBSVM} \citep{CC01a}: 
\texttt{gisette} (5,000 features and 6,000 samples),
\texttt{w8a} (300 features and 49,749 samples),  
\texttt{a9a} (123 features and  32,561 samples), and 
\texttt{mnist} (784 features and  60,000 samples).
Here, the first dataset has a large number of features but a small number of samples, while the other ones have a small number of features but a large number of samples.
Since \texttt{mnist} is designed for multi-class classification, we convert it to a binary format by grouping the even digits (i.e., $\sets{0,2,4,6,8}$) into one class and the odd digits (i.e., $\sets{1, 3, 5, 7, 9}$) into the other.
As usual, we normalize the feature vector $\hat{X}_i$ such that each sample has unit norm, and add a column of  all ones to address the bias term.

To generate ambiguous features, we apply the following procedure.
First, we choose the number of ambiguous features to be $p_2 = 10$.
Next, for each sample $i$, we take the nominal feature vector $\hat{X}_i$ and add a random noise generated from a normal distribution of zero mean and variance $\sigma^2 = 0.05^2$.
We also choose the regularization parameter $\bar{\lambda}$ to be $\bar{\lambda} := 5\times 10^{-3}$.
This $\bar{\lambda}$ lead to a reasonable sparsity pattern of an approximate solution $u^k$ to $u^{\star}$ of \eqref{eq:logistic_reg_exam}, though it does not produce a highly accurate residual norm.

\beforesubsubsec
\subsubsection{Numerical experiments}\label{subsec:numerical_experiments}
\aftersubsubsec
We carry out different experiments to test our \ref{eq:VrAFBS4NI} and \ref{eq:VrABFS4NI} for both the monotone and nonmonotone $T$ using four estimators: L-SVRG, SAGA, L-SARAH, and HSGD.

\vspace{0.75ex}
\noindent\textbf{$\mathrm{(a)}$~Comparing 4 variants of \ref{eq:VrAFBS4NI} on monotone problems.}
First, we run four variants: L-SVRG, SAGA, L-SARAH, and HSGD on two real datasets: \texttt{w8a} and \texttt{gisette} to solve a convex-concave minimax instance of \eqref{eq:logistic_reg_exam} with the $\ell_1$-regularizer $R(u) = \norms{u}_1$ (i.e., $T$ is monotone).
The results of this experiment are revealed in Figure~\ref{fig:experiment1}.

\begin{figure}[ht!]
\vspace{-0ex}
\centering
\includegraphics[width=1\textwidth]{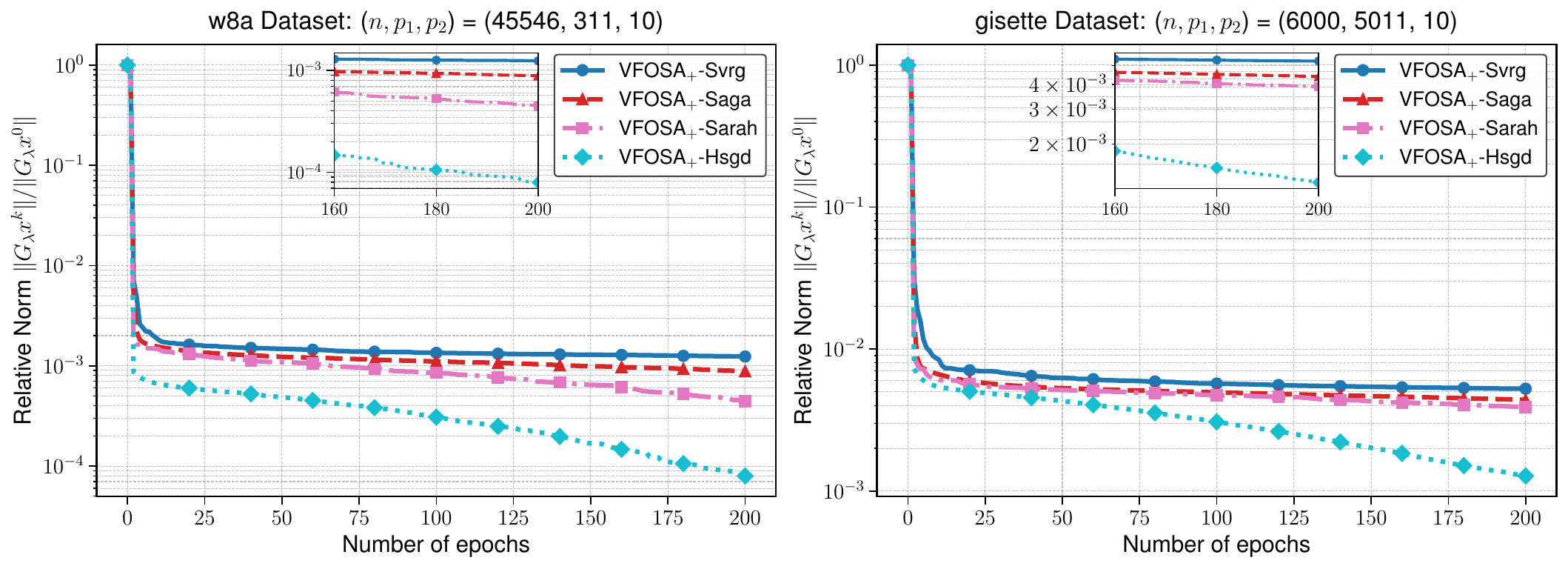}
\vspace{-5ex}
\caption{
The performance of 4 variants of \ref{eq:VrAFBS4NI}: L-SVRG, SAGA, L-SARAH, and HSGD for solving \eqref{eq:logistic_reg_exam} with the $\ell_1$-regularizer on two datasets: \texttt{w8a} and  \texttt{gisette}.
}
\label{fig:experiment1}
\vspace{-3ex}
\end{figure}

One can see from Figure~\ref{fig:experiment1} that both \texttt{\ref{eq:VrAFBS4NI}{\!}-Svrg} and \texttt{\ref{eq:VrAFBS4NI}{\!}-Saga} achieve similar performance, with \texttt{\ref{eq:VrAFBS4NI}{\!}-Saga} slightly outperforming \texttt{\ref{eq:VrAFBS4NI}{\!}-Svrg}, despite having the same oracle complexity.
\texttt{\ref{eq:VrAFBS4NI}{\!}-Sarah} performs better than both \texttt{\ref{eq:VrAFBS4NI}{\!}-Svrg} and \texttt{\ref{eq:VrAFBS4NI}{\!}-Saga}, while \texttt{\ref{eq:VrAFBS4NI}{\!}-Hsgd} appears to outperform all its competitors.
Although \texttt{\ref{eq:VrAFBS4NI}{\!}-Svrg} and \texttt{\ref{eq:VrAFBS4NI}{\!}-Sarah} occasionally compute full batches of $F$,  \texttt{\ref{eq:VrAFBS4NI}{\!}-Saga} and \texttt{\ref{eq:VrAFBS4NI}{\!}-Hsgd} require no full-batch evaluations except for the first epoch.
However, \texttt{\ref{eq:VrAFBS4NI}{\!}-Saga} must store all component mappings $F_ix^k$.
This experiment confirms that biased estimators such as L-SARAH and HSGD outperform unbiased ones like L-SVRG and SAGA, aligning with our theoretical findings.

\begin{figure}[ht!]
\vspace{-0ex}
\centering
\includegraphics[width=1\textwidth]{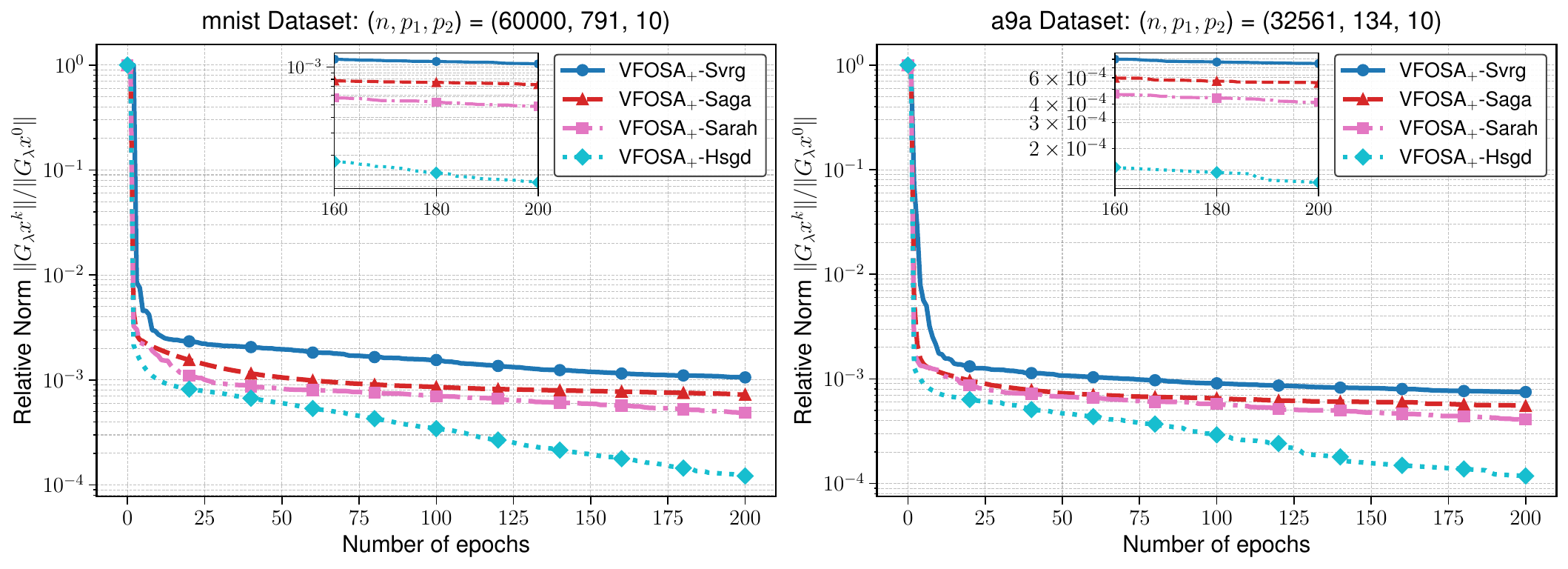}
\vspace{-5ex}
\caption{
The performance of 4 variants of \ref{eq:VrAFBS4NI}: L-SVRG, SAGA, L-SARAH, and HSGD for solving \eqref{eq:logistic_reg_exam} using the $\ell_1$-regularizer on two datasets: \texttt{mnist} and  \texttt{a9a}.
}
\label{fig:experiment2}
\vspace{-3ex}
\end{figure}

Similarly, we also compare these four algorithmic variants on the other two datasets: \texttt{mnist} and \texttt{a9a}.
The results are reported in Figure~\ref{fig:experiment2}.

Again, we observe a similar performance as in the first experiment. 
Both \texttt{\ref{eq:VrAFBS4NI}{\!}-Sarah} and \texttt{\ref{eq:VrAFBS4NI}{\!}-Hsgd} still outperform \texttt{\ref{eq:VrAFBS4NI}{\!}-Svrg} and \texttt{\ref{eq:VrAFBS4NI}{\!}-Saga}.
Overall, \texttt{\ref{eq:VrAFBS4NI}{\!}-Hsg} is still the best in this experiment.

%%% Nonmonotone experiments.
\vspace{0.75ex}
\noindent\textbf{$\mathrm{(b)}$~Comparing 4 variants of \ref{eq:VrAFBS4NI} on nonmonotone problems.}
Next, we run these four algorithmic variants on four datasets above to solve a nonconvex-concave minimax instance of \eqref{eq:logistic_reg_exam} by choosing the SCAD regularizer (i.e., nonmonotone $T$).
The results are shown in Figure~\ref{fig:experiment3a} and Figure~\ref{fig:experiment3b}, respectively for each pair of datasets.

\begin{figure}[ht!]
\vspace{-0ex}
\centering
\includegraphics[width=1\textwidth]{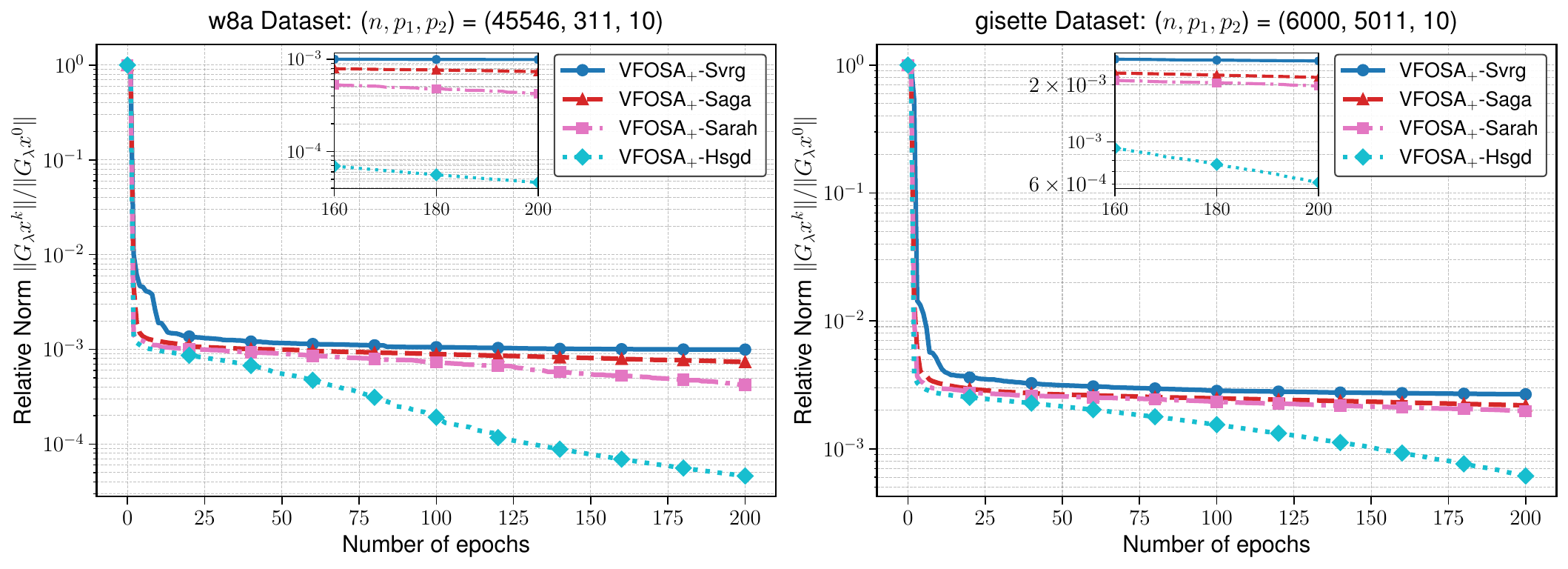}
\vspace{-5ex}
\caption{
The performance of 4 variants of \ref{eq:VrAFBS4NI}: SVRG, SAGA, SARAH, and HSGD for solving \eqref{eq:logistic_reg_exam} using the SCAD regularizer on two datasets:  \texttt{w8a} and  \texttt{gisette}.
}
\label{fig:experiment3a}
\vspace{-3ex}
\end{figure}

\begin{figure}[ht!]
\vspace{-1ex}
\centering
\includegraphics[width=1\textwidth]{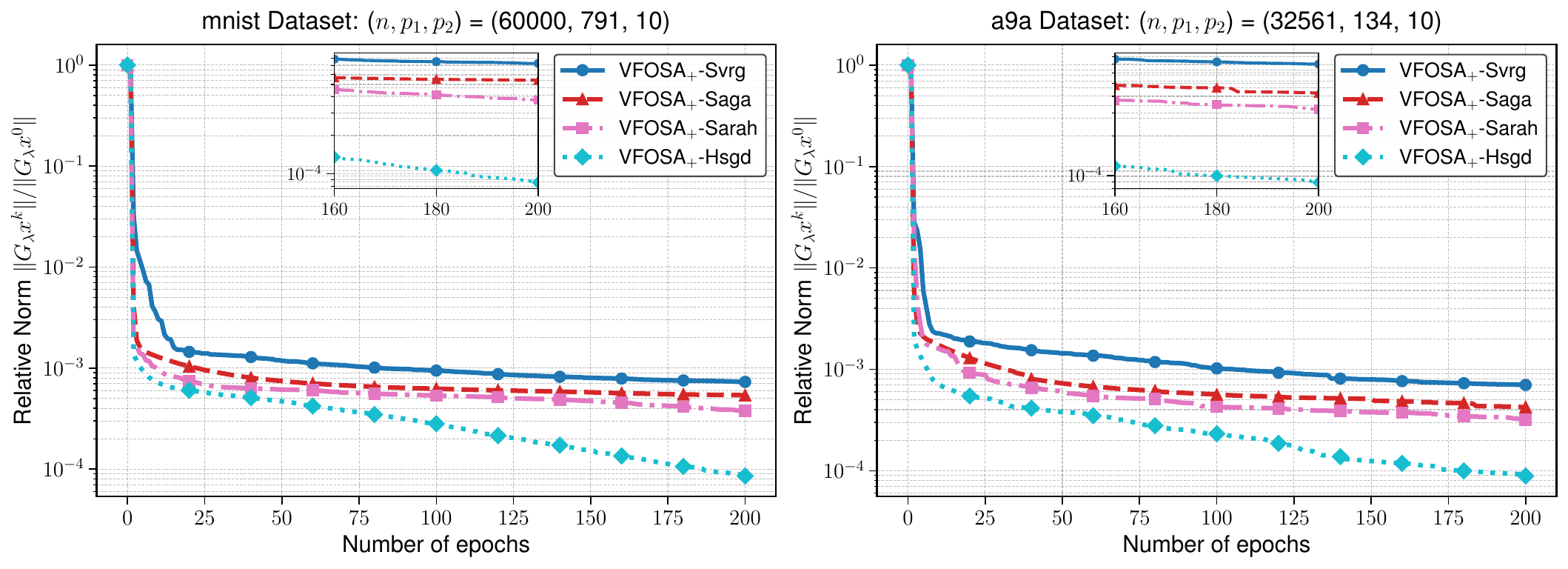}
\vspace{-5ex}
\caption{
The performance of 4 variants of \ref{eq:VrAFBS4NI}: L-SVRG, SAGA, L-SARAH, and HSGD for solving \eqref{eq:logistic_reg_exam} using the SCAD regularizer on the \texttt{mnist} and  \texttt{a9a} datasets.
}
\label{fig:experiment3b}
\vspace{-3ex}
\end{figure}

We can see from both Figures~\ref{fig:experiment3a} and \ref{fig:experiment3b} that our algorithms still work well with the SCAD regularizer, and show a similar performance as our first two experiments. 
Nevertheless, the solutions we obtain have a better sparsity pattern compared to the $\ell_1$-norm regularizer. 

\vspace{0.75ex}
\noindent\textbf{$\mathrm{(c)}$~Comparing 4 variants of \ref{eq:VrABFS4NI}.}
Alternatively, we conduct a similar type of experiments for our \ref{eq:VrABFS4NI} method.
This time we only compare four variants of \ref{eq:VrABFS4NI} on two datasets: \texttt{mnist} and  \texttt{gisette}  for both the $\ell_1$-norm regularizer (monotone case) and the SCAD regularizer (nonmonotone case).
The results of this experiment are presented in Figures~\ref{fig:experiment4a} and \ref{fig:experiment4b}, respectively.

\begin{figure}[ht!]
\vspace{-0ex}
\centering
\includegraphics[width=1\textwidth]{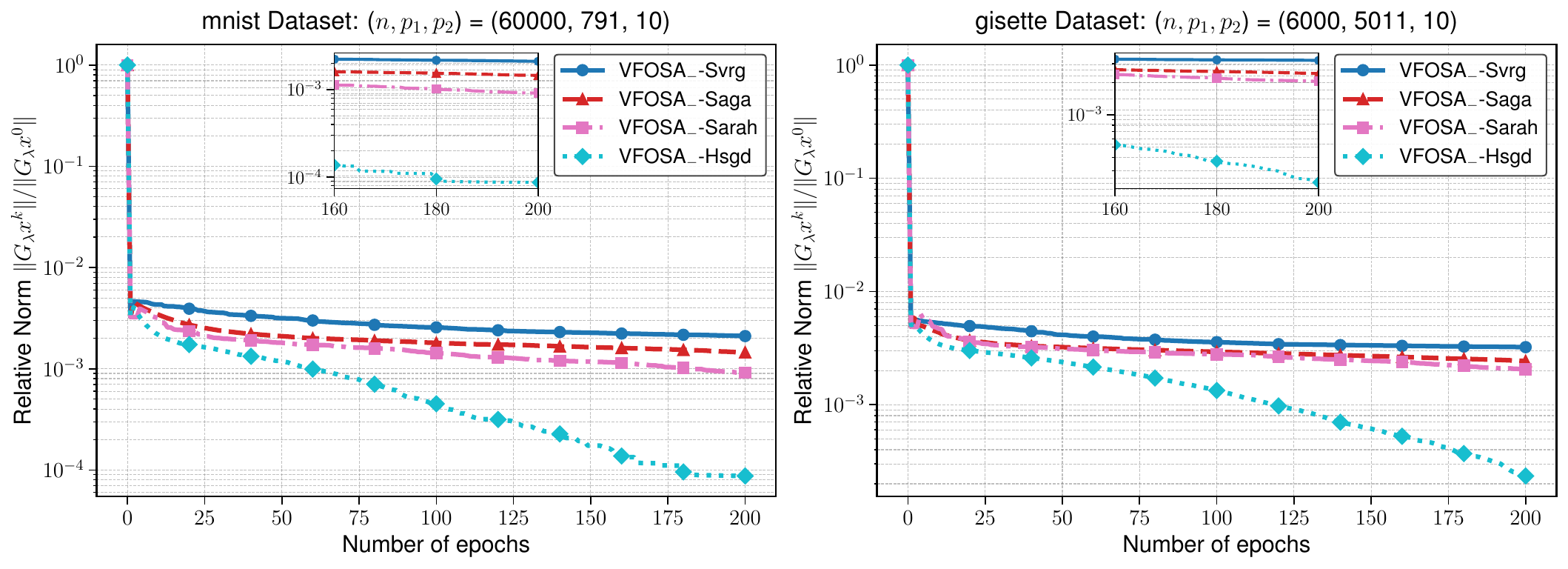}
\vspace{-5ex}
\caption{
The performance of 4 variants of \ref{eq:VrABFS4NI}: L-SVRG, SAGA, L-SARAH, and HSGD for solving \eqref{eq:logistic_reg_exam} with the $\ell_1$-norm regularizer on \texttt{mnist} and  \texttt{gisette}.
}
\label{fig:experiment4a}
\vspace{-3ex}
\end{figure}

\begin{figure}[ht!]
\vspace{-0ex}
\centering
\includegraphics[width=1\textwidth]{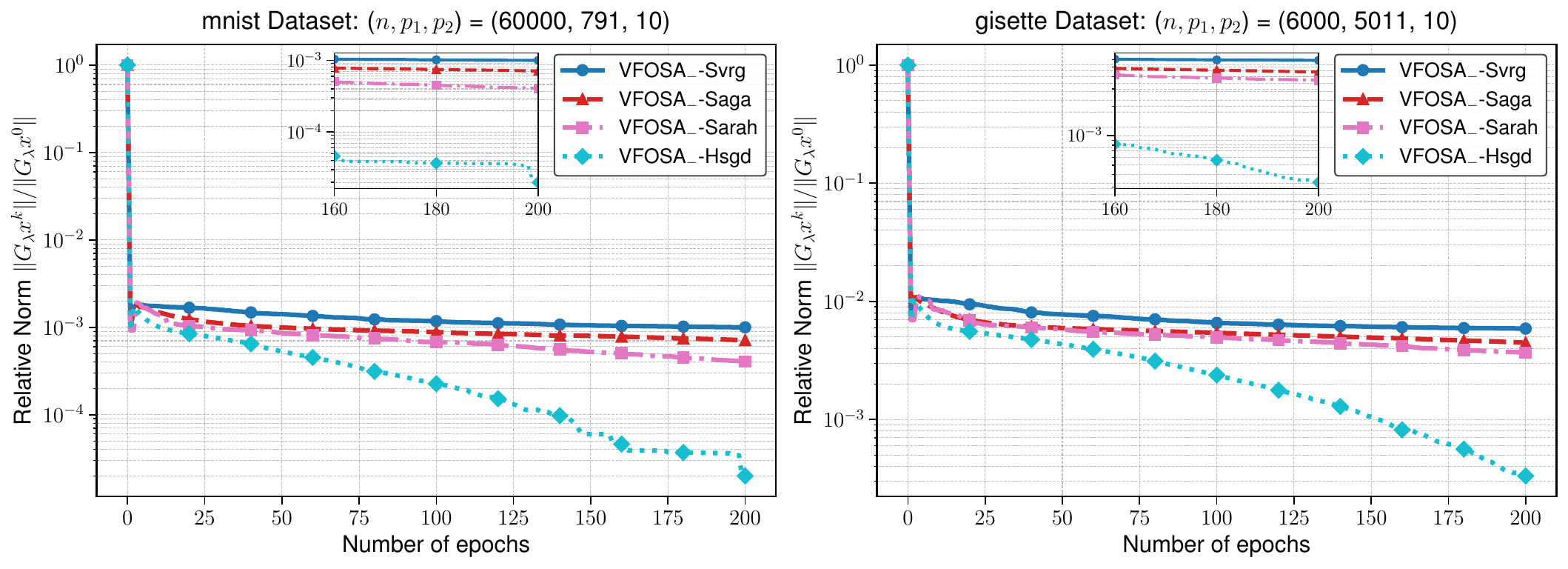}
\vspace{-5ex}
\caption{
The performance of 4 variants of \ref{eq:VrABFS4NI}: L-SVRG, SAGA, L-SARAH, and HSGD for solving \eqref{eq:logistic_reg_exam} with the SCAD regularizer on \texttt{mnist} and  \texttt{gisette}.
}
\label{fig:experiment4b}
\vspace{-4ex}
\end{figure}

We still see that our L-SARAH and HSGD work better than L-SVRG and SAGA.
They quickly reach the $10^{-2} \to 10^{-3}$ accuracies after a few epochs, but then make slow progress later. 
HSGD still outperforms its competitors, especially on the \texttt{gisette} and \texttt{a9a} datasets.

\vspace{0.75ex}
\noindent\textbf{$\mathrm{(d)}$~Comparing  \ref{eq:VrAFBS4NI},  \ref{eq:VrABFS4NI}, and \texttt{VrHalpern}.}
Finally, we compare our methods: \ref{eq:VrAFBS4NI} and  \ref{eq:VrABFS4NI} and  \texttt{VrHalpern} in \citep{cai2023variance}.
For each of our methods, we choose three variants: L-SVRG, L-SARAH, and HSGD as they do not need to store $F_i$ as in SAGA.
We run these algorithms on two datasets: \texttt{mnist} and \texttt{gisette} with the same setting as in the previous experiments for both the monotone and nonmonotone cases.
For \texttt{VrHalpern}, we choose $\lambda_k := \frac{2}{k+4}$ and  $\eta := \frac{1}{2L}$.
This $\eta$ is consistent to our learning rates, but twice larger than the suggested value $\eta = \frac{1}{4L}$ in \citep{cai2023variance}.

We run this experiment for $N_e=200$ epochs as before using both the $\ell_1$-norm and SCAD regularizers.
Figures~\ref{fig:experiment5a} and \ref{fig:experiment5b} reveal the results of these methods for each case, respectively.

\begin{figure}[ht!]
\vspace{-0ex}
\centering
\includegraphics[width=1\textwidth]{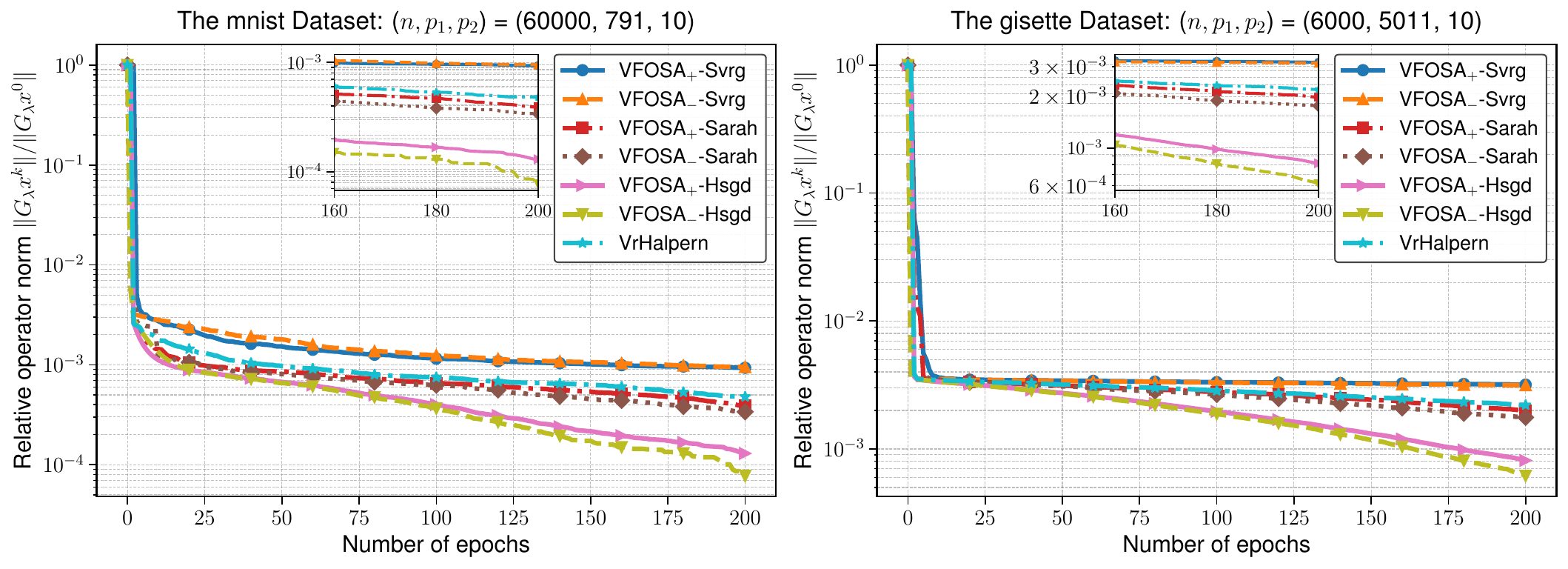}
\vspace{-5ex}
\caption{
The performance of 7 algorithms: 3 variants of each \ref{eq:VrAFBS4NI} and \ref{eq:VrABFS4NI}, and \texttt{VrHalpern} using the $\ell_1$-regularizer on two datasets:  \texttt{mnist} and  \texttt{gisette}.
}
\label{fig:experiment5a}
\vspace{-4ex}
\end{figure}

\begin{figure}[ht!]
\vspace{0ex}
\centering
\includegraphics[width=1\textwidth]{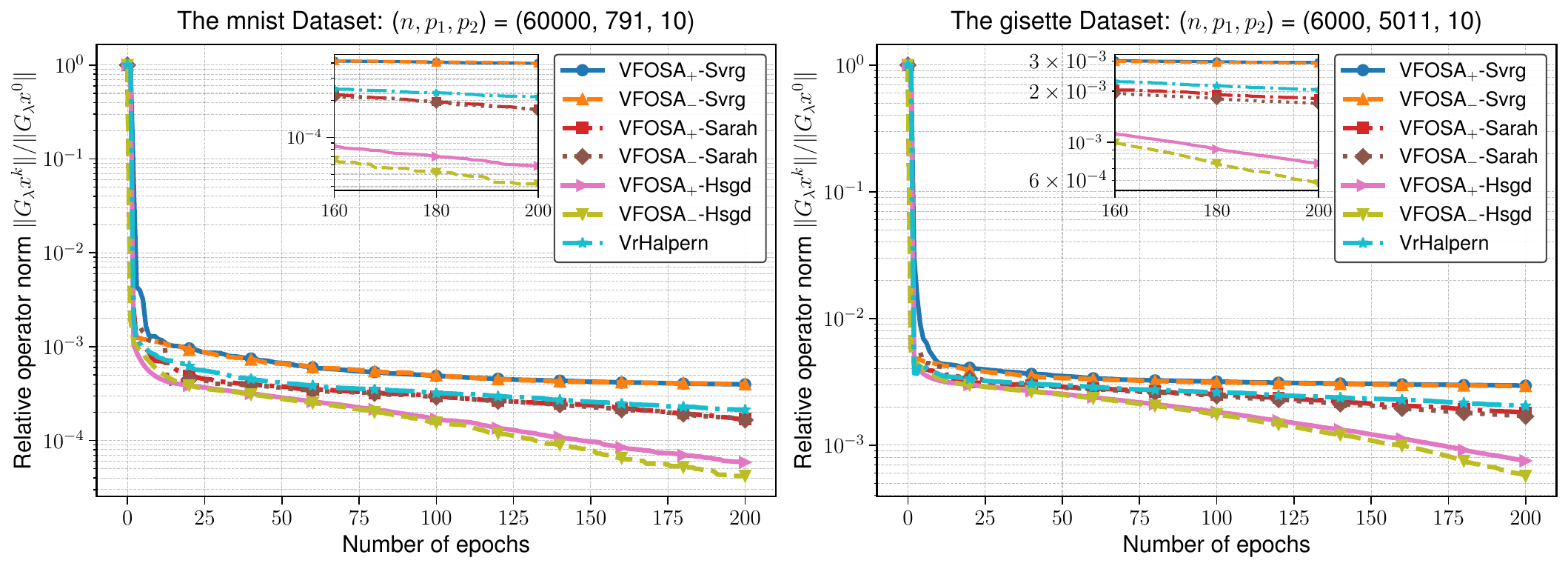}
\vspace{-5ex}
\caption{
The performance of 7 algorithms: 3 variants of each \ref{eq:VrAFBS4NI} and \ref{eq:VrABFS4NI}, and \texttt{VrHalpern} using the SCAD regularizer on 2 datasets:  \texttt{mnist} and  \texttt{gisette}.
}
\label{fig:experiment5b}
\vspace{-4ex}
\end{figure}

As observed in both Figures~\ref{fig:experiment5a} and \ref{fig:experiment5b}, our L-SARAH and HSGD variants perform better than L-SVRG and \texttt{VrHalpern} on the \texttt{mnist} and \texttt{gisette} datasets.
They also outperform L-SVRG and \texttt{VrHalpern} in both the monotone and nonmonotone cases.
Our \texttt{\ref{eq:VrABFS4NI}{\!}-Hsgd} appears to perform slightly better than \texttt{\ref{eq:VrAFBS4NI}{\!}-Hsgd}.
A key reason for the superior performance of the HSGD variants is that they avoid full-batch computations except for the first epoch, reducing the number of $F_i$ evaluations.
\texttt{VrHalpern} performs better than our L-SVRG variants, which aligns with the theoretical complexity results in both methods.

%%%%%%%%%%%%%%%%%%%%%%%%%%%%%%%%%%%%%%%%%%%%%%%%%%%%%%%%
%%%% 6.2. Comparison with other methods --- Bilinear matrix game
%%%%%%%%%%%%%%%%%%%%%%%%%%%%%%%%%%%%%%%%%%%%%%%%%%%%%%%%
\rv{
\beforesubsec
\subsection{Policeman vs. Burglar problem --- Comparison between different methods}\label{subsec:matrix_game}
\aftersubsec
\noindent\textbf{$\mathrm{(a)}$~Mathematical model.}
Following \citet{nemirovski2013mini}, we consider the Policeman vs. Burglar problem as follows.
There are $p_1$ houses in a city, where the $i$-th house has wealth $w_i \in \R_{+}$.
Every evening, the Burglar chooses a house $i$ to attack, and the Policeman chooses his post near a house $j$ for all $1 \leq i, j \leq p_1$.
After the burglary begins, the Policeman becomes aware of where it is happening, and his probability of catching the Burglar is $\mathbf{p}_c := \exp\sets{ -\theta \cdot \textrm{dist}(i, j)}$, where $\textrm{dist}(i, j)$ is the distance between houses $i$ and $j$.
On the other hand, the Burglar seeks to maximize his expected profit $\mbf{L}_{ij} = w_i\big(1 - \exp\set{ -\theta \cdot \textrm{dist}(i, j)}  \big)$, while the Policeman’s interest is completely opposite.

Let $\mbf{L}$ be a $p_1\times p_1$ symmetric matrix such that $\mbf{L}_{ij} := w_i(1 - \exp\{-\theta \cdot \text{dist}(i, j)\})$ for $1 \leq i, j \leq p_1$, and let $\Uc = \Vc := \Delta_{p_1}$ be the standard simplex in $\R^{p_1}$.
Then, the above Policeman vs. Burglar problem can be formulated into the following two-person game:
\begin{equation}\label{eq:matrix_game}
\min_{u \in \mathcal U} \max_{v \in \mathcal V}\big\{ \Hc(u, v) :=  \iprods{ \mbf{L} u, v } \big\},
\end{equation}
where $u$ and $v$ represent mixed strategies of the Policeman and the Burglar, respectively.

The optimality condition  of \eqref{eq:matrix_game} is $ 0 \in Fx + Tx$, which is  a special case of \eqref{eq:NI}, where $x := [u, v]$, $Fx := [ \mbf{L}^\top v; -\mbf{L}u]$, and $Tx := [\partial{\delta}_{\mathcal{U}}(u) , \partial{\delta}_{\mathcal{V}}(v)]$ with $\delta_{\mathcal{X}}$ being the indicator of $\mathcal{X}$. 
Clearly, $F$ is skew-symmetric and thus monotone, but it is not average co-coercive.
To make $F$ average co-coercive, we add a small regularizer so that $Fx = [\epsilon u + \mbf{L}^\top v; \epsilon v -\mbf{L}u]$ for $\epsilon = 10^{-8}$.
This modification does not significantly interfere with \eqref{eq:matrix_game}.

\vspace{0.75ex}
\noindent\textbf{$\mathrm{(b)}$~Generating input data.}
First, we choose $\textrm{dist}(i, j) := \vert i - j \vert$ and set $\theta := 0.8$ that reflects a reasonable probability $\mbf{p}_c$ of catching the Burglar.
Next, we generate a vector $\hat{w} \in \mathbb{R}^{p_1}_{+}$ randomly from a standard normal distribution, followed by taking the absolute value to ensure nonnegativity.
We call this vector the nominal wealth.
Then, the wealth vector $w$ is generated by $w := \vert \hat{w} + \sigma \cdot \texttt{randn}(q)\vert$ as a nonnegative random vector, where $\sigma^2 := 0.05$ is the variance of the noise. 
Now, assume that $\mbf{L} := \frac{1}{n}\sum_{s=1}^n\mbf{L}_s$ is the mean of $n$ samples $\mbf{L}_s$ generated from the samples $w_s$ of $w$ for $s = 1, \cdots, n$.
Using this procedure, we generate two sets of problems corresponding to two experiments as follows.
\begin{compactitem}
\item \textit{Experiment 1.} Choose $p_1 := 100$ houses (on a $10\times 10$ grid) and $n = 1000$ samples, and generate $10$ problem instances of size $p = 2p_1 = 200$.
\item \textit{Experiment 2.} Choose $p_1 := 225$ houses (on a $15\times 15$ grid) and $n = 2000$ samples, and also generate $10$ problem instances of size $p = 2p_1 = 450$.
\end{compactitem}
%%%% 
\vspace{0.75ex}
\noindent\textbf{$\mathrm{(c)}$~Our algorithms and their competitors.}
We select the following five competitors.
\begin{compactitem}
\item The optimistic gradient method, e.g., in \citep{daskalakis2018training}, abbreviated by \texttt{OG}.
It is a non-accelerated deterministic variant of Popov's past-extragradient method.
\item The fast Krasnosel'ki\v{i}-Mann (KM) method in \citet{bot2022bfast,tran2022connection}, called \texttt{FKM}.
This is a Nesterov's accelerated variant of the KM scheme.
\item The variance-reduced forward-reflected-backward splitting (FRBS) algorithm in \citet{alacaoglu2021forward}, abbreviated by \texttt{VrFRBS}.
\item The variance-reduced extragradient (EG) algorithm in \citet{alacaoglu2021stochastic}, abbreviated by \texttt{VrEG}.
This is a non-accelerated variance-reduced EG method.
\item The variance-reduced Halpern fixed-point method in \citet{cai2023variance}, called \texttt{VrHalpern}.
\end{compactitem}
Since the stochastic competitors either use L-SVRG or L-SARAH, we implement two variants of our \ref{eq:VrAFBS4NI}: \texttt{\ref{eq:VrAFBS4NI}{\!}-Svrg} and \texttt{\ref{eq:VrAFBS4NI}{\!}-Sarah} and compare them with the above competitors.
We run each experiment on $10$ problem instances and report the mean of the relative FBS residual norm $\norms{G_{\lambda}x^k}/\norms{G_{\lambda}x^0}$ against the number of epochs for $200$ epochs.
%%%

\vspace{0.75ex}
\noindent\textbf{$\mathrm{(d)}$~Parameter selection.}
We test all the algorithms using the recommended parameters from their theory.
More specifically, for all the L-SVRG variants, we choose the probability for snapshot points $\tilde{x}^k$ as $\mbf{p}_k = \frac{1}{2n^{1/3}}$ and the mini-batch size $b_k := \lfloor \frac{n^{2/3}}{2} \rfloor$, while for all the L-SARAH variants, we choose $\mbf{p}_k = \frac{1}{2\sqrt{n}}$ and $b_k := \lfloor \frac{\sqrt{n}}{2} \rfloor$.
We also choose $x^0 := \frac{2}{p}\cdot \textrm{ones}(p)$ as the initial point in all methods so that it is feasible to \eqref{eq:matrix_game}.
The learning rate of both  \texttt{OG} and \texttt{FKM} is $\eta = \frac{1}{L}$ as suggested by their theory.
The learning rate of \texttt{VrFRBS} is $\eta = 0.99 \cdot \frac{1-\sqrt{1-\mbf{p}_k}}{2L}$ as recommended in \citet{alacaoglu2021forward}.
The learning rate of \texttt{VrEG} is $\eta = 0.99 \cdot \frac{\sqrt{1-\alpha}}{L}$ for $\alpha := 1 - \mbf{p}_k$ as shown in \citet{alacaoglu2021stochastic}.
Note that the learning rate of both \texttt{VrFRBS} and \texttt{VrEG} was derived for the single sample case, while we use it here for the mini-batch case.
However, since $\mbf{p}_k$ is larger than the theoretical value $\frac{1}{n}$ or $\frac{2}{n}$, this learning rate is larger than the one in their paper.
The learning rate of \texttt{VrHalpern} is $\eta := \frac{1}{4L}$ as in \citet{cai2023variance}.
For our \ref{eq:VrAFBS4NI} methods, since $\rho = 0$, we choose $\mu := 0.95 \cdot \frac{2}{3}$, $r := 2 + \frac{1}{\mu}$, $\lambda := \frac{1}{L}$, $\bar{\beta} := \frac{\lambda(4-L\lambda)}{4}$, and $\beta := \frac{(2-\mu)\bar{\beta}}{2+\mu}$ as suggested by our theory in  Theorem~\ref{th:VrAFBS4NI_convergence}.

\vspace{0.75ex}
\noindent\textbf{$\mathrm{(e)}$~Numerical results.}
The performance of all the algorithms are reported in Figure~\ref{fig:matrix_game_1}.

\begin{figure}[ht!]
\vspace{-0ex}
\centering
\includegraphics[width=1\linewidth]{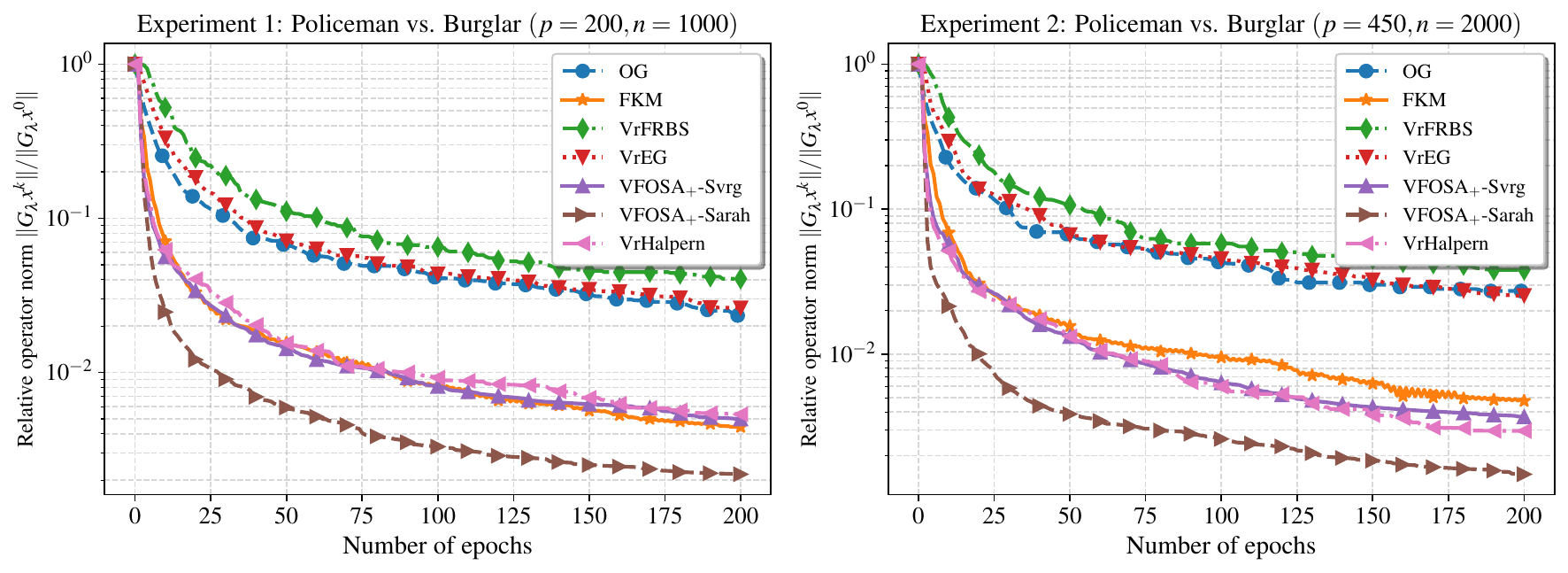}
\vspace{-5ex}
\caption{The performance of $2$ variants of \ref{eq:VrAFBS4NI} and 5 competitors for solving \eqref{eq:matrix_game} using theoretical parameters. 
The average of $10$ problems in each experiment.}
\label{fig:matrix_game_1}
\vspace{-3ex}
\end{figure}

As shown in Figure~\ref{fig:matrix_game_1}, the three accelerated methods consistently outperform their non-accelerated counterparts, including both deterministic and variance-reduced variants.
The three non-accelerated schemes: \texttt{OG}, \texttt{VrFRBS}, and \texttt{VrEG}, exhibit similar performance across both experiments when using their respective theoretical parameters.
Among the accelerated methods, the deterministic algorithm \texttt{FKM} still performs well and is comparable to our \texttt{VFOSA$_{+}${\!}-Svrg} and \texttt{VrHalpern}.
However, our \texttt{VFOSA$_{+}${\!}-Sarah} achieves the best performance, attaining the lowest relative residual norm among all methods.

Finally, we reduce both $\mbf{p}_k$ and $b_k$ by half to obtain a smaller probability and mini-batch size, respectively.
We then rerun both experiments to evaluate how these parameters influence the performance of the stochastic methods.
The results are presented in Figure~\ref{fig:matrix_game_2}.

\begin{figure}[ht!]
\vspace{-0ex}
\centering
\includegraphics[width=1\linewidth]{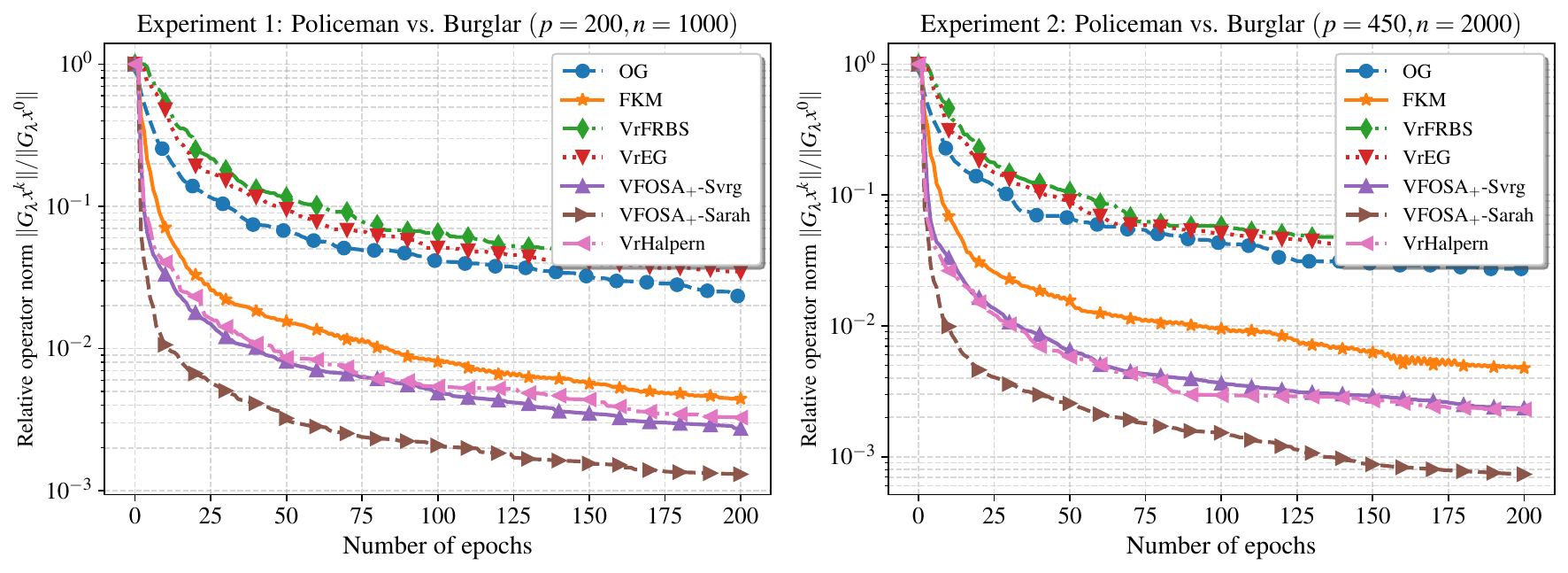}
\vspace{-5ex}
\caption{The performance of $2$ variants of \ref{eq:VrAFBS4NI} and 5 competitors for solving \eqref{eq:matrix_game} using smaller $\mbf{p}_k$ and $b_k$. 
The average of $10$ problems in each experiment.}
\label{fig:matrix_game_2}
\vspace{-3ex}
\end{figure}

As shown in Figure~\ref{fig:matrix_game_2}, the variance-reduced accelerated methods show improved performance compared to the previous run.
They also outperform the deterministic \texttt{FKM} method from the earlier experiment.
This is because the smaller values of $\mbf{p}_k$ and $b_k$ increase the number of iterations within the same number of epochs, enhancing performance.
}
%%%%%%%%%%%%%%%%%%%%%%%%%%%%%%%%%%%%%%%%%%%%%%%%%%%%%%%%
%%%% Conclusions
%%%%%%%%%%%%%%%%%%%%%%%%%%%%%%%%%%%%%%%%%%%%%%%%%%%%%%%%
\rv{
\beforesec
\section{Conclusions}\label{sec:conclusions}
\aftersec
We have developed two fast operator splitting frameworks with variance reduction to solve a class of generalized equations in both finite-sum and expectation settings, covering certain nonmonotone problems.
Our methods exploit both the forward-backward and the backward-forward splitting schemes and support a wide range of variance-reduced estimators, covering both unbiased and biased instances.
We have established convergence rates of order $\BigOs{1/k^2}$ and $\SmallOs{1/k^2}$ in expectation for our methods. 
Then, we have also proved almost sure $\SmallOs{1/k^2}$ convergence rates along with almost sure convergence of the iterates to a solution of our problem.
%%%
Our frameworks comprise popular estimators such as L-SVRG, SAGA, L-SARAH, and Hybrid-SGD, and we have derived oracle complexity results that match or closely approach the best-known in the literature.
%%%
Several interesting questions remain open. 
For example, can our approach be extended to extragradient methods and their variants to weaken the co-coercivity of $F$? 
Can adaptive schemes be developed to remove the need for estimating the co-coercivity constant $L$ and the co-hypomonotonicity modulus $\rho$? 
We plan to explore these topics in our future work.
}

%%%%%%%%%%%%%%%%%%%%%%%%%%%%%%%%%%%%%%%%%%%%%%%%%%%%%%%%
%%%% Acknowledgments.
%%%%%%%%%%%%%%%%%%%%%%%%%%%%%%%%%%%%%%%%%%%%%%%%%%%%%%%%
\vspace{1ex}
\noindent\textbf{Acknowledgements.}
This work is  partially supported by the National Science Foundation (NSF), grant no. NSF-RTG DMS-2134107 and the Office of Naval Research (ONR), grant No. N00014-23-1-2588 (2023-2026).
The author gratefully acknowledges Mr. Nghia Nguyen-Trung for his careful proofreading of this work.

%%%%%%%%%%%%%%%%%%%%%%%%%%%%%%%%%%%%%%%%%%%%%%%%%
%%%% APPENDIX 
%%%%%%%%%%%%%%%%%%%%%%%%%%%%%%%%%%%%%%%%%%%%%%%%%
\appendix
\beforesec
\section{Technical Lemmas and Proof of Lemma~\ref{le:FBS_cocoerciveness}}\label{apdx:A1:useful_lemmas}
\aftersec
This appendix recalls necessary technical results and gives the full proof of Lemma~\ref{le:FBS_cocoerciveness}.

%%%%%%%%%%%%%%%%%%%%%%%%%%%%%%%%%%%%%%%%
%%% A1. Technical lemma.
\beforesubsec
\subsection{Technical lemmas}\label{apdx:tech_lemmas}
\aftersubsec
We need the following technical results for our convergence analysis in the sequel.

%%% Lemma A.1.
\begin{lemma}[\citet{Bauschke2011}, Lemma 5.31]\label{le:A1_descent}
Let $\sets{\alpha_k}$, $\sets{\zeta_k}$, $\sets{\gamma_k}$, and $\sets{\varepsilon_k}$ be nonnegative sequences such that $\sum_{k=0}^{\infty}\gamma_k < +\infty$ and $\sum_{k=0}^{\infty} \varepsilon_k < +\infty$.
In addition, for all $k\geq 0$, we assume that
\begin{equation}\label{eq:lm_A1_cond}
\alpha_{k+1} \leq (1 + \gamma_k)\alpha_k - \zeta_k + \varepsilon_k.
\end{equation}
Then, we conclude that $\lim_{k\to\infty}\alpha_k$ exists and $\sum_{k=0}^{\infty}\zeta_k < +\infty$.
\end{lemma}

%%% Lemma A2.
\begin{lemma}\label{le:A2_sum}
Given a nonnegative sequence $\sets{\alpha_k}$ and $\omega \geq 0$ such that $\lim_{k\to\infty} k^{\omega + 1} \alpha^k$ exists and $\sum_{k=0}^{\infty}k^{\omega} \alpha_k < +\infty$.
Then, we conclude that $\lim_{k\to\infty}k^{\omega + 1} \alpha^k = 0$.
\end{lemma}

%%% Proof of Lemma A2.
\begin{proof}
Since $\alpha_k \geq 0$, suppose by contradiction that $\lim_{k\to\infty}k^{\omega + 1} \alpha^k = \alpha > 0$.
For any $0 < \epsilon < \alpha$, there exists $k_0$  sufficiently large such that $k^{\omega+1}\alpha_k \geq \alpha - \epsilon > 0$ for all $k \geq k_0$.
Hence, we get $k^{\omega}\alpha_k \geq \frac{\alpha - \epsilon}{k}$.
However, since $\sum_{k=0}^{\infty}k^{\omega} \alpha_k < +\infty$, the last relation leads to
\begin{equation*}
\begin{array}{lcl}
+\infty < \sum_{k=k_0}^{\infty}\frac{\alpha - \epsilon}{k} \leq  \sum_{k=k_0}^{\infty}k^{\omega} \alpha_k < +\infty.
\end{array}
\end{equation*}
This relation shows a contradiction. 
Thus, we conclude that $\lim_{k\to\infty}k^{\omega + 1} \alpha^k = \alpha = 0$.
\end{proof}
%%% End of Proof.

We also need the well-known Robbins-Siegmund  supermartingale theorem \citep{robbins1971convergence}, which we state it here as a technical lemma.

%%% Lemma A.3.
\begin{lemma}\label{le:RS_lemma}
Let $\sets{U_k}$, $\sets{\gamma_k}$, $\sets{V_k}$ and $\sets{E_k}$ be  sequences of nonnegative integrable random variables on some arbitrary probability space and adapted to the filtration $\set{\Fc_k}_{k\geq 0}$ with $\sum_{k=0}^{\infty}\gamma_k < +\infty$ and $\sum_{k=0}^{\infty}E_k < +\infty$ almost surely, and 
\begin{equation}\label{eq:RS_martingale_cond}
\Expn{ U_{k+1} \mid \Fc_k} \leq (1 + \gamma_k)U_k - V_k + E_k, 
\end{equation}
almost surely for all $k \geq 0$.
Then, $\sets{U_k}$ almost surely converges to a random variable and $\sum_{k=0}^{\infty}V_k <+\infty$ almost surely.
\end{lemma}

The following lemma is Proposition 4.1 of \citet{davis2022variance}.
It was proven for a demiclosed mapping $G$, but we recall it here for the case $G$ is continuous in a finite-dimensional space.

\begin{lemma}[\citet{davis2022variance}, Proposition 4.1]\label{le:A3_lemma}
Suppose that $G : \R^p\to\R^p$ is continuous and $\zer{G}\neq\emptyset$.
Let $\sets{x^k}$ be a sequence of random vectors such that for all $x^{\star} \in \zer{G}$, the sequence $\sets{\norms{x^k - x^{\star}} }$ almost surely converges to a $[0, \infty)$-valued random variable. 
In addition, assume that $\sets{\norms{Gx^k}}$ also almost surely converges to zero.
Then, $\sets{x^k}$ almost surely converges to a $\zer{G}$-valued random variable.
\end{lemma}

%%%%%%%%%%%%%%%%%%%%%%%%%%%%%%%%%%%%%%%%
%%%% A.1. The proof of Lemma 1.
\beforesubsec
\subsection{The proof of Lemma~\ref{le:FBS_cocoerciveness} --- Equivalent reformulations of \eqref{eq:NI}}\label{apdx:le:FBS_cocoerciveness}
\aftersubsec
%% Proof of Lemma 3.1.
\begin{proof}
Let $A_{\lambda T}x := \frac{1}{\lambda}(x - J_{\lambda T}x)$ be the Moreau-Yosida approximation of $\lambda T$.
First, we show that $A_{\lambda T}$ is $(\lambda-\rho)$-co-coercive, provided that $\lambda > \rho$.
The \rv{co-coercivity} of $A_{\lambda T}$ was proven in \citet{attouch2018backward}, but we give a short proof here for completeness.

Indeed, for any $x$ and $y$, we denote by $u := J_{\lambda T}x$ and $v := J_{\lambda T}y$.
Then, we have $A_{\lambda T}x = \frac{1}{\lambda}(x - u) \in Tu$ and $A_{\lambda T}y = \frac{1}{\lambda}(y - v) \in Tv$.
Since $T$ is $\rho$-co-hypomonotone, we have $\iprods{A_{\lambda T}x - A_{\lambda T}y, u - v} \geq -\rho\norms{A_{\lambda T}x - A_{\lambda T}y}^2$.
Substituting $u = x - \lambda A_{\lambda T}x$ and $v = y - \lambda A_{\lambda T}y$ into this inequality, and rearranging the result, we get
$\iprods{A_{\lambda T}x - A_{\lambda T}y, x - y} \geq  (\lambda - \rho)\norms{A_{\lambda T}x - A_{\lambda T}y}^2$, which proves that $A_{\lambda T}$ is $(\lambda-\rho)$-co-coercive, provided that $\lambda > \rho$.

(i)~To prove \eqref{eq:G_cocoerciveness}, from \eqref{eq:FBS_residual} we have  $A_{\lambda T}(x - \lambda Fx) = \frac{1}{\lambda }( x - \lambda Fx  - J_{\lambda T}(x - \lambda Fx) ) = G_{\lambda}x - Fx$ and $A_{\lambda}(y - \lambda Fy) = G_{\lambda}y - Fy$. 
By the $(\lambda-\rho)$-\rv{co-coercivity} of  $A_{\lambda T}$, we have 
\begin{equation*}
\iprods{ G_{\lambda}x - G_{\lambda}y - (Fx - Fy), x - y - \lambda(Fx - Fy)} \geq (\lambda - \rho) \norms{G_{\lambda}x - G_{\lambda}y - (Fx - Fy)}^2.
\end{equation*}
Expanding this inequality and rearranging the result, we get
\begin{equation*} 
\arraycolsep=0.2em
\begin{array}{lcl}
\iprods{G_{\lambda}x - G_{\lambda}y, x - y} & \geq &  (\lambda-\rho)\norms{G_{\lambda}x - G_{\lambda}y}^2 -  (\lambda-2\rho)\iprods{G_{\lambda}x - G_{\lambda}y, Fx - Fy}  \vspace{1ex}\\
&& + {~} \iprods{Fx - Fy, x - y} - \rho\norms{Fx - Fy}^2 \vspace{1ex}\\
& = & (\lambda - \rho)\norms{G_{\lambda}x - G_{\lambda}y}^2 -  (\lambda - 2\rho) \iprods{G_{\lambda}x - G_{\lambda}y, Fx - Fy}  \vspace{1ex}\\
&& + {~} \frac{L}{\hat{L}}\iprods{Fx - Fy, x - y}  + \frac{\hat{L} - L}{\hat{L}}\iprods{Fx - Fy, x - y} - \rho\norms{Fx - Fy}^2.
\end{array} 
\end{equation*}
Since $F$ is $\frac{1}{L}$-co-coercive by our assumption, the last inequality leads to
\begin{equation*}
\arraycolsep=0.2em
\begin{array}{lcl}
\iprods{G_{\lambda}x - G_{\lambda}y, x - y} &\geq & (\lambda - \rho)\norms{G_{\lambda}x - G_{\lambda}y}^2 + \big(\frac{1}{\hat{L}} - \rho\big)\norms{Fx - Fy}^2 \vspace{1ex}\\
&& - {~} (\lambda - 2\rho) \iprods{G_{\lambda}x - G_{\lambda}y, Fx - Fy}  + \frac{\hat{L} - L}{\hat{L}}\iprods{Fx - Fy, x - y} \vspace{1ex}\\
& = & \frac{4(1- \hat{L} \rho)(\lambda - \rho) - (\lambda - 2\rho)^2\hat{L}}{4(1 - \hat{L}\rho)} \norms{G_{\lambda}x - G_{\lambda}y}^2  + \frac{\hat{L} - L}{\hat{L}}\iprods{Fx - Fy, x - y} \vspace{1ex}\\
&& + {~} \frac{1 - \hat{L} \rho }{\hat{L}}\big\Vert Fx - Fy - \frac{(\lambda-2\rho)\hat{L}}{2(1- \hat{L}\rho )}(G_{\lambda}x - G_{\lambda}y) \big\Vert^2 \vspace{1ex}\\
& \geq &  \frac{\lambda(4 - \hat{L}\lambda)  - 4\rho }{4(1 - \hat{L} \rho )}  \norms{G_{\lambda}x - G_{\lambda}y}^2 + \frac{\hat{L} - L}{\hat{L}}\iprods{Fx - Fy, x - y},
\end{array} 
\end{equation*}
which exactly proves \eqref{eq:G_cocoerciveness}, where $\bar{\beta} :=  \frac{\lambda(4  - \hat{L}\lambda) - 4\rho}{4(1 - \hat{L} \rho )} \geq 0$ and $\Lambda := \frac{\hat{L} - L}{L\hat{L}}$, provided that $\rho\hat{L} < 1$ and $\rho < \lambda \leq \frac{2 + 2\sqrt{1 - \hat{L}\rho}}{\hat{L}}$.

(ii)~To prove \eqref{eq:S_cocoerciveness}, 
we denote by $x := J_{\lambda T}u$ and $y := J_{\lambda T}v$ for given $u, v \in \dom{T}$, where $\lambda > \rho$.
Then, we have $A_{\lambda}u := \frac{1}{\lambda}(u - J_{\lambda T}u) = \frac{1}{\lambda}(u - x)$.
Now, by \eqref{eq:BFS_residual}, we have $A_{\lambda}u = S_{\lambda}u - Fx$.
Similarly, we also have $A_{\lambda}v = S_{\lambda}v - Fy$.
Using these two relations and the $(\lambda-\rho)$-co-coercivity of $A_{\lambda}$, we can show that
\begin{equation*}
\arraycolsep=0.2em
\begin{array}{lcl}
\iprods{S_{\lambda}u - S_{\lambda}v - (Fx - Fy), u - v} & \geq  & (\lambda - \rho)\norms{S_{\lambda}u - S_{\lambda}v - (Fx - Fy)}^2.
\end{array} 
\end{equation*}
Utilizing again $x - \lambda Fx = u - \lambda S_{\lambda}u$ from \eqref{eq:BFS_residual},
the last inequality leads to 
\begin{equation*}
\arraycolsep=0.1em
\begin{array}{lcl}
\iprods{S_{\lambda}u - S_{\lambda}v, u - v} & \geq  &  (\lambda-\rho) \norms{S_{\lambda}u - S_{\lambda}v}^2  +  (\lambda-\rho)\norms{Fx - Fy}^2 \vspace{1ex}\\
&& - {~} (\lambda - 2\rho)\iprods{S_{\lambda}u - S_{\lambda}v, Fx - Fy} +  \iprods{Fx - Fy, u - v - \lambda (S_{\lambda}u  -  S_{\lambda}v)}\vspace{1ex}\\
& = & (\lambda-\rho) \norms{S_{\lambda}u - S_{\lambda}v}^2    -  (\lambda - 2\rho)\iprods{S_{\lambda}u - S_{\lambda}v, Fx - Fy} \vspace{1ex}\\
&& + {~}   \frac{L}{\hat{L}} \iprods{Fx - Fy, x - y } +  \frac{\hat{L} - L}{\hat{L}} \iprods{Fx - Fy, x - y } - \rho \norms{Fx - Fy}^2.
\end{array} 
\end{equation*}
Substituting $\iprods{Fx - Fy, x - y} \geq \frac{1}{L}\norms{Fx - Fy}^2$ from the $\frac{1}{L}$-co-coercivity of $F$ into the last inequality, we can further derive that
\begin{equation*}
\arraycolsep=0.2em
\begin{array}{lcl}
\iprods{S_{\lambda}u - S_{\lambda}v, u - v} & \geq  &   (\lambda - \rho)\norms{S_{\lambda}u - S_{\lambda}v}^2 + \big(\frac{1}{\hat{L}} - \rho\big)\norms{Fx - Fy}^2 \vspace{1ex}\\
&& - {~} (\lambda - 2\rho) \iprods{S_{\lambda}u - S_{\lambda}v, Fx - Fy}  + \frac{\hat{L} - L}{\hat{L}}\iprods{Fx - Fy, x - y} \vspace{1ex}\\
& = & \frac{4(1 - \hat{L}\rho )(\lambda - \rho) - (\lambda - 2\rho)^2\hat{L}}{4(1 -  \hat{L}\rho )} \norms{S_{\lambda}u - S_{\lambda}v}^2  + \frac{(\hat{L} - L)}{\hat{L}}\iprods{Fx - Fy, x - y} \vspace{1ex}\\
&& + {~} \frac{1- \hat{L}\rho}{\hat{L}}\norms{Fx - Fy - \frac{(\lambda-2\rho)\hat{L}}{2(1- \hat{L}\rho)}(S_{\lambda}u - S_{\lambda}v)}^2 \vspace{1ex}\\
& \geq &  \frac{\lambda(4 - \hat{L}\lambda)  - 4\rho}{4(1 -  \hat{L}\rho)}  \norms{S_{\lambda}u - S_{\lambda}v}^2 + \frac{\hat{L} - L}{\hat{L}}\iprods{Fx - Fy, x - y}.
\end{array} 
\end{equation*}
This proves \eqref{eq:S_cocoerciveness} with $\bar{\beta} :=  \frac{\lambda(4 - \hat{L}\lambda) - 4\rho}{4(1 -  \hat{L}\rho)}  \geq 0$ and $\Lambda := \frac{\hat{L} - L}{L\hat{L}} \geq 0$ as in Statement (ii). 

(iii)~Finally, since $J_{\lambda T}x = x - \lambda A_{\lambda T}x$ and $J_{\lambda T}y = y - \lambda A_{\lambda T}y$, using the $(\lambda-\rho)$-\rv{co-coercivity} of $A_{\lambda T}$, we can show that
\begin{equation*} 
\arraycolsep=0.2em
\begin{array}{lcl}
\norms{J_{\lambda T}x - J_{\lambda T}y}^2 &= & \norms{x - y - \lambda (A_{\lambda T}x - A_{\lambda T}y)}^2 \vspace{1ex}\\
&= & \norms{x - y}^2 - 2\lambda\iprods{A_{\lambda T}x - A_{\lambda T}y, x - y} + \lambda^2\norms{A_{\lambda T}x - A_{\lambda T}y}^2 \vspace{1ex}\\
& \leq & \norms{x - y}^2 - \lambda\big(\lambda - 2\rho)\norms{A_{\lambda T}x - A_{\lambda T}y}^2.
\end{array} 
\end{equation*}
Thus, if $\lambda \geq 2\rho$, then $\norms{J_{\lambda T}x - J_{\lambda T}y} \leq \norms{x - y}$, implying that $J_{\lambda T}$ is nonexpansive.
\end{proof}
%%% End of Proof.

%%%%%%%%%%%%%%%%%%%%%%%%%%%%%%%%%%%%%%%%%%%%%%%%%%
%%% A. The Proof of Technical Results in Section B.
%%%%%%%%%%%%%%%%%%%%%%%%%%%%%%%%%%%%%%%%%%%%%%%%%%
\vspace{-2ex}
\beforesec
\section{The Proof of Technical Results in Section~\ref{sec:VR_estimators}}\label{apdx:sec:VR_estimators}
\aftersec
We provide the full proof of all technical lemmas in the main text of Section~\ref{sec:VR_estimators}.
%In what follows, we denote by $\Fc_k := \sigma(x^0, \cdots, x^k)$ as the $\sigma$-field generated by $x^0, \cdots, x^k$.

%%%%%%%%%%%%%%%%%%%%%%%%%%%%%%%%%%%%%%%%
\vspace{-0.5ex}
\beforesubsec
\subsection{The proof of Lemma~\ref{le:loopless_svrg_bound} --- The L-SVRG estimator}\label{apdx:proof_SVRG_estimator_bound}
\aftersubsec
%%% Proof of Lemma 7.
\begin{proof}
At the $k$-th iteration, we have $3$ independent random quantities: a mini-batch $\Sc_k$, a mega-batch $\bar{\Sc}_k$ to form $\bar{F}\tilde{x}^k$, and a Bernoulli's random variable $i_k$ following the rule \eqref{eq:xy_hat}.

Define $\widehat{F}^k := F\tilde{x}^k + \Fb(x^k, \Sc_k) -  \Fb(\tilde{x}^k,  \Sc_k)$.
Then, from \eqref{eq:loopless_svrg}, we have $\widetilde{F}^k = \widehat{F}^k + \bar{F}\tilde{x}^k - F\tilde{x}^k$.
By Young's inequality, for any $\tau > 0$, we can show that
\begin{equation}\label{eq:loopless_svrg_proof1}
\Expsn{(\Sc_k, \bar{\Sc}_k)}{\norms{\widetilde{F}^k - Fx^k}^2 } \leq (1 + \tau)\Expsn{\Sc_k}{ \norms{\widehat{F}^k - Fx^k}^2 } + \tfrac{1 + \tau}{\tau}\Expsn{\bar{\Sc}_k}{ \norms{ \bar{F}\tilde{x}^k - F\tilde{x}^k }^2 }.
\end{equation}
For the expectation setting \eqref{eq:expectation_form}, we consider $X^k(\xi) := \Fb(x^k, \xi) - \Fb(\tilde{x}^k, \xi) - (Fx^k - F\tilde{x}^k)$.
Then, we have $\Expsn{\xi}{X^k(\xi)} = 0$.
Since $\Sc_k$ is i.i.d., we can show that
\begin{equation*}
\arraycolsep=0.2em
\begin{array}{lcl}
\Expsn{\Sc_k}{\norms{\widehat{F}^k - Fx^k}^2 } & = & \Expsn{\Sc_k}{\norms{ F\tilde{x}^k  + \Fb(x^k, \Sc_k) -  \Fb(\tilde{x}^k,  \Sc_k) - Fx^k}^2 } \vspace{1ex}\\
& = & \Expsn{\Sc_k}{ \Vert \frac{1}{b_k}\sum_{\xi_i \in\Sc_k}  \big[ \Fb(x^k, \xi_i) -  \Fb(\tilde{x}^k,  \xi_i) - (Fx^k - F\tilde{x}^k) \big] \Vert^2  } \vspace{1ex}\\
& = & \Expsn{\Sc_k}{  \Vert \frac{1}{b_k}\sum_{\xi_i \in\Sc_k} X^k(\xi_i) \Vert^2 }\vspace{1ex}\\
& = & \frac{1}{b_k} \Expsn{\xi}{ \norms{ X^k(\xi)}^2 } \vspace{1ex}\\
& \leq & \frac{1}{b_k}\Expsn{\xi}{ \norms{\Fb(x^k, \xi) - \Fb(\tilde{x}^k, \xi)}^2 }.
\end{array}
\end{equation*}
For the finite-sum case \eqref{eq:finite_sum_form},  we denote by $X^k_i := F_ix^k - F_i\tilde{x}^k - (Fx^k - F\tilde{x}^k)$.
Then, $\Expsn{i}{X^k_i} = 0$.
\rv{Similar to \citet[Lemma 2]{Pham2019}, we can show that}
\begin{equation*}
\arraycolsep=0.2em
\begin{array}{lcl}
\Expsn{\Sc_k}{\norms{\widehat{F}^k - Fx^k}^2  } & = & \Expsn{\Sc_k}{\norms{ F\tilde{x}^k  + \Fb(x^k, \Sc_k) -  \Fb(\tilde{x}^k,  \Sc_k) - Fx^k}^2 } \vspace{1ex}\\
& = & \Expsn{\Sc_k}{ \Vert \frac{1}{b_k}\sum_{i \in\Sc_k} \big[ F_ix^k -  F_i\tilde{x}^k - (Fx^k - F\tilde{x}^k) \big] \Vert^2 } \vspace{1ex}\\
& = & \Expsn{\Sc_k}{ \Vert \frac{1}{b_k}\sum_{i \in\Sc_k} X^k_i \Vert^2   } \vspace{1ex}\\
& \leq & \frac{n-b_k}{(n-1)b_k} \cdot \frac{1}{n} \sum_{i=1}^n\norms{F_ix^k - F_i\tilde{x}^k}^2 \vspace{1ex}\\
& \leq & \ \frac{1}{b_k n} \sum_{i=1}^n\norms{F_ix^k - F_i\tilde{x}^k}^2 \vspace{1ex}\\
& = & \frac{1}{b_k}\Expsn{\xi}{ \norms{\Fb(x^k, \xi) - \Fb(\tilde{x}^k, \xi)}^2 }, \quad \textrm{where $\mbf{F}(x,\xi) := F_ix$}.
\end{array}
\end{equation*}
Combining either the first or second relation above and \eqref{eq:loopless_svrg_proof1}, and then taking the conditional expectation $\Expsn{k}{\cdot}$ on both sides of the result, we get
\begin{equation}\label{eq:loopless_svrg_proof1b}
\hspace{-1ex}
\arraycolsep=0.2em
\begin{array}{lcl}
\Expsn{k}{\norms{\widetilde{F}^k - Fx^k}^2 } \leq \frac{1 + \tau }{b_k}\Expsn{k}{ \Expsn{\xi}{ \norms{\Fb(x^k, \xi) - \Fb(\tilde{x}^k, \xi)}^2 } } + \tfrac{1 + \tau}{\tau}\Expsn{k}{ \norms{ \bar{F}\tilde{x}^k - F\tilde{x}^k }^2 }.
\end{array}
\hspace{-1ex}
\end{equation}
If we define $\hat{\Delta}_k := \frac{1}{b_k}\Expsn{\xi}{ \norms{\Fb(x^k, \xi) - \Fb(\tilde{x}^k, \xi)}^2 }$, then \eqref{eq:loopless_svrg_proof1b} implies \eqref{eq:loopless_svrg_bound2}.

Next, for a Bernoulli's random variable $i_k$ following the rule \eqref{eq:xy_hat}, we have
\begin{equation*}
\arraycolsep=0.2em
\begin{array}{lcl}
\Expsn{\xi,i_k}{\norms{ \Fb(x^k, \xi) - \Fb(\tilde{x}^k, \xi)}^2 } & = &  \mbf{p}_k \Expsn{\xi}{\norms{ \Fb(x^k, \xi) - \Fb(x^{k-1}, \xi)}^2} \vspace{1ex} \\
&& + {~} (1- \mbf{p}_k)\Expsn{\xi}{\norms{ \Fb(x^k, \xi) - \Fb(\tilde{x}^{k-1}, \xi)}^2 }.
\end{array}
\end{equation*}
Now, for any $c > 0$, by Young's inequality, we have
\begin{equation*}
\arraycolsep=0.2em
\begin{array}{lcl}
\Expsn{\xi}{\norms{ \Fb(x^k, \xi) - \Fb(\tilde{x}^{k-1}, \xi)}^2 } & \leq & (1 + c)\Expsn{\xi}{\norms{ \Fb(x^{k-1}, \xi) - \Fb(\tilde{x}^{k-1}, \xi)}^2 }  \vspace{1ex} \\
&& + {~} \left(1 +  \frac{1}{c} \right) \Expsn{\xi}{\norms{ \Fb(x^k, \xi) - \Fb(x^{k-1}, \xi)}^2 }.
\end{array}
\end{equation*}
Combining the last two expressions, taking the conditional expectation $\Expsn{k}{\cdot}$ on both sides of the result, and using the definition of $\hat{\Delta}_k$ and $b_{k-1} \leq b_k$, we can show that
\begin{equation*}
\arraycolsep=0.2em
\begin{array}{lcl}
\Expsn{k}{ \hat{\Delta}_k }  &\leq &  (1 + c)(1-\mbf{p}_k) \hat{\Delta}_{k-1} + \frac{1}{b_k}\left[ (1-\mbf{p}_k )\left( 1 + \frac{1}{c} \right) + \mbf{p}_k  \right] \Expsn{\xi}{\norms{ \Fb(x^k, \xi) - \Fb(x^{k-1}, \xi)}^2}.
\end{array}
\end{equation*}
Let us choose $c := \frac{(1-\alpha)\mbf{p}_k}{1-\mbf{p}_k}$ for some $\alpha \in (0, 1)$.
Then, we get $(1 + c)(1- \mbf{p}_k) = 1 - \alpha \mbf{p}_k$ and $(1- \mbf{p}_k)\left(1 + \frac{1}{c} \right) + \mbf{p}_k = \frac{1 - (1+\alpha)\mbf{p}_k  + \mbf{p}_k^2}{(1-\alpha)\mbf{p}_k} \leq \frac{1}{(1-\alpha)\mbf{p}_k}$ for any $0 \leq \mbf{p}_k \leq 1$.
Hence, we obtain
\begin{equation*}
\arraycolsep=0.2em
\begin{array}{lcl}
\Expsn{k}{ \hat{\Delta}_k } &\leq & (1 - \alpha \mbf{p}_k ) \hat{\Delta}_{k-1}  + \frac{ 1 }{(1-\alpha) b_k \mbf{p}_k} \cdot \Expsn{\xi}{\norms{ \Fb(x^k, \xi) - \Fb(x^{k-1}, \xi)}^2},
\end{array}
\end{equation*}
which proves \eqref{eq:loopless_svrg_bound}.

Next, since $\Expsn{\bar{\Sc}_k}{ \norms{ \bar{F}\tilde{x}^k - F\tilde{x}^k }^2 } \leq \frac{\sigma^2}{n_k}$, \eqref{eq:loopless_svrg_proof1b} implies
\begin{equation}\label{eq:loopless_svrg_proof1c}
\arraycolsep=0.2em
\begin{array}{lcl}
\Expsn{k}{\norms{\widetilde{F}^k - Fx^k}^2 } \leq (1 + \tau)\Expsn{k}{ \hat{\Delta}_k } + \tfrac{(1 + \tau)\sigma^2}{\tau n_k}.
\end{array}
\end{equation}
Let us define $\Delta_k := (1 + \tau)\hat{\Delta}_k + \frac{(1 + \tau)\sigma^2}{\tau n_k}$.
Then, plugging $\Delta_k$ from \eqref{eq:loopless_svrg_proof1c} into \eqref{eq:loopless_svrg_bound} and using $n_{k-1}\leq n_k$, we can show that
\begin{equation}\label{eq:loopless_svrg_proof1e}
\arraycolsep=0.2em
\begin{array}{lcl}
\Expsn{k}{ \Delta_k } &\leq & (1 - \alpha \mbf{p}_k )\Delta_{k-1}  + \frac{1+ \tau}{(1-\alpha)b_k \mbf{p}_k } \Expsn{\xi}{\norms{ \Fb(x^k, \xi) - \Fb(x^{k-1}, \xi)}^2} + \frac{(1 + \tau)\alpha \mbf{p}_k \sigma^2}{\tau n_k}.
\end{array}
\end{equation}
If  we choose $\tau := 1$, then using \eqref{eq:loopless_svrg_proof1c} and \eqref{eq:loopless_svrg_proof1e}, we can show that $\widetilde{F}^k$ satisfies Definition~\ref{de:VR_Estimators} with $\Delta_k := 2\hat{\Delta}_k + \frac{2\sigma^2}{n_k}$, $\kappa_k := \alpha\mbf{p}_k$, $\Theta_k := \frac{2}{(1-\alpha)b_k\mbf{p}_k}$,  and $\sigma_k^2 := \frac{2\alpha\mbf{p}_k \sigma^2}{n_k}$.

Finally, if $\bar{F}\tilde{x}^k = F\tilde{x}^k$, then by setting $\tau = 0$ in \eqref{eq:loopless_svrg_bound2} and combining the result and \eqref{eq:loopless_svrg_bound}, they imply that $\widetilde{F}^k$ satisfies Definition~\ref{de:VR_Estimators} with $\Delta_k = \hat{\Delta}_k$,  $\kappa_k = \alpha\mbf{p}_k$, $\Theta_k = \frac{1}{(1-\alpha)b_k\mbf{p}_k}$,  and $\sigma_k^2 = 0$.
\end{proof}
%%% End of the proof.

%%%%%%%%%%%%%%%%%%%%%%%%%%%%%%%%%%%%%%%%
%%% Proof of Lemma B.2.
\beforesubsec
\subsection{The proof of Lemma~\ref{le:SAGA_estimator_full} --- The SAGA estimator}\label{apdx:proof_SAGA_estimator_bound}
\aftersubsec
\begin{proof} 
Let $X^k_i := F_ix^k -  \hat{F}_i^k$ for all $i \in [n]$.
Then, we have $\Expsn{i}{X^k_i} = Fx^k - \frac{1}{n}\sum_{j=1}^n\hat{F}^k_j$ for any $i \in [n]$.
Therefore, we can derive
\begin{equation*}
\arraycolsep=0.2em 
\begin{array}{lcl}
\Expsn{\Sc_k}{ \norms{  \widetilde{F}^k - Fx^k }^2 } & = & \Expsn{\Sc_k}{\norms{ \frac{1}{b_k }\sum_{i \in\Sc_k}X^k_i  -  \big[ Fx^k  -  \frac{1}{n}\sum_{j=1}^n \hat{F}_j^k \big] }^2 } \vspace{1ex} \\
& = & \Expsn{\Sc_k}{ \norms{ \frac{1}{b_k}\sum_{i \in\Sc_k}\big( X^k_i  -  \Expsn{i}{X^k_i} ] \big) }^2   } \vspace{1ex} \\
& = & \frac{1}{b_k^2}  \sum_{i\in\Sc_k} \Expsn{i}{\norms{X^k_i - \Expsn{i}{X^k_i} }^2   } \vspace{1ex} \\
& \leq & \frac{1}{b_k^2} \sum_{i\in\Sc_k} \Expsn{i}{ \norms{  X_i^k }^2  }  \vspace{1ex} \\
& = & \frac{1}{n b_k}  \sum_{i = 1}^n  \norms{  F_ix^k - \hat{F}_i^k }^2.
\end{array}
\end{equation*}
Taking the conditional expectation $\Expsn{k}{\cdot}$ on both sides of this inequality and using $\Delta_k := \frac{1}{n b_k}  \sum_{i = 1}^n  \norms{  F_ix^k - \hat{F}_i^k }^2$, we obtain the first line of \eqref{eq:SAGA_variance}.

Now, from the definition of $\Delta_k$ and the update rule \eqref{eq:SAGA_ref_points}, for any $c > 0$, by Young's inequality, we can show that
\begin{equation*}
\arraycolsep=0.2em
\begin{array}{lcl}
\Expsn{\Sc_k}{ \Delta_k } & = &  \frac{1}{n b_k}  \sum_{i = 1}^n \Expsn{\Sc_k}{  \norms{  F_ix^k - \hat{F}_i^k }^2 } \vspace{1ex} \\
& \overset{ \eqref{eq:SAGA_ref_points} }{ = } & \big(1 - \frac{b_k}{n}\big)  \frac{1}{n b_k}  \sum_{i = 1}^n  \norms{  F_ix^k - \hat{F}_i^{k-1} }^2  +   \frac{b_k}{n} \cdot  \frac{1}{n b_k}  \sum_{i = 1}^n   \norms{  F_ix^k - F_ix^{k-1} }^2  \vspace{1ex} \\
& \leq &   \frac{(1+c)b_{k-1}}{b_k} \big(1 - \frac{b_k}{n}\big) \frac{1}{nb_{k-1}}  \sum_{i = 1}^n  \norms{  F_ix^{k-1} - \hat{F}_i^{k-1} }^2   \vspace{1ex}\\
&& + {~}  \frac{(1 + c)}{c n b_k}\big(1 - \frac{b_k}{n}\big)  \sum_{i = 1}^n \norms{  F_ix^k -  F_ix^{k-1} }^2 + \frac{1}{n^2}   \sum_{i = 1}^n  \norms{  F_ix^k -  F_ix^{k-1} }^2  \vspace{1ex} \\
&  =  &  \frac{(1 + c)b_{k-1}}{b_k}  \big(1 - \frac{b_k}{n}\big) \Delta_{k-1} + \big[ \frac{1}{n} + \big(1 - \frac{b_k}{n}\big)\frac{(1+c)}{c b_k} \big] \frac{1}{n} \sum_{i = 1}^n  \norms{  F_ix^k - F_ix^{k-1} }^2 .
\end{array}
\end{equation*}
For any $\alpha \in (0, 1)$, if we choose $c := \frac{(n - \alpha b_k)b_k}{(n - b_k)b_{k-1}} - 1$, then $\frac{(1+c)b_{k-1}}{b_k}(1 - \frac{b_k}{n}) = 1 - \frac{\alpha b_k}{n}$.
In addition, we can also compute $\underline{C}_k :=  \frac{1}{n} + \big(1 - \frac{b_k}{n}\big)\frac{(1+c)}{c b_k}$ as $\underline{C}_k  = \frac{1}{n} +\frac{(n-b_k)(n -\alpha b_k)}{n[ n(b_k-b_{k-1}) + b_kb_{k-1} - \alpha b_k^2]}$.
Hence, taking $\Expsn{k}{\cdot}$ on both sides, we obtain from the last inequality that
\begin{equation}\label{eq:saga_estimator_proof3}
\arraycolsep=0.2em
\begin{array}{lcl}
\Expsn{k}{ \Delta_k } & \leq &  \big(1 -  \frac{\alpha b_k}{n} \big)  \Delta_{k-1}  +  \frac{\underline{C}_k}{n} \sum_{i=1}^n \norms{F_ix^k - F_ix^{k-1}}^2.
\end{array}
\end{equation}
Since $b_{k-1} \geq b_k \geq b_{k-1} - \frac{(1-\alpha)b_kb_{k-1}}{2n}$, we have 
\begin{equation*}
\arraycolsep=0.2em
\begin{array}{lcl}
n(b_k - b_{k-1}) + b_kb_{k-1} - \alpha b_k^2 \geq b_kb_{k-1} -\alpha b_k^2 - \frac{(1-\alpha)}{2}b_kb_{k-1} = \frac{1+\alpha}{2}b_kb_{k-1} - \alpha b_k^2 \geq \frac{(1-\alpha)b_k^2}{4}.
\end{array}
\end{equation*}
This implies that $\underline{C}_k \leq \frac{1}{n} + \frac{2(n-b_k)(n-\alpha b_k)}{n(1-\alpha)b_k^2} \leq  \frac{1}{n} + \frac{2n}{(1-\alpha)b_k^2} \leq \frac{(3-\alpha)n}{(1-\alpha)b_k^2} =: \Theta_k$.
Using this bound, we obtain from \eqref{eq:saga_estimator_proof3} the second bound in  \eqref{eq:SAGA_variance}.

Consequently, $\widetilde{F}^k$ satisfies the $\textbf{VR}(\kappa_k, \Theta_k, \Delta_k, \sigma_k)$ property in Definition~\ref{de:VR_Estimators} with $\kappa_k := \frac{\alpha b_k }{n} \in (0, 1]$, $\Delta_k$ and $\Theta_k$ given above, and $\sigma_k^2 = 0$.
%\Eproof
\end{proof}
%%% End of Proof.

%%%%%%%%%%%%%%%%%%%%%%%%%%%%%%%%%%%%%%%%
%%% Proof of Lemma 8.
\aftersubsec
\subsection{The proof of Lemma~\ref{le:loopless_sarah_bound} --- The L-SARAH estimator}\label{apdx:proof_SARAH_estimator_bound}
\aftersubsec
\begin{proof} 
We prove for the expectation setting \eqref{eq:expectation_form}.
The proof of the finite-sum setting \eqref{eq:finite_sum_form} is similar to the ones in \citet{driggs2019accelerating,li2020page}.
Let $i_k$ be the Bernoulli's random variable following the switching rule in \eqref{eq:loopless_sarah}.
Then, we have 
\begin{equation*}
\arraycolsep=0.2em
\begin{array}{lcl}
\Expsn{i_k}{\norms{ \widetilde{F}^k - Fx^k }^2 } & = & (1 - \mbf{p}_k) \norms{ \widetilde{F}^{k-1}  + \Fb(x^k, \Sc_k) - \Fb(x^{k-1}, \Sc_k) - Fx^k }^2 + \mbf{p}_k  \norms{ \bar{F}x^k - Fx^k }^2.
\end{array}
\end{equation*}
By the proof of the loopless SARAH estimator \cite[Lemma~3]{li2020page}, we have
\begin{equation*}
\arraycolsep=0.2em
\begin{array}{lcl}
A_k &:= & \Expsn{k}{\norms{ \widetilde{F}^{k-1}  + \Fb(x^k, \Sc_k) - \Fb(x^{k-1}, \Sc_k) - Fx^k }^2}   \vspace{1ex} \\
&=&  \norms{ \widetilde{F}^{k-1}  - Fx^{k-1}}^2  +  \Expsn{k}{\norms{\Fb(x^k, \Sc_k) - \Fb(x^{k-1}, \Sc_k)}^2} - \norms{Fx^k - Fx^{k-1}}^2  \vspace{1ex} \\
&\leq &  \norms{ \widetilde{F}^{k-1}  - Fx^{k-1}}^2 +  \frac{1}{b_k}\Expsn{\xi}{\norms{\Fb(x^k, \xi) - \Fb(x^{k-1}, \xi)}^2}.
\end{array}
\end{equation*}
Taking the conditional expectation $\Expsn{k}{\cdot}$ on both sides of the first estimate and combining the result with the second expression, we obtain \eqref{eq:loopless_sarah_var_bound}.

Finally, note that $\Expsn{\bar{\Sc}_k}{\norms{ \bar{F}x^k - Fx^k }^2} \leq \frac{\sigma^2}{n_k}$ and $1-\mbf{p}_k \leq 1$, \eqref{eq:loopless_sarah_var_bound} shows that $\widetilde{F}^k$ satisfies Definition~\ref{de:VR_Estimators} with $\Delta_k :=   \norms{\widetilde{F}^k - Fx^k }^2$,  $\kappa_k =  \mbf{p}_{k}$,  $\Theta_k  := \frac{1}{b_k}$, and $\sigma_k^2 := \frac{\mbf{p}_{k} \sigma^2}{n_k}$.
However, if we choose $\bar{F}x^k = Fx^k$, then we can set $\sigma_k = 0$ since $\Expsn{\bar{\Sc}_k}{\norms{ \bar{F}x^k - Fx^k }^2} = 0$.
\end{proof}
%%% End of the proof.

%%%%%%%%%%%%%%%%%%%%%%%%%%%%%%%%%%%%%%%%
%%% The proof of Lemma 1.
\beforesubsec
\subsection{The proof of Lemma~\ref{le:HSGD_estimator_bound} --- The HSGD estimator}\label{apdx:proof_HSGD_estimator_bound}
\aftersubsec
\begin{proof} 
Let $e^k := \widetilde{F}^k - Fx^k$, $X^k(\xi) := \Fb(x^k, \xi) - \Fb(x^{k-1}, \xi) - (Fx^k - Fx^{k-1})$, and $Y^k := \bar{F}x^k - Fx^k$.
Then, we have $\Expsn{\xi}{X^k(\xi)} = 0$ and $\Expsn{\hat{\Sc}_k}{Y^k} = 0$ by our assumption.
In addition, if we denote $X^k := \frac{1}{b_k}\sum_{\xi \in \Sc_k}X^k(\xi)$, then we also have $\Expsn{\Sc_k}{ \norms{X^k}^2 } \leq \frac{1}{b_k}\bar{\Ec}_k$ for $\bar{\Ec}_k$ defined in Definition~\ref{de:VR_Estimators}.
From \eqref{eq:HSGD_estimator}, we can write
\begin{equation*}
\arraycolsep=0.2em
\begin{array}{lcl}
e^k & = &   \widetilde{F}^k - Fx^k =  (1-\tau_k) \widetilde{F}^{k-1} + (1-\tau_k)[ \Fb(x^k, \Sc_k) - \Fb(x^{k-1}, \Sc_k)]  + \tau_k\bar{F}x^k - Fx^k \vspace{1ex}\\
& = &  (1-\tau_k)( \widetilde{F}^{k-1} - Fx^{k-1} ) + (1-\tau_k)[ \Fb(x^k, \Sc_k) - \Fb(x^{k-1}, \Sc_k) - (Fx^k - Fx^{k-1})] \vspace{1ex}\\
&& + {~} \tau_k(\bar{F}x^k - Fx^k ) \vspace{1ex}\\
& = &  (1-\tau_k)e^{k-1}   +  (1-\tau_k)X^k + \tau_k Y^k.
\end{array}
\end{equation*}
This expression leads to 
\begin{equation*}
\arraycolsep=0.2em
\begin{array}{lcl}
\norms{e^k}^2 & = &  (1-\tau_k)^2\norms{e^{k-1}}^2 +  (1-\tau_k)^2\norms{X^k}^2 + \tau_k^2\norms{Y^k}^2 \vspace{1ex}\\
&& + {~} 2(1-\tau_k)^2\iprods{e^{k-1}, X^k} + 2(1-\tau_k)\tau_k \iprods{X^k, Y^k} + 2(1-\tau_k)\tau_k \iprods{e^{k-1}, Y^k}.
\end{array}
\end{equation*}
Taking $\Expsn{(\Sc_k,\hat{\Sc}_k)}{\cdot}$ on both sides of this expression and using $\Expsn{(\Sc_k,\hat{S}_k)}{X^k} = \Expsn{\hat{\Sc}_k}{\Expsn{\Sc_k}{X^k  \mid \hat{\Sc}_k } } = 0$ and $\Expsn{(\Sc_k,\hat{S}_k)}{Y^k} = \Expsn{\Sc_k}{\Expsn{\hat{\Sc}_k}{Y^k  \mid \Sc_k } } = 0$, we can show that
\begin{equation}\label{eq:HSGD_proof5}
\arraycolsep=0.2em
\begin{array}{lcl}
\Expsn{(\Sc_k,\hat{\Sc}_k)}{\norms{e^k}^2} & = &  (1-\tau_k)^2\norms{e^{k-1}}^2 +  (1-\tau_k)^2\Expsn{\Sc_k}{ \norms{X^k}^2 } + \tau_k^2 \Expsn{\hat{\Sc}_k}{ \norms{Y^k}^2 } \vspace{1ex}\\
&& + {~}  2(1-\tau_k)\tau_k \Expsn{(\Sc_k,\hat{\Sc}_k)}{ \iprods{X^k, Y^k} }.
\end{array}
\end{equation}
Here, we have used the facts that $X^k$ only depends on $\Sc_k$ and $Y^k$ only depends on $\hat{\Sc}_k$.

\noindent
Now, we consider two cases.

(i)~If $\Sc_k$ and $\hat{\Sc}_k$ are independent, then $\Expsn{(\Sc_k,\hat{\Sc}_k)}{ \iprods{X^k, Y^k} \mid \Fc_k} = 0$.
Using this fact, $\Expsn{\Sc_k}{ \norms{X^k}^2 } \leq \frac{1}{b_k}\bar{\Ec}_k$, and $\delta_k^2 :=  \Expsn{\hat{\Sc}_k}{ \norms{Y^k}^2 }$ into \eqref{eq:HSGD_proof5}, and then taking the conditional expectation $\Expsn{k}{\cdot}$ on both sides of the result, we obtain \eqref{eq:HSGD_key_est_a}.

(ii)~If $\Sc_k$ and $\hat{\Sc}_k$ are not independent, then by Young's inequality, we have
\begin{equation*}
\arraycolsep=0.2em
\begin{array}{lcl}
2(1-\tau_k)\tau_k \Expsn{(\Sc_k,\hat{\Sc}_k)}{ \iprods{X^k, Y^k} } & \leq &  (1-\tau_k)^2 \Expsn{\Sc_k}{ \norms{X^k}^2 } + \tau_k^2 \Expsn{\hat{\Sc}_k}{ \norms{Y^k}^2 }.
\end{array}
\end{equation*}
Substituting this inequality, $\Expsn{\Sc_k}{ \norms{X^k}^2 } \leq \frac{1}{b_k}\bar{\Ec}_k$, and $\delta_k^2 :=  \Expsn{\hat{\Sc}_k}{ \norms{Y^k}^2 }$ into \eqref{eq:HSGD_proof5}, and taking the conditional expectation $\Expsn{k}{\cdot}$ on both sides of the result, we obtain \eqref{eq:HSGD_key_est_b}.

Finally, to prove (iii), we utilize $\Expsn{\hat{\Sc}_k}{ \delta_k^2 } \leq \frac{\sigma^2}{\hat{b}_k}$ to  obtain \eqref{eq:VR_property} in Definition~\ref{de:VR_Estimators} from \eqref{eq:HSGD_key_est_a} and \eqref{eq:HSGD_key_est_b}, respectively.
\end{proof}
%%% End of the proof.

%%%%%%%%%%%%%%%%%%%%%%%%%%%%%%%%%%%%%%%%%%%%%%%%%%%
%%% Apendix C. -- The Proof of Technical Results in Section 4.
%%%%%%%%%%%%%%%%%%%%%%%%%%%%%%%%%%%%%%%%%%%%%%%%%%%
\beforesec
\section{The Proof of Technical Results in Section~\ref{sec:VR_AFBS_method}}\label{apdx:sec:VrAFBS_method}
\aftersec
This appendix presents the full proof of technical lemmas and theorems in Section~\ref{sec:VR_AFBS_method}.

%%%%%%%%%%%%%%%%%%%%%%%%%%%%%%%%%%%%%%%%
%%% Proof of Lemma 2.
\beforesubsec
\subsection{The proof of Lemma~\ref{le:VrAFBS4NI_descent_property}}\label{apdx:le:VrAFBS4NI_descent_property}
\aftersubsec
\begin{proof} 
From the first two lines of \eqref{eq:VrAFBS4NI}, we can easily show that  $t_kx^{k+1} = (t_k-1)x^k + z^k - t_k\eta_k\widetilde{G}_{\lambda}^k$.
Rearranging this expression and using $z^k = z^{k+1} - \nu(x^{k+1} - y^k) = z^{k+1} + \nu\eta_k\widetilde{G}_{\lambda}^k$ from the last line of \eqref{eq:VrAFBS4NI}, we obtain
\begin{equation}\label{eq:VrAFBS4NI_lm2_proof1} 
\arraycolsep=0.2em
\left\{\begin{array}{lcl}
t_k(t_k-1)(x^{k+1} - x^k) &= & - (t_k-1) (x^k - z^k) - t_k(t_k-1)\eta_k \widetilde{G}_{\lambda}^k, \vspace{1ex}\\
t_k(t_k-1)(x^{k+1} - x^k) & = & - t_k(x^{k+1} - z^k) - t_k^2\eta_k \widetilde{G}_{\lambda}^k \vspace{1ex}\\
& = & -t_k(x^{k+1} - z^{k+1}) - t_k(t_k - \nu) \eta_k \widetilde{G}_{\lambda}^k.
\end{array}\right.
\end{equation}
Let us define $\Ec_{k+1} := L \iprods{Fx^{k+1} - Fx^k, x^{k+1} - x^k}$ as in Lemma~\ref{le:VrAFBS4NI_descent_property}.
Then, from \eqref{eq:G_cocoerciveness}, we have 
\begin{equation*}
\arraycolsep=0.2em
\begin{array}{lcl}
\widetilde{\Tc}_{[1]} &:= & t_k(t_k-1)\iprods{G_{\lambda}x^{k+1}, x^{k+1} - x^k} - t_k(t_k-1) \iprods{G_{\lambda}x^k, x^{k+1} - x^k} \vspace{1ex}\\
& \geq &  \beta t_k(t_k-1)\norms{G_{\lambda}x^{k+1} - G_{\lambda}x^k}^2 + (\bar{\beta} - \beta) t_k(t_k-1)\norms{G_{\lambda}x^{k+1} - G_{\lambda}x^k}^2 \vspace{1ex}\\
&& + {~} \Lambda t_k(t_k-1) \Ec_{k+1}.
\end{array}
\end{equation*}
Substituting \eqref{eq:VrAFBS4NI_lm2_proof1} into $\widetilde{\Tc}_{[1]}$ and using $e_{\lambda}^k := \widetilde{G}^k_{\lambda} - G_{\lambda}x^k$ and $\theta_k := \frac{t_k-1}{t_k-\nu}$, we can show that
\begin{equation*}
\arraycolsep=0.2em
\begin{array}{lcl}
\widehat{\Tc}_{[1]} &:= & (t_k-1)\iprods{G_{\lambda}x^k, x^k - z^k } - t_k\iprods{G_{\lambda}x^{k+1}, x^{k+1} - z^{k+1} } \vspace{1ex}\\
& \geq & \eta_k t_k(t_k - \nu) \iprods{G_{\lambda}x^{k+1}, \widetilde{G}_{\lambda}^k}  - \eta_k t_k(t_k-1) \iprods{G_{\lambda}x^k, \widetilde{G}_{\lambda}^k} + \beta t_k(t_k-1) \norms{G_{\lambda}x^{k+1} - G_{\lambda}x^k}^2 \vspace{1ex}\\
&& + {~} (\bar{\beta} - \beta)t_k(t_k-1)\norms{G_{\lambda}x^{k+1} - G_{\lambda}x^k}^2 + \Lambda t_k(t_k-1) \Ec_{k+1} \vspace{1ex}\\
& = & \eta_kt_k(t_k - \nu ) \iprods{G_{\lambda}x^{k+1}, G_{\lambda}x^k}  - \eta_k t_k(t_k-1) \norms{G_{\lambda}x^k}^2 + \beta t_k(t_k-1) \norms{G_{\lambda}x^{k+1} - G_{\lambda}x^k}^2   \vspace{1ex}\\
&& + {~}  (\bar{\beta} - \beta)t_k(t_k-1)\norms{G_{\lambda}x^{k+1} - G_{\lambda}x^k}^2 +  \Lambda t_k(t_k-1) \Ec_{k+1} \vspace{1ex}\\
&& + {~} \eta_k t_k(t_k-\nu) \iprods{G_{\lambda}x^{k+1} - \theta_k G_{\lambda}x^k, e_{\lambda}^k}.
\end{array}
\end{equation*}
By Young's inequality, for any $s > 0$, we can derive from $\widehat{\Tc}_{[1]}$ that
\begin{equation*}
\arraycolsep=0.2em
\begin{array}{lcl}
\widehat{\Tc}_{[1]} &:= & (t_k-1)\iprods{G_{\lambda}x^k, x^k - z^k } - t_k\iprods{G_{\lambda}x^{k+1}, x^{k+1} - z^{k+1} } \vspace{1ex}\\
& \geq & \eta_k t_k(t_k - \nu) \iprods{G_{\lambda}x^{k+1}, G_{\lambda}x^k}  - \eta_k t_k(t_k-1) \norms{G_{\lambda}x^k}^2 +  \Lambda t_k(t_k-1) \Ec_{k+1} \vspace{1ex}\\
&& + {~} \beta t_k(t_k-1) \norms{G_{\lambda}x^{k+1} - G_{\lambda}x^k}^2   +  (\bar{\beta} - \beta)t_k(t_k-1)\norms{G_{\lambda}x^{k+1} - G_{\lambda}x^k}^2  \vspace{1ex}\\
&& - {~} s\beta t_k(t_k-\nu)  \norms{G_{\lambda}x^{k+1} -  \theta_k G_{\lambda}x^k}^2   - \frac{\eta_k^2t_k(t_k-\nu)}{4s\beta}\norms{e_{\lambda}^k}^2.
\end{array}
\end{equation*}
Substituting the following identity
\begin{equation*}
\arraycolsep=0.2em
\begin{array}{lcl}
\norms{G_{\lambda}x^{k+1} - \theta_k G_{\lambda}x^k}^2 & = &  (1-\theta_k)\norms{G_{\lambda}x^{k+1}}^2  - \theta_k(1-\theta_k)\norms{G_{\lambda}x^k}^2 + \theta_k\norms{G_{\lambda}x^{k+1} - G_{\lambda}x^k}^2 \vspace{1ex}\\
& = & \frac{1-\nu}{t_k-\nu}\norms{G_{\lambda}x^{k+1}}^2  - \frac{(1-\nu)(t_k-1)}{(t_k-\nu)^2}\norms{G_{\lambda}x^k}^2 + \frac{t_k-1}{t_k-\nu}\norms{G_{\lambda}x^{k+1} - G_{\lambda}x^k}^2
\end{array}
\end{equation*}
into the last expression $\widehat{\Tc}_{[1]}$, one can show that
\begin{equation*}
\arraycolsep=0.2em
\begin{array}{lcl}
\Tc_{[1]} &:= & t_{k-1}\iprods{G_{\lambda}x^k, x^k - z^k} - t_k\iprods{G_{\lambda}x^{k+1}, x^{k+1} - z^{k+1}}  \vspace{1ex}\\
& \geq &   \beta t_k\big[ t_k - 1  -  s(1-\nu) \big] \norms{G_{\lambda}x^{k+1}}^2 - t_k(t_k-1) \big[ \eta_k - \beta  - \frac{s\beta(1-\nu)}{t_k-\nu}  \big]\norms{G_{\lambda}x^k}^2 \vspace{1ex}\\ 
&& + {~} t_k \big[ \eta_k(t_k-\nu) - 2\beta(t_k-1) \big] \iprods{G_{\lambda}x^{k+1}, G_{\lambda}x^k} + (t_{k-1} - t_k + 1)\iprods{G_{\lambda}x^k, x^k - z^k} \vspace{1ex}\\
&& + {~} \big[ \bar{\beta} - (1+s)\beta \big] t_k(t_k-1) \norms{G_{\lambda}x^{k+1} - G_{\lambda}x^k}^2  +  \Lambda t_k(t_k-1) \Ec_{k+1}  - \frac{\eta_k^2t_k(t_k-\nu)}{4s\beta}\norms{e_{\lambda}^k}^2.
\end{array}
\end{equation*}
Since $z^{k+1} - z^k = -\nu\eta_k \widetilde{G}_{\lambda}^k$, by Young's inequality again,  we have
\begin{equation*}
\arraycolsep=0.2em
\begin{array}{lcl}
\Tc_{[2]} &:= & \frac{1-\mu}{2\nu\eta_k}\norms{z^k - x^{\star}}^2 - \frac{1-\mu}{2\nu\eta_{k+1}} \norms{z^{k+1} - x^{\star}}^2 \vspace{1ex}\\
& = & \frac{1-\mu}{2\nu\eta_k} \big[ \norms{z^k - x^{\star}}^2 -  \norms{z^{k+1} - x^{\star}}^2 \big] + \frac{(1-\mu)}{2\nu}\big(\frac{1}{\eta_k} - \frac{1}{\eta_{k+1}} \big) \norms{z^{k+1} - x^{\star}}^2 \vspace{1ex}\\
& = & -\frac{1-\mu}{\nu\eta_k}\iprods{z^{k+1} - z^k, z^k - x^{\star}} - \frac{1-\mu}{2\nu\eta_k}\norms{z^{k+1} - z^k}^2 + \frac{(1-\mu)}{2\nu}\big(\frac{1}{\eta_k} - \frac{1}{\eta_{k+1}} \big) \norms{z^{k+1} - x^{\star} }^2 \vspace{1ex}\\
&= & (1-\mu) \iprods{\widetilde{G}_{\lambda}^k, z^k - x^{\star}} - \frac{(1-\mu)\nu\eta_k}{2} \norms{\widetilde{G}_{\lambda}^k}^2 + \frac{(1-\mu)}{2\nu}\big(\frac{1}{\eta_k} - \frac{1}{\eta_{k+1}} \big)\norms{z^{k+1} - x^{\star} }^2  \vspace{1ex}\\
&= & (1-\mu) \iprods{G_{\lambda}x^k, z^k - x^{\star}} +  (1-\mu) \iprods{e_{\lambda}^k, z^k - x^{\star}}  - \frac{(1-\mu)\nu\eta_k}{2} \norms{\widetilde{G}_{\lambda}^k}^2 \vspace{1ex}\\
&& + {~} \frac{(1-\mu)}{2\nu}\big(\frac{1}{\eta_k} - \frac{1}{\eta_{k+1}} \big) \norms{z^{k+1} - x^{\star} }^2  \vspace{1ex}\\
& \geq & (1-\mu) \iprods{G_{\lambda}x^k, z^k - x^{\star}} - \frac{(1-\mu)\nu\beta(t_{k-1}-1)(t_k-1)}{\mu(1-\nu)}\norms{e_{\lambda}^k}^2 - \frac{\mu(1-\mu)(1-\nu)}{4\nu\beta(t_{k-1}-1)(t_k-1)}\norms{z^k - x^{\star}}^2 \vspace{1ex}\\
&& - {~}  (1-\mu)\nu\eta_k \norms{G_{\lambda}x^k}^2 - (1-\mu)\nu\eta_k \norms{e_{\lambda}^k}^2 + \frac{(1-\mu)}{2\nu}\big(\frac{1}{\eta_k} - \frac{1}{\eta_{k+1}} \big) \norms{z^{k+1} - x^{\star} }^2.
\end{array}
\end{equation*}
Since $\eta_k = \frac{2\beta(t_k - 1)}{t_k - \nu}$ and $t_k = \mu(k + r)$ due to \eqref{eq:VrAFBS4NI_tk_etak_update}, we have  $\eta_k(t_k - \nu) - 2\beta(t_k-1) = 0$ and $t_{k-1} - t_k + 1 = 1 - \mu$, respectively.
In addition, we also have $\frac{1}{\eta_k} - \frac{1}{\eta_{k+1}} = \frac{\mu(1-\nu)}{2\beta(t_k-1)(t_{k+1}-1)}$.
Adding $\Tc_{[1]}$ to $\Tc_{[2]}$  and then using the last three relations, we can derive
\begin{equation*}
\arraycolsep=0.2em
\begin{array}{lcl}
\Tc_{[3]} &:= & t_{k-1}\iprods{G_{\lambda}x^k, x^k - z^k} - t_k\iprods{G_{\lambda}x^{k+1}, x^{k+1} - z^{k+1}} \vspace{1ex}\\
&& + {~} \frac{1-\mu}{2\nu\eta_k}\norms{z^k - x^{\star}}^2 - \frac{1-\mu}{2\nu\eta_{k+1}} \norms{z^{k+1} - x^{\star}}^2  \vspace{1ex}\\
%%%
& \geq & \beta t_k \big[ t_k - 1 - s(1 - \nu) \big] \norms{G_{\lambda}x^{k+1}}^2 \vspace{1ex}\\
&& - {~}  \frac{\beta (t_k-1)}{t_k - \nu} \big[ t_k(t_k - 2 + \nu - s(1-\nu)) + 2(1-\mu)\nu \big]  \norms{G_{\lambda}x^k}^2 \vspace{1ex}\\
&& + {~} \big[ \bar{\beta} - (1+s)\beta \big] t_k(t_k - 1) \norms{G_{\lambda}x^{k+1} - G_{\lambda}x^k}^2 + \Lambda t_k(t_k-1) \Ec_{k+1}    
\vspace{1ex}\\
&& - {~} \big[ \frac{ \beta t_k(t_k-1)^2}{s(t_k - \nu)} + \frac{(1-\mu)\nu\beta(t_{k-1} - 1)(t_k - 1) }{\mu(1-\nu)} +  \frac{2\beta\nu(1-\mu)(t_k-1)}{t_k - \nu} \big]\norms{e_{\lambda}^k}^2 \vspace{1ex}\\
&& - {~}   \frac{\mu(1-\mu)(1-\nu)}{4\nu\beta(t_{k-1} - 1)(t_k - 1) }\norms{z^k - x^{\star}}^2 + \frac{\mu(1-\mu)(1-\nu)}{4\nu\beta(t_k-1)(t_{k+1}-1)}\norms{z^{k+1} - x^{\star}}^2 \vspace{1ex}\\
&& + {~} (1-\mu) \iprods{G_{\lambda}x^k, x^k - x^{\star}}.
\end{array}
\end{equation*}
Since $\frac{1}{\eta_k} + \frac{\mu(1-\nu)}{2\beta(t_{k-1} - 1)(t_k - 1) }  = \frac{(t_k-\nu)(t_{k-1}-1) + \mu(1 - \nu)}{2\beta(t_{k-1}-1)(t_k-1)}$, using $\Lc_k$ from \eqref{eq:VrAFBS_Lyapunov_func} and $\norms{e_{\lambda}^k} \leq \norms{\widetilde{F}^k - Fx^k} = \norms{e^k}$ from \eqref{eq:G_estimator_bound1}, the last inequality leads to
\begin{equation*}
\arraycolsep=0.2em
\begin{array}{lcl}
\Lc_k - \Lc_{k+1} & \geq & \beta \varphi_k \cdot \norms{G_{\lambda}x^k}^2 + (1 - \mu) \iprods{G_{\lambda}x^k, x^k - x^{\star}} + \Lambda t_k(t_k-1) \Ec_{k+1}   \vspace{1ex}\\
&& + {~} \big[ \bar{\beta} - (1+s)\beta  \big] t_k(t_k-1) \norms{G_{\lambda}x^{k+1} - G_{\lambda}x^k}^2  - \psi_k  \norms{e^k}^2,
\end{array}
\end{equation*}
which proves \eqref{eq:VrAFBS4NI_desecent_property}, where $\varphi_k$ and $\psi_k$ are respectively 
\begin{equation*}
\arraycolsep=0.2em
\begin{array}{lcl}
\varphi_k & := & t_{k-1} \big[ t_{k-1} - 1 - s(1 - \nu) \big] - \frac{(t_k-1)}{t_k - \nu} \big[ t_k(t_k - 2 + \nu - s(1-\nu)) + 2\nu(1-\mu) \big],  \vspace{1ex}\\
\psi_k & := & \beta(t_k-1)\big[ \frac{t_k(t_k-1)}{s(t_k - \nu)} +  \frac{2\nu(1-\mu)}{t_k-\nu} + \frac{\nu(1-\mu)(t_{k-1} - 1) }{\mu(1-\nu)}  \big],
\end{array}
\end{equation*}
as given in \eqref{eq:VrAFBS4NI_mk_param}.
\end{proof}
%%% End of proof.

%%%%%%%%%%%%%%%%%%%%%%%%%%%%%%%%%%%%%%%%
%%% Proof of Lemma 3.3.
\vspace{-0.5ex}
\beforesubsec
\subsection{The proof of Lemma~\ref{le:VrAFBS4NI_Lk_lowerbound}}\label{apdx:le:VrAFBS4NI_Lk_lowerbound}
\aftersubsec
\begin{proof} 
First, since $F$ is $\frac{1}{L}$-co-coercive, it is monotone.
From \eqref{eq:G_cocoerciveness}, $G_{\lambda}x^{\star} = 0$ for any $x^{\star} \in \zer{\Phi}$, and the monotonicity of $F$, we have $\iprods{G_{\lambda}x^k, x^k - x^{\star}} \geq \bar{\beta} \norms{G_{\lambda}x^k}^2$. 

Next, utilizing the last inequality, \eqref{eq:VrAFBS_Lyapunov_func}, and $\eta_k = \frac{2\beta(t_k-1)}{t_k-\nu}$ from \eqref{eq:VrAFBS4NI_tk_etak_update}, we can show that
\vspace{-0.5ex}
\begin{equation*} 
\arraycolsep=0.2em
\begin{array}{lcl}
\Lc_k & := &  \beta a_k \norms{G_{\lambda}x^k}^2 + t_{k-1} \iprods{G_{\lambda}x^k, x^k - z^k} +  \frac{(1-\mu)[ (t_k-\nu)(t_{k-1}-1) + \mu(1-\nu) ]}{4\nu\beta(t_{k-1}-1)(t_k-1)} \norms{z^k - x^{\star}}^2 \vspace{1ex}\\
& = & \frac{\beta}{2}\big( 2a_k - t_{k-1}^2) \norms{G_{\lambda}x^k}^2  + \big[  \frac{(1-\mu)[ (t_k-\nu)(t_{k-1}-1) + \mu(1-\nu) ]}{4\nu\beta(t_{k-1}-1)(t_k-1)}  - \frac{1}{2\beta}\big] \norms{z^k - x^{\star}}^2 \vspace{1ex}\\
&& + {~} \frac{\beta t_{k-1}^2}{2}\norms{G_{\lambda}x^k}^2 + t_{k-1} \iprods{G_{\lambda}x^k, x^k - x^{\star}}   - t_{k-1} \iprods{G_{\lambda}x^k, z^k - x^{\star}} + \frac{1}{2\beta}\norms{z^k - x^{\star}}^2 \vspace{1ex}\\
& = & \frac{\beta(2a_k - t_{k-1}^2)}{2}\norms{G_{\lambda}x^k}^2 + \frac{(t_{k-1}-1)[(1-\mu-2\nu)t_k + \nu(1+\mu)] + \mu(1-\mu)(1-\nu)}{4\nu\beta(t_{k-1}-1)(t_k-1)} \norms{z^k - x^{\star}}^2  \vspace{1ex}\\
&& + {~}  t_{k-1} \iprods{G_{\lambda}x^k, x^k - x^{\star}} + \frac{1}{2\beta}\norms{z^k - x^{\star} - \beta t_{k-1}G_{\lambda}x^k }^2 \vspace{1ex}\\
& \geq & \frac{\beta(2a_k - t_{k-1}^2) + 2\bar{\beta}t_{k-1}}{2}\norms{G_{\lambda}x^k}^2 +  \frac{(t_{k-1}-1)[(1-\mu-2\nu)t_k + \nu(1+\mu)] + \mu(1-\mu)(1-\nu)}{4\nu\beta(t_{k-1}-1)(t_k-1)}   \norms{z^k - x^{\star}}^2.
\end{array}
\vspace{-0.5ex}
\end{equation*}
Now, we have $A_k := \beta(2a_k - t_{k-1}^2) + 2\bar{\beta}t_{k-1} =   \beta t_{k-1}[ t_{k-1} - 2 - 2s(1 - \nu)] +   2\bar{\beta}t_{k-1}$. 
Using this relation into the last estimate, we obtain \eqref{eq:VrAFBS4NI_Lyapunov_func_lowerbound}.

Finally, from the definitions of $\Pc_k$ and $\Qc_k$ in \eqref{eq:VrAFBS_Lyapunov_func}, the nonnegativity of the last terms in $\Pc_k$ and $\Qc_k$, and Assumption~\ref{as:A2},  we can easily show that $\Pc_k \geq \Qc_k \geq \Lc_k \geq 0$.
\end{proof}
%%% End of Proof.

%%%%%%%%%%%%%%%%%%%%%%%%%%%%%%%%%%%%%%%%
%%% Proof of Lemma 3.3.
\beforesubsec
\subsection{The proof of Lemma~\ref{le:VrAFBS4NI_descent_property2}}\label{apdx:le:VrAFBS4NI_descent_property2}
\aftersubsec
\begin{proof}
First, taking the conditional expectation $\Expsn{k}{\cdot}$ on both sides of \eqref{eq:VrAFBS4NI_desecent_property}, we obtain
\begin{equation*} 
\arraycolsep=0.2em
\begin{array}{lcl}
\Lc_k  & \geq &  \Expsn{k}{ \Lc_{k+1} } + \beta \varphi_k \norms{G_{\lambda}x^k}^2   + (1-\mu) \iprods{G_{\lambda}x^k, x^k - x^{\star}}  \vspace{1ex}\\
&& + {~} \big[ \bar{\beta} - (1+s)\beta \big] t_k(t_k-1) \Expsn{k}{ \norms{G_{\lambda}x^{k+1} - G_{\lambda}x^k }^2 } \vspace{1ex}\\
&& + {~}  \Lambda t_k(t_k-1) \Expsn{k}{ \Ec_{k+1} }  - \psi_k \Expsn{k}{ \norms{e^k}^2 }.
\end{array}
\end{equation*}
Here, $\psi_k$ and $\varphi_k$ from \eqref{eq:VrAFBS4NI_mk_param} are respectively given by
\begin{equation*} 
\arraycolsep=0.2em
\begin{array}{lcl}
\psi_k & := & \beta(t_k-1)\big[ \frac{t_k(t_k-1)}{s(t_k - \nu)} +  \frac{2\nu(1-\mu)}{t_k-\nu} + \frac{\nu(1-\mu)(t_{k-1} - 1) }{\mu(1-\nu)}  \big], \vspace{1ex}\\
\varphi_k & := & t_{k-1} \big[ t_{k-1} - 1 - s(1 - \nu) \big] - \frac{(t_k-1)}{t_k-\nu} \big[ t_k(t_k - 2 + \nu - s(1-\nu)) + 2(1-\mu)\nu \big].
\end{array}
\end{equation*}
Adding $\big[ \bar{\beta} - (1+s) \beta \big] t_{k-1}(t_{k-1} -1) \norms{G_{\lambda}x^k - G_{\lambda}x^{k-1} }^2 + \Lambda t_{k-1}(t_{k-1}-1)  \Ec_k$ to both sides of the last inequality, and using $\Qc_k$ from \eqref{eq:VrAFBS_Lyapunov_func} we have
\begin{equation}\label{eq:VrAFBS4NI_ds1}
\arraycolsep=0.2em
\begin{array}{lcl}
\Qc_k  & \geq &  \Expsn{k}{ \Qc_{k+1} } + \beta \varphi_k \norms{G_{\lambda}x^k}^2   + (1-\mu) \iprods{G_{\lambda}x^k, x^k - x^{\star}}  \vspace{1ex}\\
&& + {~} \big[ \bar{\beta} - (1+s)\beta \big] t_{k-1} (t_{k-1} - 1) \norms{G_{\lambda}x^k - G_{\lambda}x^{k-1} }^2 \vspace{1ex}\\
&& + {~}  \Lambda t_{k-1}(t_{k-1} - 1)  \Ec_k - \psi_k \Expsn{k}{ \norms{e^k}^2 }.
\end{array}
\end{equation}
Next, from \eqref{eq:G_estimator_bound1} and Definition~\ref{de:VR_Estimators}, we have 
\begin{equation}\label{eq:VrAFBS4NI_ds1b} 
\arraycolsep=0.2em
\begin{array}{lcl}
\Expsn{k}{ \norms{e^k}^2 } & = &  \Expsn{k}{ \norms{\widetilde{F}^k - Fx^k}^2 } \leq \Expsn{k}{ \Delta_k }. 
\end{array}
\end{equation}
Let $\bar{\Ec}_k$ be defined in Definition~\ref{de:VR_Estimators}.
Then, by Assumption~\ref{as:A2}, we have
\begin{equation}\label{eq:VrAFBS4NI_ds1c}
\arraycolsep=0.2em
\begin{array}{lcl}
 \Ec_k  & = &  L \cdot  \iprods{Fx^k - Fx^{k-1}, x^k - x^{k-1} }  \geq  \bar{\Ec}_k.
\end{array}
\end{equation}
Now, from \eqref{eq:VR_property} in Definition~\ref{de:VR_Estimators} and \eqref{eq:VrAFBS4NI_ds1b}, for $\Gamma_k \geq 0$ in \eqref{eq:VrAFBS_Lyapunov_func}, we can show that
\begin{equation*}
\arraycolsep=0.2em
\begin{array}{lcl}
\frac{(1 - \kappa_k)\Gamma_kt_{k-1}(t_{k-1}-1)}{2} \Delta_{k-1} & \geq & \frac{\Gamma_k t_{k-1}(t_{k-1}-1)}{2} \Expsn{k}{ \Delta_k } -  \frac{\Theta_k\Gamma_kt_{k-1}(t_{k-1} - 1) }{2}  \bar{\Ec}_k - \frac{\Gamma_kt_{k-1}(t_{k-1} - 1) }{2} \sigma_k^2 \vspace{1ex}\\
& \geq & \psi_k \Expsn{k}{ \norms{e^k}^2 } + \frac{\Gamma_kt_{k-1}(t_{k-1} - 1)  - 2\psi_k}{2} \Expsn{k}{ \Delta_k } -  \frac{\Theta_k\Gamma_kt_{k-1}(t_{k-1} - 1) }{2}  \bar{\Ec}_k \vspace{1ex}\\
&& - {~} \frac{\Gamma_kt_{k-1}(t_{k-1}-1)}{2} \sigma_k^2.
\end{array}
\end{equation*}
Adding this inequality to \eqref{eq:VrAFBS4NI_ds1}, and then using \eqref{eq:VrAFBS4NI_ds1c} we get
\begin{equation}\label{eq:VrAFBS4NI_ds2}
\hspace{-0.2ex}
\arraycolsep=0.2em
\begin{array}{lcl}
\Tc_{[1]} &:= & \Qc_k + \frac{(1 - \kappa_k)\Gamma_kt_{k-1}(t_{k-1}-1)}{2} \Delta_{k-1} \vspace{1ex}\\
& \geq &  \Expsn{k}{ \Qc_{k+1}} + \frac{[ \Gamma_k t_{k-1}(t_{k-1} - 1) - 2 \psi_k]}{2} \Expsn{k}{ \Delta_k } + \beta \varphi_k \norms{G_{\lambda}x^k}^2 \vspace{1ex}\\
&& + {~} \big[ \bar{\beta} - (1+s)\beta  \big] t_{k-1} (t_{k-1}-1) \norms{G_{\lambda}x^k - G_{\lambda}x^{k-1} }^2  + (1-\mu) \iprods{G_{\lambda} x^k, x^k - x^{\star}} \vspace{1ex}\\
&& + {~} \frac{t_{k-1}(t_{k-1}-1)}{2}\big( 2\Lambda   - \Theta_k\Gamma_k  \big) \bar{ \Ec}_k   - \frac{\Gamma_k t_{k-1}(t_{k-1} - 1)}{2}  \sigma_k^2.
\end{array}
\hspace{-4ex}
\end{equation}
Assume that $[\Gamma_{k+1}(1 -\kappa_{k+1}) + \mu^{-1}\beta] t_k(t_k-1) \leq \Gamma_kt_{k-1}(t_{k-1} - 1) - 2\psi_k$, which is exactly the condition \eqref{eq:VrAFBS4NI_param2}.
Then, using $\Pc_k$ from \eqref{eq:VrAFBS_Lyapunov_func} and this condition, we obtain from \eqref{eq:VrAFBS4NI_ds2} that
\begin{equation*} 
\hspace{-0.2ex}
\arraycolsep=0.2em
\begin{array}{lcl}
\Pc_k &:= & \Qc_k + \frac{[\mu(1-\kappa_k)\Gamma_k + \beta] t_{k-1}(t_{k-1}-1) }{2\mu}\Delta_{k-1}  \vspace{1ex}\\
& \geq &  \Expsn{k}{ \Qc_{k+1}} + \frac{[\mu(1-\kappa_{k+1})\Gamma_{k+1} + \beta] t_k(t_k-1)}{2\mu} \Expsn{k}{ \Delta_k } + \frac{\beta t_{k-1}(t_{k-1}-1) }{2\mu}\Delta_{k-1} \vspace{1ex}\\
&& + {~} \beta \varphi_k \norms{G_{\lambda}x^k}^2 + (1-\mu) \iprods{G_{\lambda} x^k, x^k - x^{\star}}  \vspace{1ex}\\
&& + {~} \big[ \bar{\beta} - (1+s)\beta \big] t_{k-1}(t_{k-1} - 1)  \norms{G_{\lambda}x^k - G_{\lambda}x^{k-1} }^2  \vspace{1ex}\\
&& + {~} \frac{t_{k-1}(t_{k-1}-1)}{2} \big( 2\Lambda   - \Gamma_k \Theta_k  \big)  \bar{\Ec}_k - \frac{\Gamma_k t_{k-1}(t_{k-1}-1) }{2} \sigma_k^2.
\end{array}
\end{equation*}
Finally, using again $\Pc_{k+1} := \Qc_{k+1} + \frac{[\mu(1-\kappa_{k+1})\Gamma_{k+1} + \beta] t_k(t_k-1)}{2\mu}\Delta_k$ from \eqref{eq:VrAFBS_Lyapunov_func}, the last expression implies \eqref{eq:VrAFBS4NI_desecent_property2}.
\end{proof}
%%% End of Proof.

%%%%%%%%%%%%%%%%%%%%%%%%%%%%%%%%%%%%%%%%
\beforesubsec
\subsection{The proof of Theorem~\ref{th:VrAFBS4NI_convergence} --- Key estimates}\label{apdx:subsec:th:VrAFBS4NI_convergence}
\aftersubsec
%%%% Proof of Theorem 3.1.
\begin{proof}
Suppose that we choose $\nu := \frac{\mu}{2}$,  $s := \frac{2(\mu - \nu)}{1 - \nu} = \frac{2\mu}{2-\mu}$, and $0 < \beta \leq \frac{\bar{\beta}}{1+s} = \frac{(2-\mu)\bar{\beta}}{2 + \mu}$.
Using these choices and  $t_k := \mu(k + r)$ from \eqref{eq:VrAFBS4NI_param_update} for $r \geq 2+\frac{1}{\mu}$, $\varphi_k$ and $\psi_k$ defined by \eqref{eq:VrAFBS4NI_mk_param} respectively become 
\begin{equation*} 
\arraycolsep=0.2em
\begin{array}{lcl}
\varphi_k & := & t_{k-1}(t_{k-1} - 1  - 2(\mu-\nu)) - \frac{ t_k-1}{t_k - \nu} \big[ t_k^2 - (2 + 2\mu - 3\nu)t_k + 2(1-\mu)\nu \big] \vspace{1ex}\\
& \geq & \frac{(2 - 3\mu)t_k}{2} + 3\mu^2, \vspace{1ex}\\
\psi_k &:= & \beta(t_k-1)\big[ \frac{(1-\nu) t_k(t_k-1)}{2(\mu -\nu)(t_k - \nu)} +  \frac{2(1-\mu)\nu}{t_k-\nu} + \frac{(1-\mu)\nu(t_{k-1} - 1) }{\mu(1-\nu)}  \big] \vspace{1ex}\\
& \leq &  \frac{2\beta}{\mu} t_k(t_k-1).
\end{array}
\end{equation*}
If we choose $0 < \mu < \frac{2}{3}$ as in \eqref{eq:VrAFBS4NI_param_update}, then from the first expression, we get $\varphi_k > 0$.

From the second expression, we can check that  \eqref{eq:VrAFBS4NI_param2} of Lemma~\ref{le:VrAFBS4NI_descent_property2} holds if we impose
\begin{equation*} 
\arraycolsep=0.2em
\begin{array}{lcl}
\big[ \Gamma_{k+1}(1-\kappa_{k+1}) +  \frac{5\beta}{\mu} \big] t_k(t_k-1) \leq  \Gamma_k t_{k-1}(t_{k-1}-1).
\end{array}
\end{equation*}
This condition is equivalent to $\kappa_{k+1} \geq 1 - \frac{\Gamma_k t_{k-1}(t_{k-1} - 1)}{\Gamma_{k+1} t_k(t_k - 1)} + \frac{5\beta}{\mu\Gamma_{k+1}}$, which is exactly the first condition in \eqref{eq:VrAFBS4NI_param_cond}.

Using $\varphi_k$, we can derive from \eqref{eq:VrAFBS4NI_desecent_property2} that
\begin{equation}\label{eq:VrAFBS4NI_th1_proof5} 
\hspace{-0.5ex}
\arraycolsep=0.2em
\begin{array}{lcl}
\Pc_k  & \geq &  \Expsn{k}{ \Pc_{k+1} }  +  \frac{\beta}{2}\big[(2 - 3\mu)t_k  + 6\mu^2 \big] \norms{G_{\lambda}x^k}^2 \vspace{1ex}\\
&& - {~}  \frac{\Gamma_k t_{k-1}(t_{k-1}-1) }{2}   \sigma_k^2 + \frac{\beta t_{k-1}(t_{k-1}-1)}{2\mu} \Delta_{k-1} + \frac{(2\Lambda   - \Gamma_k \Theta_k)t_{k-1}(t_{k-1}-1)}{2} \bar{\Ec}_k  \vspace{1ex}\\
&& + {~} \big[ \bar{\beta} - \frac{(2+\mu)\beta}{2-\mu} \big] t_{k-1}(t_{k-1} - 1)  \norms{G_{\lambda}x^k -  G_{\lambda}x^{k-1} }^2. 
\end{array}
\hspace{-4ex}
\end{equation}
Let us denote by $\hat{\psi}_k := \beta\big[(2 - 3\mu)t_k + 6\mu^2 \big] \geq 0$.
Then, \eqref{eq:VrAFBS4NI_th1_proof5} is equivalent to
\begin{equation}\label{eq:VrAFBS4NI_th1_proof6a} 
\arraycolsep=0.2em
\begin{array}{lcl}
\Expsn{k}{ \Pc_{k+1} } & \leq & \Pc_k - \frac{ \hat{\psi}_k }{2}  \norms{G_{\lambda}x^k}^2 - \frac{\beta t_{k-1}(t_{k-1}- 1) }{2\mu} \Delta_{k-1}   +  \frac{\Gamma_k t_{k-1}(t_{k-1}- 1)  }{2}  \sigma_k^2   \vspace{1ex}\\
&& - {~}  \big[ \bar{\beta} - \frac{(2 + \mu)\beta}{2 - \mu} \big]  t_{k-1}(t_{k-1} - 1) \norms{G_{\lambda}x^k - G_{\lambda}x^{k-1} }^2 \vspace{1ex}\\
&& - {~} \frac{(2\Lambda   - \Gamma_k \Theta_k)t_{k-1}(t_{k-1}-1)}{2} \bar{\Ec}_k. 
\end{array}
\end{equation}
Since $2\Lambda   - \Gamma_k \Theta_k \geq 0$ from the second condition of \eqref{eq:VrAFBS4NI_param_cond}, dropping the last term of \eqref{eq:VrAFBS4NI_th1_proof6a}, and then taking the total expectation $\Expn{\cdot}$ on both sides of the result, we get
\begin{equation}\label{eq:VrAFBS4NI_th1_proof6b} 
\arraycolsep=0.2em
\begin{array}{lcl}
\Expn{ \Pc_{k+1} } & \leq &   \Expn{ \Pc_k } - \frac{ \hat{\psi}_k }{2}  \Expn{ \norms{G_{\lambda}x^k}^2 }   - \frac{\beta t_{k-1}(t_{k-1}-1)}{2\mu} \Expn{ \Delta_{k-1} } +  \frac{ t_{k-1}(t_{k-1}-1)  }{2}  \Gamma_k \sigma_k^2  \vspace{1ex}\\
&& - {~}  \big[ \bar{\beta} - \frac{(2 + \mu)\beta}{2-\mu} \big] t_{k-1}(t_{k-1} - 1) \Expn{ \norms{G_{\lambda}x^k - G_{\lambda}x^{k-1} }^2 }. 
\end{array}
\end{equation}
Let $\underline{B}_K := \frac{1 }{2\beta}\sum_{k=0}^K  \Gamma_k t_{k-1}(t_{k-1}-1) \sigma_k^2$.
Then, since $\Gamma_k \leq \frac{2\Lambda}{\Theta_k}$ by \eqref{eq:VrAFBS4NI_param_cond}, we can show that
\begin{equation*}
\arraycolsep=0.2em
\begin{array}{lcl}
\underline{B}_K := \frac{1 }{2\beta}\sum_{k=0}^K   t_{k-1}(t_{k-1}-1) \Gamma_k\sigma_k^2 &\leq & B_K := \frac{\Lambda}{\beta} \sum_{k=0}^Kt_{k-1}(t_{k-1} - 1)\frac{\sigma_k^2}{\Theta_k}.
\end{array}
\end{equation*}
By induction and the last relation, we obtain from \eqref{eq:VrAFBS4NI_th1_proof6b} that
\begin{equation}\label{eq:VrAFBS4NI_th1_proof6} 
\hspace{-0.2ex}
\arraycolsep=0.2em
\begin{array}{llcl}
& \Expn{ \Pc_K } & \leq & \Pc_0 + \beta B_{K-1}\vspace{1ex}\\
& \sum_{k=0}^K \hat{\psi}_k\Expn{\norms{ G_{\lambda}x^k }^2 } & \leq &  2\Expn{ \Pc_0 } + 2\beta B_K, \vspace{1ex}\\
& \sum_{k=0}^K t_{k-1}(t_{k-1}-1) \Expn{ \Delta_{k-1} } & \leq & \frac{2\mu}{\beta}\big( \Expn{\Pc_0 } + \beta B_K \big), \vspace{1ex}\\
& \sum_{k=0}^K \big[ \bar{\beta} - \frac{(2+\mu)\beta}{2-\mu} \big] t_{k-1}(t_{k-1}-1) \Expn{ \norms{G_{\lambda}x^k - G_{\lambda}x^{k-1} }^2 } & \leq &  \Expn{ \Pc_0 } + \beta B_K.
\end{array}
\hspace{-2ex}
\end{equation}
Because $x^0 = z^0$, we have
\begin{equation*} 
\arraycolsep=0.2em
\begin{array}{lcl}
\Lc_0 & := & \beta a_0\norms{G_{\lambda}x^0}^2 + \frac{1-\mu}{2\nu\eta_0}\norms{x^0 - x^{\star}}^2 \vspace{1ex}\\
& \leq & \beta \mu(r-1)(\mu r - 2\mu - 1) \norms{G_{\lambda}x^0}^2  + \frac{2r-1}{4\beta(\mu r - 1)}\norms{x^0 - x^{\star}}^2 \vspace{1ex}\\
& \leq & \beta \mcal{R}_0^2, \quad \text{where}  \ \  \mcal{R}_0^2 := \mu^2r^2 \norms{G_{\lambda}x^0}^2  + \frac{2r-1}{4\beta^2(\mu r - 1)}\norms{x^0 - x^{\star}}^2.
\end{array}
\end{equation*}
Since $\kappa_0 \in (0, 1]$,  $x^{-1} = x^0$, $\Delta_{-1} = \Delta_0$, and $\Expn{\Delta_0} \leq \sigma_0^2$, from \eqref{eq:VrAFBS_Lyapunov_func}, we get
\begin{equation*} 
\arraycolsep=0.2em
\begin{array}{lcl}
\Expn{ \Pc_0 } &= & \Expn{ \Qc_0 } = \Expn{ \Lc_0 } + \frac{(r-1)(\mu r - \mu - 1)[(1-\kappa_0)\mu\Gamma_0 + \beta]}{2}\Expn{ \Delta_0 } \vspace{1ex}\\
& \leq & \beta \big[ \Expn{ \mcal{R}_0^2 } + \frac{\mu r^2(\mu\Gamma_0 + \beta) }{2\beta} \sigma_0^2 \big] \vspace{1ex}\\
& = & \beta \big( \Psi_0^2 + E_0^2),
\end{array}
\end{equation*}
where $\Psi_0^2 := \Expn{ \mcal{R}_0^2 }$ and $E_0^2 := \frac{\mu r^2(\mu\Gamma_0 + \beta) }{2\beta} \sigma_0^2$ are given in \eqref{eq:VrAFBS4NI_th1_convergence1_init}.

Now, since $A_k := \beta t_{k-1}(t_{k-1} - 2 - 2\mu) + 2\bar{\beta}t_{k-1} \geq \beta[ t_{k-1}^2 - 2(1+\mu)t_{k-1} + \frac{2(2+\mu)}{2-\mu}t_{k-1}] \geq \beta t_{k-1}^2$ and $1 - \mu - 2\nu \geq 0$, from Lemma~\ref{le:VrAFBS4NI_Lk_lowerbound}, we also have
\begin{equation*} 
\arraycolsep=0.2em
\begin{array}{lcl}
\Lc_k &\geq & \frac{A_k}{2}\norms{G_{\lambda}x^k}^2 \geq  \frac{\beta t_{k-1}^2}{2}\norms{G_{\lambda}x^k}^2 =  \frac{\beta\mu^2(k+r-1)^2}{2}\norms{G_{\lambda}x^k}^2.
\end{array}
\end{equation*}
Combining the last two bounds and the first estimate of \eqref{eq:VrAFBS4NI_th1_proof6}, we can derive that
\begin{equation*} 
\arraycolsep=0.2em
\begin{array}{lcl}
\frac{\beta\mu^2(K+r - 1)^2}{2}\Expn{ \norms{G_{\lambda}x^K }^2 } & \leq &  \Expn{ \Lc_K }  \leq \Expn{ \Pc_K } \leq  \beta \big( \Psi_0^2 + E_0^2 +  B_K \big).
\end{array}
\end{equation*}
This expression leads to \eqref{eq:VrAFBS4NI_th1_convergence1}.

Next, since $\hat{\psi}_k = \beta\big[(2 - 3\mu)t_k + 6\mu^2 \big] > 0$ due to $0 < \mu < \frac{2}{3}$, we obtain from the second estimate of \eqref{eq:VrAFBS4NI_th1_proof6} that
\begin{equation*} 
\arraycolsep=0.2em
\begin{array}{lcl}
\sum_{k=0}^K \beta \big[ (2 - 3\mu)\mu(k+r) + 6\mu^2 \big] \Expn{\norms{ G_{\lambda}x^k }^2 } \leq  2\Expn{ \Pc_0} + 2\beta B_K.
\end{array}
\end{equation*}
Substituting $\Expn{ \Pc_0 } \leq \beta (\Psi_0^2 + E_0^2)$ into this inequality, we obtain the first line of \eqref{eq:VrAFBS4NI_th1_summable_bound1}.
The second bound of \eqref{eq:VrAFBS4NI_th1_summable_bound1} is a consequence of the third line of \eqref{eq:VrAFBS4NI_th1_proof6} and the facts that $\Expn{ \norms{e^{k-1}}^2 } \leq \Exp{\Delta_{k-1}}$ and $\Expn{ \Pc_0} \leq \beta (\Psi_0^2 + E_0^2)$.
Similarly, the third line of \eqref{eq:VrAFBS4NI_th1_summable_bound1} is a direct consequence of the last  line of \eqref{eq:VrAFBS4NI_th1_proof6} and $\Expn{ \Pc_0 } \leq \beta (\Psi_0^2 + E_0^2)$.
\end{proof}
%%%% End of Proof.

%%%%%%%%%%%%%%%%%%%%%%%%%%%%%%%%%%%%%%%%
\beforesubsec
\subsection{The proof of Theorem~\ref{th:VrAFBS4NI_o_rates_convergence2} --- The $\BigO{1/k^2}$ and $\SmallOs{1/k^2}$-convergence rates}\label{apdx:subsec:th:VrAFBS4NI_o_rates_convergence2}
\aftersubsec
%%%% Proof of Theorem 4.2.
\begin{proof} 
Let us divide this proof into several steps as follows.

%%% Step 1.
\vspace{0.5ex}
\noindent\textit{\textbf{Step 1. Proving \eqref{eq:VrAFBS4NI_BigO_rate}.}}
Since $B_{\infty} := \frac{\Lambda}{\beta}\sum_{k=0}^{\infty} t_{k-1}(t_{k-1}-1)\frac{\sigma_k^2}{\Theta_k}  < +\infty$ by \eqref{eq:summable_variance}, \eqref{eq:VrAFBS4NI_BigO_rate} follows directly from \eqref{eq:VrAFBS4NI_th1_convergence1} and $B_k \leq B_{\infty}$.

%%% Step 2.
\vspace{0.5ex}
\noindent\textit{\textbf{Step 2. The first two summability bounds of \eqref{eq:VrAFBS4NI_summable_bound2}.}}
By our condition \eqref{eq:summable_variance}, we have $\Psi_0^2 + B_K \leq \Psi_0^2 + B_{\infty}$.
Combining this fact,  $e^k_{\lambda} := \widetilde{G}_{\lambda}^k - G_{\lambda}x^k$, $e^k := \widetilde{F}^k - Fx^k$ and $\norms{e^k_{\lambda}} \leq \norms{e^k}$ from \eqref{eq:G_estimator_bound1}, we can show from \eqref{eq:VrAFBS4NI_th1_summable_bound1} that
\begin{equation}\label{eq:VrAFBS4NI_th2_proof3} 
\arraycolsep=0.2em
\begin{array}{lcl}
\sum_{k=0}^{\infty} t_k \Expn{\norms{ G_{\lambda}x^k }^2 } & < & +\infty, \vspace{1ex}\\
\sum_{k=0}^{\infty} t_k^2 \Expn{\norms{e^k_{\lambda} }^2 } & < & +\infty, \vspace{1ex}\\
\sum_{k=0}^{\infty} t_k^2 \Expn{ \norms{e^k }^2 } & < & +\infty, \vspace{1ex}\\
\sum_{k=0}^{\infty} t_k^2  \Expn{ \norms{ G_{\lambda}x^{k+1} - G_{\lambda}x^k }^2 } & < & +\infty. 
\end{array}
\end{equation}
These prove the first and second lines of \eqref{eq:VrAFBS4NI_summable_bound2}.
Here, the last line of \eqref{eq:VrAFBS4NI_th2_proof3} requires $0 < \beta < \frac{(2-\mu)\bar{\beta}}{2+\mu}$. 
Combining $\norms{\widetilde{G}_{\lambda}^k}^2 \leq 2\norms{e_{\lambda}^k}^2 + 2\norms{G_\lambda x^k}^2$ and the first two lines of \eqref{eq:VrAFBS4NI_th2_proof3} we get.
\begin{equation}\label{eq:VrAFBS4NI_th2_proof3b} 
\arraycolsep=0.2em
\begin{array}{lcl}
\sum_{k=0}^{\infty} t_k \Expn{\norms{ \widetilde{G}_{\lambda}^k }^2 } & < & +\infty.
\end{array}
\end{equation}
%%% Step 3.
\vspace{0.5ex}
\noindent\textit{\textbf{Step 3. The last summability  bound of \eqref{eq:VrAFBS4NI_summable_bound2}.}}
From \eqref{eq:VrAFBS4NI} we can show that
\begin{equation*} 
\arraycolsep=0.2em
\begin{array}{lcl}
t_k(x^{k+1} - x^k + \eta_k\widetilde{G}_\lambda^k) &= & z^k - x^k, \vspace{1ex}\\
(t_{k-1} - 1)(x^k - x^{k-1} + \eta_{k-1}\widetilde{G}_\lambda^{k-1}) &= & z^{k-1} - x^k - \eta_{k-1}\widetilde{G}_\lambda^{k-1} \vspace{1ex}\\
& = & z^k - x^k - (1-\nu)\eta_{k-1}\widetilde{G}_\lambda^{k-1}.
\end{array}
\end{equation*}
Combining these two expressions, we get
\begin{equation}\label{eq:VrAFBS4NI_th12_proof100} 
\arraycolsep=0.2em
\begin{array}{lcl}
t_k(x^{k+1} - x^k + \eta_k\widetilde{G}_{\lambda}^k) & = & (t_{k-1}-1)(x^k - x^{k-1} + \eta_{k-1}\widetilde{G}_{\lambda}^{k-1}) + (1 - \nu) \eta_{k-1}\widetilde{G}_{\lambda}^{k-1}.
\end{array}
\end{equation}
If we denote $v^{k} := x^{k+1} - x^k + \eta_k\widetilde{G}_{\lambda}^k$, then we can rewrite the last expression as
\begin{equation*} 
\arraycolsep=0.2em
\begin{array}{lcl}
t_kv^{k} & := &  \big(1 - \frac{1}{t_{k-1}}\big)t_{k-1}v^{k-1} + \frac{(1-\nu) \eta_{k-1}t_{k-1}}{t_{k-1}}\widetilde{G}_{\lambda}^{k-1}.
\end{array}
\end{equation*}
By the convexity of $\norms{\cdot}^2$ and $\frac{1}{t_{k-1}} \in (0, 1]$, since $\eta_k = \frac{2\beta(t_k-1)}{t_k-\nu} \leq 2\beta$, we have
\begin{equation}\label{eq:VrAFBS4NI_th12_proof102}  
\arraycolsep=0.2em
\begin{array}{lcl}
t_k^2\norms{v^k}^2 & \leq & \big(1 - \frac{1}{t_{k-1}}\big)t_{k-1}^2 \norms{v^{k-1}}^2 + (1-\nu)^2\eta_{k-1}^2t_{k-1}\norms{\widetilde{G}_{\lambda}^{k-1}}^2 \vspace{1ex}\\
& \leq &  t_{k-1}^2\norms{v^{k-1}}^2 - t_{k-1}\norms{v^{k-1}}^2 + 4(1-\nu)^2\beta^2t_{k-1}\norms{ \widetilde{G}_{\lambda}^{k-1} }^2.
\end{array}
\end{equation}
Taking the total expectation of this inequality, and then applying Lemma~\ref{le:A1_descent} with the fact that $\sum_{k=0}^{\infty}t_k \Expn{ \norms{\widetilde{G}_{\lambda}^k}^2 } < +\infty$ from   \eqref{eq:VrAFBS4NI_th2_proof3b}, we conclude that
\begin{equation}\label{eq:VrAFBS4NI_th2_proof3c}  
\arraycolsep=0.2em
\begin{array}{lcl}
\lim_{k\to\infty}t_k^2 \Expn{\norms{v^k}^2} ~\textrm{exists} \quad \textrm{and} \quad \sum_{k=0}^{\infty}t_k \Expn{ \norms{v^k}^2 } < +\infty.
\end{array}
\end{equation}
%This proves the fifth summable bound in \eqref{eq:VrAFBS4NI_summable_bound2}.
By Young's inequality and $\eta_k \leq 2\beta$, we have
\begin{equation*} 
\arraycolsep=0.2em
\begin{array}{lcl}
\norms{x^{k+1} - x^k}^2 \leq 2\norms{x^{k+1} - x^k + \eta_k\widetilde{G}_{\lambda}^k}^2 + 2\eta_k^2\norms{\widetilde{G}_{\lambda}^k}^2 \leq 2\norms{v^k}^2 + 8\beta^2\norms{\widetilde{G}_{\lambda}^k}^2.
\end{array}
\end{equation*}
Combining this inequality,  \eqref{eq:VrAFBS4NI_th2_proof3b}, \eqref{eq:VrAFBS4NI_th2_proof3c}, we get the last summability bound of \eqref{eq:VrAFBS4NI_summable_bound2}.

%%% Step 4.
\vspace{0.5ex}
\noindent\textit{\textbf{Step 4. The limit of $t_k^2 \Expn{ \norms{v^k}^2 }$.}}
Since $\sum_{k=0}^{\infty}t_k \Expn{ \norms{v^k}^2 } < +\infty$ and $\lim_{k\to\infty}t_k^2 \Expn{ \norms{v^k}^2 }$ exists, 
applying Lemma~\ref{le:A2_sum}, we conclude that 
\begin{equation}\label{eq:VrAFBS4NI_small_o_rates_line1}
\lim_{k\to\infty}t_k^2 \Expn{ \norms{v^k}^2 } = \lim_{k\to\infty}k^2 \Expn{ \norms{v^k}^2 } = 0. %%%which proves the first line of \eqref{eq:VrAFBS4NI_small_o_rates}, showing $\SmallO{1/k^2}$ rates of $\Expn{\norms{v^k}^2}$.
\end{equation}
%%% Step 5.
\vspace{0.5ex}
\noindent\textit{\textbf{Step 5. The first limit of \eqref{eq:VrAFBS4NI_small_o_rates}.}}
From \eqref{eq:VrAFBS4NI_th12_proof100}, we can show that
\begin{equation*} 
\arraycolsep=0.15em
\begin{array}{lcl}
t_k[x^{k+1} - x^k + \eta_k(\widetilde{G}_\lambda^k - G_\lambda x^k)]  
&=& (t_{k-1} - 1)\big[ x^k - x^{k-1} + \eta_{k-1}(\widetilde{G}_\lambda^{k-1} - G_\lambda x^{k-1}) \big] \vspace{1ex}\\
&& + {~} (1 - \nu)\eta_{k-1} \widetilde{G}_\lambda^{k-1} - t_k \eta_k G_\lambda x^k + (t_{k-1} - 1)\eta_{k-1} G_\lambda x^{k-1}.
\end{array}
\end{equation*}
Let us define $w^k := x^{k+1} - x^k + \eta_ke_{\lambda}^k$ for $e_{\lambda}^k := \widetilde{G}_\lambda^k - G_\lambda x^k$.
Then this expression becomes
\begin{equation*} 
\arraycolsep=0.2em
\begin{array}{lcl}
t_k w^k &=& (t_{k-1} - 1) w^{k-1} + (1 - \nu)\eta_{k-1} \widetilde{G}_\lambda^{k-1} - t_k \eta_k G_\lambda x^k + (t_{k-1} - 1)\eta_{k-1} G_\lambda x^{k-1} \vspace{1ex}\\
&=& \big(1 - \frac{1}{t_{k-1}} \big) t_{k-1} \big[ w^{k-1} - \frac{t_k \eta_k}{t_{k-1} - 1} (G_\lambda x^k - G_\lambda x^{k-1}) \big] \vspace{1ex}\\
&& + {~} \frac{1}{t_{k-1}} \left[ t_{k-1}[\eta_{k-1}( t_{k-1} - 1)  - t_k \eta_k ] G_\lambda x^{k-1} + (1-\nu) t_{k-1} \eta_{k-1} e_{\lambda}^{k-1} \right].
\end{array}
\end{equation*}
By the convexity of $\norms{\cdot}^2$, Young's inequality, the \rv{$\bar{\beta}$-co-coercivity} of $G_{\lambda}$ from  \eqref{eq:G_cocoerciveness}, and $\norms{e^k_{\lambda}} \leq \norms{e^k}$ from \eqref{eq:G_estimator_bound1}, we can deduce that
\begin{equation*} 
\arraycolsep=0.15em
\begin{array}{lcl}
	t_k^2 \norms{w^k}^2 &\leq& \left(t_{k-1} - 1\right) t_{k-1} \norms{ w^{k-1} - \frac{t_k \eta_k}{t_{k-1} - 1} (G_\lambda x^k - G_\lambda x^{k-1})}^2  \vspace{1ex}\\
	&& + {~} t_{k-1} \norms{[\eta_{k-1}(t_{k-1} - 1) - t_k \eta_k] G_\lambda x^{k-1} + (1-\nu) \eta_{k-1} e_{\lambda}^{k-1}}^2 \vspace{1ex}\\
	&\leq& \left(t_{k-1} - 1\right) t_{k-1} \norms{w^{k-1}}^2 + \frac{t_{k-1} t_k^2 \eta_k^2}{t_{k-1} - 1} \norms{G_\lambda x^k - G_\lambda x^{k-1}}^2 \vspace{1ex}\\
	&& - {~} 2 t_{k-1} t_k \eta_k \iprods{G_\lambda x^k - G_\lambda x^{k-1}, x^k - x^{k-1}} - 2 t_{k-1} t_k \eta_{k-1} \eta_k \iprods{G_\lambda x^k - G_\lambda x^{k-1}, e_{\lambda}^{k-1}} \vspace{1ex} \\
	&& + {~} 2 t_{k-1} [\eta_{k-1}( t_{k-1} - 1) - t_k \eta_k]^2 \norms{G_\lambda x^{k-1}}^2 + 2 (1 - \nu)^2 t_{k-1} \eta_{k-1}^2 \norms{e^{k-1}_{\lambda}}^2 \vspace{1ex}\\
	&\leq& \left(t_{k-1} - 1\right) t_{k-1} \norms{w^{k-1}}^2 + t_{k-1} t_k \eta_k \left( \frac{t_k \eta_k}{t_{k-1} - 1} - 2\bar{\beta} + \eta_{k-1} \right) \norms{G_\lambda x^k - G_\lambda x^{k-1}}^2 \vspace{1ex}\\
	&& + {~} 2t_{k-1} (\eta_{k-1}( t_{k-1} - 1) - t_k \eta_k )^2 \norms{G_\lambda x^{k-1}}^2 \vspace{1ex}\\
	&& + {~} t_{k-1}\eta_{k-1} [2(1-\nu)^2 \eta_{k-1} + t_k \eta_k] \norms{e^{k-1}}^2.
\end{array}
\end{equation*}
Let us denote by $Z_k := s_k \norms{G_\lambda x^k - G_\lambda x^{k-1}}^2 + \hat{s}_k \norms{G_\lambda x^{k-1}}^2 + \tilde{s}_k \norms{e^{k-1}}^2$, where
\begin{equation*} 
\arraycolsep=0.2em
\begin{array}{lclcl}
	s_k &:=& t_{k-1} t_k \eta_k \big( \frac{t_k \eta_k}{t_{k-1} - 1} - 2\bar{\beta} + \eta_{k-1} \big) & \leq &   8\beta^2 t_k^2, \vspace{1ex}\\
	\hat{s}_k &:=& 2t_{k-1}[\eta _{k-1}(t_{k-1} - 1) - t_k \eta_k]^2 & \leq & 8\beta^2 (\mu + \nu)^2 t_{k-1}, \vspace{1ex}\\
	\tilde{s}_k &:=&  t_{k-1}\eta_{k-1} [2(1-\nu)^2 \eta_{k-1} + t_k \eta_k] & \leq & 4\beta^2 t_{k-1} (t_k + 2) \leq 8\beta^2 t_k^2.
\end{array}
\end{equation*} 
Then, the last expression can be briefly rewritten as
\begin{equation}\label{eq:VrAFBS4NI_th12_proof101} 
\arraycolsep=0.2em
\begin{array}{lcl}
	t_k^2 \norms{w^k}^2 &\leq & t_{k-1}^2 \norms{w^{k-1}}^2 - t_{k-1} \norms{w^{k-1}}^2 + Z_k.
\end{array}
\end{equation}
Taking the total expectation $\Expn{\cdot}$ on both sides of this inequality, we get
\begin{equation*} 
\arraycolsep=0.2em
	\begin{array}{lcl}
		t_k^2 \Expn{\norms{w^k}^2} \leq t_{k-1}^2 \Expn{\norms{w^{k-1}}^2} - t_{k-1}\Expn{\norms{w^{k-1}}^2} + \Expn{ Z_k }.
	\end{array}
\end{equation*}
Using the upper bound of the coefficients $s_k$, $\hat{s}_k$, and $\tilde{s}_k$, and \eqref{eq:VrAFBS4NI_th2_proof3}, we can easily show that $\sum_{k=0}^\infty \Expn{ Z_k } < + \infty$. 
Applying again Lemma~\ref{le:A1_descent} to the last inequality, we conclude that
\begin{equation*} 
\arraycolsep=0.2em
\begin{array}{lcl}
\lim_{k\to\infty}t_k^2 \Expn{\norms{w^k}^2} ~\textrm{exists} \quad \textrm{and} \quad \sum_{k=0}^{\infty}t_k \Expn{ \norms{w^k}^2 } < +\infty.
\end{array}
\end{equation*}
Applying Lemma~\ref{le:A2_sum}, these  statements imply that $\lim_{k \to \infty} t_k^2 \Expn{\norms{w^k}^2} = 0$.

By Young's inequality and $\norms{e^k_{\lambda}} \leq \norms{e^k}$, one has
\begin{equation}\label{eq:VrAFBS4NI_th12_proof105}  
\arraycolsep=0.2em
\begin{array}{lcl}
\norms{x^{k+1} - x^k}^2 \leq 2\norms{x^{k+1} - x^k + \eta_k e_{\lambda}^k}^2 + 2 \eta_k^2 \norms{e_{\lambda}^k}^2 \leq 2\norms{w^k}^2 + 8\beta^2\norms{e^k}^2.
\end{array}
\end{equation}
Combining this fact, $\lim_{k \to \infty} t_k^2 \Expn{\norms{w^k}^2} = 0$ and \eqref{eq:VrAFBS4NI_th2_proof3}, we prove the second limit of \eqref{eq:VrAFBS4NI_small_o_rates}.

%%% Step 6.
\vspace{0.5ex}
\noindent\textit{\textbf{Step 6. The second limit of \eqref{eq:VrAFBS4NI_small_o_rates}.}}
Finally, since $\eta_k = \frac{2\beta(t_k-1)}{t_k-\nu} \geq \frac{4\beta(r\mu-1)}{\mu(2r -1)}$, by Young's inequality, we have
\begin{equation*} 
\arraycolsep=0.2em
\begin{array}{lcl}
\frac{16\beta^2(r\mu-1)^2}{\mu^2(2r-1)^2} \norms{G_\lambda x^k}^2 & \leq & \eta_k^2\norms{G_\lambda x^k}^2 \vspace{1ex}\\
& \leq & 2 \norms{x^{k+1} - x^k + \eta_k (\widetilde{G}_\lambda^k - G_\lambda x^k)}^2 + 2  \norms{x^{k+1} - x^k + \eta_k \widetilde{G}_\lambda^k}^2 \vspace{1ex}\\
& \leq & 2 \norms{w^k}^2 + 2 \norms{v^k}^2.
\end{array}
\end{equation*}
This fact,  \eqref{eq:VrAFBS4NI_small_o_rates_line1}, and $\lim_{k \to \infty} t_k^2 \Expn{\norms{w^k}^2} = 0$ imply the first limit of  \eqref{eq:VrAFBS4NI_small_o_rates}.
\end{proof}
%%%% End of Proof.

%%%%%%%%%%%%%%%%%%%%%%%%%%%%%%%%%%%%%%%%
\beforesubsec
\subsection{The proof of Theorem~\ref{th:VrAFBS4NI_o_rates_convergence3} --- Almost sure convergence}\label{apdx:th:VrAFBS4NI_o_rates_convergence3}
\aftersubsec
\begin{proof}
Again, we divide this proof into the following steps.

%%% Step 1.
\vspace{0.5ex}
\noindent\textit{\textbf{Step 1. The first summability bound in \eqref{eq:VrAFBS4NI_summable_bound2_as}.}}
First, applying our assumption $\Gamma_k\Theta_k \leq \Lambda$ from \eqref{eq:VrAFBS4NI_param_cond_new}, we can deduce  from \eqref{eq:VrAFBS4NI_th1_proof6a} that
\begin{equation}\label{eq:VrAFBS4NI_th3_proof1} 
\hspace{-0.5ex}
\arraycolsep=0.2em
\begin{array}{lcl}
\Expn{ \Pc_{k+1} \mid \Fc_k } & \leq &  \Pc_k - \frac{ \hat{\psi}_k }{2}  \norms{G_{\lambda}x^k}^2   - \frac{\beta t_{k-1}(t_{k-1}-1) }{2 \mu} \Delta_{k-1} -  \frac{\Lambda t_{k-1}(t_{k-1}-1)}{2} \bar{\Ec}_k \vspace{1ex}\\
&& - {~} \big(\bar{\beta} - \frac{(2+\mu)\beta}{2-\mu} \big)t_{k-1}(t_{k-1} - 1) \norms{G_{\lambda}x^k - G_{\lambda}x^{k-1} }^2 +  \frac{\Lambda t_{k-1}(t_{k-1}-1) \sigma_k^2 }{2\Theta_k}. 
\end{array}
\hspace{-2ex}
\end{equation}
Next, we denote  $U_k := \Pc_k$,  $\gamma_k :=  0$, $E_k := \frac{\Lambda t_{k-1}(t_{k-1}-1) \sigma_k^2 }{2\Theta_k}$, and 
\begin{equation*}
\arraycolsep=0.2em
\begin{array}{lcl}
V_k & := & \frac{ \hat{\psi}_k }{2}  \norms{G_{\lambda}x^k}^2 + \frac{\beta t_{k-1}(t_{k-1}-1)}{2\mu} \Delta_{k-1} + (\bar{\beta} - 2\beta)t_{k-1}(t_{k-1} - 1) \norms{G_{\lambda}x^k - G_{\lambda}x^{k-1} }^2 \vspace{1ex}\\
&& + {~}  \frac{ \Lambda  t_{k-1}(t_{k-1}-1)}{2} \bar{\Ec}_k. 
\end{array}
\end{equation*}
These quantities are nonnegative.
By \eqref{eq:summable_variance}, we have $\sum_{k=0}^{\infty}E_k = \frac{\beta}{2} B_{\infty} < +\infty$ surely.
In addition,  $\sum_{k=0}^{\infty}\gamma_k = 0 < +\infty$ and $\hat{\psi}_k = 2\beta[ (2-5\mu)t_k - (2 - 3\mu - 2\mu^2)] = \BigOs{t_k} > 0$ surely.
Moreover, $\Fc_k$ is a filtration.
By Lemma~\ref{le:RS_lemma}, we conclude that
\vspace{-0.5ex}
\begin{equation}\label{eq:VrAFBS4NI_th3_proof2} 
\arraycolsep=0.2em
\begin{array}{lcl}
\lim_{k\to\infty} \Pc_k \quad \textrm{exists almost surely}, \vspace{1ex}\\
\sum_{k=0}^{\infty}t_k   \norms{G_{\lambda}x^k}^2 & < & +\infty \quad \textrm{almost surely},  \vspace{1ex}\\
\sum_{k=0}^{\infty}t_k^2   \norms{G_{\lambda}x^k - G_{\lambda}x^{k-1}}^2 & < & +\infty \quad \textrm{almost surely}, \vspace{1ex}\\
\sum_{k=0}^{\infty}t_k^2  \Delta_k & < & +\infty \quad \textrm{almost surely}, \vspace{1ex}\\
\sum_{k=0}^{\infty}t_k^2 \bar{\Ec}_k  & < & + \infty \quad \textrm{almost surely}.
\end{array}
\vspace{-0.5ex}
\end{equation}
The second line of \eqref{eq:VrAFBS4NI_th3_proof2} is the first summability bound of \eqref{eq:VrAFBS4NI_summable_bound2_as}.

%%% Step 2.
\vspace{0.5ex}
\noindent\textit{\textbf{Step 2. The second summability bound in \eqref{eq:VrAFBS4NI_summable_bound2_as}.}}
Since $\Expsn{k}{\norms{e^k}^2} \leq \Expsn{k}{ \Delta_k}$ by the first line of \eqref{eq:VR_property} and $\Gamma_k\Theta_k \leq \Lambda$ by \eqref{eq:VrAFBS4NI_param_cond_new}, utilizing these relations and the second line of \eqref{eq:VR_property}, we get
\begin{equation*} 
\arraycolsep=0.2em
\begin{array}{lcl}
\Gamma_k \Expsn{k}{\norms{e^k}^2} & \leq &  \Gamma_k \Expsn{k}{ \Delta_k } \overset{\tiny\eqref{eq:VR_property} }{ \leq } \Gamma_k(1-\kappa_k) \Delta_{k-1} + \Gamma_k\Theta_k\bar{\Ec}_k + \Gamma_k\sigma_k^2  \vspace{1ex}\\
&\leq & \Gamma_k(1-\kappa_k) \Delta_{k-1} + \Lambda\bar{\Ec}_k + \frac{\Lambda\sigma_k^2}{\Theta_k}.
\end{array}
\end{equation*}
Multiplying both sides of this inequality by $t_{k-1}(t_{k-1}-1)$ and utilizing $(1-\kappa_k)\Gamma_kt_{k-1}(t_{k-1}-1) \leq \Gamma_{k-1}t_{k-2}(t_{k-2}-1)$ from \eqref{eq:VrAFBS4NI_param2} and $\Gamma_{k-1} \leq \frac{\Lambda}{\underline{\Theta}}$ from \eqref{eq:VrAFBS4NI_param_cond_new}, we can show that
\begin{equation*} 
\arraycolsep=0.2em
\begin{array}{lcl}
\Gamma_kt_{k-1}(t_{k-1}-1) \Expsn{k}{\norms{e^k}^2}
&\leq &  \Gamma_{k-1}t_{k-2}(t_{k-2}-1) \Delta_{k-1} + \Lambda t_{k-1}(t_{k-1}-1) \big( \bar{\Ec}_k + \frac{\sigma_k^2}{\Theta_k}\big) \vspace{1ex}\\
&\leq & \Gamma_{k-1}t_{k-2}(t_{k-2}-1) \norms{e^{k-1}}^2 -  \Gamma_{k-1}t_{k-2}(t_{k-2}-1) \norms{e^{k-1}}^2 \vspace{1ex}\\
&& + {~} \frac{\Lambda}{\underline{\Theta}}t_{k-2}(t_{k-2}-1) \Delta_{k-1} + \Lambda t_{k-1}(t_{k-1}-1)  \big( \bar{\Ec}_k + \frac{\sigma_k^2}{\Theta_k}\big).
\end{array}
\end{equation*}
Let us denote by
\begin{equation*} 
\arraycolsep=0.2em
\begin{array}{lcl}
E_k &:= &  \frac{\Lambda}{\underline{\Theta}} t_{k-2}(t_{k-2}-1) \Delta_{k-1} + \Lambda t_{k-1}(t_{k-1}-1) \bar{\Ec}_k + \Lambda t_{k-1}(t_{k-1}-1) \frac{\sigma_k^2}{\Theta_k}.
\end{array}
\end{equation*}
Then, by \eqref{eq:VrAFBS4NI_th3_proof2} and \eqref{eq:summable_variance}, we have $\sum_{k=0}^{\infty}E_k \leq \frac{\Lambda}{\underline{\Theta}} \sum_{k=0}^{\infty}t_k^2\Delta_k + \Lambda\sum_{k=0}^{\infty}t_k^2\bar{\Ec}_k + \beta B_{\infty} < +\infty$ almost surely.
Thus, applying again Lemma~\ref{le:RS_lemma} with $U_k :=  \Gamma_{k-1}t_{k-2}(t_{k-2}-1) \norms{e^{k-1}}^2$, $\gamma_k := 0$, $V_k :=  \Gamma_{k-1}t_{k-2}(t_{k-2}-1) \norms{e^{k-1}}^2$, and $E_k$ given above,  we conclude that 
\begin{equation*}%\label{eq:VrAFBS4NI_th3_proof3} 
\arraycolsep=0.2em
\begin{array}{lcl}
\sum_{k=0}^{\infty}\Gamma_k t^2_k  \norms{e^k}^2 & < & +\infty \quad  \textrm{almost surely}.
\end{array}
\end{equation*}
However, since $\Gamma_k \geq \underline{\Gamma} > 0$ by \eqref{eq:VrAFBS4NI_param_cond_new}, this implies the second summability bound in \eqref{eq:VrAFBS4NI_summable_bound2_as}.

%%% Step 3.
\vspace{0.5ex}
\noindent\textit{\textbf{Step 3. The last summability bound in \eqref{eq:VrAFBS4NI_summable_bound2_as}.}}
Since $\norms{\widetilde{G}^k_{\lambda}}^2 \leq 2\norms{G_{\lambda}x^k}^2 + 2\norms{e_{\lambda}^k}^2 \leq 2\norms{G_{\lambda}x^k}^2 + 2\norms{e^k}^2$, using the second lines of \eqref{eq:VrAFBS4NI_th3_proof2} and \eqref{eq:VrAFBS4NI_summable_bound2_as}, this inequality implies that 
\begin{equation}\label{eq:VrAFBS4NI_th3_proof3d} 
\arraycolsep=0.2em
\begin{array}{lcl}
\sum_{k=0}^{\infty}t_k \norms{\widetilde{G}^k_{\lambda}}^2 < +\infty \quad  \textrm{almost surely}.
\end{array}
\end{equation}
Let $v^k := x^{k+1} - x^k + \eta_k\widetilde{G}_{\lambda}^k$.
Taking the conditional expectation $\Expsn{k}{\cdot}$ of \eqref{eq:VrAFBS4NI_th12_proof102},  we get
\begin{equation*} 
\arraycolsep=0.2em
\begin{array}{lcl}
\Expsn{k}{ t_k^2\norms{v^k}^2} & \leq &  t_{k-1}^2\norms{v^{k-1}}^2 - t_{k-1}\norms{v^{k-1}}^2 + 4(1-\nu)^2\beta^2t_{k-1}\norms{ \widetilde{G}_{\lambda}^{k-1} }^2.
\end{array}
\end{equation*}
Since  $\sum_{k=0}^{\infty}t_k \norms{\widetilde{G}^k_{\lambda}}^2 < +\infty$ by \eqref{eq:VrAFBS4NI_th3_proof3d}, applying Lemma~\ref{le:RS_lemma} to the last inequality, and then using Lemma~\ref{le:A2_sum}, we conclude that the following statements hold almost surely:
\begin{equation}\label{eq:VrAFBS4NI_th3_proof4}  
\arraycolsep=0.2em
\begin{array}{lcl}
\lim_{k\to\infty}t_k^2 \norms{v^k}^2 = 0 \quad \textrm{and} \quad \sum_{k=0}^{\infty}t_k  \norms{v^k}^2 < +\infty.
\end{array}
\end{equation}
Let $w^k := x^{k+1} - x^k + \eta_ke_{\lambda}^k$ as before.
Following similar arguments here, but applying them to \eqref{eq:VrAFBS4NI_th12_proof101}, we can also prove that the following statements hold almost surely:
\begin{equation}\label{eq:VrAFBS4NI_th3_proof4b}  
\arraycolsep=0.2em
\begin{array}{lcl}
\lim_{k\to\infty}t_k^2 \norms{w^k}^2 = 0 \quad \textrm{and} \quad \sum_{k=0}^{\infty}t_k  \norms{w^k}^2 < +\infty.
\end{array}
\end{equation}
Combining \eqref{eq:VrAFBS4NI_th12_proof105}, \eqref{eq:VrAFBS4NI_th3_proof4b}, and the second line of \eqref{eq:VrAFBS4NI_summable_bound2_as}, we can easily prove that $\sum_{k=0}^{\infty}t_k\norms{x^{k+1} - x^k}^2 < +\infty$ almost surely.
This is the last line of \eqref{eq:VrAFBS4NI_summable_bound2_as}.

%%% Step 4.
\vspace{0.5ex}
\noindent\textit{\textbf{Step 4. The limits of \eqref{eq:VrAFBS4NI_small_o_rates_as}.}}
Combining the limits in \eqref{eq:VrAFBS4NI_th3_proof4} and \eqref{eq:VrAFBS4NI_th3_proof4b} and Young's inequality, similar to \textbf{Step 6} in the proof of Theorem~\ref{th:VrAFBS4NI_o_rates_convergence2}, we can easily show that
\begin{equation}\label{eq:VrAFBS4NI_th3_proof5}  
\arraycolsep=0.2em
\begin{array}{lcl}
\lim_{k\to\infty}t_k^2 \norms{G_{\lambda}x^k}^2 = 0 \quad \textrm{almost surely.}
\end{array}
\end{equation}
This limit is the first limit in \eqref{eq:VrAFBS4NI_small_o_rates_as}.
The second limit of \eqref{eq:VrAFBS4NI_small_o_rates_as} follows from \eqref{eq:VrAFBS4NI_th12_proof105}, the first limit of \eqref{eq:VrAFBS4NI_th3_proof4b}, and $\lim_{k\to\infty} t_k^2\norms{e^k}^2 = 0$ almost surely from \eqref{eq:VrAFBS4NI_summable_bound2_as}.

%%% Step 5.
\vspace{0.5ex}
\noindent\textit{\textbf{Step 5. The existence of $\lim_{k\to\infty}\norms{z^k - x^{\star}}^2$.}}
From \eqref{eq:VrAFBS4NI}, we can easily shows that $z^k - x^k = t_k(x^{k+1} - x^k) + t_k\eta_k\widetilde{G}_{\lambda}^k = t_kv^k$.
Since $\lim_{k\to\infty}t_k^2\norms{v^k}^2 = 0$ almost surely as shown in \eqref{eq:VrAFBS4NI_th3_proof4}, we conclude that $\lim_{k\to\infty}\norms{z^k - x^k} = 0$ almost surely.
Using this limit and $\lim_{k\to\infty}t_k^2\norms{G_{\lambda}x^k}^2 = 0$ from \eqref{eq:VrAFBS4NI_th3_proof5}, we have
\begin{equation*} 
\arraycolsep=0.2em
\begin{array}{lcl}
2\vert t_k\iprods{G_{\lambda}x^k, x^k - z^k} \leq \norms{x^k - z^k}^2 + t_k^2\norms{G_{\lambda}x^k}^2 \to 0 \quad\textrm{as} \ k \to \infty \quad \textrm{almost surely}.
\end{array}
\end{equation*}
Hence, we get $\lim_{k\to\infty}t_{k-1}\vert \iprods{G_{\lambda}x^k, x^k - z^k} \vert = 0$ almost surely.

Next, as discussed in Assumption~\ref{as:A2}, $F$ is $L$-Lipschitz continuous, we have 
\begin{equation*} 
\arraycolsep=0.2em
\begin{array}{lcl}
\vert \iprods{Fx^k - Fx^{k-1}, x^k - x^{k-1}} \vert \leq \norms{Fx^k - Fx^{k-1}}\norms{x^k - x^{k-1}} = L\norms{x^k - x^{k-1}}^2.
\end{array}
\end{equation*}
Since $\lim_{k\to\infty} t_{k-1}^2 \norms{x^k - x^{k-1}}^2 = 0$ due to the second limit of \eqref{eq:VrAFBS4NI_small_o_rates_as}, we conclude that $\lim_{k\to\infty} t_{k-1}(t_{k-1} - 1) \vert \iprods{Fx^k - Fx^{k-1}, x^k - x^{k-1}} \vert = 0$ almost surely.

Now, from \eqref{eq:VrAFBS_Lyapunov_func} we can write
\begin{equation}\label{eq:VrAFBS4NI_th3_proof6}   
\vspace{-0.5ex}
\arraycolsep=0.2em
\begin{array}{lcl}
\Pc_k & = &  \beta a_k \norms{G_{\lambda}x^k}^2 + t_{k-1} \iprods{G_{\lambda}x^k, x^k - z^k}  +  \frac{[ \mu(1-\kappa_k)\Gamma_k + \beta] t_{k-1}(t_{k-1}-1)}{2\mu}\Delta_{k-1} \vspace{1ex}\\
&&  + {~} \big[ \bar{\beta} - \frac{(2+\mu)\beta}{2-\mu} \big] t_{k-1} (t_{k-1} -1) \norms{G_{\lambda}x^k - G_{\lambda}x^{k-1} }^2 \vspace{1ex}\\
&& + {~} \Lambda L t_{k-1}(t_{k-1}-1) \iprods{Fx^k - Fx^{k-1}, x^k - x^{k-1}} + \frac{c_k}{2\nu\beta}\norms{z^k - x^{\star}}^2.
\end{array}
\vspace{-0.25ex}
\end{equation}
Collecting all the necessary almost sure limits proven above, we have, almost surely
\vspace{-0.5ex}
\begin{equation*} 
\arraycolsep=0.2em
\begin{array}{lcl}
\lim_{k\to\infty}\Pc_k \ \ \textrm{exists as shown in \eqref{eq:VrAFBS4NI_th3_proof2}} , \vspace{1ex}\\
\lim_{k\to\infty}a_k \norms{G_{\lambda}x^k}^2 = \lim_{k\to\infty}t_k^2 \norms{G_{\lambda}x^k}^2 = 0 \quad \textrm{by \eqref{eq:VrAFBS4NI_th3_proof5}}, \vspace{1ex}\\
\lim_{k\to\infty}t_{k-1}\vert \iprods{G_{\lambda}x^k, x^k - z^k} \vert = 0, \vspace{1ex}\\
\lim_{k\to\infty} t_{k-1}(t_{k-1}-1) \vert \iprods{Fx^k - Fx^{k-1}, x^k - x^{k-1}} \vert = 0, \vspace{1ex}\\
\lim_{k\to\infty}t_k^2   \norms{G_{\lambda}x^k - G_{\lambda}x^{k-1}}^2 = 0 \quad \textrm{from \eqref{eq:VrAFBS4NI_th3_proof2}}, \vspace{1ex}\\
\lim_{k\to\infty}t_k^2\Delta_k = 0 \quad \textrm{from \eqref{eq:VrAFBS4NI_th3_proof2}}.
\end{array}
\vspace{-0.5ex}
\end{equation*}
Combining these limits and \eqref{eq:VrAFBS4NI_th3_proof6}, and noting that $c_k = \BigOs{1}$ due to \eqref{eq:VrAFBS_coefficients}, we conclude that  $\lim_{k\to\infty}\norms{z^k - x^{\star}}^2$ exists almost surely.

\vspace{0.5ex}
\noindent\textit{\textbf{Step 6.  Almost sure convergence of $\sets{x^k}$ and $\sets{z^k}$.}}
Since $\lim_{k\to\infty}\norms{x^k - z^k} = 0$  and $\lim_{k\to\infty}\norms{z^k - x^{\star}}$ exists almost surely, combining these facts and $\vert \norms{x^k - x^{\star}} - \norms{z^k - x^{\star}} \vert \leq \norms{x^k - z^k}$, we conclude that $\lim_{k\to\infty}\norms{x^k - x^{\star}}^2$ exists almost surely.

Finally, since $\lim_{k\to\infty}\norms{x^k - x^{\star}}^2$ exists almost surely, $\lim_{k\to\infty}\norms{G_{\lambda}x^k} = 0$ almost surely by the second line of \eqref{eq:VrAFBS4NI_th3_proof2}, and $G_{\lambda}$ is continuous, applying Lemma~\ref{le:A3_lemma}, we conclude that $\sets{x^k}$ converges to  a random variable $\bar{x}^{\star} \in \zer{G_{\lambda}}$ almost surely.
However, since $\bar{x}^{\star} \in \zer{G_{\lambda}}$ if and only if $\bar{x}^{\star} \in \zer{\Phi}$ surely, we also conclude that  $\sets{x^k}$  almost surely converges to a solution $\bar{x}^{\star} \in \zer{\Phi}$.
In addition, since $\lim_{k\to\infty}\norms{z^k - x^k} = 0$ almost surely, we can state that $\sets{z^k}$  almost surely converges to  $\bar{x}^{\star} \in \zer{\Phi}$.
\end{proof}

%%%%%%%%%%%%%%%%%%%%%%%%%%%%%%%%%%%%%%%%%%%%
%%% D. The Proof of Corollaries in Subsections 4.2 and 4.3.
%%%%%%%%%%%%%%%%%%%%%%%%%%%%%%%%%%%%%%%%%%%%
\vspace{-2ex}
\beforesec
\section{The Proof of Corollaries in Subsections~\ref{subsec:complexity_bounds1} and \ref{subsec:complexity_bounds2}}\label{apdx:subsec:concrete_estimator_complexity}
\aftersec
We now present the full proofs of Corollaries in Subsections~\ref{subsec:complexity_bounds1} and \ref{subsec:complexity_bounds2}.

%%%%%%%%%%%%%%%%%%%%%%%%%%%%%%%%%%%%%%%%
\beforesubsec
\subsection{The finite-sum setting \eqref{eq:finite_sum_form}}\label{apdx:subsec:complexity_bounds1}
\aftersubsec
This appendix presents the proof of Corollaries \ref{co:SVRG_complexity}, \ref{co:SAGA_complexity}, and \ref{co:SARAH_complexity}, respectively.

%%%%%%%%%%%%%%%%%%%%%%%%%%%%%%%%%%%%%%%%
%%% Proof of Corollary 4.1.
\beforesubsubsec
\subsubsection{The proof of Corollary~\ref{co:SVRG_complexity} --- The L-SVRG variant of \ref{eq:VrAFBS4NI}}\label{apdx:co:SVRG_complexity}
\aftersubsubsec
\begin{proof} 
Since the  \ref{eq:loopless_svrg} estimator is used and $\bar{F}\tilde{x}^k := F\tilde{x}^k$, by Lemma~\ref{le:loopless_svrg_bound}, we have
\begin{equation*}
\arraycolsep=0.2em
\begin{array}{ll}
& \Delta_k = \hat{\Delta}_k, \quad \kappa_k = \alpha \mbf{p}_k, \quad  \Theta_k =  \frac{1}{(1-\alpha)b_k\mbf{p}_k} \geq \underline{\Theta} := \frac{1}{(1-\alpha)n} > 0,  \quad \textrm{and} \quad \sigma_k^2 = 0.
\end{array}
\end{equation*}
Thus, the quantity $B_K$ in Theorem~\ref{th:VrAFBS4NI_convergence} becomes $B_K := \frac{\Lambda}{\beta}\sum_{k=0}^K t_{k-1}(t_{k-1}-1)  \frac{\sigma_k^2}{\Theta_k}  = 0$.
Moreover, since $x^0 = \tilde{x}^0$, we also have $\Delta_0 = \hat{\Delta}_0 := \frac{1}{nb_0}\sum_{i=1}^n \norms{F_ix^0 - F_i\tilde{x}^0}^2 \leq \sigma_0^2 = 0$.

Next, let us choose $\alpha = \frac{1}{2}$ and assume that $\Gamma_k = \frac{5c_p\beta n^{\omega}}{\mu} =: \underline{\Gamma} > 0$ for all $k\geq 0$ and given $c_p > 0$ and $\omega > 0$.
We recall the first condition of \eqref{eq:VrAFBS4NI_param_cond} in Theorem~\ref{th:VrAFBS4NI_convergence} as follows:
\begin{equation*}
\arraycolsep=0.2em
\begin{array}{lcl}
\kappa_k & = & \alpha\mbf{p}_k = \frac{\mbf{p}_k}{2} \geq 1 - \frac{\Gamma_{k-1}t_{k-2}(t_{k-2}-1)}{\Gamma_k t_{k-1}(t_{k-1}-1)} + \frac{5\beta}{\mu \Gamma_k } = 1 + \frac{5\beta}{\mu\Gamma_k} - \frac{t_{k-2}(t_{k-2}-1)}{t_{k-1}(t_{k-1}-1)} \vspace{1ex}\\
& = & \frac{1}{c_pn^{\omega}} + \frac{2\mu}{\mu(k+r-1) - 1} - \frac{\mu+1}{(k+r-1)(\mu(k+r-1) - 1)}.
\end{array}
\end{equation*}
This condition holds if $\mbf{p}_k \geq \frac{2}{c_p n^{\omega}} + \frac{4\mu}{\mu(k + r-1)-1}$.

Now, for $r > 5 + \frac{1}{\mu}$ and $n^{\omega} \geq \frac{1}{c_p} \max\big\{ \frac{2\mu(r-1)-2}{\mu(r-5)-1}, \frac{\mu(r-1) -1 }{4\mu} \big\}$, if we choose
\vspace{-0.5ex}
\begin{equation*}
\mbf{p}_k := \begin{cases}
\frac{2}{c_p n^{\omega}} + \frac{4\mu}{\mu(k+r-1) - 1} &\textrm{if}~0 \leq k \leq K_0 := \lfloor 4c_p n^{\omega} + 1 + \mu^{-1} - r \rfloor, \\
\frac{3}{c_p n^{\omega}}  & \textrm{otherwise},
\end{cases}
\vspace{-0.25ex}
\end{equation*}
then $0 <  \frac{2}{c_p n^{\omega}} + \frac{4\mu}{\mu(k + r-1) - 1} \leq \mbf{p}_k \leq 1$.
Consequently, the first condition in \eqref{eq:VrAFBS4NI_param_cond} holds.

Since $\Gamma_k\Theta_k \leq \frac{10 c_p\beta n^{\omega}}{\mu b_k \mbf{p}_k} \leq \frac{5\beta c_p^2n^{2\omega}}{\mu b_k}$, the second condition of \eqref{eq:VrAFBS4NI_param_cond} in Theorem~\ref{th:VrAFBS4NI_convergence} and \eqref{eq:VrAFBS4NI_param_cond_new} in Theorem~\ref{th:VrAFBS4NI_o_rates_convergence3} both hold if 
\begin{equation*} 
\arraycolsep=0.2em
\begin{array}{lcl}
\frac{5\beta c_p^2n^{2\omega}}{\mu b} \leq \Lambda.
\end{array}
\end{equation*}
Clearly, if we choose $b_k = b := \big\lfloor \frac{5\beta c_p^2 n^{2\omega}}{\mu\Lambda} \big\rfloor = \lfloor c_b n^{2\omega}\rfloor$ with $c_b := \frac{5\beta c_p^2}{\mu\Lambda}$, then this condition holds.
Consequently, the second condition of \eqref{eq:VrAFBS4NI_param_cond} and \eqref{eq:VrAFBS4NI_param_cond_new} hold.

Since $B_K = 0$ and $E_0^2 = 0$, for a given $\epsilon > 0$, to guarantee  $\Expn{ \norms{G_{\lambda}x^{K}}^2} \leq \epsilon^2$,  from \eqref{eq:VrAFBS4NI_th1_convergence1} we can impose that $\frac{2\Psi_0^2 }{\mu^2(K+r-1)^2} \leq \epsilon^2$.
The last condition holds if we choose $K := \big\lfloor \frac{\sqrt{2}\Psi_0}{\mu \epsilon} \big\rfloor - 1$.

The expected total number of oracle calls is $\Expn{\Tc_K} := n + \Expn{\hat{\Tc}_K}$, where
\begin{equation*}
\arraycolsep=0.2em
\begin{array}{lcl}
\Expn{ \hat{\Tc}_K } & = & \sum_{k=0}^K ( n \mbf{p}_k + 2b ) \vspace{1ex}\\
& = & n\sum_{k=0}^K\mbf{p}_k  + 2b  (K+1) \vspace{1ex}\\
& = & 2b (K+1) + n \sum_{k=0}^{K_0} \mbf{p}_k + n \sum_{k=K_0+1}^K\mbf{p}_k\vspace{1ex}\\
& \leq & 2 b (K+1) + n \big[ \sum_{k=0}^{K_0} \frac{2}{c_p n^{\omega}} + 4\sum_{k=0}^{K_0}\frac{\mu}{\mu(k+r-1)-1} + \sum_{k=K_0+1}^K \frac{3}{c_p n^{\omega}} \big] \vspace{1ex}\\
& \leq &  \frac{2\sqrt{2}c_b\Psi_0 n^{2\omega} }{\mu \epsilon} + n \big[ \frac{2(4c_p n^{\omega} - r + 1 + \mu^{-1})}{c_p n^{\omega}} + 4\ln(4c_p n^{\omega} - r+2+\mu^{-1}) +  \frac{3\sqrt{2}\Psi_0}{c_p n^{\omega}\mu \epsilon} \big] \vspace{1ex}\\
& \leq & 4 n\big[ 2 + \ln(4c_p n^{\omega}) \big] + \frac{\sqrt{2}\Psi_0}{\mu \epsilon}  \big( 2c_bn^{2\omega} + \frac{ 3n^{1-\omega}}{c_p} \big).
\end{array}
\end{equation*}
Here, we have used $r - 1 - \frac{1}{\mu} \geq 1$ and $\sum_{k=0}^{K_0}\frac{\mu}{\mu(k+r-1)-1} \leq \ln(1 + (r-1-\frac{1}{\mu})^{-1}K_0) \leq \ln(K_0 + 1)$ in the last inequality.
Therefore, we conclude that the expected total number of oracle calls $F_i$ and evaluations $J_{\lambda T}$ in \eqref{eq:VrAFBS4NI}  is at most 
\begin{equation*}
\arraycolsep=0.2em
\begin{array}{lcl}
\Expn{ \Tc_K } & := & \Big\lfloor n + 4 n\big[ 2 + \ln(4c_p n^{\omega}) \big] + \frac{\sqrt{2}\Psi_0}{\mu \epsilon}  \big( 2c_bn^{2\omega} + \frac{ 3n^{1-\omega} }{c_p} \big) \Big\rfloor
\end{array}
\end{equation*}
to achieve $\Expn{ \norms{G_{\lambda}x^{K}}^2} \leq \epsilon^2$.
In particular, if we choose $\omega := \frac{1}{3}$, then we get $\Expn{ \Tc_K } :=  \big\lfloor n + 4 n\big[ 2 + \ln(4c_p n^{1/3}) \big] + \frac{\sqrt{2}\Psi_0 (3 + 2c_bc_p) n^{2/3} }{c_p \mu \epsilon} \big\rfloor$.
\end{proof}
%%% End of Proof.

%%%%%%%%%%%%%%%%%%%%%%%%%%%%%%%%%%%%%%%%
%%% Proof of Corollary 4.2.
\beforesubsubsec
\subsubsection{The proof of Corollary~\ref{co:SAGA_complexity} --- The SAGA variant of \ref{eq:VrAFBS4NI}}\label{apdx:co:SAGA_complexity}
\aftersubsubsec
\begin{proof}
Since the \ref{eq:SAGA_estimator} estimator is used, by Lemma~\ref{le:SAGA_estimator_full}, we have
\begin{equation*}
\arraycolsep=0.2em
\begin{array}{lcl}
\kappa_k = \frac{\alpha b_k }{n}, \quad \Theta_k =   \frac{(3-\alpha)n}{(1-\alpha)b_k^2} \geq \underline{\Theta} := \frac{3-\alpha}{(1-\alpha)n} > 0, \quad \textrm{and} \quad \sigma_k^2 = 0.
\end{array}
\end{equation*}
Note that $B_K$ in Theorem~\ref{th:VrAFBS4NI_convergence} becomes $B_K := \frac{\Lambda}{\beta}\sum_{k=0}^K t_{k-1}(t_{k-1}-1) \frac{ \sigma_k^2 }{\Theta_k} = 0$.
Moreover, since $\hat{F}_i^0 = F_ix^0$ for all $i \in [n]$, we also have $\Delta_0 :=   \frac{1}{nb_0} \sum_{i=1}^n \norms{ F_ix^0 -  \hat{F}_{i}^0 }^2  \leq \sigma_0^2 = 0$.

Next, let us choose $\alpha := \frac{1}{2}$ and $\Gamma_k := \frac{5\beta n^{\omega}}{\mu c_b} =: \underline{\Gamma} > 0 $ for all $k\geq 0$ and given $\omega > 0$ and $c_b > 0$.
We recall the first condition of \eqref{eq:VrAFBS4NI_param_cond} in Theorem~\ref{th:VrAFBS4NI_convergence} as follows:
\begin{equation*}
\arraycolsep=0.2em
\begin{array}{lcl}
\kappa_k = \frac{\alpha b_k}{n} = \frac{b_k}{2n} \geq 1 - \frac{\Gamma_{k-1}t_{k-2}(t_{k-2}-1)}{\Gamma_k t_{k-1}(t_{k-1} - 1)} + \frac{5\beta}{\mu\Gamma_k } = \frac{c_b}{n^{\omega}} + \frac{2\mu}{\mu(k+r-1) - 1} - \frac{1+\mu}{(k+r-1)(\mu(k+r-1)-1)}.
\end{array}
\end{equation*}
Let us choose 
\begin{equation*}
b_k := \begin{cases}
2c_bn^{1-\omega} + \frac{4\mu n}{\mu(k+r-1)-1} & \textrm{if}~~k \leq K_0 := \lfloor 4n^{\omega} + 1 + \mu^{-1} - r \rfloor, \\
3c_bn^{1-\omega}, &\textrm{otherwise}.
\end{cases}
\end{equation*}
Then, the last condition holds.
Consequently, the first condition of \eqref{eq:VrAFBS4NI_param_cond} holds.
Clearly, we have $b_k \leq b_{k-1}$ for all $k\geq 1$.
Moreover, using the update of $b_k$, we have
\begin{equation*}
\arraycolsep=0.2em
\begin{array}{lcl}
b_{k-1} - b_k = \frac{4\mu^2 n}{[\mu(k+r-1)-1][\mu(k+r-2)-1]} \leq  \frac{b_kb_{k-1}}{4n}.
\end{array}
\end{equation*}
Therefore, we finally get $b_{k-1} - \frac{(1-\alpha)b_kb_{k-1}}{2n} \leq b_k \leq b_{k-1}$, which verifies the condition of Lemma~\ref{le:SAGA_estimator_full}.
Note that since $1\leq b_k \leq n$ and $K_0 \geq 0$,  we require $r > 5 + \frac{1}{\mu}$ and $n^{\omega} \geq \max \big\{\frac{\mu(r-1)-1}{4\mu}, \frac{2c_b[\mu(r-1)-1]}{\mu(r-5)-1} \big\}$ as stated in Corollary~\ref{co:SAGA_complexity}.

From the definition of $\Theta_k$ above and $b_k \geq 2c_bn^{1-\omega}$, we can easily show that $\Theta_k =  \frac{(3-\alpha)n}{(1-\alpha) b_k^2} = \frac{5n}{b_k^2} \leq \frac{5}{4c_b^2n^{1-2\omega}}$.
Since  $\Gamma_k\Theta_k  \leq  \frac{ 25 \beta n^{\omega}}{4 \mu c_b^2n^{1-2\omega}}$, the second condition of \eqref{eq:VrAFBS4NI_param_cond} in Theorem~\ref{th:VrAFBS4NI_convergence} and the condition \eqref{eq:VrAFBS4NI_param_cond_new} in Theorem~\ref{th:VrAFBS4NI_o_rates_convergence3} both hold if $ \frac{ 25\beta }{4 \mu c_b^2n^{1 - 3\omega}} \leq \Lambda$.
Let us choose $\omega := \frac{1}{3}$ and  $c_b := \frac{5}{2} \sqrt{ \frac{\beta}{\mu\Lambda}}$.
Then, the last condition holds.
Consequently, the second condition of \eqref{eq:VrAFBS4NI_param_cond} and \eqref{eq:VrAFBS4NI_param_cond_new} are both satisfied.

Since $B_K = 0$ and $E_0^2 = 0$, for a given $\epsilon > 0$, to guarantee  $\Expn{ \norms{G_{\lambda}x^{K}}^2} \leq \epsilon^2$, from \eqref{eq:VrAFBS4NI_th1_convergence1}, we can impose $\frac{2\Psi_0^2 }{\mu^2(K+r-1)^2} \leq \epsilon^2$.
This condition holds if we choose $K := \big\lfloor \frac{\sqrt{2}\Psi_0}{\mu \epsilon} \big\rfloor - 1$.

The expected total number of oracle calls becomes $\Expn{\Tc_K} := n + \Expn{\hat{\Tc}_K}$, where $n$ is the number of $F_i$ evaluations for the first epoch, and 
\begin{equation*}
\arraycolsep=0.2em
\begin{array}{lcl}
\Expn{\hat{\Tc}_K} & = & \sum_{k=0}^K b_k = \sum_{k=0}^{K_0}b_k + \sum_{k=K_0+1}^Kb_k \vspace{1ex}\\
& \leq & 2c_bn^{2/3}(K_0 + 1) + \sum_{k=0}^{K_0}\frac{4\mu n}{\mu(k+r-1) - 1} + 3c_b(K-K_0)n^{2/3} \vspace{1ex}\\
& \leq & 8c_bn^{2/3}n^{1/3} + 4 n \ln(K_0 + 1) + 3c_bKn^{2/3} \vspace{1ex}\\
& \leq & 8c_bn + 4n\ln(4n^{1/3}) + \frac{3\sqrt{2}c_b \Psi_0 n^{2/3}}{\mu\epsilon}.
\end{array}
\end{equation*}
Thus, we get $\Expn{\Tc_K} :=  \big\lfloor  [8c_b + 4\ln(4n^{1/3})] n   + \frac{3\sqrt{2} c_b \Psi_0 n^{2/3}}{\mu\epsilon} \big\rfloor$.
We conclude that the total number of $F_i$ and $J_{\lambda T}$ evaluations of \eqref{eq:VrAFBS4NI}  is at most $\Expn{\Tc_K} := \BigOs{ n\ln(n) + \frac{n^{2/3}}{\epsilon} }$ to achieve $\Expn{ \norms{G_{\lambda}x^{K}}^2} \leq \epsilon^2$.
\end{proof}
%%% End of Proof.

%%%%%%%%%%%%%%%%%%%%%%%%%%%%%%%%%%%%%%%%
%%% Proof of Corollary 4.3.
\beforesubsubsec
\subsubsection{The proof of Corollary~\ref{co:SARAH_complexity} --- The L-SARAH variant of \ref{eq:VrAFBS4NI}}\label{apdx:co:SARAH_complexity}
\aftersubsubsec
\begin{proof} 
Since the \ref{eq:loopless_sarah} estimator is used and $\bar{F}x^k = Fx^k$, by Lemma~\ref{le:loopless_sarah_bound}, we have
\begin{equation*}
\arraycolsep=0.2em
\begin{array}{lcl}
\kappa_k =  \mbf{p}_k, \quad \Theta_k =  \frac{1}{b_k} \geq \underline{\Theta} := \frac{1}{n} > 0,  \quad \textrm{and} \quad \sigma_k^2 = 0.
\end{array}
\end{equation*}
The quantity $B_K$ in Theorem~\ref{th:VrAFBS4NI_convergence} becomes $B_K := \frac{\Lambda}{\beta}\sum_{k=0}^K t_{k-1}(t_{k-1}-1)\frac{\sigma_k^2}{\Theta_k} = 0$.
Moreover, since $\widetilde{F}^0 := Fx^0$, we have $\norms{\widetilde{F}^0 - Fx^0}^2 \leq \sigma_0^2 = 0$.

Next, let us choose $\Gamma_k = \frac{5c_p\beta n^{\omega}}{\mu} =: \underline{\Gamma} > 0$ for all $k\geq 0$ and  given $c_p > 0$ and $\omega > 0$.
We recall the first condition of \eqref{eq:VrAFBS4NI_param_cond} in Theorem~\ref{th:VrAFBS4NI_convergence} as follows:
\begin{equation*}
\arraycolsep=0.2em
\begin{array}{lcl}
\kappa_k = \mbf{p}_k \geq 1 - \frac{\Gamma_{k-1} t_{k-2}(t_{k-2}-1)}{\Gamma_{k} t_{k-1}(t_{k-1} - 1)} + \frac{5\beta }{\mu\Gamma_k} = \frac{1}{c_p n^{\omega}}  + \frac{2\mu}{\mu(k+r-1) - 1} - \frac{1+\mu}{(k+r-1)(\mu(k+r-1)-1)}.
\end{array}
\end{equation*}
If $r > 3 + \frac{1}{\mu}$ and $n^{\omega} \geq \frac{1}{c_p} \max\sets{\frac{\mu(r-1) - 1}{\mu(r-3)-1}, \frac{\mu(r-1)-1}{2\mu}}$, then by choosing
\begin{equation*}
\mbf{p}_k := \begin{cases}
\frac{1}{c_p n^{\omega}} + \frac{2\mu}{\mu(k+r-1) - 1} &\textrm{if}~k \leq K_0 := \lfloor 2c_p n^{\omega} - r + 1 + \mu^{-1} \rfloor, \\
\frac{2}{c_p n^{\omega}}  & \textrm{otherwise},
\end{cases}
\end{equation*}
we have $0 < \mbf{p}_k \leq 1$ and the last condition holds.
Thus the first condition of \eqref{eq:VrAFBS4NI_param_cond} holds.

Since $\Gamma_k\Theta_k = \frac{5\beta c_p  n^{\omega}}{\mu b_k}$,
the second condition of \eqref{eq:VrAFBS4NI_param_cond} in Theorem~\ref{th:VrAFBS4NI_convergence} and the condition \eqref{eq:VrAFBS4NI_param_cond_new} in Theorem~\ref{th:VrAFBS4NI_o_rates_convergence3} both hold if $ \frac{5 c_p\beta n^{\omega}}{\mu b_k} \leq  \Lambda$.
Clearly, if we choose $b_k = b := \big\lfloor\frac{5\beta c_p n^{\omega}}{\mu\Lambda} \big\rfloor = \lfloor c_b n^{\omega}\rfloor$ with $c_b := \frac{5\beta c_p}{\mu\Lambda}$, then this condition holds.
Consequently, the second condition of \eqref{eq:VrAFBS4NI_param_cond} and \eqref{eq:VrAFBS4NI_param_cond_new} are both satisfied.

Since $B_K = 0$ and $E_0^2 = 0$, for a given $\epsilon > 0$, to guarantee  $\Expn{ \norms{G_{\lambda}x^{K}}^2} \leq \epsilon^2$, from \eqref{eq:VrAFBS4NI_th1_convergence1}, we can impose $\frac{2\Psi_0^2 }{\mu^2(K+r-1)^2} \leq \epsilon^2$.
This condition holds if we choose $K := \big\lfloor \frac{\sqrt{2}\Psi_0}{\mu \epsilon} \big\rfloor - 1$.

We can estimate the expected total number of oracle calls as $\Expn{\Tc_K} := n + \Expn{ \hat{\Tc}_K}$, where
\begin{equation*}
\arraycolsep=0.2em
\begin{array}{lcl}
\Expn{\hat{\Tc}_K} & = & \sum_{k=0}^K[ \mbf{p}_k n + 2(1 - \mbf{p}_k )b] \vspace{1ex}\\
& = & n\sum_{k=0}^K\mbf{p}_k + 2b \sum_{k=0}^K(1 - \mbf{p}_k ) \vspace{1ex}\\
& = & 2b (K+1) + (n - 2b)\sum_{k=0}^{K_0} \mbf{p}_k + (n - 2b)\sum_{k=K_0+1}^K \mbf{p}_k \vspace{1ex}\\
& \leq & 2b (K+1) + n \big[ \sum_{k=0}^{K_0} \frac{1}{c_p n^{\omega}} + 2\sum_{k=0}^{K_0}\frac{\mu}{\mu(k+r-1) - 1} + \sum_{k=K_0+1}^K \frac{2}{c_p n^{\omega}} \big] \vspace{1ex}\\
& \leq &  \frac{\sqrt{2} c_b \Psi_0 n^{\omega} }{\mu \epsilon} + n \big[ \frac{2c_p n^{\omega} - r + 1 + \mu^{-1}}{c_p n^{\omega}} + 2\ln(2c_p n^{\omega} - r + 1+\mu^{-1}) +  \frac{2\sqrt{2}\Psi_0}{c_p n^{\omega}\mu \epsilon} \big] \vspace{1ex}\\
& \leq & 2n\big[ 2 + \ln(c_p n^{\omega}) \big] + \frac{\sqrt{2}\Psi_0}{\mu \epsilon}  \big( c_bn^{\omega} + \frac{ 2n^{1-\omega}}{c_p} \big).
\end{array}
\end{equation*}
Here, we have used $r - 1 - \frac{1}{n} \geq 1$ in the last inequality.
Hence, we conclude that the expected total number of oracle calls $F_i$ and $J_{\lambda T}$ evaluations of \eqref{eq:VrAFBS4NI}  is at most 
\begin{equation*}
\arraycolsep=0.2em
\begin{array}{lcl}
\Expn{\Tc_K} := \Big\lfloor n + 2n\big[ 2 + \ln( c_p n^{\omega}) \big] + \frac{\sqrt{2}\Psi_0}{\mu \epsilon}  \big( c_bn^{\omega} + \frac{ 2n^{1-\omega}}{c_p} \big) \Big\rfloor
\end{array}
\end{equation*}
to achieve $\Expn{ \norms{G_{\lambda}x^{K}}^2} \leq \epsilon^2$.
Clearly, if we choose $\omega = \frac{1}{2}$, then we obtain $\Expn{\Tc_K} := \BigOs{ n\ln(n^{1/2}) + \frac{n^{1/2}}{\epsilon} }$.
\end{proof}
%%% End of Proof.

%%%%%%%%%%%%%%%%%%%%%%%%%%%%%%%%%%%%%%%%
\beforesubsec
\subsection{The expectation setting \eqref{eq:expectation_form}}\label{apdx:subsec:complexity_bounds2}
\aftersubsec
This appendix presents the full proof of Corollaries \ref{co:SVRG_complexity_Esetting}, \ref{co:SARAH_complexity_Esetting}, and \ref{co:HSGD_complexity_Esetting}, respectively.

%%%%%%%%%%%%%%%%%%%%%%%%%%%%%%%%%%%%%%%%
%%% Proof of Corollary 4.1.
\beforesubsubsec
\subsubsection{The proof of Corollary~\ref{co:SVRG_complexity_Esetting} --- The L-SVRG variant of \ref{eq:VrAFBS4NI}}\label{apdx:co:SVRG_complexity_Esetting}
\aftersubsubsec
\begin{proof} 
Since \ref{eq:loopless_svrg} is used, substituting $\tau := 1$ and $\alpha := \frac{1}{2}$ into Lemma~\ref{le:loopless_svrg_bound}, we have
\begin{equation*}
\arraycolsep=0.2em
\begin{array}{ll}
\Delta_k := 2\hat{\Delta}_k + \frac{2\sigma^2}{n_k}, \quad \kappa_k := \alpha \mbf{p}_k = \frac{\mbf{p}_k}{2}, \quad \Theta_k :=   \frac{2}{(1-\alpha)b_k\mbf{p}_k} = \frac{4}{b_k\mbf{p}_k},  \quad \textrm{and} \quad \sigma_k^2 := \frac{2\alpha\mbf{p}_k \sigma^2}{n_k} = \frac{\mbf{p}_k \sigma^2}{n_k},
\end{array}
\end{equation*}
to fulfill the $\textbf{VR}(\kappa_k, \Theta_k, \Delta_k, \sigma_k)$ condition of Definition~\ref{de:VR_Estimators}.

Next, let us choose $\Gamma_k := \frac{5\beta}{\mu  \epsilon^{\omega}} =: \underline{\Gamma} > 0$ for a given tolerance $\epsilon \in (0, 1)$ and a given $\omega > 0$.
We recall the first condition of \eqref{eq:VrAFBS4NI_param_cond} in Theorem~\ref{th:VrAFBS4NI_convergence} as follows:
\begin{equation*}
\arraycolsep=0.2em
\begin{array}{lcl}
\kappa_k = \alpha\mbf{p}_k = \frac{\mbf{p}_k}{2} \geq 1 - \frac{\Gamma_{k-1} t_{k-2}(t_{k-2}-1)}{\Gamma_k t_{k-1}(t_{k-1}-1)} + \frac{5\beta}{\mu\Gamma_k }=  \epsilon^{\omega} + \frac{2\mu}{\mu(k+r-1)-1} - \frac{1+\mu}{(k+r-1)(\mu(k+r-1)-1)}.
\end{array}
\end{equation*}
From this relation, we can see that if we choose
\begin{equation*}
\arraycolsep=0.2em
\begin{array}{lcl}
\mbf{p}_k & := & 2\epsilon^{\omega} + \frac{4\mu}{\mu(k+r-1) - 1} \geq \underline{\mbf{p}} := 2\epsilon^{\omega},
\end{array}
\end{equation*}
then the first condition in \eqref{eq:VrAFBS4NI_param_cond} holds and $\mbf{p}_k \leq 1$, provided that $r > 5 + \frac{1}{\mu}$ and $\epsilon^{\omega} \leq \frac{\mu(r-5)-1}{2\mu(r-1)-2}$.

Since $\Gamma_k\Theta_k \leq \frac{10\beta }{\mu(1-\alpha)b \mbf{p}_k \epsilon^{\omega}} \leq \frac{10\beta}{\mu b\epsilon^{2\omega}}$, the second condition of \eqref{eq:VrAFBS4NI_param_cond} in Theorem~\ref{th:VrAFBS4NI_convergence} and the condition \eqref{eq:VrAFBS4NI_param_cond_new} in Theorem~\ref{th:VrAFBS4NI_o_rates_convergence3} both hold if $\frac{10\beta}{\mu b\epsilon^{2\omega}} \leq \Lambda$.
Therefore, if we choose $b_k = b := \big\lfloor \frac{10\beta}{\mu\Lambda \epsilon^{2\omega}} \big\rfloor = \lfloor\frac{c_b}{\epsilon^{2\omega}}\rfloor $ with $c_b := \frac{10\beta }{\mu\Lambda}$, then the last condition holds.
Consequently, the second condition of \eqref{eq:VrAFBS4NI_param_cond} and \eqref{eq:VrAFBS4NI_param_cond_new} are both satisfied.
Moreover, since $b_k = b > 0$ for all $k\geq 0$, we have $\Theta_k \geq \frac{4}{b} =: \underline{\Theta} > 0$, which fulfills  \eqref{eq:VrAFBS4NI_param_cond_new}.

If we fix $n_k = n > 0$, then the quantity $B_K$ in Theorem~\ref{th:VrAFBS4NI_convergence} becomes 
\begin{equation*} 
\arraycolsep=0.2em
\begin{array}{lcl}
B_K & := & \frac{\Lambda}{\beta}\sum_{k=0}^K t_{k-1}(t_{k-1}-1) \frac{\sigma_k^2}{\Theta_k} = \frac{5 \sigma^2}{\mu n \epsilon^{2\omega}}\sum_{k = 0}^K \mbf{p}_k^2  t_{k-1}(t_{k-1} - 1) \vspace{1ex}\\
& \leq & \frac{10\sigma^2}{ n \epsilon^{\omega}}\sum_{k=0}^K\big( \epsilon^{\omega} + \frac{2\mu}{\mu (k+r-1)-1}\big)(k+r-1)[\mu(k+r-1) - 1] \vspace{1ex}\\
& \leq & \frac{10\mu \sigma^2}{ n\epsilon^{\omega}}\sum_{k=0}^K\big[ \epsilon^{\omega}(k+r-1)^2 + 2(k+r-1) \big] \vspace{1ex}\\
& \leq & \frac{10\mu \sigma^2}{ n } \big( K  + \epsilon^{-\omega} \big) (K+r-1)^2.
\end{array}
\end{equation*}
Moreover, from the construction \eqref{eq:loopless_svrg} of $\widetilde{F}^0$, we have $E_0^2 = \frac{\mu r^2(\mu\Gamma_0 + \beta)\sigma^2}{2\beta n}$.
We also have $\Gamma_0\mu + \beta = \frac{5\beta}{\epsilon^{\omega}} + \beta \leq \frac{6\beta}{\epsilon^{\omega}}$.
Using these bounds,  to guarantee  $\Expn{ \norms{G_{\lambda}x^{K}}^2} \leq \epsilon^2$, from \eqref{eq:VrAFBS4NI_th1_convergence1}, we can impose
\begin{equation*} 
\arraycolsep=0.2em
\begin{array}{lcl}
\Expn{ \norms{G_{\lambda}x^K}^2}   & \leq &   \frac{2( \Psi_0^2 + B_{K-1})}{\mu^2(K+r-1)^2} + \frac{\mu r^2(\Gamma_0\mu + \beta)}{\mu^2 n \beta(K + r - 1)^2} \sigma^2  \vspace{1ex}\\
& \leq &  \frac{2 \Psi_0^2}{\mu^2(K+r-1)^2} + \frac{6 r^2 \sigma^2}{\mu (K + r - 1)^2\epsilon^{\omega}n} + \frac{3 \sigma^2}{\mu n}\big(K + \epsilon^{-\omega}\big)  \vspace{1ex}\\
& \leq & \epsilon^2.
\end{array}
\end{equation*}
This condition holds if 
\begin{equation*} 
\arraycolsep=0.2em
\begin{array}{lcl}
\frac{2 \Psi_0^2}{\mu^2(K+r-1)^2} \leq \frac{\epsilon^2}{4}, \quad \frac{3\sigma^2K }{\mu  n} \leq  \frac{\epsilon^2}{4}, \quad  \frac{3\sigma^2  }{n\epsilon^{\omega}} \leq  \frac{\epsilon^2}{4}, \quad \textrm{and} \quad \frac{ 6 r^2\sigma^2}{\mu (K + r - 1)^2\epsilon^{\omega}n} \leq \frac{\epsilon^2}{4}.
\end{array}
\end{equation*}
These four conditions are simultaneously satisfied if we choose 
\begin{equation*} 
\arraycolsep=0.2em
\begin{array}{lcl}
K :=  \big\lfloor \frac{2\sqrt{2}\Psi_0}{\mu \epsilon} - 1\big\rfloor \quad\textrm{and} \quad n \geq \sigma^2 \max\set{ \frac{3\mu r^2}{\Psi_0^2\epsilon^{\omega}},  \frac{12}{\epsilon^{2+\omega}}, \frac{24\sqrt{2}\Psi_0}{\mu^2\epsilon^{3}} }.  
\end{array}
\end{equation*}
Hence, we can choose $\omega := 1$ and $n := \frac{c_n}{\epsilon^3} = \BigOs{\frac{1}{\epsilon^3}}$ for given $c_n := 12 \sigma^2 \max\sets{1, \frac{2\sqrt{2}\Psi_0}{\mu^2}}$.

Since we fix $n_k = n > 0$ and $b_k = b > 0$ for all $k \geq 0$, the expected total number of  oracle calls becomes $\Expn{ \Tc_K} = n + \Expn{ \hat{\Tc}_K }$, where
\begin{equation*}
\arraycolsep=0.2em
\begin{array}{lcl}
\Expn{ \hat{\Tc}_K } & = & \sum_{k=0}^K ( n \mbf{p}_k + 2b ) = 2b (K + 1) + n \sum_{k=0}^{K} \mbf{p}_k \vspace{1ex}\\
& \leq & 2 b (K + 1) + 2n \big[ \sum_{k=0}^{K} \epsilon  + 2 \sum_{k=0}^{K}\frac{\mu}{\mu(k+r-1)-1} \big] \vspace{1ex}\\
& \leq &  \frac{4c_b\sqrt{2}\Psi_0 }{\mu \epsilon^3} + \frac{2c_n}{\epsilon^{3}} \big[  \epsilon (K + 1)  + 2\ln(K + 1) \big] \vspace{1ex}\\
& \leq & \frac{4\sqrt{2} c_b \Psi_0 }{\mu \epsilon^3} + \frac{2c_n}{\epsilon^{3}} \big[ \frac{2\sqrt{2}\Psi_0}{\mu} + 2\ln\big(\frac{2\sqrt{2}\Psi_0}{\mu \epsilon}\big) \big].
\end{array}
\end{equation*}
Therefore, we obtain $\Expn{ \hat{\Tc}_K } = \BigOs{ \epsilon^{-3} +  \epsilon^{-3}\ln(\epsilon^{-1})} = \widetilde{\mcal{O}}(\epsilon^{-3})$.
This inequality shows that the expected total number of oracle calls $\Fb(\cdot, \xi)$ and $J_{\lambda T}$ evaluations of \eqref{eq:VrAFBS4NI}  is at most $\Expn{ \Tc_K } :=  \BigOs{ \epsilon^{-3} +  \epsilon^{-3}\ln(\epsilon^{-1})}$ to achieve $\Expn{ \norms{G_{\lambda}x^{K}}^2} \leq \epsilon^2$.
\end{proof}
%%% End of Proof.

%%%%%%%%%%%%%%%%%%%%%%%%%%%%%%%%%%%%%%%%
%%% Proof of Corollary 4.3.
\beforesubsubsec
\subsubsection{The proof of Corollary~\ref{co:SARAH_complexity_Esetting} --- The L-SARAH variant of \ref{eq:VrAFBS4NI}}\label{apdx:co:SARAH_complexity_Esetting}
\aftersubsubsec
\begin{proof} 
Since the \ref{eq:loopless_sarah} estimator is used, by Lemma~\ref{le:loopless_sarah_bound}, we have
\vspace{-0.5ex}
\begin{equation*}
\arraycolsep=0.2em
\begin{array}{lcl}
\kappa_k :=  \mbf{p}_k, \quad \Theta_k :=  \frac{1}{b_k},  \quad \textrm{and} \quad \sigma_k^2 := \frac{\mbf{p}_k \sigma^2}{n}.
\end{array}
\vspace{-0.5ex}
\end{equation*}
Next, let us choose $\Gamma_k = \frac{5\beta }{\mu \epsilon^{\omega}} =: \underline{\Gamma} > 0$ for all $k\geq 0$, a given tolerance $\epsilon \in (0, 1)$, and $\omega > 0$.
We recall the first condition of \eqref{eq:VrAFBS4NI_param_cond} in Theorem~\ref{th:VrAFBS4NI_convergence} as follows:
\vspace{-0.5ex}
\begin{equation*}
\arraycolsep=0.2em
\begin{array}{lcl}
\kappa_k  & = & \mbf{p}_k  \geq 1 - \frac{\Gamma_{k-1}t_{k-2}(t_{k-2}-1)}{\Gamma_kt_k(t_k-1)} + \frac{5\beta}{\mu\Gamma_k} = \frac{5\beta}{\mu\Gamma_k}  + \frac{2\mu}{\mu(k+r-1)-1} - \frac{1+\mu}{(k+r-1)(\mu(k+r-1)-1)} \vspace{1ex}\\
&  = &  \epsilon^{\omega} + \frac{2\mu}{\mu(k+r-1)-1} - \frac{1+\mu}{(k+r-1)(\mu(k+r-1)-1)}.
\end{array}
\vspace{-0.5ex}
\end{equation*}
This condition holds if we choose $\mbf{p}_{k+1} :=  \epsilon^{\omega} + \frac{2\mu}{\mu(k+r-1)-1}$, provided that $r > 3 + \frac{1}{\mu}$ and $0 < \epsilon^{\omega} \leq \frac{\mu(r-3)-1}{\mu(r-1)-1}$.
In addition, we also have $0 < \underline{\mbf{p}} := \epsilon^{\omega} \leq \mbf{p}_k \leq 1$.

Since $\Gamma_k \Theta_k = \frac{5\beta}{\mu b_k \epsilon^{\omega}}$, the second condition of \eqref{eq:VrAFBS4NI_param_cond} in Theorem~\ref{th:VrAFBS4NI_convergence} and the condition \eqref{eq:VrAFBS4NI_param_cond_new} in Theorem~\ref{th:VrAFBS4NI_o_rates_convergence3}  both hold if $\frac{5\beta}{\mu b_k \epsilon^{\omega}} \leq \Lambda$.
Therefore, if we choose $b_k = b := \big\lfloor \frac{5\beta}{\mu\Lambda\epsilon^{\omega}}  \big\rfloor = \big\lfloor \frac{c_b}{\epsilon^{\omega}}\big\rfloor$ with $c_b := \frac{5\beta}{\mu\Lambda}$, then the last condition holds.
Consequently, the second condition of \eqref{eq:VrAFBS4NI_param_cond} and \eqref{eq:VrAFBS4NI_param_cond_new} are both satisfied.
Since $b_k = b > 0$, we have $\Theta_k \geq \frac{1}{b} =: \underline{\Theta} > 0$.

If we fix $n_k = n > 0$, then the quantity $B_K$ in Theorem~\ref{th:VrAFBS4NI_convergence} becomes 
\begin{equation*} 
\arraycolsep=0.2em
\begin{array}{lcl}
B_K & := & \frac{\Lambda}{\beta}\sum_{k=0}^K  t_{k-1}(t_{k-1}-1) \frac{\sigma_k^2}{\Theta_k} = \frac{5\sigma^2}{2\mu n \epsilon^{\omega}}\sum_{k=0}^n\mbf{p}_k t_{k-1}(t_{k-1}-1) \vspace{1ex}\\
& = & \frac{5\sigma^2}{2 n \epsilon^{\omega}}\sum_{k=0}^K\big(\epsilon^{\omega} + \frac{2\mu}{\mu(k+r-1)-1}\big)(k+r-1)[\mu(k+r-1)-1] \vspace{1ex}\\
& \leq & \frac{5\mu\sigma^2}{2n}\big(K + 2\epsilon^{-\omega}\big)(K+r-1)^2.
\end{array}
\end{equation*}
Moreover, we also have $E_0^2 = \frac{\mu r^2(\mu\Gamma_0 + \beta)\sigma^2}{2\beta n}$ and $\Gamma_0\mu + \beta = \frac{5\beta}{\epsilon^{\omega}} + \beta \leq \frac{6\beta}{\epsilon^{\omega}}$.
Using these bounds, from \eqref{eq:VrAFBS4NI_th1_convergence1}, to guarantee  $\Expn{ \norms{G_{\lambda}x^{K}}^2} \leq \epsilon^2$, we impose
\begin{equation*} 
\arraycolsep=0.2em
\begin{array}{lcl}
\Expn{ \norms{G_{\lambda}x^K}^2}   & \leq &   \frac{2( \Psi_0^2 + B_{K-1})}{\mu^2(K+r-1)^2} + \frac{\mu r^2(\Gamma_0\mu + \beta)}{ \mu^2 n \beta(K + r - 1)^2} \sigma^2  \vspace{1ex}\\
& \leq &  \frac{2\Psi_0^2}{\mu^2(K+r-1)^2} + \frac{3 \sigma^2}{\mu n}\big(K + 2\epsilon^{-\omega}\big) + \frac{6 r^2\sigma^2}{\mu (K + r - 1)^2\epsilon^{\omega}n} \vspace{1ex}\\
& \leq & \epsilon^2.
\end{array}
\end{equation*}
This condition holds if 
\begin{equation*} 
\arraycolsep=0.2em
\begin{array}{lcl}
\frac{2\Psi_0^2}{\mu^2(K+r-1)^2} \leq \frac{\epsilon^2}{4}, \quad \frac{3\sigma^2 K }{\mu  n} \leq  \frac{\epsilon^2}{4}, \quad  \frac{6\sigma^2  }{\mu \beta n\epsilon^{\omega}} \leq  \frac{\epsilon^2}{4}, \quad \textrm{and} \quad \frac{6 r^2 \sigma^2}{\mu (K + r - 1)^2\epsilon^{\omega}n} \leq \frac{\epsilon^2}{4}.
\end{array}
\end{equation*}
These four conditions are simultaneously satisfied  if we choose 
\begin{equation*} 
\arraycolsep=0.2em
\begin{array}{lcl}
K :=  \big\lfloor \frac{2\sqrt{2}\Psi_0}{\mu \epsilon} - 1 \big\rfloor \quad\textrm{and} \quad n \geq \sigma^2 \max\set{  \frac{24\sqrt{2}\Psi_0}{\mu^2\epsilon^3}, \frac{24}{\mu\beta\epsilon^{2+\omega}}, \frac{3 r^2\mu}{\Psi_0^2\epsilon^{\omega}}}.  
\end{array}
\end{equation*}
Thus, we can choose $\omega = 1$ and $n := \frac{c_n}{\epsilon^3} = \BigOs{\frac{1}{\epsilon^3}}$ for given $c_n := 24\sigma^2 \max\set{  \frac{\sqrt{2}\Psi_0}{\mu^2}, \frac{1}{\mu\beta} }$.

Since we fix $n_k = n > 0$ and $b_k = b > 0$, the expected total number of oracle calls becomes $\Expn{ \Tc_K}  = n + \Expn{ \hat{\Tc}_K }$, where
\begin{equation*}
\arraycolsep=0.2em
\begin{array}{lcl}
\Expn{\hat{\Tc}_K} & = & \sum_{k=0}^K[ n \mbf{p}_k + 2b(1 - \mbf{p}_k)] = n\sum_{k=0}^K\mbf{p}_k  + 2b \sum_{k=0}^K(1 - \mbf{p}_k) \vspace{1ex}\\
& = & 2b (K + 1) + (n - 2b)\sum_{k=0}^{K} \mbf{p}_k \vspace{1ex}\\
& \leq & 2 b (K + 1) + n \big[ \sum_{k=0}^{K} \epsilon  + 2 \sum_{k=0}^{K}\frac{\mu}{\mu(k+r-1)-1} \big] \vspace{1ex}\\
& \leq &  \frac{4c_b\sqrt{2}\Psi_0 }{\mu \epsilon^2} + \frac{c_n}{\epsilon^{3}} \big[  \epsilon (K + 1) + 2\ln(K+r) \big] \vspace{1ex}\\
& \leq & \frac{4c_b\sqrt{2}\Psi_0 }{\mu \epsilon^2} + \frac{c_n}{ \epsilon^{3}} \big[ \frac{2\sqrt{2}\Psi_0}{\mu} + 2\ln\big(\frac{3\sqrt{2}\Psi_0}{\mu \epsilon}\big) \big].
\end{array}
\end{equation*}
Therefore, we obtain $\Expn{\hat{\Tc}_K } = \BigOs{ \epsilon^{-2} +  \epsilon^{-3} +  \epsilon^{-3}\ln(\epsilon^{-1})} = \widetilde{\mcal{O}}(\epsilon^{-3})$.
This inequality shows that the expected total number of oracle calls $\Fb(\cdot, \xi)$ and $J_{\lambda T}$ evaluations of \eqref{eq:VrAFBS4NI}  is at most $\Expn{\Tc_K} :=  \BigOs{ \epsilon^{-2} + \epsilon^{-3} +  \epsilon^{-3}\ln(\epsilon^{-1})}$ to achieve $\Expn{ \norms{G_{\lambda}x^{K}}^2} \leq \epsilon^2$.
\end{proof}
%%% End of Proof.

%%%%%%%%%%%%%%%%%%%%%%%%%%%%%%%%%%%%%%%%
%%% Proof of Corollary 4.4.
\beforesubsubsec
\subsubsection{The proof of Corollary~\ref{co:HSGD_complexity_Esetting} --- The HSGD variant of \ref{eq:VrAFBS4NI}}\label{apdx:co:HSGD_complexity_Esetting}
\aftersubsubsec
\begin{proof}
Since \eqref{eq:HSGD_estimator} is used for $Fx^k$, by Lemma~\ref{le:HSGD_estimator_bound}, we have
\begin{equation*}
\arraycolsep=0.2em
\begin{array}{lcl}
\kappa_k :=  1 - (1-\tau_k)^2, \quad \Theta_k := \frac{2(1-\tau_k)^2}{b_k},  \quad \textrm{and} \quad \sigma_k^2 := \frac{2\tau_k^2\sigma^2 }{ \hat{ b}_k}.
\end{array}
\end{equation*}
Now, let us choose $\Gamma_k = \frac{5\beta t_k(t_k-1)}{\theta \mu t_{k-1}(t_{k-1}-1)} \geq \frac{5\beta}{\theta\mu} =: \underline{\Gamma} > 0$ for all $k\geq 0$ and some $\theta \in (0, 1)$.
Then, the first condition of \eqref{eq:VrAFBS4NI_param_cond} in Theorem~\ref{th:VrAFBS4NI_convergence} is equivalent to
\begin{equation*}
\arraycolsep=0.2em
\begin{array}{lcl}
\kappa_k  =  1 - (1-\tau_k)^2  & \geq &  1 - \frac{\Gamma_{k-1} t_{k-2}(t_{k-2}-1)}{\Gamma_k t_{k-1}(t_{k-1}-1)} + \frac{5\beta}{\mu \Gamma_k }  =   1 - \frac{(1-\theta) t_{k-1}(t_{k-1}-1)}{t_k(t_k-1)}.
\end{array}
\end{equation*}
This condition is equivalent to 
\begin{equation*}
\arraycolsep=0.2em
\begin{array}{lcl}
\frac{(1-\theta)t_{k-1}(t_{k-1}-1)}{t_k(t_k-1)} \geq  (1-\tau_k)^2.
\end{array}
\end{equation*}
If we choose $\tau_k := 1 - \sqrt{\frac{(1-\theta)t_{k-1}(t_{k-1}-1)}{t_k(t_k-1)}}$ as stated in Corollary~\ref{co:HSGD_complexity_Esetting}, then the last condition holds.
Consequently, the  first condition of \eqref{eq:VrAFBS4NI_param_cond} in Theorem~\ref{th:VrAFBS4NI_convergence} holds.

Since $\Theta_k = \frac{2(1-\tau_k)^2}{b_k} = \frac{2(1-\theta)t_{k-1}(t_{k-1} - 1)}{b_kt_k(t_k-1)}$, we have  $\Gamma_k\Theta_k = \frac{10\beta(1-\theta)}{\mu\theta b_k}$.
The second condition of \eqref{eq:VrAFBS4NI_param_cond} in Theorem~\ref{th:VrAFBS4NI_convergence} and the condition \eqref{eq:VrAFBS4NI_param_cond_new} in Theorem~\ref{th:VrAFBS4NI_o_rates_convergence3} both hold if $ \frac{10\beta(1-\theta)}{\mu\theta b_k} \leq \Lambda$.
Therefore, if we choose $b_k = b \geq \frac{10\beta}{\mu\Lambda\theta}$, then the last condition holds.
Consequently, the second condition of \eqref{eq:VrAFBS4NI_param_cond} and \eqref{eq:VrAFBS4NI_param_cond_new} are both satisfied.
Moreover, since $\tau_k \geq 1 - \sqrt{1-\theta}$, $b_k = b > 0$ and $\hat{b}_k = \hat{b} > 0$ for all $k\geq 0$, we have $\Theta_k \geq \underline{\Theta} := \frac{\hat{c} (1-\theta)}{b} > 0$ for a given $\hat{c} > 0$.

Now, since $t_k = t_{k-1} + \mu$, if $t_{k-1} \geq \frac{1}{1 - \mu^2}$ (which holds if $r \geq 1 + \frac{1}{\mu(1 -\mu^2)}$), then we have $\sqrt{t_k(t_k - 1)}  \leq \sqrt{t_{k-1}(t_{k-1}- 1)} + 1$.
Therefore, we get
\begin{equation*} 
\arraycolsep=0.2em
\begin{array}{lcl}
\Tc_{[1]} &:= & \big[ \sqrt{t_k(t_k-1)} - \sqrt{(1-\theta)t_{k-1}(t_{k-1}-1)} \big]^2 \vspace{1ex}\\
& \leq & \big[ (1-\sqrt{1-\theta})\sqrt{t_{k-1}(t_{k-1}- 1)}  + 1\big]^2 \vspace{1ex}\\
& \leq & \big( \theta t_{k - 1} + 1 \big)^2 = \mu^2\theta^2(k+r-1)^2 + 2\mu\theta(k+r-1) + 1.
\end{array}
\end{equation*}
Since we choose $r \geq 5 + \frac{1}{\mu}$, we also have $r \geq 1 + \frac{1}{\mu(1-\mu^2)}$ and $\Gamma_0 \leq  \frac{5c_0\beta}{\mu\theta}$ for a constant $c_0 > 0$.

In this case, we can bound the quantity $B_K$ in Theorem~\ref{th:VrAFBS4NI_convergence} as follows:
\begin{equation*} 
\arraycolsep=0.2em
\begin{array}{lcl}
B_K & := & \frac{\Lambda}{\beta}\sum_{k=0}^K t_{k-1}(t_{k-1}-1) \frac{  \sigma_k^2}{\Theta_k} \vspace{1ex}\\
& = & \frac{\Lambda\sigma^2 b}{\beta(1-\theta)\hat{b} } \sum_{k=0}^K t_k(t_k-1)\Big(1 - \sqrt{\frac{(1-\theta)t_{k-1}(t_{k-1}-1)}{t_k(t_k-1)}} \Big)^2 \vspace{1ex}\\
& = &  \frac{\Lambda\sigma^2 b}{\beta(1-\theta)\hat{b} } \sum_{k=0}^K \big[ \sqrt{t_k(t_k-1)} - \sqrt{(1-\theta)t_{k-1}(t_{k-1}-1)} \big]^2 \vspace{1ex}\\
& \leq &  \frac{\Lambda\sigma^2 b}{\beta(1-\theta)\hat{b} } \sum_{k=0}^K \big[ \mu^2\theta^2(k+r-1)^2 + 2\mu\theta(k+r-1) + 1 \big] \vspace{1ex}\\
& \leq &    \frac{\Lambda\sigma^2 b}{\beta(1-\theta)\hat{b} } \big[ \mu^2\theta^2 K(K+r-1)^2 + 2\mu\theta(K+r-1)^2 + (K+1) \big].
\end{array}
\end{equation*}
In addition, we  also have $E_0^2 = \frac{\mu r^2(\mu\Gamma_0 + \beta)\sigma^2}{2\beta n_0}$.
Using these bounds and  $\Gamma_0 \leq \frac{5c_0\beta}{\mu\theta}$, to guarantee  $\Expn{ \norms{G_{\lambda}x^{K}}^2} \leq \epsilon^2$ for any $\epsilon \in (0, 1/2]$, from \eqref{eq:VrAFBS4NI_th1_convergence1}, we impose
\begin{equation*} 
\arraycolsep=0.2em
\begin{array}{lcl}
\Expn{ \norms{G_{\lambda}x^K}^2}   & \leq &   \frac{2( \Psi_0^2 + B_{K-1})}{\mu^2(K+r-1)^2} + \frac{\mu r^2(\Gamma_0\mu + \beta) }{\mu^2n_0 \beta(K + r - 1)^2}\sigma^2 \vspace{1ex}\\
& \leq &  \frac{2\Psi_0^2}{\mu^2(K+r-1)^2} +   \frac{2\Lambda\sigma^2 b}{\mu^2\beta(1-\theta)\hat{b} } \big(\mu^2\theta^2K + 2\mu \theta + \frac{1}{K}\big) + \frac{ r^2\sigma^2(5c_0 + \theta)}{(K + r - 1)^2\mu \theta n_0} \vspace{1ex}\\
& \leq & \epsilon^2.
\end{array}
\end{equation*}
If we assume that $\theta \in (0, 1/2]$ and $\theta \leq c_0$, then the last condition holds if 
\begin{equation*} 
\arraycolsep=0.2em
\begin{array}{lcl}
\frac{2\Psi_0^2}{\mu^2(K+r-1)^2} \leq \frac{\epsilon^2}{5}, \ \ \frac{4\Lambda \sigma^2 \theta^2 b K }{\beta \hat{b}} \leq  \frac{\epsilon^2}{5}, \ \ \frac{8\Lambda\theta\sigma^2b}{\mu\beta \hat{b}} \leq  \frac{\epsilon^2}{5}, \quad \frac{4\Lambda\sigma^2 b}{\mu^2\beta \hat{b} K} \leq  \frac{\epsilon^2}{5}, \ \  \textrm{and} \ \ \frac{6c_0r^2 \sigma^2}{(K + r - 1)^2\mu \theta n_0 } \leq \frac{\epsilon^2}{5}.
\end{array}
\end{equation*}
These five conditions are simultaneously satisfied  if we choose 
\begin{equation*} 
\arraycolsep=0.2em
\begin{array}{lcl}
K :=   \big\lfloor \frac{\sqrt{10}\Psi_0}{\mu \epsilon} - 1 \big\rfloor, \quad \theta := \epsilon, \quad n_0 :=  \frac{3c_0\mu r^2 \sigma^2}{\Psi_0^2\epsilon}, \quad\textrm{and} \quad \frac{\hat{b}}{b} \geq 2\sigma^2 \max\set{ \frac{10\sqrt{10}\Lambda\Psi_0}{\mu\beta\epsilon}, \frac{20\Lambda\sigma^2 }{\mu\beta\epsilon },  \frac{\sqrt{10}\Lambda\sigma^2}{\mu\beta\Psi_0\epsilon}}.  
\end{array}
\end{equation*}
Combining this condition and $b_k = b \geq \frac{3\beta}{\mu\Lambda\theta} = \frac{3\beta}{\mu\Lambda\epsilon}$ above, we conclude that $b_k = b := \big\lfloor \frac{c_b}{\epsilon} \big\rfloor$ for $c_b := \frac{3\beta}{\mu\Lambda}$ and 
$\hat{b}_k = \hat{b} := \lfloor \frac{ \hat{c}_b}{\epsilon^2} \rfloor$ for $\hat{c}_b :=  \frac{6\sqrt{10} \sigma^2}{\mu^2} \max\set{ 10 \Psi_0,  2\sqrt{10},  \frac{1}{\Psi_0}  }$ and all $k\geq 1$, and $n_0 := \big\lfloor \frac{c_n}{\epsilon}\big\rfloor$ for $c_n :=  \frac{3c_0\mu r^2 \sigma^2}{\Psi_0^2}$.
Since $\epsilon \in (0, 1/2]$ and $\theta = \epsilon$, we also have $\theta \in (0, 1/2]$.

Finally, since $b_k = b > 0$ and $\hat{b}_k = \hat{b} > 0$ for all $k\geq 0$, the expected total number of oracle calls is $\Expn{ \Tc_K}  = \Expn{ \Tc_0 } + \Expn{ \hat{\Tc}_k }$, where $\hat{\Tc}_k$ is
\vspace{-0.5ex}
\begin{equation*}
\arraycolsep=0.2em
\begin{array}{lcl}
\Expn{\hat{\Tc}_K}  = & (2b + \hat{b})(K + 1) = \big( \frac{2c_b}{\epsilon} + \frac{\hat{c}_b}{\epsilon^2 }\big)\frac{\sqrt{10}\Psi_0}{\mu \epsilon} \leq  \frac{(2c_b + \hat{c}_b) \sqrt{10}\Psi_0}{\mu \epsilon^3},
\end{array}
\vspace{-0.5ex}
\end{equation*}
and $\Expn{\Tc_0 } = 2b + n_0 =  \BigOs{\epsilon^{-1}}$ is the expected total number  of oracle calls for evaluating $\widetilde{F}^0$.
Overall, the expected total number of oracle calls $\Fb(\cdot, \xi)$ and $J_{\lambda T}$ evaluations of \eqref{eq:VrAFBS4NI}  is at most $\Expn{ \Tc_K } :=  \BigOs{ \epsilon^{-1} + \epsilon^{-3} }$ to achieve $\Expn{ \norms{G_{\lambda}x^{K}}^2} \leq \epsilon^2$.
\end{proof}
%%% End of Proof.

%%%%%%%%%%%%%%%%%%%%%%%%%%%%%%%%%%%%%%%%%%%%%%%%%%%%%%
%%%% E. The Convergence of BFS Method with Variance Reduction.
%%%%%%%%%%%%%%%%%%%%%%%%%%%%%%%%%%%%%%%%%%%%%%%%%%%%%%
\beforesec
\section{The Convergence Analysis of Algorithm~\ref{alg:A2}}\label{apdx:sec:convergence_of_VrABFS4NI_method}
\aftersec
This appendix provides the full proofs of the results in Section~\ref{sec:VrABFS_method}.

%%%%%%%%%%%%%%%%%%%%%%%%%%%%%%%%%%%%%%%%
\vspace{-1ex}
\beforesubsec
\subsection{The proof of Theorem~\ref{th:VrABFS4NI_O_rates} --- Key estimates}\label{apdx:th:VrABFS4NI_O_rates}
\aftersubsec
\begin{proof}%%%%%
To analyze the convergence of \eqref{eq:VrABFS4NI}, we construct the following functions:
\vspace{-0.5ex}
\begin{equation*}%\label{eq:VrABFS_Lyapunov_func}
\arraycolsep=0.2em
\begin{array}{lcl}
\widehat{\Lc}_k & := &  \beta a_k \norms{S_{\lambda}u^k}^2 + t_{k-1} \iprods{S_{\lambda}u^k, u^k - s^k} + \frac{c_k}{2\nu\beta}\norms{s^k - u^{\star}}^2, \vspace{1ex}\\
\widehat{\Qc}_k &:= & \widehat{\Lc}_k + \big[ \bar{\beta} - (1+s) \beta  \big] t_{k-1} (t_{k-1} -1) \norms{S_{\lambda}u^k - S_{\lambda}u^{k-1} }^2 \vspace{1ex}\\
&& + {~} \Lambda L t_{k-1}(t_{k-1}-1) \iprods{Fx^k - Fx^{k-1}, x^k - x^{k-1}}, \vspace{1ex}\\
\widehat{\Pc}_k &:= & \widehat{\Qc}_k + \frac{[\mu(1 - \kappa_k)\Gamma_k + \beta] t_{k-1}(t_{k-1}-1) }{2 \mu }\Delta_{k-1},
\end{array}
\vspace{-0.5ex}
\end{equation*}
where $u^{\star} \in \zer{S_{\lambda}}$, $a_k := t_{k-1}[ t_{k-1} - 1 - s(1-\nu)]$,  $\mu \in (0, 1]$, $s > 0$, $c_k$ is given in \eqref{eq:VrAFBS_coefficients}, and $\Gamma_k > 0$ are given parameters.
The quantity $\Delta_k$ and the parameter $\kappa_k$ are given in Definition~\ref{de:VR_Estimators}.
The functions $\widehat{\Lc}_k$, $\widehat{\Qc}_k$, and $\widehat{\Pc}_k$ are similar to the ones $\Lc_k$, $\Qc_k$, and $\Pc_k$ in \eqref{eq:VrAFBS_Lyapunov_func}, respectively.

In this case, $S_{\lambda}u$ and $(u^k, s^k)$ play the same role as $G_{\lambda}x$ and $(x^k, z^k)$ in Section~\ref{sec:VR_AFBS_method} and the results of Theorem~\ref{th:VrAFBS4NI_convergence} still hold for $S_{\lambda}u$ and $\sets{(u^k, s^k)}$.
Using the relation $\norms{Fx^k + \xi^k} = \norms{ F(J_{\lambda T}u^k) + \lambda^{-1}(u^k - J_{\lambda T}u^k)} = \norms{S_{\lambda}u^k}$, we obtain \eqref{eq:VrABFS4NI_th1_convergence1} from \eqref{eq:VrAFBS4NI_th1_convergence1}.
The estimates in \eqref{eq:VrABFS4NI_th1_summable_bound1} follow from \eqref{eq:VrAFBS4NI_th1_summable_bound1} and the relation $\norms{Fx^k + \xi^k} = \norms{S_{\lambda}u^k}$.
Finally, since $u \in J_{\lambda T}u + \lambda T(J_{\lambda T}u)$, it is obvious to check that $\xi^k := \frac{1}{\lambda}(u^k - J_{\lambda T}u^k) = \frac{1}{\lambda}(u^k - x^k) \in T(J_{\lambda T}u^k) =  Tx^k$.
\end{proof}
%%% End of Proof.

%%%%%%%%%%%%%%%%%%%%%%%%%%%%%%%%%%%%%%%%
%%%% The proof of Theorem 19.
\vspace{-3ex}
\beforesubsec
\subsection{The proof of Theorem~\ref{th:VrABFS4NI_o_rates}  --- The $\BigOs{1/k^2}$ and $\SmallOs{1/k^2}$-rates}\label{apdx:th:VrABFS4NI_o_rates}
\aftersubsec
\begin{proof}
The conclusions of Theorem \ref{th:VrABFS4NI_o_rates} follow from the results of Theorem~\ref{th:VrAFBS4NI_o_rates_convergence2}, respectively and the relations $\norms{x^{k+1} - x^k} = \norms{J_{\lambda T}u^{k+1} - J_{\lambda T}u^k} \leq \norms{u^{k+1} - u^k}$ and $\norms{Fx^k + \xi^k} = \norms{S_{\lambda}u^k}$ for $\xi^k := \frac{1}{\lambda}(u^k - x^k) \in Tx^k$.
We omit the details to avoid a repetition.
\end{proof}
%%% End of Proof.

%%%%%%%%%%%%%%%%%%%%%%%%%%%%%%%%%%%%%%%%
%%%% The proof of Theorem 19.
\vspace{-3ex}
\beforesubsec
\subsection{The proof of Theorem~\ref{th:VrABFS4NI_almost_sure_convergence} --- Almost sure convergence}\label{apdx:th:VrABFS4NI_almost_sure_convergence}
\aftersubsec
\begin{proof}
First, the $\SmallOs{1/k^2}$ almost sure  convergence rates also follow from those in Theorem~\ref{th:VrAFBS4NI_o_rates_convergence3}, respectively and the relations $\norms{x^{k+1} - x^k} = \norms{J_{\lambda T}u^{k+1} - J_{\lambda T}u^k} \leq \norms{u^{k+1} - u^k}$ and $\norms{Fx^k + \xi^k} = \norms{S_{\lambda}u^k}$ for $\xi^k := \frac{1}{\lambda}(u^k - x^k) \in Tx^k$.

Next, we have  $\lim_{k\to\infty}\norms{u^k - u^{\star}}$ exists almost surely for all $u^{\star} \in \zer{S_{\lambda}}$ as in Theorem~\ref{th:VrAFBS4NI_o_rates_convergence3}.
Moreover, we also have $\lim_{k\to\infty}\norms{S_{\lambda}u^k} = 0$ almost surely as in  Theorem~\ref{th:VrAFBS4NI_o_rates_convergence3}.
Applying Lemma~\ref{le:A3_lemma} again, we conclude that $\sets{u^k}$ almost surely converges to  a $ \zer{S_{\lambda}}$-valued random variable $\bar{u}^{\star} \in \zer{S_{\lambda}}$.

Finally, since $x^k = J_{\lambda T}u^k$, $x^{\star} = J_{\lambda T}u^{\star} \in \zer{\Phi}$ almost surely, and $J_{\lambda T}$ is continuous (since it is nonexpansive), it is easy to show that $\sets{x^k}$ also converges to $x^{\star}$ almost surely by the classical Continuous Mapping Theorem \citep[Exercise 1.3.3]{durrett2019probability}.
\end{proof}
%%% End of Proof.

\bibliographystyle{plain}
\itemsep=-0.1em
%\bibliography{/Users/quoc/Dropbox/E-Books/tran_bibtex_new}

\end{document}